\documentclass[makeidx,12pt]{book} 
\usepackage{graphicx}
\usepackage{mathrsfs, amsthm, amsmath, amssymb}
\usepackage{appendix}
\usepackage{bm, hyperref}
\usepackage{cite}
\usepackage{url}
\usepackage{epsfig} 
\usepackage{euscript}
\usepackage{amsfonts}
\usepackage{dsfont}
\usepackage{pict2e}
\usepackage{multicol}
\usepackage{lscape}
\usepackage{cancel}
 
\usepackage{calligra}
\usepackage[T1]{fontenc}

\usepackage{pgfplots}
\pgfplotsset{compat=1.17}

\usepackage{mathtools} 

\usepackage[margin=4cm]{geometry}

\usepackage[margin=1.5cm]{caption} 
\usepackage[labelfont=bf]{caption} 

\numberwithin{figure}{section}

\usepackage{tikz}
\usetikzlibrary{arrows.meta}

\usepackage{xcolor}
\definecolor{dark-red}{rgb}{0.65,0.15,0.15}
\definecolor{dark-blue}{rgb}{0.15,0.15,0.65}
\definecolor{medium-blue}{rgb}{0.1,0.1,0.9}
\hypersetup{
    colorlinks, linkcolor={medium-blue},
    citecolor={medium-blue}, urlcolor={medium-blue}
}

\newtheorem{theorem}{Theorem}[section] 

\newtheorem{lemma}[theorem]{Lemma}
\newtheorem{corollary}[theorem]{Corollary}

\newtheorem{proposition}[theorem]{Proposition}

\theoremstyle{definition}
\newtheorem{definition}[theorem]{Definition}

\newtheorem{example}[theorem]{Example}
\newtheorem{xca}[theorem]{Exercise}

\newtheorem{axiom}[theorem]{Axiom}
\newtheorem{prob}[theorem]{Problem}
 
\theoremstyle{remark}
\newtheorem{remark}[theorem]{Remark}

\newtheorem{notation}[theorem]{Notation}

\newtheorem{scratch}[theorem]{Scratch Work}

\numberwithin{equation}{section} 
\numberwithin{section}{chapter} 
\numberwithin{subsection}{section} 

\newcommand{\vs}{\vspace{1cm}}
\newcommand{\s}{\vspace{3mm}}

\newcommand{\N}{\mathbb{N}}
\newcommand{\Q}{\mathbb{Q}}
\newcommand{\R}{\mathbb{R}}
\newcommand{\Z}{\mathbb{Z}}

\newcommand{\bfa}{\mathbf{a}} 
\newcommand{\bfb}{\mathbf{b}}
\newcommand{\bfc}{\mathbf{c}}

\newcommand{\bfe}{\mathbf{e}}

\newcommand{\bfp}{\mathbf{p}}

\newcommand{\bft}{\mathbf{t}}
\newcommand{\bfu}{\mathbf{u}} 
\newcommand{\bfv}{\mathbf{v}} 
\newcommand{\bfw}{\mathbf{w}} 
\newcommand{\bfx}{\mathbf{x}}
\newcommand{\bfy}{\mathbf{y}}
\newcommand{\bfz}{\mathbf{z}}

\DeclareMathOperator{\acl}{acl} 
\DeclareMathOperator{\awf}{awf} 
\DeclareMathOperator*{\Coda}{Coda} 
\DeclareMathOperator*{\Slim}{Slim} 

\makeindex

\begin{document}

\begin{titlepage}
\begin{center}
 {\huge\bfseries Arbitrarily Close\\}
 \vspace{1.5cm}
 {\Large John~A.~Rock}\\[5pt]
 jarock@cpp.edu\\[14pt]
 \vspace{1.5cm}
\begin{tikzpicture} 
\draw[dashed, fill=blue!15] (-3,-1.41) rectangle (1.41,1.41);
	\draw (-0.8,0) node {$B$};
	\draw (-3.02,1.42) node {$\bullet$};
	\draw (-3.02,-1.42) node {$\bullet$};
	\draw (1.44,-1.42) node {$\bullet$};		
\begin{scope}[thick, dashed, red] 
\draw (1.45,1.43) circle (2.1cm); 
\draw (1.45,1.43) circle (1.4cm); 
\draw (1.45,1.43) circle (0.7cm); 
\end{scope}	
\draw[semithick] (-3,-1.41) -- (-3,1.41);
\draw[semithick] (-3,-1.41) -- (1.41,-1.41);
	\draw[fill=white] (1.41,1.41) circle (0.08cm);
	\draw (1.8,1.41) node {$\bfy$};
\end{tikzpicture}
\vfill
\end{center}
\end{titlepage}













\mainmatter

\tableofcontents   


\setcounter{page}{5}

\chapter{Preface}
\label{ch:preface}

Who are mathematicians? First and foremost, you are. In my opinion, anyone who is willing to read a book on real analysis is automatically a mathematician.

But your opinion matters more than mine.

\section{Communities and resources}
\label{sec:communitiesandresources}

Mathematicians come from all over and from minoritized groups of all kinds. Please take a look at some of the following resources to help you find mathematicians whose lives and experiences might reflect your own.

\begin{itemize}
	\item \href{https://www.maa.org/press/ebooks/living-proof-stories-of-resilience-along-the-mathematical-journey-2}{Living Proof: Stories of Resilience Along the Mathematical Journey}:\\
	``The stories in Living Proof are intended to provide support and inspiration for mathematics students experiencing struggle and despair. If students keep working, if they keep seeking, they’ll be rewarded by serendipity, which is really just, as these stories remind us, the habit of mind to be engaged and to notice when something good has happened.''
	\item \href{https://www.mathad.com/home}{Mathematicians of the African Diaspora (MATHAD)}:\\
	``MATHAD is dedicated to promoting and highlighting the contributions of members of the African diaspora to mathematics, especially contributions to current mathematical research.''
	\item \href{https://www.meetamathematician.com/home}{M$\Sigma\Sigma$T a Mathematician}:\\
	``The mission of M$\Sigma\Sigma$T a Mathematician is to share stories of mathematicians from different backgrounds, especially from historically excluded groups, with the aim of introducing students to role models and fostering a sense of community.''
	\item \href{https://www.meetamathematician.com/words-of-wisdom}{M$\Sigma\Sigma$T a Mathematician - Words of Wisdom}: (Short videos)
	\item \href{https://mathematicallygiftedandblack.com/}{Mathematically Gifted and Black}: \\
	`` `To be young, gifted and black'; a common phrase used in the Black community. It comes from a song famously sung by Ms. Nina Simone and co-written by Weldon Irvine. To me, the first few lines of the song always felt like a love song written to those of us in the Black community: `To be young, gifted and black/Oh what a lovely precious dream/To be young, gifted and black,/Open your heart to what I mean.' ''
	\item \href{https://www.lathisms.org/}{Lathisms}:\\ 
	``A vibrant, inclusive, and diverse Mathematical community where the Latinx and Hispanic culture is valued, cultivated, and celebrated.''
	\item \href{https://indigenousmathematicians.org/}{Indigenous Mathematicians}: \\
	``Our digital hale (house) is our space to connect, network, and inspire the next generation of Indigenous Mathematicians! Together we can share, narrate, and tell our stories on our journeys in math.''	
	\item \href{https://www.hiddennorms.com/}{Hidden NORMS}: (Monthly webinar) \\
	``Brings together a star-studded lineup of speakers and panelists for academic and professional success for you the student doing this `math thing'.''
	\item \href{https://awm-math.org/}{Association for Women in Mathematics}: \\
	``Since its founding in 1971 by a small but passionate group of women mathematicians, the Association for Women in Mathematics (AWM) has grown into a leading society for women in the mathematical sciences, and is one of the societies comprising the Conference Board of the Mathematical Sciences.''
	\item \href{http://lgbtmath.org/}{Spectra: the Association for LGBTQ+ Mathematicians}: \\
	``Welcome to the website for LGBTQ+ mathematicians and their allies. This arose from a need for recognition and community for Gender and Sexual Minority mathematicians, and we hope that this will be a resource for our community. Check back as our organization grows and evolves!''
	\item \href{https://gradsubgroups.org/}{SUBgroups}:\\ 
	``Online peer groups for first-year math grad students.''
\end{itemize}

\section{Target audiences}
\label{sec:targetaudiences}

This is not a novel, it's a resource designed with lots of potential readers in mind. Please make use of it as you see fit. To be clear, I assume you have some familiarity with inequalities, basic algebra, trigonometry, some set theory, and writing mathematical proofs.

For readers who have never taken a course on real analysis, I suggest you take your time, read as much as you can, and take notes as you go. Lots of details and figures are included to help you along the way. 

For readers with more experience including instructors, please skip around as you see fit. However, I recommend you do not skip Sections \ref{sec:acl}, \ref{sec:completeorderedfield}, and \ref{sec:arbitrarilycloseineuclideanspaces}. These are vital to the development of the entire book since they are where the kernel of analysis---{\em arbitrarily close}---is defined for the real line $\R$ and for Euclidean spaces $\R^m$, respectively.

In my experience, there is more content provided here than can be squeezed into a single semester. So again, please make use of this textbook as you see fit. 

To all my readers, please let me know what you think! If you spot any errors or have recommendations for ways to improve any aspect of what you find here, please reach out to me via email: jarock@cpp.edu.

\section{Walkthroughs}
\label{sec:walkthroughs}

You are encouraged to do {\em walkthroughs} for the statements and proofs of every Definition, Theorem, Corollary, etc., found in each section. Consider some or all of these activities, in no particular order: 
\begin{itemize}
\item Rewrite statements in your own words.
\item Come up with examples and nonexamples.
\item Draw figures to accompany or supplement the given material.
\item Create activities using Desmos, GeoGebra, WolframAlpha, or other freely available apps and websites.
\item Come up with scratch work that may or may not fit a given proof.
\item Write proofs that are more thorough.
\item Write proofs that are more concise.
\item Do any activity you can think of that will help you understand what's going on.
\end{itemize}

Ultimately, walkthroughs are what you make of them. Do what makes sense for you. I believe the more you supplement your reading with walkthroughs, whatever a walkthrough means to you, the more the mathematics will make sense and the more beautiful it will become.

\section{Motivation and goals}
\label{sec:motivationandgoals}

A goal of this book is to provide a development of the fundamental results in calculus and analysis based on a concrete notion of what it means for a point to be {\em arbitrarily close} to a set. See Definitions \ref{def:aclreal} and \ref{def:acl}. Along the way, I strive to include lots of details and motivation for the way definitions are made and the way results are proven. 

The simple question that drove the development of this book is: 
\begin{quote}
How close can a point be to a set?
\end{quote}
The answer I've come up with---which has redefined my perspective on analysis---is: 
\begin{quote}
So close there's no distance between them:\\
\textit{arbitrarily close.}
\end{quote}
Once I defined {\em arbitrarily close} in terms of {\em comparing a point to a set}, as opposed to comparing one point and another, I felt like I was really onto to something.

As an overview, here are some of the concepts explored in this book:

\begin{itemize}
	\item The supremum of a set of real numbers is the upper bound arbitrarily close to the set (Definition \ref{def:supremumacl}).
	\item Zero is the only real number arbitrarily close to the sets of positive and negative real numbers (Lemma \ref{lem:zeroacl}).
	\item A limit of a sequence is arbitrarily close to its sequence (Theorem \ref{thm:exercise0}). In fact, the precise connection between the concepts of limits for sequences and points arbitrarily close to the range of a sequence is quite deep and forms a foundation for many of the results presented in the book .
	\item A closed set comprises the points arbitrarily close to the set. See Definition \ref{def:closureclosed} and Theorem \ref{thm:closedopen}.
	\item A function preserves closeness at a point  if and only if it is continuous at the point. See Definitions \ref{def:preservecloseness} and \ref{def:continuity} as well as Theorem \ref{thm:continuityequivalence}.
\end{itemize}

Further connections to arbitrarily close that are currently under development include:
\begin{itemize}
	\item A limit of a function is arbitrarily close to the range of the function. 
	\item A derivative is arbitrarily close to the set of difference quotients. 
	\item An integral is arbitrarily close to sets of sums of areas of rectangles. 
	\item In certain settings, a continuous function is arbitrarily close the set of polynomials. 
	\item The sum of a series is arbitrarily close to the set of partial sums. 
\end{itemize}

The setting is limited to Euclidean spaces $\R^m$ in general, and the real line $\R=\R^1$ when the restriction is warranted somehow. Results stated in terms of $\R^m$ hold for $\R$ (as well as complete metric spaces in most cases). The intuition of distance between points provided by metric spaces is at the heart of the concept of arbitrarily close, but the definition itself is topological in nature. See Definitions \ref{def:aclreal} and \ref{def:acl}.

Chapter \ref{ch:kernelofanalysis} develops the main topic: A formal definition for the phrase {\em arbitrarily close}.


\section{Scratch Work}
\label{sec:scratchwork}

Throughout the book, scratch work accompanies examples, theorems, and other results that come with proofs. I don't really have a specific structure in mind when it comes to scratch work, and that's the point. Mathematics can be messy before we figure out nice ways of explaining what we have in mind. I think we should embrace the mess with a playful attitude and allow ourselves to explore ideas before trying to write the most beautiful and elegant proofs we can. 

To me, we should strive to write proofs that are as clear and concise as we can make them, but as with all writing endeavors we should expect to write multiple drafts before settling on a finished product. This is where scratch work comes in.

Scratch work provides a place for me to hash out ideas before delving into proofs. Sometimes this amounts to gathering details before reorganizing them into a proof, sometimes I explore things I know will {\em not} work out, and sometimes I simply provide a little foreshadowing of the proofs that follow. 

What does your scratch work look like? Does it help you find your own perspective on the mathematics you're exploring?

\section{Mistakes, play, and learning}
\label{sec:mistakesplayandlearning}

The exercises are there to play with! Do scratch work,  draw stuff, and make mistakes---make {\em lots} of mistakes---before worrying about writing proofs. Mistakes are an unavoidable and essential part  of learning and writing mathematics. I encourage you to read Dr.~Francis Su's wonderful book titled ``Mathematics for Human Flourishing''. 

Do you have any idea how much material in this book started off poorly formed, haphazard, and sometimes silly? Pretty much all of it. You will undoubtedly find some typos and even mathematical mistakes as you read this book. And you know what, {\em that's a good thing}. 

If you find mistakes I've made, you're learning. If you try something that ends up not working, you're learning. If an idea works for a bit and you get stuck, put the work aside and do something else. You're still learning. Your brain is amazing, it'll keep churning on the math even if you're not actively thinking about it.

For people who play video games, do you ever read the entire manual before playing a new game? Probably not. Even if you do, do you find that reading the manual made you into an expert right away? I seriously doubt it. My guess is you just grabbed a controller and went for it. Why not take the same approach with mathematics?

{\em Play with the ideas!} Playing and learning go hand-in-hand. Grab the controller, crash your character into a wall, get a ``Game Over'', and try again. When you figure something out, share it with your friends, classmates, and instructors. When you get stuck, ask them what they did or would do. Communicating your ideas (poorly formed, haphazard, and sometimes silly) and thinking through the ideas shared by your friends, classmates, and instructors will get you through.

And how about driving a car? You can watch me do it, I can explain to you how I go about the way I drive, you can read about it. But none of that compares to how much you will learn by getting behind the wheel yourself.

I suppose a point I'm trying to make is that I can explain how I see the mathematics of real analysis, but it matters more for you to discover your own perspective on the material and to find ways to explain it to the mathematicians around you. 

Most importantly, {\em have fun!}


\chapter[Kernel of Analysis]{Kernel of Analysis}
\label{ch:kernelofanalysis}

What's the kernel of analysis? A fundamental goal of this book is to help you find your own answer to this question. If you already have one, I hope you find a fresh perspective on some classic mathematics.

To me, the kernel of analysis is this: 
\begin{quote}
A point is {\em arbitrarily close} to a set if\\ every neighborhood of the point intersects the set.
\end{quote}
Defining the phrase {\em arbitrarily close} in this way provides a foundation for classic results in real analysis dealing with closure, limits and convergence, connectedness,  continuity, differentiation and integration. These ideas and more are explored throughout the textbook.\footnote{{\em Arbitrarily  close} is actually topological in nature since we can take {\em neighborhoods} to be open sets containing the point in question.}

\vs
\section{Arbitrarily close}
\label{sec:acl}

Let's start with a couple of intervals and see if either has a largest element. Note that the set of real numbers is denoted by $\R$ throughout the book, and the notation ``$x\in\R$'' means $x$ is a real number.

\begin{prob}\label{prob:intervals}
Consider the closed interval $F$ and open interval $G$ in Figure \ref{fig:intervals} given by 
\begin{align}
	F&=[0,3140]=\{x\in\R:0\leq x\leq 3140\} \quad\textnormal{and}\label{eqn:intervalF}\\ 
	G&=(0,3140)=\{x\in\R:0<x<3140\}.\label{eqn:intervalG}
\end{align}

\begin{figure}
\centering
\begin{tikzpicture}
\draw (-2,0) node {$F$};
	\draw[-,semithick] (0,0) -- (4,0);
	\draw (0,0) node {$[$};
	\draw (4,0) node {$]$};
	\draw (0,-0.5) node {$0$};
	\draw (4,-0.5) node {$3140$};
\draw (-2,-1.5) node {$G$};
	\draw[-,semithick] (0,-1.5) -- (4,-1.5);
	\draw (0.02,-1.5) node {$($};
	\draw (3.98,-1.5) node {$)$};
	\draw (0,-2) node {$0$};
	\draw (4,-2) node {$3140$};
\end{tikzpicture}
\caption{The closed interval $F=[0,3140]$ and open interval $G=(0,3140)$ from Problem \ref{prob:intervals}.}
\label{fig:intervals}
\end{figure}

\noindent Which interval has a largest element? That is, which has a {\em maximum}?
\end{prob}

To guide our reasoning more concretely, we could use a formal definition for maximum to gives us properties to check against.  For instance, do you know what I mean by ``largest element'', exactly? How am using ``large'' in this context? 

To codify our intuition and prove what we believe to be true in a rigorous way, consider the following definitions of {\em upper bound}, {\em maximum}, {\em lower bound}, and {\em minimum}. Please keep in mind, this book is about finding rigorous ways to define, justify, understand, and expand results from calculus.

\begin{definition}\label{def:max}
A real number $b$ is an {\em upper bound} for a set of real numbers $S$ if for every $x$ in $S$ we have $x\leq b$. In this case, we say $S$ is {\em bounded above}. A real number $q$ is the {\em maximum} of $S$ if $q$ is an upper bound for $S$ and $q$ is an element of $S$. That is,
\begin{itemize}
	\item[(i)] for every $x$ in $S$ we have $x\leq q$, and 
	\item[(ii)] $q$ is in $S$.
\end{itemize} 
In this case, we write $q=\max{S}$. 

Similarly, a real number $a$ is a {\em lower bound} for a set of real numbers $S$ if for every $x$ in $S$ we have $a \leq x$. In this case, we say $S$ is {\em bounded below}. A real number $v$ is the {\em minimum} of $S$ if $v$ is a lower bound for $S$ and $v$ is an element of $S$. That is,
\begin{itemize}
	\item[(iii)] for every $x$ in $S$ we have $v\leq x$, and 
	\item[(iv)] $v$ is in $S$.
\end{itemize} 
In this case, we write $v=\min{S}$. 
\end{definition}

Let's see how we can use Definition \ref{def:max} to prove a couple of facts about $F$ and $G$.

\begin{example}\label{eg:closedinterval}
The closed interval $F=[0,3140]$ has a largest element. In other words, $\max{F}$ exists.
\end{example}

\begin{proof}[Proof for Example \ref{eg:closedinterval}]
By checking the properties in Definition \ref{def:max}, since (i) $x\leq 3140$ for every $x$ in $F$ (thus 3140 is an upper bound for $F$) and (ii) 3140 is in $F$, the interval $F$ has a maximum---its largest element---given by $\max{F}=3140$.
\end{proof}
 
What about an open interval like $G=(0,3140)$? My intuition suggests the maximum might be $3140$. The problem is no real number is both (i) an upper bound for $G$ and (ii) an element of $G$. 

\begin{example}\label{eg:openinterval}
The open interval $G=(0,3140)$ has no largest element. In other words, $\max{G}$ does not exist.
\end{example}

\begin{proof}[Proof for Example \ref{eg:openinterval}]
Every real number greater than or equal to 3140 is not in $G$, so (ii) fails in this case. Also,  any real number in $G$ is strictly less than 3140, so there's always a larger number in $G$. For instance, the midpoint between 3140 and the given real number in $G$ is also in $G$. Thus, no element of $G$ is an upper bound for $G$, so (i) fails in this case. 

Since no real number satisfies both parts (i) and (ii) of Definition \ref{def:max} with respect to $G$, the open interval $G$ has no maximum. In other words, $G$ {\em has no largest element.} 
\end{proof}

Even though 3140 is not the maximum of $G$, it plays a special role: The point 3140 is an upper bound for the set $G$ which is as close to $G$ as possible without actually being in $G$. In other words, 3140 is both an upper bound for $G$ and so close to $G$ there is no distance between them. But what does that mean, exactly? And how can we prove it?

In order to prove the results we can visualize and believe to be true, mathematical definitions capture our intuition and allow us to control the context of the problems we're trying to solve. Definitions allow us to be specific in addressing intuitive ideas like how close a point can be to a set or having no distance between a point and a set.

Ultimately, it is useful to put our intuition into a more concrete mathematical framework via definitions. I believe this is harder than it sounds. After all, once we create a definition, how do we check that it really aligns with our intuition? And how do we go about proving results based on our definitions?

\begin{remark}\label{rmk:firstsectionquestion}
For this first section of the book, I'm trying to drive us to the following question: What does it mean for a point to be {\em arbitrarily close} to a set? To ensure we have a sound mathematical foundation, let's define a notion for distance between points in the real line so we can eventually be precise about how close points and sets can be to each other. The distance we use relies on a formal definition for the absolute value of a real number.
\end{remark}

\begin{definition}\label{def:absolutevalue}
For every $x$ in $\R$, the {\em absolute value} of $x$, denoted by $|x|$, is given by
\begin{align}\label{eqn:absolutevalue}
	|x|=
	\begin{cases}
	\hspace{8pt}x, & \textnormal{ if } x\geq0,\\
	-x, &  \textnormal{ if } x<0.
	\end{cases}
\end{align} 
For every $x$ and $y$ in $\R$, the {\em distance} between $x$ and $y$, denoted by $d_\R(x,y)$, is defined by 
\begin{align}\label{eqn:realdistance}
d_\R(x,y)=|x-y|.
\end{align}
\end{definition}

\begin{remark}\label{rmk:absolutevalue}
It may be odd to see absolute value defined piecewise as it is in Definition \ref{def:absolutevalue}, but doing so gives a rigorous way to work with absolute values when writing proofs. Still, the ``$-x$'' part may look weird. I think of it as multiplying a negative real number $x$ by $-1$ rather than just ``dropping the negative sign'' since the latter sometimes fails to capture the process of finding the absolute value of a negative number. For instance, $x=2-\sqrt{5}$ is a negative real number with no clear negative sign to ``drop''. On the other hand, 
\begin{align*}
x=2-\sqrt{5}<0 \quad\Longrightarrow\quad |x|=|2-\sqrt{5}|=(-1)(2-\sqrt{5})=\sqrt{5}-2,
\end{align*}
where the symbol ``$\Longrightarrow$'' is read as ``implies''.

Again, by thinking of taking the absolute value of a negative real as multiplying the negative number by $-1$, we have a rigorous mathematical process we can rely on for proofs involving absolute value.
\end{remark}

The definition for distance in Definition \ref{def:absolutevalue} is based on a comparison between two points in the real line, not a point and a set as mentioned in Remark \ref{rmk:firstsectionquestion}. Still, it gives us something to work with: Given a point and a set, we can consider {\em any} amount of distance around the point and see if there are points from the set within that distance. 

Here is a definition for {\em arbitrarily close} in the real line, what I consider to be the kernel of analysis and the fundamental theme of the book. See Figure \ref{fig:aclreal}. (A more general definition in the context of Euclidean spaces is presented in Definition \ref{def:acl}.)

\begin{definition}\label{def:aclreal}
Let $y$ be a real number and let $B$ be a set of real numbers. The point $y$ is said to be \textit{arbitrarily close} to the set $B$, and we write $y \acl B$, if for every $\varepsilon > 0$ there is some $x_\varepsilon$ in $B$ such that
\begin{align}\label{eqn:aclrealinequality}
d_\R(x_\varepsilon,y)=|x_\varepsilon-y|<\varepsilon.
\end{align}
The phrase ``$B$ is {\em arbitrarily close} to $y$'' is also defined and denoted in the same way. That is, $B\acl{y}$ is taken to mean the same thing as $y\acl{B}$.\footnote{Thanks to Berit Givens for suggesting the notation ``$y\acl{B}$'' to represent the phrase ``$y$ is arbitrarily close to $B$''.} 

If some real number $z$ is not arbitrarily close to $B$, then there is some $\varepsilon_z>0$
such that for every $x$ in $B$ we have
\begin{align}\label{eqn:awfrealinequality}
d_\R(x,z)=|x-z|\geq\varepsilon_z>0.
\end{align}
In this case, we say $z$ is \textit{away from} $B$ and write $z \awf B$. The phrase ``$B$ is {\em away from} $z$'' is also defined and denoted in the same way.
\end{definition}

\begin{figure}
\centering
\begin{tikzpicture}
\draw (0,0) node {$B$};
	\draw[-,semithick] (2,0) -- (4,0);
	\draw (2,0) node {$[$};
	\draw (3.98,0) node {$)$};
	\draw[blue] (8,0) node {$\circ$};	
	\draw (2.15,0) node {$\bullet$};
	\draw (2.7,0) node {$\bullet$};
	\draw (3.5,0) node {$\bullet$};			
	\draw (8,-0.5) node {$z$};
	\draw (4.05,-0.47) node {$y$};	
	\draw (3.55,-0.5) node {$x_\varepsilon$};		
\draw[-,semithick, blue] (8.05,0.05) -- (8.71,0.71);
\draw[blue] (8.6,0.2) node {$\varepsilon_z$};
\begin{scope}[thick, dashed, blue] 
\draw (8,0) circle (1cm); 
\end{scope}	
\begin{scope}[thick, dashed, red] 
\draw (4,0) circle (1cm); 
\draw (4,0) circle (1.6cm); 
\draw (4,0) circle (2.2cm);
\end{scope}	
	\draw[-,semithick, red] (4,0) -- (4.71,0.71);
	\draw[red] (4.5,0.2) node {$\varepsilon$};
\end{tikzpicture}
\caption{Examples of real numbers $y$ and $z$ along with a set $B$ where $y\acl{B}$ and $z\awf{B}$ as in Definition \ref{def:aclreal}.}
\label{fig:aclreal}
\end{figure}

In Figure \ref{fig:aclreal}, the set $B$ is an interval that does not contain its right endpoint $y$. The point $z$ is away from $B$ since the blue dashed circle centered at $z$ has a positive radius $\varepsilon_z$ with no points of $B$ inside. On the other hand, $y$ is arbitrarily close to $B$ since circles centered at $y$ of any positive radius $\varepsilon$ have a point $x_\varepsilon$ from the set $B$ within them. Three such circles are dashed in red, and each has at least one such $x_\varepsilon$ within its radius indicated by one of the three $\bullet$. In order to keep things from getting too cluttered, only one red $\varepsilon$ appears in the figure. 

\begin{remark}\label{rmk:smallaswelike}
In Definition \ref{def:aclreal}, it may help to think of the positive real number $\varepsilon$ as the amount of error or ``wiggle room'' we'd like to allow, the idea being that we can allow \textit{any} amount of error, no matter how small. In this way, $y \acl B$ means $B$ gets as close to $y$ as we like, no matter how close that may be. So, $y \acl B$ is exactly what it means when there is no distance between a point $y$ and a set $B$. 
\end{remark}

\begin{remark}\label{rmk:subscriptornot}
How can we prove $y,z$, and $B$ satisfy these definitions in specific cases? To prove $y \acl B$, we should respond to {\em each} distance $\varepsilon>0$ with a point $x_\varepsilon$ in $B$ that is within $\varepsilon$ of $y$. To prove $z \awf B$, we need just {\em one} distance $\varepsilon_z>0$ that separates $z$ from all the points in $B$ by a distance of $\varepsilon_z$ or more. 
\end{remark}

Also, the notation for the variables in Definition \ref{def:aclreal} play subtle roles. The $x_\varepsilon$ and $\varepsilon_z$ both have a subscript indicating something special is going on. Specifically, $x_\varepsilon$ is a particular real number in $B$ that is found in response to $\varepsilon>0$, and $\varepsilon_z>0$ is a particular positive number found in response to the real number $z$ and its relationship to $B$. On the other hand, $\varepsilon$ represents {\em any} positive real number and $x$ represents {\em any} real number in $B$. This may not make much sense or seem too important yet, but I am planting seeds for concepts we will come across throughout the book. Hang in there.

\begin{figure}
\centering
\begin{tikzpicture}
\draw (-2,0) node {$G$};
	\draw[-,semithick] (0,0) -- (4,0);
	\draw (0.02,0) node {$($};
	\draw (3.98,0) node {$)$};
	\draw (2.75,0) node {$\bullet$};
	\draw (3.25,0) node {$\bullet$};
	\draw (3.75,0) node {$\bullet$};			
	\draw (0,-0.5) node {$0$};
	\draw (4,-0.5) node {$3140$};
	\draw (6,-0.5) node {$4710$};		
	\draw[-,semithick, blue] (6.05,0.05) -- (6.71,0.71);
	\draw[blue] (6,0) node {$\circ$};
\begin{scope}[thick, dashed, blue] 
\draw (6,0) circle (1cm); 
\end{scope}	
	\draw[blue] (6.2,0.65) node {$\varepsilon_z$};
\end{tikzpicture}
\caption{The open interval $G=(0,3140)$ along with the real number $z=4710$ and distance $\varepsilon_z=785$ from Example \ref{eg:openinterval2}.}
\label{fig:openinterval}
\end{figure}

\begin{example}\label{eg:openinterval2}
For the open interval $G=(0,3140)$, I claim the real number $y=3140$ is arbitrarily close to $G$ while $z=4710$ is away from $G$. It's easier to prove $4710\awf{G}$ since we need just {\em one} positive distance $\varepsilon_z$ to separate $4710$ from all the points in $G$. Let's do that first. 

To align with Figure \ref{fig:openinterval}, the blue point represents 4710 and the radius of the little blue circle is half the distance between 4710 and 3140, which looks good enough\footnote{The real number $z=4710$ was chosen so the figure would be to scale.}. So we have
\begin{align}
\varepsilon_z=\frac{|4710-3140|}{2}=\frac{1570}{2}=785>0.
\end{align}
\end{example}

\begin{proof}[Proof of $4710\awf{G}$ in Example \ref{eg:openinterval2}]
Consider the real number $z=4710$ and let $x$ be any real number in $G=(0,3140)$. Then we have
\begin{align}\label{eqn:x31404710}
	 0<x<3140<4710.
\end{align} Now let $\varepsilon_z=785>0$. Since $x-4710<0$ and $x>3140$ according to \eqref{eqn:x31404710}, by the definition of absolute value (Definition \ref{def:absolutevalue}) we have
\begin{align}
	|x-4710|&=4710-x 
	>4710-3140
	=1570
	>785=\varepsilon_z.
\end{align}
Therefore, every point in $G$ is more than $\varepsilon_z=785>0$ away from $4710$, and so $4710\awf{G}$. 
\end{proof}

Proving $3140$ is arbitrarily close to $G$ takes more effort. It's not good enough to consider just one distance. To prove $3140\acl{G}$, we should respond to {\em every} positive distance or ``error'' $\varepsilon>0$ with its own point $x_\varepsilon$ that's both in $G$ and within $\varepsilon$ of $3140$.  Scratch work will give us the pieces we need.

\begin{scratch}
\label{scr:openintervalscratch}
Always consider drawing figures when doing scratch work. Students often request I draw more figures!

\begin{figure}
\centering
\begin{tikzpicture}
\draw (-2,-1.5) node {$G$};
	\draw[-,semithick] (0,-1.5) -- (4,-1.5);
	\draw (0.02,-1.5) node {$($};
	\draw (3.98,-1.5) node {$)$};
	\draw (3.25,-1.5) node {$\bullet$};		
	\draw (0,-2) node {$0$};
	\draw (4,-2) node {$y$};
	\draw (3.25,-2) node {$x_\varepsilon$};	
	\draw[-,semithick, red] (4,-1.5) -- (5.02,-0.38);
	\draw[red] (4.7,-1.1) node {$\varepsilon$};
\begin{scope}[thick, dashed, red] 
\draw (4,-1.5) circle (1.5cm); 
\end{scope}	
\end{tikzpicture}
\caption{The open interval $G=(0,3140)$ along with a distance $\varepsilon>0$ (in red) and a real number $x_\varepsilon$ which is both in $G$ and within $\varepsilon$ of $y=3140$. See Scratch Work \ref{scr:openintervalscratch}.}
\label{fig:openintervalscratch1}
\end{figure}

In Figure \ref{fig:openintervalscratch1}, the red radius $\varepsilon$ is small enough so the real number $x_\varepsilon=3140-(\varepsilon/2)$ is both in $G$ and within $\varepsilon$ of 3140. Only one red $\varepsilon$ appears in the figure, but to prove $3140\acl{G}$, we need to show {\em every} distance $\varepsilon>0$ comes with its own $x_\varepsilon$ from the set $G$. Considering more values for $\varepsilon$ looks something like Figure \ref{fig:openintervalscratch2}.

\begin{figure}
\centering
\begin{tikzpicture}
\draw (-2,-1.5) node {$G$};
	\draw[-,semithick] (0,-1.5) -- (4,-1.5);
	\draw (0.02,-1.5) node {$($};
	\draw (3.98,-1.5) node {$)$};
	\draw (2,-1.5) node {$\bullet$};	
	\draw (2.75,-1.5) node {$\bullet$};
	\draw (3.4,-1.5) node {$\bullet$};		
	\draw (0,-2) node {$0$};
	\draw (2,-2) node {$1570$};	
	\draw (4,-2) node {$3140$};
	\draw[-,semithick, red] (4,-1.5) -- (7.43,1.43);
	\draw[red] (6,0.75) node {$\varepsilon$};
\begin{scope}[thick, dashed, red] 
\draw (4,-1.5) circle (1cm); 
\draw (4,-1.5) circle (1.5cm); 
\draw (4,-1.5) circle (4.5cm);
\end{scope}	
\end{tikzpicture}
\caption{The open interval $G$ along with multiple distances (in red) and corresponding points (the $\bullet$) as in Scratch Work \ref{scr:openintervalscratch}. The real number 1570 is within the largest circle whose radius is $\varepsilon$ (in red), but not the smaller circles.}
\label{fig:openintervalscratch2}
\end{figure}

The largest red circle in Figure \ref{fig:openintervalscratch2} has a radius $\varepsilon$ greater than 3140. So for that $\varepsilon$, I responded with $x_\varepsilon=1570$, which is good enough since 1570 is in $G$ and within $\varepsilon$ of 3140. The other two circles have smaller radii, but each has at least one point (one of the $\bullet$) that is in $G$ and also within the corresponding distance of 3140. 
\end{scratch}

On to the proof. Again, we need to verify {\em every} $\varepsilon>0$ comes with at least one point $x_\varepsilon$ that's both in $G$ and within $\varepsilon$ of 3140. 

\begin{proof}[Proof of $3140\acl{G}$ in Example \ref{eg:openinterval2}]
Let $\varepsilon>0$. (By not specifying a particular value of $\varepsilon$, we are accounting for {\em all} positive distances, or ``errors'', at the same time!) Consider the piecewise definition for the real number $x_\varepsilon$ given by
\begin{align}
	x_\varepsilon
	&=\label{eqn:choiceofx}
	\begin{cases}
	\displaystyle 3140-\frac{\varepsilon}{2}, & \textnormal{ if } 0<\varepsilon<3140,\\
	1570, &  \textnormal{ if } \varepsilon \geq 3140.
	\end{cases}
\end{align} 
So when $0<\varepsilon<3140$ and in Figure \ref{fig:openintervalscratch1}, we have $x_\varepsilon=3140-(\varepsilon/2)
$ which is in $G$ and 
\begin{align}
d_\R(x_\varepsilon,3140)&=|x_\varepsilon-3140|=3140-\left(3140-\frac{\varepsilon}{2}\right)= \frac{\varepsilon}{2}<\varepsilon.
\end{align}
When $\varepsilon \geq 3140$ as in Figure \ref{fig:openintervalscratch2}, we have $x_\varepsilon=1570$ which is in $G$ and
\begin{align}
d_\R(x_\varepsilon,3140)&=|x_\varepsilon-3140|=3140-1570=1570\leq \frac{\varepsilon}{2}<\varepsilon.
\end{align}
Therefore, $3140 \acl G$. 
\end{proof}

\begin{remark}\label{rmk:scratch}
Hang on. What just happened? Did every step make sense? Does Scratch Work \ref{scr:openintervalscratch} help you see how I came up with this proof?  Take your time reasoning through my proof and consider writing up a walkthrough to help you find your own understanding.
\end{remark}

After laying out my Scratch Work \ref{scr:openintervalscratch}, I reorganized and rewrote stuff to produce the proof showing 3140 is arbitrarily close to $G$ in Example \ref{eg:openinterval}. How would you have done the scratch work and proof? 

Moving forward, scratch work will typically appear before a corresponding proof. The amount of detail will vary, but scratch work will usually include motivation for the steps in my proofs. For you, scratch work should entail anything that helps you figure stuff out, no matter what form it may take.

With the definition for arbitrarily close, we now have enough mathematics to define {\em supremum}. Roughly speaking, a supremum is a lot like a maximum in that both are upper bounds for a given set, but the supremum is not necessarily an element of the set like the maximum has to be. An analogous statement holds for the definition of {\em infimum} which is defined in terms of a suitable lower bound and is a lot like minimum.

\begin{definition}\label{def:supremumacl}
A real number $u$ is the {\em supremum} of a nonempty set of real numbers $S$, and we write $u=\sup S$, if $u$ is an upper bound for $S$ and arbitrarily close to $S$. That is, $u=\sup S$ if
\begin{itemize}
	\item[(i)] for every $x\in S$ we have $x\leq u$, and
	\item[(ii)] $u \acl S$.
\end{itemize}

Similarly, a real number $\ell$ is the {\em infimum} of $S$, and we write $\ell=\inf S$, if $\ell$ is a lower bound for $S$ and arbitrarily close to $S$. That is, $\ell=\inf S$ if 
\begin{itemize}
	\item[(iii)] for every $x\in S$ we have $\ell \leq x$, and
	\item[(iv)] $\ell \acl S$.
\end{itemize}
\end{definition}

Have you noticed how similar the definitions for maximum and supremum are to one another? See Definitions \ref{def:max} and \ref{def:supremumacl}.

\begin{example}\label{eg:openintervalrevisit}
As seen in Example \ref{eg:openinterval}, the real number $3140$ is arbitrarily close to the open interval $G=(0,3140)$. Since $x<3140$ for each $x$ in $G$, we have $3140$ is also an upper bound for $G$. Therefore, $\sup G=3140$ and yet $\max G$ does not exist (as pointed out Example \ref{eg:closedinterval}). Similarly, $\min G$ does not exist but $\inf{G}=0$. 
\end{example}

Let's consider one more example to see what our results give us regarding the open interval $G$ from
Problem \ref{prob:intervals} as well as Examples \ref{eg:openinterval} and \ref{eg:openintervalrevisit}, this time with specific values of $\varepsilon>0$ in mind.

\begin{example}\label{eg:sanitycheck}
For the interval $G=(0,3140)$, we proved $3140\acl{G}$ with a key step of choosing
\begin{align}
	x_\varepsilon
	&=\label{eqn:choiceofxagain}
	\begin{cases}
	\displaystyle 3140-\frac{\varepsilon}{2}, & \textnormal{ if } 0<\varepsilon<3140,\\
	1570, &  \textnormal{ if } \varepsilon \geq 3140.
	\end{cases}
\end{align} 
in response to any given $\varepsilon>0$. What does this formula give us for the specific values $\varepsilon_1=1$, $\varepsilon_2=1/2$, and $\varepsilon_3=1/100$? Are the values of $x_\varepsilon$ really in $G$?

First, we can show $x_1,x_2$, and $x_3$ are in $G$ by verifying they satisfy the inequalities defining $G$ in line \eqref{eqn:intervalG}. Respectively, we have\footnote{I changed the notation to $x_1$ instead of $x_{\varepsilon_1}$ since smaller subscript is so tiny.}
\begin{align}
	0< x_1&=3140-\frac{\varepsilon_1}{2}=3140-\frac{1}{2}<3140,\\
	0< x_2&=3140-\frac{\varepsilon_2}{2}=3140-\frac{1}{4}<3140,\quad\textnormal{and}\\
	0< x_3&=3140-\frac{\varepsilon_3}{2}=3140-\frac{1}{200}<3140.
\end{align}
Hence, $x_1, x_2$ and $x_3$ are in $G=(0,3140)$.

Now, are these $x_k$ for $k=1,2,3$ close enough to 3140? Yes, but each one only needs to be within their corresponding $\varepsilon_k$ to get their jobs done.  Since $x_k<3140$ for each $k=1,2,3$, by the definition of absolute value (Definition \ref{def:absolutevalue}) we have
\begin{align}
	|x_1-3140|&=\left|-\frac{1}{2}\right|=\frac{1}{2}<1=\varepsilon_1,\\
	|x_2-3140|&=\left|-\frac{1}{4}\right|=\frac{1}{4}<\frac{1}{2}=\varepsilon_2,\quad\textnormal{and}\\
	|x_3-3140|&=\left|-\frac{1}{200}\right|=\frac{1}{200}<\frac{1}{100}=\varepsilon_3.
\end{align}
We're good!
\end{example}

Let's play around with another pair of sets that are {\em not} intervals.

\begin{example}\label{eg:twocountablesets}
Consider the following sets of real numbers $A$ and $B$ where $\N$ denotes the set of positive integers:
\begin{align}
A&=\left\{a_n=2-\frac{1}{\sqrt{n}}:n\in\N\right\} \qquad\textnormal{and}\\
B&=\left\{b_n=\left(2-\frac{1}{\sqrt{n}}\right)(-1)^n:n\in\N\right\}.
\end{align}
I claim $2\acl{A}, 2\acl{B}$, and $-2\acl{B}$ while $-2\awf{A}$. See Figure \ref{fig:twocountablesets}. 
\end{example}

\begin{figure}
\centering
\begin{tikzpicture}
\draw (-2,0) node {$A$};
\draw (4,0) node {$\circ$};
\draw (3.76,0) node {$...$};
\draw (3,-0.5) node {$1$};
\draw (4,-0.5) node {$2$};
\foreach \Point in {(3,0), (3.29,0), (3.42,0), (3.5,0)}
{
    \node at \Point {\textbullet};
}
\draw (-2,-1.5) node {$B$};
\draw (0,-1.5) node {$\circ$};
\draw (4,-1.5) node {$\circ$};
\draw (3.76,-1.5) node {...};
\draw (0.28,-1.5) node {...};
\draw (0,-2) node {$-2$};
\draw (1,-2) node {$-1$};
\draw (4,-2) node {$2$};
\foreach \Point in {(1,-1.5), (3.29,-1.5), (0.58,-1.5), (3.5,-1.5)}
{
    \node at \Point {\textbullet};
}
\end{tikzpicture}
\caption{The sets $A$ and $B$ from Example \ref{eg:twocountablesets}. Neither set contains 2, indicated by one of the $\circ$ in the figure. But how close do they get? What about $-2$? Can you prove your answers?}
\label{fig:twocountablesets}
\end{figure}

Before attempting proofs, consider some scratch work for the set $A$ along with Figure \ref{fig:twocountablesetsscratch}.

\begin{scratch}\label{scr:twocountablesets}
\begin{figure}
\centering
\begin{tikzpicture}
\draw (-2,0) node {$A$};
\draw (4,0) node {$\circ$};
\draw (4,-0.4) node {$2$};
\draw (3.76,0) node {$...$};
\draw (3,-0.4) node {$a_1$};
\draw (0,-0.4) node {$-2$};
\draw (0,0) node {$\circ$};
\draw[-,semithick, blue] (0.05,0.05) -- (0.71,0.71);
\draw[blue] (0.6,0.2) node {$1$};
\foreach \Point in {(3,0), (3.29,0), (3.42,0), (3.5,0)}
{
    \node at \Point {\textbullet};
}
\draw[thick, dashed, blue] (0,0) circle (1.0cm);
\begin{scope}[thick, dashed, red] 
\draw (4,0) circle (0.85cm); 
\draw (4,0) circle (1.4cm); 
\end{scope}	
\draw[-,semithick, red] (4.05,0.05) -- (5,1);
\draw[red] (4.97,0.6) node {$\varepsilon_1$};
\end{tikzpicture}
\caption{A figure to accompany Scratch Work \ref{scr:twocountablesets} featuring the set $A$ along with a distance $1$ around $-2$ and a distance $\varepsilon_1>1$ around $-2$.}
\label{fig:twocountablesetsscratch}
\end{figure}

Yet again, the ``away from'' part looks like it might  be easiest to prove since a radius of $\varepsilon_z=1$ is enough distance to keep $z=-2$ away from all the points in $A$. 

The ``arbitrarily close'' part---$2\acl{A}$---deserves more scrutiny. See Figure \ref{fig:twocountablesetsscratch}. For the largest red circle around $2$ (the center $\circ$ on the right) whose radius $\varepsilon_1$ is greater than $1$, the real number $x_{\varepsilon_1}=a_1$ is in $A$ and close enough to $2$. But $a_1$ is not close enough for the smaller circle whose radius is less than 1. That's okay, though. For each radius $\varepsilon>0$, we need just one point $x_\varepsilon$ that's in $A$ and within $\varepsilon$ of 2, and each $\varepsilon$ can have its own $x_\varepsilon$. 

So, if we can find a way to take a given but arbitrary $\varepsilon>0$ and respond to it with a suitably defined $x_\varepsilon$, we're good.
\end{scratch}

\begin{remark}\label{rmk:carefulchoice}
Careful! Setting 
\begin{align}
	x_\varepsilon
	&=\label{eqn:choiceofxagain}
	\begin{cases}
	\displaystyle 2-\frac{\varepsilon}{2}, & \textnormal{ if } 0<\varepsilon<2,\\
	1, &  \textnormal{ if } \varepsilon \geq 2
	\end{cases}
\end{align} 
will not be good enough to show $2\acl{A}$, even though a similar choice was made for $x_\varepsilon$ and the interval $G$ in Example \ref{eg:openinterval}. To see why, temporarily set $\varepsilon=3/2$. Since $0<\varepsilon=3/2<2$, we have
\begin{align}
	x_\varepsilon&=2-\frac{3/2}{2}=2-\frac{3}{4}=\frac{5}{4}.
\end{align} 
Since $5/4<2$, the definition of absolute value (Definition \ref{def:absolutevalue}) yields
\begin{align}
	d_\R(x_\varepsilon,2)&=|x_\varepsilon-2|= 2-\frac{5}{4}=\frac{3}{4}<\frac{3}{2}=\varepsilon,
\end{align} 
which means $x_\varepsilon=5/4$ is close enough to 2. The problem is, $5/4$ {\em is not in} $A$. 

For $x_\varepsilon=5/4$ to be an element of $A$, there must be a positive integer $n_\varepsilon$ where
\begin{align}
x_\varepsilon&=2-\left(\frac{1}{\sqrt{n_\varepsilon}}\right)=5/4.
\end{align}
But solving for $n_\varepsilon$ yields $n_\varepsilon=16/9$, which is not a positive integer. So, we need to be more careful when choosing $x_\varepsilon$ in response to $\varepsilon$.
\end{remark}

Let's start over with some new scratch work.

\begin{scratch}\label{scr:startatend}
With any scratch work, consider starting at the end. For each $\varepsilon>0$, we want to end up with $x_\varepsilon=2-(1/\sqrt{n_\varepsilon})$---which is definitely  in the set $A$---where
\begin{align}
|x_\varepsilon-2|=\left|2-\left(\frac{1}{\sqrt{n_\varepsilon}}\right)-2\right|=\frac{1}{\sqrt{n_\varepsilon}}<\varepsilon.
\end{align}
This amounts to solving the inequality for $n_\varepsilon$ and making sure we choose $n_\varepsilon$ to be a positive integer. So let's do that. We get 
\begin{align}
\frac{1}{\sqrt{n_\varepsilon}}<\varepsilon\quad \Longrightarrow\quad n_\varepsilon>\frac{1}{\varepsilon^2}.
\end{align}
Choosing a {\em positive integer} $n_\varepsilon$ large enough to satisfy $n_\varepsilon>1/\varepsilon^2$ allows us to set $x_\varepsilon=a_{n_\varepsilon}=2-(1/\sqrt{n_\varepsilon})$.
\end{scratch}

Time for a proof. While Scratch Work \ref{scr:startatend} is where I really figured things out, the proof amounts to a careful reorganization of the scratch work where the details are put in an appropriate order.

\begin{proof}[Proof of $-2\awf{A}$ and $2\acl{A}$ in Example \ref{eg:twocountablesets}] 
First, to show $-2\awf{A}$, note that for every $n$ in $\N$ we have
\begin{align}
|a_n-(-2)|&=\left|2-\left(\frac{1}{\sqrt{n}}\right)-(-2)\right|=4-\frac{1}{\sqrt{n}}\geq 3> 1>0.
\end{align}
Hence, every real number $a_n$ in $A$ is a distance of at least 1 away from $-2$.

Next, to show $2\acl{A}$, let $\varepsilon>0$. (By not specifying a particular value of $\varepsilon$, we are accounting for {\em all} positive distances, or radii or ``errors'', at the same time.) Choose a positive integer $n_\varepsilon$ large enough to satisfy $n_\varepsilon>1/\varepsilon^2$. We have 
\begin{align}
n_\varepsilon>\frac{1}{\varepsilon^2} \quad\Longrightarrow\quad \frac{1}{\sqrt{n_\varepsilon}}<\varepsilon.
\end{align}
Now consider $a_{n_\varepsilon}=2-(1/\sqrt{n_\varepsilon})$. Then $a_{n_\varepsilon}$ is in $A$ and 
\begin{align}
d_\R(a_{n_\varepsilon},2)&=\left|a_{n_\varepsilon}-2\right|=\left|2-\frac{1}{\sqrt{n_\varepsilon}}-2\right|=\frac{1}{\sqrt{n_\varepsilon}}<\varepsilon.
\end{align}
Therefore, $2\acl{A}$.
\end{proof}

\begin{remark}\label{rmk:differentn}
In the proofs of $-2\awf{A}$ and $2\acl{A}$ as stated in Example \ref{eg:twocountablesets}, I used the variables $n$ and $n_\varepsilon$ in subtly different ways, much like the differences between $\varepsilon$ versus $\varepsilon_z$ and $x$ versus $x_\varepsilon$. Both $n$ and $n_\varepsilon$ represent positive integers, but $n$ represents {\em any} positive integer while $n_\varepsilon$ is a particular positive integer chosen in response to $\varepsilon>0$ (hence the subscript) and how the scratch work turned out based on the definition of the set $A$.
\end{remark}

Proving $-2\acl{B}$ and $2\acl{B}$ from Example \ref{eg:twocountablesets} can be very similar to the proving $2\acl{A}$, but there is a significant difference. When proving $2\acl{A}$, it was enough to find a positive integer $n_\varepsilon$ in response to $\varepsilon$. But when proving $-2\acl{B}$ and $2\acl{B}$, the parity of $n_\varepsilon$---whether $n_\varepsilon$ is even or odd---changes whether the corresponding point in $B$ is close to $2$ or $-2$. Such a point in $B$ might be close enough to either $2$ or $-2$, but maybe not both.

For instance, Figure \ref{fig:Bcloseto2-2}, $b_1=-1$ is in $B$ and within the largest red circle on the left, so it is within the corresponding radius $\varepsilon$ of $-2$. Hence,  $b_1=-1$ is close enough to $-2$ for that particular $\varepsilon$. But $-1$ is not close enough to $2$ since it is outside the largest red circle on the right which has the same radius $\varepsilon$. Hence, for the $\varepsilon>0$ that gives the radius of the two largest red circles, $n_\varepsilon=1$ produces $b_1=-1$ which is close enough to $-2$ but not close enough to $2$. On the other hand, $n_\varepsilon+1=2$ produces $b_2=2-(1/\sqrt{2})$, which is close enough to $2$. 

\begin{figure}
\centering
\begin{tikzpicture}
\draw (-2,0) node {$B$};
\draw (0,0) node {$\circ$};
\draw (4,0) node {$\circ$};
\draw (3.76,0) node {...};
\draw (0.28,0) node {...};
\draw (0,-0.6) node {$-2$};
\draw (1,-0.5) node {$-1$};
\draw (4,-0.6) node {$2$};
\foreach \Point in {(1,0), (3.29,0), (0.58,0), (3.5,0)}
{
    \node at \Point {\textbullet};
}
\begin{scope}[thick, dashed, red] 
\draw (0,0) circle (0.38cm); 
\draw (0,0) circle (0.85cm); 
\draw (0,0) circle (1.4cm); 
\end{scope}
\begin{scope}[thick, dashed, red] 
\draw (4,0) circle (0.38cm); 
\draw (4,0) circle (0.85cm); 
\draw (4,0) circle (1.4cm); 
\end{scope}
\draw[-,semithick, red] (0.05,0.05) -- (1,1);
\draw[red] (.95,0.65) node {$\varepsilon$};
\draw[-,semithick, red] (4.05,0.05) -- (5,1);
\draw[red] (4.95,0.65) node {$\varepsilon$};
\end{tikzpicture}
\caption{The set $B$ from Example \ref{eg:twocountablesets} along with various distances around $2$ and $-2$.}
\label{fig:Bcloseto2-2}
\end{figure}

Loosely speaking, the $b_n$ with even indices are near $2$ while those with odd indices are near $-2$. In turn, parity of positive integers affects the proof. In the proof, by not specifying a particular value of $\varepsilon$, we are accounting for {\em all} positive distances, or ``errors'', around both $-2$ and $2$ simultaneously.

\begin{proof}[Proof of $-2\acl{B}$ and $2\acl{B}$  in Example \ref{eg:twocountablesets}]
Let $\varepsilon>0$.  Choose an {\em odd} positive integer $j_\varepsilon$ large enough to satisfy $j_\varepsilon>1/\varepsilon^2$. We have 
\begin{align}
j_\varepsilon>\frac{1}{\varepsilon^2} \quad\Longrightarrow\quad \frac{1}{\sqrt{j_\varepsilon}}<\varepsilon.
\end{align}
Now consider $b_{j_\varepsilon}=-2+(1/\sqrt{j_\varepsilon})$. Then $b_{j_\varepsilon}$ is in $B$ and 
\begin{align}
d_\R(b_{j_\varepsilon},-2)&=\left|-2+\frac{1}{\sqrt{j_\varepsilon}}-(-2)\right|=\frac{1}{\sqrt{j_\varepsilon}}<\varepsilon.
\end{align}
Hence $-2\acl{B}$.

Next, choose an {\em even} positive integer $k_\varepsilon$ large enough to satisfy $k_\varepsilon>1/\varepsilon^2$. We have 
\begin{align}
k_\varepsilon>\frac{1}{\varepsilon^2} \quad\Longrightarrow\quad \frac{1}{\sqrt{k_\varepsilon}}<\varepsilon.
\end{align}
Now consider $b_{k_\varepsilon}=2-(1/\sqrt{k_\varepsilon})$. Then $b_{k_\varepsilon}$ is in $B$ and 
\begin{align}
d_\R(b_{k_\varepsilon},2)&=\left|2-\frac{1}{\sqrt{k_\varepsilon}}-2\right|=\frac{1}{\sqrt{k_\varepsilon}}<\varepsilon.
\end{align}
Hence $2\acl{B}$ as well.
\end{proof}

Before going any deeper, Section \ref{sec:preliminary} lays out some of the assumptions I've made about the mathematical knowledge and experiences readers are expected to have. In some ways, it may have been better to put these assumptions first, but I wanted to start with a notion for arbitrarily close in the real line that serves as the foundation for the entire book. This comes at the cost of using some facts before proving them; a small price to pay for a first section on such a difficult topic. Section \ref{sec:preliminary} also establishes more notation and background material.

\vs
\section*{Exercises}
\setcounter{theorem}{0}

Exercises are for play: Do scratch work, draw stuff, and make mistakes---make {\em lots} of mistakes---before worrying about writing proofs. {\em Have fun!}

\vs
Once again, here are the sets discussed throughout Section \ref{sec:acl}:
\begin{align}
	F&=[0,3140]=\{x\in\R:0\leq x\leq 3140\},\\ 
	G&=(0,3140)=\{x\in\R:0<x<3140\},\\
	A&=\{2-(1/\sqrt{n}):n\in\N\} \quad\textnormal{and}\\
B&=\{[2-(1/\sqrt{n})](-1)^n:n\in\N\}.
\end{align}

\begin{figure}
\centering
\begin{tikzpicture}
\draw (-2,0) node {$F$};
	\draw[-,semithick] (0,0) -- (4,0);
	\draw (0,0) node {$[$};
	\draw (4,0) node {$]$};
	\draw (0,-0.5) node {$0$};
	\draw (4,-0.5) node {$3140$};
\draw (-2,-1.5) node {$G$};
	\draw[-,semithick] (0,-1.5) -- (4,-1.5);
	\draw (0.02,-1.5) node {$($};
	\draw (3.98,-1.5) node {$)$};
	\draw (0,-2) node {$0$};
	\draw (4,-2) node {$3140$};
\draw (-2,-3) node {$A$};
\draw (4,-3) node {$\circ$};
\draw (3.76,-3) node {$...$};
\draw (3,-3.5) node {$1$};
\draw (4,-3.5) node {$2$};
\foreach \Point in {(3,-3), (3.29,-3), (3.42,-3), (3.5,-3)}
{
    \node at \Point {\textbullet};
}
\draw (-2,-4.5) node {$B$};
\draw (0,-4.5) node {$\circ$};
\draw (4,-4.5) node {$\circ$};
\draw (3.76,-4.5) node {...};
\draw (0.28,-4.5) node {...};
\draw (0,-5) node {$-2$};
\draw (1,-5) node {$-1$};
\draw (4,-5) node {$2$};
\foreach \Point in {(1,-4.5), (3.29,-4.5), (0.58,-4.5), (3.5,-4.5)}
{
    \node at \Point {\textbullet};
}	
\end{tikzpicture}
\caption{The sets $F,G,A$, and $B$ are explored further in the following exercises.}
\label{fig:FGAB}
\end{figure}

\xca For the sets $F,G,A$, and $B$ given above and visualized in Figure \ref{fig:FGAB}, determine whether each set has a maximum, minimum, supremum, and infimum and prove your results. (Some cases have already been discussed and proven in Section \ref{sec:acl}, so don't prove those again unless you want to. You can treat those proofs like templates, carefully adjusting what I've done to fit a similar situation accordingly.)

\xca Consider the set $J=[3140,3150)$. For each $n\in\N$, find $y_n\in J$ such that $|y_n- 3150| < 1/2^n$. Does this prove $3150\acl{J}$?

\xca Prove for every $\varepsilon>0$ we have
\begin{align}
	\max\left\{3140-\frac{\varepsilon}{2}, 1570\right\}
	&=\label{eqn:choiceofxexercise}
	\begin{cases}
	\displaystyle 3140-\frac{\varepsilon}{2}, & \textnormal{ if } 0<\varepsilon<3140,\\
	1570, &  \textnormal{ if } \varepsilon \geq 3140.
	\end{cases}
\end{align}
Thus, the choice of real number $x_\varepsilon$ in Example \ref{eg:openintervalrevisit} can be thought of as the maximum of a set of two real numbers that depends on $\varepsilon$. In other words, $x_\varepsilon$ is a function of $\varepsilon$.

\xca The following statement is {\em false}: For any set of real numbers $S$ that is bounded above we have 
\begin{align}\label{eqn:falsemaxsup}
u=\max{S} \qquad\textnormal{if and only if}\qquad u=\sup{S}.
\end{align}
\begin{enumerate}
	\item Find a {\em counterexample} for this statement and prove your result. That is, find, describe, and draw a set of real numbers $S$ that is bounded above but for which the ``if and only if'' part of line \eqref{eqn:falsemaxsup} does not hold.
	\item Revise the above false statement to create an implication which is true and prove your result.
	\item Write a similar implication involving minimum and infimum, and prove it too. 
\end{enumerate}

\xca\label{exer:neitherintervalnorcountable} Consider the set of real numbers $T$ given by 
\begin{align}\label{eqn:neitherintervalnorcountable}
	T&=[0,2)\cup\{4-(1/n):n\in\N\}.
\end{align}
See Figure \ref{fig:neitherintervalnorcountable}. That is, every real number $x$ in $T$ satisfies either $0\leq x<2$ or $x=4-(1/n)$ for some positive integer $n$. The symbol $\cup$ stands for {\em union} and is discussed in the next section.

\begin{figure}
\centering
\begin{tikzpicture}
\draw (-2,0) node {$T$};
	\draw[-,semithick] (0,0) -- (2,0);
	\draw (0,0) node {$[$};
	\draw (2,0) node {$)$};
	\draw (0,-0.5) node {$0$};
	\draw (2,-0.5) node {$2$};
\draw (4,0) node {$\circ$};
\draw (3.76,0) node {...};
\draw (3,-0.5) node {$3$};
\draw (4,-0.5) node {$4$};
\foreach \Point in {(3,0), (3.5,0)}
{
    \node at \Point {\textbullet};
}	
\end{tikzpicture}
\caption{The set of real numbers $T$ from Exercise \ref{exer:neitherintervalnorcountable}.}
\label{fig:neitherintervalnorcountable}
\end{figure}

\begin{enumerate}
	\item Prove $\min{T}=\inf{T}=0$, $\sup{T}=4$, and $\max{T}$ does not exist.
	\item What can you say and prove about the real numbers $2$ and $3$ here? Are either of them arbitrarily close to $T$? Away from $T$? Can you prove your answers?
\end{enumerate}

\xca Find examples of sets of real numbers with the following properties:
\begin{enumerate}
	\item A set $U$ where neither $\sup{U}$ nor $\inf{U}$ exist and $U\neq\R$.
	\item A set $V$ where $\min{V}$ and $\max{V}$ exists with $\min{V}<\max{V}$, but $V$ is not an interval.
	\item A set $W$ where $\inf{W}$ exists, $\sup{W}$ exists, and we have $\inf{W}=\sup{W}$.
\end{enumerate}
Be sure to draw figures for each of your sets and justify your results. What other properties do your examples have?

\vs
\section[Preliminary concepts]{Preliminary concepts and\\ background material}
\label{sec:preliminary}

To set the stage for the more technical aspects of the book, this section provides some notation and terminology. I am assuming a fair amount of familiarity with inequalities, basic algebra, trigonometry, set theory, and writing mathematical proofs.

\begin{notation}\label{not:notation}
Here are some conventions used throughout the book.
\begin{itemize}
	\item[(i)] Lower case letters like $a,b,c,f,g,x,y,z,\varepsilon$ and $\delta$ typically denote real numbers or functions.
	\item[(ii)] Capital letters like $B,S,$ and $V$ typically denote sets. 
	\item[(iii)] Boldface lower case letters like $\bfy$ and $\bfw$ typically denote points (or vectors) in some Euclidean space $\R^m$.
	\item[(iv)] Along these lines, notation that starts with lower case letters such as ``$\max$'', ``$\sup$'', and ``$\lim$'' represent real numbers or points while ``$\Coda$'' and ``$\Slim$'' represent sets.
	\item[(v)] The notation ``$\Longrightarrow$'' is read as ``implies''. 
	\item[(vi)] The notation ``$\Longleftrightarrow$'' is read as ``if and only if''.
\end{itemize}
\end{notation}

\begin{notation}[Sets and points]\label{not:settheory}
If $S$ is any set comprising any kind of elements we like, we write $x\in S$ when $x$ is in---or $x$ is an element of, or $x$ belongs to---the set $S$. If $z$ is not in $S$, we write $z\notin S$.  Throughout the book, elements are also referred to as {\em points}. The set with no elements is called the {\em empty set}, denoted by $\varnothing$. If a set has one or more elements, it is called {\em nonempty}. 
\end{notation}

\begin{notation}[Subsets and equality of sets]\label{not:moresettheory}
If $A$ and $B$ are sets where every element of $A$ is also an element of $B$, we say $A$ is a {\em subset} of $B$ or $B$ {\em contains} $A$ and write $A\subseteq B$.  If $A$ is a subset of $B$ but there is an element of $B$ that is not in $A$, we say $A$ is a {\em proper} subset of $B$ and write $A\subsetneq B$. In the case where we have both $A\subseteq B$ and $B\subseteq A$, we say $A$ and $B$ are {\em equal} and write $A=B$. Equivalently, $A=B$ if and only if $A$ and $B$ have the exact same elements.
\end{notation}

\begin{notation}\label{not:setbuilder}
Sets can be described in many different ways, all of which are useful and provide a variety of perspectives. For instance, with $\Z$ denoting the set of integers, consider:
\begin{itemize}
\item[(i)] Let $S=\{-3,-2,-1,0,1,2,3\}$.
\item[(ii)] Let $S$ be the set of integers between $-3.5$ and $3.999$.
\item[(iii)] Let $S=\{z\in\Z:|z|\leq 3\}$.
\end{itemize}
Items (i), (ii), and (iii) define the same exact set $S$ in three different ways, each of which is perfectly valid for the purposes of writing proofs. 
\end{notation}

Another way to describe a set is to draw a figure. Many figures are provided throughout the text because I think they can be incredibly helpful for building intuition and exploring technical ideas. In fact, I regularly encourage you to draw figures to supplement proofs, examples, and pretty much everything in the book. However, figures don't suffice for proofs (...but just barely).

Unions and intersections of sets play a significant role.

\begin{definition}\label{def:moresettheory}
For sets $A$ and $B$, the \textit{union} of $A$ and $B$ is denoted by $A\cup B$ and defined by
\begin{align}
A\cup B=\{x:x\in A \textnormal{ or } x\in B\}.
\end{align}
The ``or'' within the definition of $A\cup B$ is the ``inclusive or''.  That is, $A\cup B$ is the set of elements in $A$, in $B$, or in both $A$ and $B$. The \textit{intersection} of $A$ and $B$ is denoted by $A\cap B$ and defined by
\begin{align}
A\cap B=\{x:x\in A \textnormal{ and } x\in B\}.
\end{align}
The notation $A\backslash B$ denotes the set of points in $A$ that are not in $B$. That is, 
\begin{align}
A\backslash B=\{x:x\in A \textnormal{ and } x\notin B\}.
\end{align}
\end{definition}

The variables $x_\varepsilon$, $\varepsilon_z$, and $n_\varepsilon$ from Section \ref{sec:acl} show us that indexing variables can be helpful. Similarly, we sometimes want to consider intersections and unions of collections of sets. Generally speaking, variables and sets can be indexed by other sets like $\N$, $\Z$, or the set of positive real numbers.

\begin{definition}\label{def:indexed}
Given a set $B$, a nonempty set $\Lambda$ is an \textit{index set} for $B$ if for each $\lambda\in\Lambda$ there is a subset $S_\lambda$ of $B$. The collection of these sets is called an \textit{indexed family} of sets and is denoted by $\{S_\lambda\}_{\lambda\in\Lambda}$.
\end{definition}

\begin{definition}\label{def:indexedsetoperations}
Given an indexed family of sets $\{S_\lambda\}_{\lambda\in\Lambda}$, the \textit{union} and \textit{intersection} of all of the sets in the indexed family are defined by
\begin{align}
\bigcup_{\lambda\in\Lambda}S_\lambda&=\{x:x\in S_\lambda \textnormal{ for at least one } \lambda\in\Lambda\} \quad\textnormal{and}\\
\bigcap_{\lambda\in\Lambda}S_\lambda&=\{x:x\in S_\lambda \textnormal{ for all }\lambda\in\Lambda\},\quad\textnormal{respectively}.
\end{align}
\end{definition}
In other words and slightly different notation, we have $x\in\cup_{\lambda\in\Lambda}S_\lambda$ whenever $x$ is in some of the $S_\lambda$. We have $x\in\cap_{\lambda\in\Lambda}S_\lambda$ only when $x$ is in \textit{every} $S_\lambda$. In the special case where the index set is the set of positive integers $\N$, we write
\begin{align}\label{eqn:countableunionintersection}
\bigcup_{n=1}^\infty S_n & \quad\textnormal{and}\quad \bigcap_{n=1}^\infty S_n,\quad\textnormal{respectively}.
\end{align}

\begin{notation}[Important sets]\label{not:importantsets}
The following sets appear throughout the book.\footnote{The set of positive integers is also known as the set of natural numbers, among other things.} How do you like to describe them?
\begin{align*}
\textit{positive integers:}\quad \N&=\{1,2,3,\ldots\}.\\
\textit{integers:}\quad \Z&=\{\ldots, -2,-1,0,1,2,3,\ldots\}.\\
\textit{rational numbers:}\quad \Q&=\{m/n:m,n\in\Z \textnormal{ and }n\neq 0\}.\\
\textit{real numbers:}\quad \R&=\{x:x\in\Q \textnormal{ or } x \textnormal{ is a gap near } \Q\}. 
\end{align*}
\end{notation}

Intervals are used throughout as well. 
\begin{definition}[Intervals] \label{def:intervals}
For $a,b\in\R$ with $a<b$, each of the following sets is an {\em interval}.
\begin{multicols}{2}
\begin{enumerate}
	\item $(-\infty,\infty)=\R$
	\item $(-\infty,b)=\{x\in\R:x<b\}$
	\item $(a,\infty)=\{x\in\R:a<x\}$
	\item $(-\infty,b]=\{x\in\R:x\leq b\}$
	\item $[a,\infty)=\{x\in\R:a\leq x\}$
	\item $(a,b)=\{x\in\R:a<x<b\}$
	\item $[a,b]=\{x\in\R:a\leq x\leq b\}$
	\item $(a,b]=\{x\in\R:a<x\leq b\}$
	\item $[a,b)=\{x\in\R:a\leq x<b\}$
\end{enumerate}
\end{multicols}
\end{definition}
See Figure \ref{fig:intervalplots}. Furthermore, we have:
\begin{itemize}
\item $(a,b), (-\infty,b), (a,\infty),$ and $(-\infty,\infty)$ are \textit{open intervals};
\item $[a,b], (-\infty,b], [a,\infty),$ and $(-\infty,\infty)$ are \textit{closed intervals};
\item $(-\infty,b)$ and $(-\infty,b]$ are \textit{bounded above}, but \textit{unbounded};
\item $(a,\infty)$ and $(a,\infty]$  are \textit{bounded below}, but \textit{unbounded};
\item $(a,b), [a,b], (a,b]$, and $[a,b)$ are \textit{bounded}, meaning they are both bounded above and bounded below.
\end{itemize}

It strikes some people as odd that $\R=(-\infty,\infty)$ is both an open interval and a closed interval. If this doesn't seem right to you, you're not alone. However, this odd choice of terminology is conventional and so prevalent I will not try to replace it. The study of open, closed, and other types of sets---\textit{topology}---is the focus of Chapter \ref{ch:topologyofeuclideanspaces}.

\begin{figure}
\centering
\begin{tikzpicture}
\draw (-2,0) node {$(a,b)$};
	\draw[-,semithick] (1.98,0) -- (4.02,0);
	\draw (2,0) node {$($};
	\draw (4,0) node {$)$};
	\draw (2,-0.5) node {$a$};
	\draw (4,-0.5) node {$b$};
\draw (-2,-1.5) node {$[a,b]$};
	\draw[-,semithick] (2,-1.5) -- (4,-1.5);
	\draw (2,-1.5) node {$[$};
	\draw (4,-1.5) node {$]$};
	\draw (2,-2) node {$a$};
	\draw (4,-2) node {$b$};
\draw (-2,-3) node {$(a,b]$};	
	\draw[-,semithick] (1.98,-3) -- (4,-3);
	\draw (2,-3) node {$($};
	\draw (4,-3) node {$]$};
	\draw (2,-3.5) node {$a$};
	\draw (4,-3.5) node {$b$};
\draw (-2,-4.5) node {$[a,b)$};	
	\draw[-,semithick] (2,-4.5) -- (4.02,-4.5);
	\draw (2,-4.5) node {$[$};
	\draw (4,-4.5) node {$)$};
	\draw (2,-5) node {$a$};
	\draw (4,-5) node {$b$};
\draw (-2,-6) node {$(-\infty,b)$};	
	\draw[<-,semithick] (0,-6) -- (4,-6);
	\draw (2,-6) node {$|$};
	\draw (4,-6) node {$)$};
	\draw (0,-6.5) node {$-\infty$};
	\draw (2,-6.5) node {$a$};
	\draw (4,-6.5) node {$b$};
\draw (-2,-7.5) node {$(a,\infty)$};	
	\draw[->,semithick] (2,-7.5) -- (6,-7.5);
	\draw (2,-7.5) node {$($};
	\draw (4,-7.5) node {$|$};
	\draw (2,-8) node {$a$};
	\draw (4,-8) node {$b$};
	\draw (6,-8) node {$\infty$};
\draw (-2,-9) node {$(-\infty,b]$};	
	\draw[<-,semithick] (0,-9) -- (4,-9);
	\draw (2,-9) node {$|$};
	\draw (4,-9) node {$]$};
	\draw (0,-9.5) node {$-\infty$};
	\draw (2,-9.5) node {$a$};
	\draw (4,-9.5) node {$b$};
\draw (-2,-10.5) node {$[a,\infty)$};	
	\draw[->,semithick] (2,-10.5) -- (6,-10.5);
	\draw (2,-10.5) node {$[$};
	\draw (4,-10.5) node {$|$};
	\draw (2,-11) node {$a$};
	\draw (4,-11) node {$b$};
	\draw (6,-11) node {$\infty$};	
\draw (-2,-12) node {$(-\infty,\infty)$};
	\draw[<->,semithick] (0,-12) -- (6,-12);
	\draw (2,-12) node {$|$};
	\draw (4,-12) node {$|$};
	\draw (0,-12.5) node {$-\infty$};
	\draw (2,-12.5) node {$a$};
	\draw (4,-12.5) node {$b$};
	\draw (6,-12.5) node {$\infty$};
\end{tikzpicture}
\caption{Here are plots of all nine types of intervals. See Definition \ref{def:intervals}.}
\label{fig:intervalplots}
\end{figure}

Other sets playing a prominent role in the book are the set of {\em irrational numbers} denoted by $\R\backslash\Q$ and {\em Euclidean spaces} denoted by $\R^m$ where $m$ is a positive integer. We have
\begin{align*}
\R\backslash\Q&=\{x:x\in\R \textnormal{ and } x\notin\Q\} \quad\textnormal{and}\\
\R^m&=\left\{
\bfx=\left[
	\begin{array}{c}
	x_1\\
	x_2\\
	\vdots\\
	x_m
\end{array}
\right]
: m\in\N \textnormal{ and } x_1, x_2,\ldots,x_m\in \R
\right\},
\end{align*}
where the real numbers $x_1,x_2,\ldots,x_m$ are the {\em coordinates} (or {\em components} or {\em entries}) of the {\em point} (or {\em vector}) $\bfx\in\R^m$. Similarly, the coordinates of a point $\bfy\in\R^m$ are denoted by $y_1,y_2,\ldots,y_m$. In every Euclidean space $\R^m$, the special {\em zero vector} $\mathbf{0}$ is the vector whose coordinates all are $0$.

There are many deep relationships between the sets described so far. For instance we have,
\begin{align}
\N \subseteq \Z \subseteq \Q \subseteq \R,
\end{align}
and each of these relationships represents a proper subset. 

\begin{remark}\label{rmk:settings}
Many of the results presented throughout are in the setting of an arbitrary Euclidean space $\R^m$ where the positive integer $m$ is left unspecified. This allows use to discuss and prove results in all of these spaces simultaneously. In particular, for $m=1$ we have $\R^1=\R$ (the real line) and for $m=2$ we have $\R^2$ (the plane). 

With that said, there are some important differences between these sets that influence the way I choose to state theorems and write proofs. For instance, the set of rational numbers $\Q$ has addition, multiplication, and inequalities that all play nicely together, but the set has gaps. The real line $\R$ has addition, multiplication, and inequalities that all play nicely together while having no gaps. (We will look at this much more closely in the next section.) 
\end{remark}

\begin{remark}\label{rmk:inequalitiesusingnorms}
We will often use the fact that $\R$ and $\R^m$ for any $m\in\N$ are \textit{vector spaces}. Basically, this means their points can be multiplied by constants (called \textit{scalars}) and their points can added together, both processes creating new points as follows: For any {\em scalar} $\alpha\in\R$ and any $\bfx,\bfy\in\R^m$ we have
\begin{align}
\alpha\bfx&=\left[
	\begin{array}{c}
	\alpha x_1\\
	\alpha x_2\\
	\vdots\\
	\alpha x_m
\end{array}
\right]
\qquad\textnormal{and}\qquad
	\bfx+\bfy=\left[
	\begin{array}{c}
	x_1+y_1\\
	x_2+y_2\\
	\vdots\\
	x_m+y_m
\end{array}
\right].
\end{align}
\end{remark}

\begin{notation}\label{not:distance}
Distance is a big, recurring theme in analysis. The types of distance we consider are limited to the standard notions in the real line $\R$ and Euclidean spaces of the form $\R^m$. In all cases, these distances stem from the standard {\em norms} and {\em metrics} on Euclidean spaces $\R^m$, denoted by $\|\cdot\|_m$ and $d_m$, respectively.\footnote{The concept of arbitrarily close can be explored in the much more general settings of metric spaces (where a wide variety of distances are taken into consideration) and even topological spaces (where notions of distance are not necessarily in play).}

Given a point $\bfx\in\R^m$, we have the {\em norm} of $\bfx$ given by
\begin{align}\label{eqn:norm}
	\|\bfx\|_m &= \sqrt{x_1^2+x_2^2+\cdots+x_m^2}.
\end{align}
Given a pair of points $\bfx,\bfy\in\R^m$, the {\em distance} between $\bfx$ and $\bfy$ is given by the classic Pythagorean distance formula:
\begin{align}\label{eqn:distance}
	d_m(\bfx,\bfy) &= \|\bfx -\bfy\|_m \nonumber \\
	& = \sqrt{(x_1-y_1)^2+(x_2-y_2)^2+\cdots+(x_m-y_m)^2}.
\end{align}

The notation for points, norms, and distances in $\R^m$ represents \textit{any} Euclidean space, including the real line $\R^1=\R$ and the plane $\R^2$ as special cases. When working in the real line $\R$, we use absolute value of the difference for distance:
\begin{align}\label{eqn:distanceinline}
	d_\R(x,y)=|x-y|.
\end{align}
\end{notation}

The distances $d_\R$ and $d_m$ are \textit{metrics}: They satisfy a special set of properties which will be used throughout the book. 

\begin{definition}\label{def:metric}
Suppose $S$ is a nonempty set. A function $d$ is a \textit{metric} on $S$ if it satisfies all of the following properties for any $x,y,$ and $z$ in $S$:
\begin{align}
	 &d(x,y) \geq 0;\label{eqn:metricnonnegative}\\
	 &d(x,y) = 0 \quad \Longleftrightarrow \quad x=y;\label{eqn:metriczero}\\
	 &d(x,y) = d(y,x);\label{eqn:metricsymmetric} \quad\textnormal{and}\\
	 &d(x,y) \leq  d(x,z) + d(z,y).\label{eqn:triangleinequality} 
\end{align}
Inequality \eqref{eqn:triangleinequality} is called the \textit{triangle inequality}.
\end{definition}


\begin{remark}\label{rmk:metricrecap}
To rephrase the defining properties of a metric:
\begin{itemize}
	\item \eqref{eqn:metricnonnegative} Distances given by a metric are nonnegative.
	\item \eqref{eqn:metriczero} The distance between two points is zero if and only if the points are the same.
	\item \eqref{eqn:metricsymmetric} Metrics are symmetric: The order does not matter.
	\item \eqref{eqn:triangleinequality} The distance between two points is the same or shorter than the total distance when an additional point is considered.
\end{itemize}
\end{remark}

Our notions of distance are metrics, but the proof of this fact is omitted.

\begin{theorem}\label{thm:distancesaremetrics}
The functions $d_\R$ and $d_m$ (for any positive integer $m$) are metrics on $\R$ and $\R^m$, respectively.
\end{theorem}

The following corollary stems from Theorem \ref{thm:distancesaremetrics} and, likewise, the proof is omitted.

\begin{corollary}\label{cor:distancesaremetrics}
For any scalar $\alpha\in\R$ and any three points $\bfx,\bfy,\bfc\in\R^m$, the triangle inequality \eqref{eqn:triangleinequality} combined with properties of vector spaces, norms, and metrics yield:
\begin{align}
\|\alpha\bfx\|_m &=|\alpha|\|\bfx\|_m;\label{eqn:scalarfactor}\\
\|\bfx-\bfc\|_m &=\|\bfx\underbrace{-\bfy+\bfy}_{\textnormal{add }\mathbf{0}}-\bfc\|_m\leq\|\bfx-\bfy\|_m+\|\bfy-\bfc\|_m;\label{eqn:addzero}\\
\|\bfx+\bfc\|_m &\leq\|\bfx\|_m+\|\bfc\|_m;
\label{eqn:triangleinequality2}\\
\|\bfx-\bfc\|_m &\leq\|\bfx\|_m+\|\bfc\|_m;
\label{eqn:triangleinequality3}
\quad\textnormal{and both}\\
\|\bfx\|_m-\|\bfc\|_m &\leq\|\bfx-\bfc\|_m
\quad\textnormal{and}\quad 
\|\bfc\|_m-\|\bfx\|_m\leq\|\bfx-\bfc\|_m
.\label{eqn:reversetriangleinequality}
\end{align} 
\end{corollary}

\begin{figure}
\centering
\begin{tikzpicture} 
	\draw (0,0) node {$\bullet$};
	\draw (-0.25,-0.3) node {$\bfx$};
	\draw (3,0) node {$\bullet$};
	\draw (3.25,-0.3) node {$\bfc$};
	\draw (2,1.5) node {$\bullet$};
	\draw (2,1.8) node {$\bfy$};	
	\draw[semithick] (0,0) --(3,0);
	\draw (1.5,-0.5) node {$\|\bfx-\bfc\|_m$};	
	\draw[semithick] (0,0) --(2,1.5);  
	\draw (0,1) node {$\|\bfx-\bfy\|_m$};
	\draw[semithick] (3,0) --(2,1.5);	
	\draw (3.5,1) node {$\|\bfy-\bfc\|_m$};	
\end{tikzpicture}
\caption{The version of the triangle inequality found in line \eqref{eqn:addzero}.}
\label{fig:addzero}
\end{figure}

\begin{remark}\label{rmk:triangleinequalities}
Inequality \eqref{eqn:addzero} is another version of the triangle inequality \eqref{eqn:triangleinequality} which, in my opinion, is appropriately named. The intermediate step in \eqref{eqn:addzero} amounts to adding $\mathbf{0}$ (the vector whose coordinates are all $0$) inside the norm before applying the triangle inequality \eqref{eqn:triangleinequality}. We'll use little techniques like this often when writing proofs. Adding a nice version of $\mathbf{0}$ in conjunction with the triangle inequality is particularly common. 

Inequalities \eqref{eqn:triangleinequality2} and \eqref{eqn:triangleinequality3} are also particular instances of the triangle inequality \eqref{eqn:triangleinequality}.

Inequalities \eqref{eqn:reversetriangleinequality} are each a version of the {\em reverse triangle inequality}. Combined, they are equivalent to 
\begin{align}
	|\|\bfx\|_m-\|\bfx\|_m|\leq\|\bfx-\bfc\|_m
.\label{eqn:reversetriangleinequalityabsolute}
\end{align}
\end{remark}

The following section focuses on introducing properties of the real line $\R$, especially the subtle notion of {\em completeness}.

\vs
\section*{Exercises}
\setcounter{theorem}{0}

Exercises are for play: Do scratch work, draw stuff, and make mistakes---make {\em lots} of mistakes---before worrying about writing proofs. {\em Have fun!}
%

\xca Prove Theorem \ref{thm:distancesaremetrics}. Draw figures for each statement, too.

\xca Prove Corollary \ref{cor:distancesaremetrics}. Draw figures for each statement, too.

\vs
\section{The real line $\R$ is a complete ordered field}
\label{sec:completeorderedfield}

To help motivate a rigorous investigation into real analysis, we will work with the following underlying assumption:

\begin{quote}
The real line $\R$ is a {\em complete ordered field}.
\end{quote}

But what does this mean? 

The goal of this section is to define each part of the phrase ``complete ordered field''. Loosely speaking: a {\em field} is a set of mathematical objects where both addition and muliplication are defined and play nicely together; a set is {\em ordered} if inequalities make sense in a concrete and (hopefully) familiar way; and a set is {\em complete} if it knows its limits in some specific sense. The assumption above amounts to the existence of a set with these three properties, and we call this set the real line $\R$. 

A formal description of what makes $\R$ an ordered field is provided by Axiom \ref{ax:orderedfield} below. I am assuming properties therein are familiar and will work them without explicitly citing them.

But what about completeness? Among the concepts of fields, order, and completeness, the latter is very much at the heart of real analysis but it is possibly the least familiar. We'll get to it in a bit.

\begin{axiom}[$\R$ is an ordered field]\label{ax:orderedfield}
There exists a set called the \textit{real line} or the \textit{set of real numbers} which is denoted by $\R$ and has the following properties involving addition, multiplication, and inequalities for any $x,y,z\in\R$: 
\begin{enumerate}
\item \textit{Commutativity}:\\
\indent $x+y=y+x$ and $xy=yx$.
\item \textit{Associativity}:\\
\indent $(x+y)+z=x+(y+z)$ and $(xy)z=x(yz)$.
\item \textit{Distributive property}:\\
\indent $x(y+z)=xy+xz$.
\item \textit{Additive identity}:\\
\indent There is a unique $0\in\R$ such that\\ 
\indent for every $x\in\R$ we have $x+0=x$ .
\item \textit{Additive inverse}:\\
\indent For each $x\in\R$, there is a unique $y\in\R$ where $x+y=0$.\\ \indent (We write $y=-x$.)
\item \textit{Multiplicative identity}:\\
\indent There is a unique $1\in\R$ such that $1\neq 0$ and\\
\indent for every $x\in\R$ we have $x(1)=x$.
\item \textit{Multiplicative inverse}:\\
\indent For each $x\in\R$ where $x\neq 0$, there is a unique $y\in\R$ where $xy=1$.\\ \indent  (We write $y=1/x$ or $y=x^{-1}$.)
\item \textit{Translation invariance}:\\
\indent If $x<y$, then $x+z<y+z$.
\item \textit{Transitivity}:\\
\indent If $x<y$ and $y<z$, then $x<z$.
\item \textit{Trichotomy}:\\
\indent For any $x,y\in\R$, exactly one of the following is true:\\
\indent $x=y$, $x<y$, or $x>y$. 
\item \textit{Multiplication inequality}:\\
\indent If $x<y$ and $z>0$, then $xz<yz$.
\end{enumerate}
\end{axiom}

Here are some consequences of the assumption that $\R$ is an ordered field (Axiom \ref{ax:orderedfield}). I believe these ideas are familiar from calculus and other courses, so the proof is left as an exercise for those who would like to explore the details. 

\begin{theorem}\label{thm:inequalityproperties}
For any $x,y,z\in\R$ we have:
\begin{enumerate}
\item If $x<y$, then $-y<-x$.
\item $0<1$.
\item If $0<x<y$, then $0<1/y<1/x$. 
\item If $x<y$ and $z<0$, then $xz>yz$.
\item $x^2\geq 0$.
\end{enumerate}
\end{theorem}

To get into what it means to say the real line $\R$ is complete, first consider my perspective on an important and subtle feature of the set of rational numbers $\Q$. I hope it helps you find your own perspective:

\begin{quote}
The set of rational numbers $\Q$ has arbitrarily small gaps.
\end{quote}

For instance, no rational number is the square root of 2. And yet there are rational numbers that are as close together as we like whose squares are less than and greater than 2, respectively. To make this pair of assertions more concrete, consider the following theorem and problem.

\begin{theorem}\label{thm:squarerootoftwo}
If $r$ is a rational number, then $r^2\neq2$.
\end{theorem}

Let's prove this with a classic contradiction argument. 

\begin{proof}
First, if $r=0$, then $r^2=0\neq 2$.
Now, to set up a contradiction, suppose $r$ is a nonzero rational number where $r^2=2$. We can write $r=m/n$ where $m$ and $n$ are integers that are not both even and neither is zero. We then have $r^2=m^2/n^2=2$ and so  $m^2=2n^2$, thus $m^2$ is even. This means $m$ itself is even as well since the square of an odd integer is odd.

As an even number, we can write $m=2k$ for some integer $k$. So now we have
\begin{align}
m^2=2n^2=4k^2.
\end{align}
The right hand side simplifies to
\begin{align}
n^2=2k^2.
\end{align}
This means $n^2$ is even, which implies $n$ is even, too. 

We have arrived at our contradiction: We assumed $m$ and $n$ are not both be even, yet both must be even when $r^2=m^2/n^2=2$. Therefore, 2 is not the square of any rational number.
\end{proof}

Loosely speaking, Theorem \ref{thm:squarerootoftwo} says that the set  of rational numbers $\Q$ has a gap at the square root of 2. And yet, the size of that gap is arbitrarily small. 

\begin{prob}\label{prob:closetoroottwo}
Consider the pair of rational numbers 
\begin{align}
a_1=1.5=15/10 \quad\textnormal{and}\quad b_1=1.4=14/10.
\end{align} We have 
\begin{align}
|a_1-b_1|=|1.5-1.4|=0.1=1/10
\end{align}
while 
\begin{align}
b_1^2=1.96<2<2.25=a_1^2.
\end{align}
Next, consider $a_2=1.42=142/100$ and $b_2=1.41=141/100$. Then
\begin{align}
|a_2-b_2|=|1.42-1.41|=0.01=1/100
\end{align}
while
\begin{align}
b_2^2=1.9881<2<2.0164=a_2^2.
\end{align}
The process results in something like Figure \ref{fig:closetoroottwo}.

\begin{figure}
\centering
\begin{tikzpicture}
\draw (-3,0) node {$\bullet$};
\draw (-3.1,-0.5) node {$b_1$};
\draw (-2.4,0) node {$\bullet$};
\draw (-2.5,-0.5) node {$b_2$};
\draw (3,0) node {$\bullet$};
\draw (3,-0.5) node {$a_1$};
\draw (-1.8,0) node {$\bullet$};
\draw (-1.7,-0.5) node {$a_2$};
\draw (-2.1,0) node {$\circ$};
\draw (-2.1,-0.5) node {?};
\end{tikzpicture}
\caption{The rational numbers $a_1,a_2,b_1$, and $b_2$ from Problem \ref{prob:closetoroottwo} giving initial approximations to the square root of two.}
\label{fig:closetoroottwo}
\end{figure}

Try to describe how to continue finding pairs of rational numbers $a_n$ and $b_n$ where, for each positive integer $n$,  we have 
\begin{align}
|a_n-b_n|=1/10^n \quad \textnormal{and}\quad b_n^2<2<a_n^2.
\end{align}
There's no need to find formulas for $a_n$ and $b_n$, but you should try to describe any process you use. 
\end{prob}

\begin{remark}\label{rmk:arbitrarilysmall}
I believe that, intuitively, $1/10^n$ can be made as small as we like by taking a positive integer $n$ to be as large as we need. So the inclusion of the bound $1/10^n$ in Problem \ref{prob:closetoroottwo} is my way of indicating the size of the gap in the rationals at the square root of 2 is indeed arbitrarily small. We can and will prove this concretely once we have more mathematical tools---definitions, theorems, etc.---at our disposal.
\end{remark}

Between any two rational numbers there is another, despite gaps like the square root of 2.

\begin{lemma}\label{lem:rationalbetweenrationals}
Suppose $p$ and $q$ are rational numbers where $p<q$. Then there is some rational number $r$ where $p<r<q$.
\end{lemma}

Taking $r$ to be the average of $p$ and $q$ gives us the result.

\begin{proof} 
Suppose $p=m/n$ and $q=s/t$ where $m,n,s,$ and $t$ are integers where $n$ and $t$ are nonzero. Let 
\begin{align}
r=\frac{p+q}{2}.
\end{align}
With $p<q$ and after applying some algebra we have
\begin{align}
p=\frac{p+p}{2}<\frac{p+q}{2}<\frac{q+q}{2}=q,
\end{align} 
and so $p<r<q$. Also, by finding a common denominator we get
\begin{align}
r=\frac{p+q}{2}=\frac{mt+ns}{2nt}.
\end{align}
Since $mt+ns$ and $2nt$ are integers and $2nt$ is nonzero, we have $r$ is rational.
\end{proof}

To recap, the set of rational numbers $\Q$ has arbitrarily small gaps and yet between any two there are always more. So what about the set of real numbers $\R$?

\begin{remark}\label{rmk:reallineisrationalswithgaps}
The assumption that the real line $\R$ is complete is a way to ensure, from the start, that it has no gaps. In my opinion: 
\begin{quote}
The real line $\R$ is the set of rational numbers and the arbitrarily small gaps between them.
\end{quote} 
There are a number of interesting, robust, and beautiful ways to construct the real line $\R$ from the set of rational numbers $\Q$. But to me they all amount to starting with the rational numbers then identifying the gaps from different perspectives.\footnote{Dedekind cuts use sets to partition $\Q$ in nice ways to identify the gaps. Equivalence classes of certain sequences (called Cauchy sequences) of rational numbers can be used to identify the gaps in a very different way.} Identifying an arbitrarily small gap in the rational numbers amounts to identifying an {\em irrational} number. 
\end{remark}

Finally, here is a formal description of what it means for the real line $\R$ to be complete. 

\begin{axiom}[Axiom of Completeness]
\label{ax:axiomofcompleteness}
Every nonempty subset of the real line that is bounded above has a supremum.
\end{axiom}

The Axiom of Completeness \ref{ax:axiomofcompleteness} formally describes what it means that the real line $\R$ has no arbitrarily small gaps (unlike the set of rational numbers $\Q$, see Problem \ref{prob:closetoroottwo}). However, you're not expected  to see why this means the real line has no gaps just yet. This perspective is a goal to be approached gradually as we explore the structure of the real line. 

Also, the Axiom of Completeness \ref{ax:axiomofcompleteness} ensures every rational number and every arbitrarily small gap in the rational numbers can be identified as a real number: Each is the supremum of a set of rational numbers that is bounded above. 

A more conventional way to set up an Axiom of Completeness is to use the notion of a {\em least upper bound} to define supremum (unlike, but equivalent to, Definition \ref{def:supremumacl}). Convention is often and deliberately broken throughout the book.

\begin{definition}\label{def:leastupperbound}
Suppose $S$ is a nonempty subset of the real line $\R$. A real number $b$ is the {\em least upper bound} of $S$ if 
\begin{itemize}
	\item[(i)] for every $x\in S$ we have $x\leq b$, and
	\item[(ii)] if $y<b$, then $y$ is not an upper bound for $S$.
\end{itemize}
\end{definition}
Statement (i) says $b$ is an upper bound for $S$ while (ii) says no real number smaller than $b$ is an upper bound for $S$. When a real number $y$ is not an upper bound for $S$, there is some $x_y$ in $S$ where $y$ is less than $x_y$. See Figure \ref{fig:leastupperbound}.

\begin{figure}
\centering
\begin{tikzpicture} 
\draw (-2,0) node {$S$};
	\draw[-,semithick] (0,0) -- (2,0);
	\draw (0,0) node {$[$};
	\draw (2,0) node {$)$};
\draw (4,0) node {$\circ$};
\draw (4,-0.45) node {$b$};
\draw (3.76,0) node {...};
\foreach \Point in {(3,0), (3.5,0)}
{
    \node at \Point {\textbullet};
}	
\draw (3.5,-0.55) node {$x_y$};
\draw (2.5,0) node {$\circ$};
\draw (2.4,-0.5) node {$y$};
\end{tikzpicture}
\caption{As in Definition \ref{def:leastupperbound}, when $b$ is the least upper bound of a set $S$ and $y<b$, $y$ is not an upper bound for $S$. So, there must be a point $x_y$ in $S$ where $y<x_y$.}
\label{fig:leastupperbound}
\end{figure}

On the other hand, any real number that's greater than $b$ is another upper bound for $S$. Thus, (i) and (ii) combine to identify $b$ as the \textit{unique} least upper bound for $S$. {\em Try drawing stuff!}

The concept of supremum as presented in Definition \ref{def:supremumacl} is equivalent to that of least upper bound in Definition \ref{def:leastupperbound}, so the Axiom of Completeness \ref{ax:axiomofcompleteness} is equivalent to more conventional notions found in other texts such as \cite{Abbott} and \cite{Rudin}. The following theorem codifies this equivalence.

\begin{theorem}\label{thm:equivalentaxiomofcompleteness}
Suppose $u\in\R$ is an upper bound for a set $S\subseteq\R$. Then $u=\sup{S}$ if and only if $u$ is the least upper bound of $S$. 
\end{theorem}

See Figure \ref{fig:equivalentaxiomcompleteness} for a set of real numbers $S$, which is neither an interval nor a sequence, along with its supremum $u$.

\begin{figure}
\centering
\begin{tikzpicture}
\draw (-2,0) node {$S$};
	\draw[-,semithick] (0,0) -- (2,0);
	\draw (0,0) node {$[$};
	\draw (2,0) node {$)$};
\draw (4,0) node {$\circ$};
\draw (4,-0.5) node {$u$};
\draw (3.76,0) node {...};
\foreach \Point in {(3,0), (3.5,0)}
{
    \node at \Point {\textbullet};
}	
\end{tikzpicture}
\caption{A set of real numbers $S$ along with its supremum $u$  which is also its least upper bound. See Theorem \ref{thm:equivalentaxiomofcompleteness}.}
\label{fig:equivalentaxiomcompleteness}
\end{figure}

The proof of Theorem \ref{thm:equivalentaxiomofcompleteness} uses the definitions of absolute value, arbitrarily close, supremum, and least upper bound (Definitions \ref{def:absolutevalue}, \ref{def:aclreal}, \ref{def:supremumacl}, and \ref{def:leastupperbound}, respectively). For this particular proof, my goal is to be thorough and indicate where some definitions are used as clearly as I can, though this makes the proof longer than it would otherwise need to be. That's okay, especially this early in the book.

\begin{proof} First, assume $u$ is an upper bound for $S$ and $u=\sup{S}$. Suppose $y<u$ with the goal to show that $y$ is not an upper bound for $S$. Then $u-y>0$ and we let $\varepsilon_y=u-y$, thinking of this as a specific distance or ``error'' we want to allow. See Figure \ref{fig:supanddistance}.

\begin{figure}
\centering
\begin{tikzpicture} 
\draw (-2,0) node {$S$};
	\draw[-,semithick] (0,0) -- (2,0);
	\draw (0,0) node {$[$};
	\draw (2,0) node {$)$};
\draw (4,0) node {$\circ$};
\draw (4,-0.5) node {$u$};
\draw (3.76,0) node {...};
\foreach \Point in {(3,0), (3.5,0)}
{
    \node at \Point {\textbullet};
}	
\draw (2.5,0) node {$\circ$};
\draw (2.4,-0.5) node {$y$};
\draw[thick,  dashed,  red] (4,0) circle (1.5cm); 
\draw[-,semithick, red] (4.05,0.05) -- (5.05,1.05);
\draw[red] (4.8,0.3) node {$\varepsilon_y$};
\end{tikzpicture}
\caption{As in the proof of Theorem \ref{thm:equivalentaxiomofcompleteness}, any real number $y$ strictly less than the supremum $u$ of a set $S$ creates a positive distance $\varepsilon_y=u-y$.}
\label{fig:supanddistance}
\end{figure}

Since $u\acl{S}$ ((ii) in Definition \ref{def:supremumacl}), by the definition of arbitrarily close (Definition \ref{def:aclreal}) there must be some $x_y$ in $S$ where
\begin{align}
d_\R(x_y,u)=|x_y-u|<\varepsilon_y=u-y.
\end{align}
See Figure \ref{fig:supdistanceandpoint}.

\begin{figure}
\centering
\begin{tikzpicture} 
\draw (-2,0) node {$S$};
	\draw[-,semithick] (0,0) -- (2,0);
	\draw (0,0) node {$[$};
	\draw (2,0) node {$)$};
\draw (4,0) node {$\circ$};
\draw (4,-0.5) node {$u$};
\draw (3.76,0) node {...};
\foreach \Point in {(3,0), (3.5,0)}
{
    \node at \Point {\textbullet};
}	
\draw (3.5,-0.55) node {$x_y$};
\draw (2.5,0) node {$\circ$};
\draw (2.4,-0.5) node {$y$};
\draw[thick,  dashed,  red] (4,0) circle (1.5cm); 
\draw[-,semithick, red] (4.05,0.05) -- (5.05,1.05);
\draw[red] (4.8,0.3) node {$\varepsilon_y$};
\end{tikzpicture}
\caption{As in the proof of Theorem \ref{thm:equivalentaxiomofcompleteness}, since $u\acl{S}$, there must be a point $x_y$ in $S$ within $\varepsilon_y=u-y$ of $u=\sup{S}$.}
\label{fig:supdistanceandpoint}
\end{figure}

Since $u$ is also an upper bound for $S$, we have $x\leq u$ for every $x$ in $S$. So $x_y-u\leq 0$, and by the definition of absolute value (Definition \ref{def:absolutevalue}) we have
\begin{align}
|x_y-u|=-(x_y-u)=u-x_y<u-y.
\end{align}
By subtracting $u$ from the right hand side of the inequality above then multiplying the result by $-1$, we get $y<x_y$. (This might seem clear from Figure \ref{fig:supdistanceandpoint}, but the figure falls just short for the purposes of this proof.) Thus, as we wanted to show, $y$ is not an upper bound for $S$. That is, no real number smaller than $u$ is an upper bound for $S$. So, $u$ satisfies both parts of Definition \ref{def:leastupperbound} and is the least upper bound of $S$.

Next, for the other direction of the proof, assume $u$ is an upper bound for $S$ that is also the least upper bound. Now let $\varepsilon>0$. (We let $\varepsilon>0$ to allow for any amount of ``error''  and set up a verification of the definition of arbitrarily close, Definition \ref{def:aclreal}). Then we have $u-\varepsilon < u$. See Figure \ref{fig:lubanddistances}.

\begin{figure}
\centering
\begin{tikzpicture}
\draw (-2,0) node {$S$};
	\draw[-,semithick] (0,0) -- (2,0);
	\draw (0,0) node {$[$};
	\draw (2,0) node {$)$};
\draw (4,0) node {$\circ$};
\draw (4,-0.5) node {$u$};
\draw (3.76,0) node {...};
\foreach \Point in {(3,0), (3.5,0)}
{
    \node at \Point {\textbullet};
}	
\draw (2.5,0) node {$\circ$};
\draw (2.05,-0.5) node {$u-\varepsilon$};
\draw[thick,  dashed,  red] (4,0) circle (1.5cm); 
\draw[-,semithick, red] (4.05,0.05) -- (5.05,1.05);
\draw[red] (4.95,0.6) node {$\varepsilon$};
\draw[thick,  dashed,  red] (4,0) circle (0.7cm); 
\draw[thick,  dashed,  red] (4,0) circle (0.3cm); 
\end{tikzpicture}
\caption{As in the proof of Theorem \ref{thm:equivalentaxiomofcompleteness}, the least upper bound $u$ is within every positive distance $\varepsilon$ of the set $S$. Only one $\varepsilon$ is drawn to keep things from getting too cluttered.}
\label{fig:lubanddistances}
\end{figure}

Since $u$ is the {\em least} upper bound of $S$, $u-\varepsilon$ is not an upper for $S$. This means there is a real number $x_\varepsilon$ in $S$ where 
\begin{align}\label{eqn:pointwithinepsilonofsup}
u-\varepsilon<x_\varepsilon.
\end{align}
See Figure \ref{fig:pointwithinepsilonoflub}. 

\begin{figure}
\centering
\begin{tikzpicture}
\draw (-2,0) node {$S$};
	\draw[-,semithick] (0,0) -- (2,0);
	\draw (0,0) node {$[$};
	\draw (2,0) node {$)$};
\draw (4,0) node {$\circ$};
\draw (4,-0.5) node {$u$};
\draw (3.76,0) node {...};
\foreach \Point in {(3,0), (3.5,0)}
{
    \node at \Point {\textbullet};
}	
\draw (3,-0.5) node {$x_\varepsilon$};
\draw (2.5,0) node {$\circ$};
\draw (2.05,-0.5) node {$u-\varepsilon$};
\draw[thick,  dashed,  red] (4,0) circle (1.5cm); 
\draw[-,semithick, red] (4.05,0.05) -- (5.05,1.05);
\draw[red] (4.95,0.6) node {$\varepsilon$};
\draw[thick,  dashed,  red] (4,0) circle (0.7cm); 
\draw[thick,  dashed,  red] (4,0) circle (0.3cm); 
\end{tikzpicture}
\caption{As in the proof of Theorem \ref{thm:equivalentaxiomofcompleteness}, $x_\varepsilon$ is only close enough to the least upper bound $u$ for the largest value of $\varepsilon$, as shown. It is not close enough for smaller values of $\varepsilon$ (thus the smaller radii), but the proof only requires us to find one $x_\varepsilon$ in the set $S$ for each $\varepsilon$ distance, separately.}
\label{fig:pointwithinepsilonoflub}
\end{figure}

Rearranging inequality \eqref{eqn:pointwithinepsilonofsup} slightly yields
\begin{align}
u-x_\varepsilon<\varepsilon.
\end{align}
Since $u$ is an upper bound for $S$ and $x_\varepsilon$ is in $S$ we have $x_\varepsilon\leq u$, and so $u-x_\varepsilon\geq 0$. By the definition of absolute value (Definition \ref{def:absolutevalue}), we have 
\begin{align}
|x_\varepsilon-u|=-(x_\varepsilon-u)=u-x_\varepsilon<\varepsilon.
\end{align}
We have shown $u$ is an upper bound for $S$ and  arbitrarily close to $S$. Therefore, $u=\sup{S}$ according to Definition \ref{def:supremumacl}.
\end{proof}

To summarize the above proof: When $u$ is the supremum of $S$, no smaller number can be an upper bound for $S$, so $u$ is the least upper bound. Also, when $u$ is the least upper bound for $S$, it must be arbitrarily close to $S$ since otherwise some smaller number would be an upper bound for $S$.

\begin{remark}\label{rmk:walkthrough}
Did the proof of Theorem \ref{thm:equivalentaxiomofcompleteness} make sense to you? \textit{Every line of it?} At first, reading and understanding proofs in analysis can take \textit{a long time}. Please be patient. There are often so many details to parse and double-check, or so many steps left out, that it can be difficult to get a feeling for what's going on. But this is okay, and to me learning how to deal with parsing and understanding proofs is part of the development of every mathematician.

When you're confronted with a proof you find hard to follow, I suggest you write up a {\em walkthrough}. This can and should be whatever you want it to be, but basically a walkthrough should reflect your own thoughts and perspectives on the proof you've been given. Maybe drawing a figure will help, or writing out some algebraic steps that were left out, or writing out the definition or statement of a theorem when it is cited in the proof, or just rewriting the proof in your own words. If you can't see why a certain step works, don't hesitate to ask a friend, a professor, or even me what's going. If you email me about analysis, I'll email you back! 
\end{remark}

\vs
\section*{Exercises}
\setcounter{theorem}{0}

Exercises are for play: Do scratch work, draw stuff, and make mistakes---make {\em lots} of mistakes---before worrying about writing proofs. {\em Have fun!}


\xca Prove that for nonempty bounded sets of real numbers $A$ and $B$ we have
\[
A\subseteq B \qquad\Longrightarrow\qquad  \sup A\leq \sup B \quad and \quad \inf B\leq \inf A.
\]
Also, find various examples where the inequalities are strict and others where equality holds.

\xca Prove Theorem \ref{thm:inequalityproperties}. Hint: Make heavy use of Axiom \ref{ax:orderedfield}.

\vs
\section{Implications of completeness}
\label{sec:implicationsofcompleteness}

What does completeness do for us? Generally speaking, it provides a way for us to ensure the \textit{existence} of real numbers: They could rational numbers, positive integers, or something else, but the big idea is for us to be able to assert their existence when the time is right. 

This section explores a collection of results stemming from the Axiom of Completeness \ref{ax:axiomofcompleteness}, including its mirror image in terms of infimum.

\begin{theorem}\label{thm:infimumexists}
Every nonempty set of real numbers that is bounded below has an infimum.
\end{theorem}

An idea behind the proof of Theorem \ref{thm:infimumexists} is how multiplying real numbers by $-1$ flips them around 0 and reverses the order. So, we can identify an infimum of one set as the negative of the supremum of another.

\begin{proof}
Suppose $S$ is a set of real numbers that is bounded below by $a$. Let  $T=\{-s:s\in S\}$. Loosely speaking, $T$ is the ``mirror image'' of $S$ comprising the negatives of every point in $S$. See Figure \ref{fig:infimumexists}. 

\begin{figure}
\centering
\begin{tikzpicture}
\draw (-2,0) node {$S$};
	\draw (-0.5,0) node {$\circ$};
	\draw (-0.5,-0.5) node {$a$};
	\draw[-,semithick] (0,0) -- (2,0);
	\draw (0,0) node {$[$};
	\draw (2,0) node {$)$};
	\draw (2,-0.5) node {$0$};
\draw (4,0) node {$\circ$};
\draw (3.76,0) node {...};
\foreach \Point in {(3,0), (3.5,0)}
{
    \node at \Point {\textbullet};
}	
\draw (-2,-1.5) node {$T$};
	\draw (4.5,-1.5) node {$\circ$};
	\draw (4.5,-2) node {$-a$};
	\draw[-,semithick] (2,-1.5) -- (4,-1.5);
	\draw (2.02,-1.5) node {$($};
	\draw (2,-2) node {$0$};	
	\draw (4,-1.5) node {$]$};
\draw (0,-1.5) node {$\circ$};
\draw (0.24,-1.5) node {...};
\foreach \Point in {(1,-1.5), (0.5,-1.5)}
{
    \node at \Point {\textbullet};
}
\end{tikzpicture}
\caption{A set of real numbers $S$ along with a lower bound $a$ and their ``mirror images'' $T$ and $-a$. See the proof of Theorem \ref{thm:infimumexists}.}
\label{fig:infimumexists}
\end{figure}

Since $a\leq s$ for every $s\in S$, by part (i) of Theorem \ref{thm:inequalityproperties} we have $-s\leq -a$, so $-a$ is an upper bound for $T$. By the Axiom of Completeness \ref{ax:axiomofcompleteness}, $T$ has a supremum, so let $u=\sup{T}$.\s

\begin{figure}
\centering
\begin{tikzpicture}	
\draw (-2,0) node {$T$};
	\draw[-,semithick] (2,0) -- (4,0);
	\draw (2.02,0) node {$($};
	\draw (2,-0.5) node {$0$};	
	\draw (4,0) node {$]$};
	\draw (4,-0.5) node {$u$};	
\draw (0,0) node {$\circ$};
\draw (0.24,0) node {...};
\foreach \Point in {(1,0), (0.5,0)}
{
    \node at \Point {\textbullet};
}
\draw (-2,-1.5) node {$S$};
	\draw[-,semithick] (0,-1.5) -- (2,-1.5);
	\draw (0,-1.5) node {$[$};
	\draw (0,-2) node {$-u$};	
	\draw (2,-1.5) node {$)$};
	\draw (2,-2) node {$0$};
\draw (4,-1.5) node {$\circ$};
\draw (3.76,-1.5) node {...};
\foreach \Point in {(3,-1.5), (3.5,-1.5)}
{
    \node at \Point {\textbullet};
}
\end{tikzpicture}
\caption{The set of real numbers $T$ along with its supremum $u$ and their ``mirror images'' $S$ and $-u$, where $-u$ is the infimum of $S$. See the proof of Theorem \ref{thm:infimumexists}.}
\label{fig:infimumfromsupremum}
\end{figure}

For the final step we show $-u=\inf{S}$. Since $u$ is an upper bound for $T$, we have $-s\leq u$ for every $s$ in $S$. Hence, $-u\leq s$ for every $s$ in $S$, so $-u$ is a lower bound for $S$. See Figure \ref{fig:infimumfromsupremum}. We also have $u\acl{T}$, so for every $\varepsilon>0$ there is some $x_\varepsilon$ in $S$ where 
\begin{align}
|u-(-x_\varepsilon)|=|(-u)-x_\varepsilon|<\varepsilon.
\end{align}
Hence, $-u\acl{S}$ and we have $-u=\inf{S}$.
\end{proof}

Theorem \ref{thm:equivalentaxiomofcompleteness} has a mirror image as well, this one is in terms of infimum and {\em greatest lower bound}.

\begin{definition}\label{def:greatestlowerbound}
Suppose $S$ is a nonempty subset of the real line $\R$. A real number $w$ is the {\em greatest lower bound} of $S$ if 
\begin{itemize}
	\item[(i)] for every $x\in S$ we have $w\leq x$, and
	\item[(ii)] if $w<y$, then $y$ is not a lower bound for $S$.
\end{itemize}
\end{definition}

\begin{theorem}\label{thm:equivalentinfimum}
Suppose $\ell\in\R$ is a lower bound for a set $S\subseteq\R$. Then $\ell=\inf{S}$ if and only if $\ell$ is the greatest lower bound of $S$. 
\end{theorem}

Theorem \ref{thm:equivalentinfimum} is so similar to Theorem \ref{thm:equivalentaxiomofcompleteness} even their proofs are mirror images of each other. So the proof of Theorem \ref{thm:equivalentinfimum} is left as an exercise. A careful consideration and modification of one of the proofs can lead to a proof of the other, making for some good practice.

An intuitive idea is that given any real number, there is a larger positive integer. Actually, we already used this idea in the proofs for Example \ref{eg:twocountablesets}. 

\begin{theorem}[Archimedean Property]\label{thm:archimedeanproperty}
Given a real number $x$ there is some positive integer $n_x$ where $x<n_x$.
\end{theorem}

\begin{proof}
To set up a contradiction, let $x\in\R$ and assume the set of positive integers $\N$ is bounded above by $x$ so that $n\leq x$ for every $n$ in $\N$. By the Axiom of Completeness \ref{ax:axiomofcompleteness}, $\N$ has a supremum $u=\sup\N$.  Since $7>0$ and $u$ is both an upper bound for $\N$ and arbitrarily close to $\N$, there must be some $n\in\N$ where 
\begin{align}\label{eqn:bothupperboundandacl}
|u-n|=u-n<7.
\end{align}
Rearranging the inequality yields $u<n+7$, which implies $u$ is \textit{not} an upper bound for $\N$ since $n+7$ is also a positive integer. This contradicts the assertion that $u=\sup\N$, meaning $x$ cannot be an upper bound for $\N$. Therefore, there must be some $n_x$ in $\N$ where $x<n_x$.
\end{proof}

\begin{remark}\label{rmk:seven}
Why $7$? Did $7$ play a special role in the proof of the Archimedean Property \ref{thm:archimedeanproperty} somehow? No, not really. It was good enough and \textit{that's all I needed}. While preparing the scratch work for the previous proof, there was a choice I was free to make. I needed a positive number to play the role of $\varepsilon$ from the definition of arbitrarily close (Definition \ref{def:aclreal}) so the inequality in line \eqref{eqn:bothupperboundandacl} is valid, and I needed this positive number to be an integer so that $n+7$ is also a positive integer. So  $7$ was good enough, but any positive integer could have played the role just as well.
\end{remark}

\textit{There is no smallest positive real number}. This subtle and perhaps surprising fact follows immediately from the Archimedean Property \ref{thm:archimedeanproperty}. A formal statement is provided by the next corollary which also gives us a way to guarantee there is a positive integer large enough that its reciprocal (a rational number) is as small as we like. See Figure \ref{fig:archimedeanproperty}.

\begin{figure}
\centering
\begin{tikzpicture}	
\draw (0,0) node {$\circ$};
\draw (0,-0.3) node {$0$};
\draw (0.4,0.01) node {......};
\foreach \Point in {(5,0), (2.5,0), (1.67,0), (1.25, 0), (1,0), (0.83,0)}
{
    \node at \Point {\textbullet};
}
\draw (1.8,-0.5) node {$1/n_\varepsilon$};
\begin{scope}[thick, dashed, red] 
\draw (0,0) circle (2.3cm); 
\draw (0,0) circle (1.4cm); 
\draw (0,0) circle (0.7cm);
\end{scope}	
\draw[-,semithick, red] (0.05,0.05) -- (1.63,1.63);
\draw[red] (1.5,1.2) node {$\varepsilon$};
\end{tikzpicture}
\caption{Corollary \ref{cor:archimedeanproperty} shows there is no smallest real number. In particular, no matter how small a positive real number $\varepsilon$ is, there's a rational number of the form $1/n_\varepsilon$ that's less than $\varepsilon$.}
\label{fig:archimedeanproperty}
\end{figure}

\begin{corollary}\label{cor:archimedeanproperty}
Given $\varepsilon>0$, there is a positive integer $n_\varepsilon$ where 
\begin{align}
0<\frac{1}{n_\varepsilon}<\varepsilon.
\end{align}
\end{corollary}

\begin{proof}
Suppose $\varepsilon>0$. Then $1/\varepsilon$ is a real number and we can apply the Archimedean Property \ref{thm:archimedeanproperty}. So there is a positive integer $n_\varepsilon$ such that $1/\varepsilon<n_\varepsilon$. Using some algebraic properties of inequalities and noting $n_\varepsilon$ is positive yields
\begin{align}
0<\frac{1}{n_\varepsilon}<\varepsilon.
\end{align}
See Figure \ref{fig:archimedeanproperty}.
\end{proof}

Yet another intuitive idea we can formally prove is the notion that every real number is between two consecutive integers. See Figure \ref{fig:betweenconsecutiveintegers}.

\begin{corollary}\label{cor:betweenconsecutiveintegers}
For every $x\in\R$ there is some $m_x\in\Z$ such that
\begin{align}\label{eqn:betweenconsecutiveintegers}
m_x\leq x<m_x+1.
\end{align}
\end{corollary}

\begin{figure}
\centering
\begin{tikzpicture}
\draw (-3,0) node {$(-\infty,\infty)$};
	\draw[<->,semithick] (-0.5,0) -- (6.5,0);
	\draw (3,0) node {$\bullet$};
	\draw (3,-0.5) node {$x$};		
\draw (-3,-1.5) node {$[m_x,m_x+1)$};
	\draw[-,semithick] (2,-1.5) -- (4.5,-1.5);
	\draw (3,-1.5) node {$\bullet$};
	\draw (3,-2) node {$x$};
	\draw (2,-1.5) node {$[$};
	\draw (2,-2) node {$m_x$};
	\draw (4.48,-1.5) node {$)$};
	\draw (4.5,-2) node {$m_x+1$};
\end{tikzpicture}
\caption{Each given real number $x$ lies between a pair of consecutive integers $m_x$ and $m_x+1$. See Corollary \ref{cor:betweenconsecutiveintegers}.}
\label{fig:betweenconsecutiveintegers}
\end{figure}

The proof addresses a subtlety: The set to which we want to apply the Axiom of Completeness may or may not be empty.

\begin{proof}
Let $x\in\R$ and let $S_x=\{z\in\Z:z\leq x\}$. Since $-x\in\R$, by the Archimedean Property \ref{thm:archimedeanproperty}, there is an $n_{-x}\in\N$ where $-x<n_{-x}$. Hence, $-n_{-x}<x$ and $-n_{-x}\in\Z$. Thus, $-n_{-x}\in S_x$, so $S_x$ is nonempty and bounded above by $x$.

Now, since $S_x$ is nonempty and bounded above, by the Axiom of Completeness \ref{ax:axiomofcompleteness} we have $u=\sup{S_x}$ exists. By Theorem \ref{thm:equivalentaxiomofcompleteness}, $u$ is the least upper bound for $S_x$, which implies $u-1$ is not an upper bound for $S_x$. So, there is some $m_x\in S_x\subseteq\Z$ where $u-1<m_x$. This implies $u<m_x+1$, and since $u$ is an upper bound for $S_x$, we have $m_x+1$ is an integer that is not in $S_x$. Hence, $x<m_x+1$. Therefore, $m_x$ is an integer where $m_x\leq x<m_x+1$.
\end{proof}

No matter how close two distinct real numbers are, there is always a rational number between them. 

\begin{theorem}[Density of $\Q$ in $\R$]\label{thm:densityofQinR}
Let $x,y\in\R$ where $x<y$. Then there is some $r\in\Q$ where $x<r<y$.
\end{theorem}

The proof below has some steps which may not make sense on a first reading. So, before that, let me share some of my scratch work with you. 

\begin{scratch}\label{scr:densityofQinR}
The goal is to come up with an integer $m$ and a positive integer $n$ where 
\begin{align}\label{eqn:densitygoal}
x<r=\frac{m}{n}<y.
\end{align}
How can we prove this in a mathematically rigorous way? By that I mean, can we get the result we want by relying on the Axioms, Definitions, Theorems, Corollaries, and other technical statements we've come across so far? 

We can use Corollary \ref{cor:archimedeanproperty} to find an $n\in\N$ where $1/n$ is small enough to be less than the distance between $x$ and $y$, that is
\begin{align}\label{eqn:1overnlessthan}
1/n<|x-y|=y-x.
\end{align} 
See Figure \ref{fig:1overnlessthan}.

\begin{figure}
\centering
\begin{tikzpicture} 
	\draw[<->,semithick] (-1,0) -- (7,0);
	\draw (3,0) node {$\bullet$};
	\draw (3,-0.5) node {$x$};
	\draw (4,0) node {$\bullet$};
	\draw (4,-0.5) node {$y$};	
	\draw[-,semithick] (3.25,0.4) -- (3.75,0.4);
	\draw (3.25,0.4) node {$|$};
	\draw (3.75,0.4) node {$|$};		
	\draw (3.5,.9) node {$1/n$};
\end{tikzpicture}
\caption{The distance between any two real numbers $x$ and $y$ is larger than the reciprocal $1/n$ of some positive integer $n$. See line \eqref{eqn:1overnlessthan} in Scratch Work \ref{scr:densityofQinR}.} 
\label{fig:1overnlessthan}
\end{figure}

But we still need a numerator $m\in\Z$ to produce the inequalities in line \eqref{eqn:densitygoal}. We can aim for $m$ to specifically give us 
\begin{align}\label{eqn:mnxy}
\frac{m-1}{n}&\leq x <\frac{m}{n}<y.
\end{align}
Such a numerator $m$ will be furnished by Corollary \ref{cor:betweenconsecutiveintegers}, stemming from a real number that lines up nicely for the conclusion of the proof. But which real number exactly? Solving for $m$ in the leftmost inequality in \eqref{eqn:mnxy} yields  $m\leq nx+1$. Now for the proof. 
\end{scratch}

\begin{proof}[Proof of Theorem \ref{thm:densityofQinR}]
Suppose $x,y\in\R$ where $x<y$. Then $y-x>0$, so Corollary \ref{cor:archimedeanproperty} guarantees there is an $n\in\N$ where
$1/n<y-x$. We can rearrange the inequality to get
\begin{align}\label{eqn:densitymiddlestep}
nx+1<ny.
\end{align}

Since $nx+1$ is a real number, it is between consecutive integers. See Figure \ref{fig:betweenm}. By Corollary \ref{cor:betweenconsecutiveintegers}, there is an $m\in\Z$ where 
\begin{align}\label{eqn:betweenm}
m\leq nx+1<m+1.
\end{align}

\begin{figure}
\centering
\begin{tikzpicture} 
	\draw[-,semithick] (0,0) -- (6,0);
	\draw (0,0) node {$[$};
	\draw (0,-0.5) node {$m$};
	\draw (2,0) node {$\bullet$};
	\draw (2,-0.5) node {$nx+1$};
	\draw (5.98,0) node {$)$};
	\draw (6,-0.5) node {$m+1$};	
\end{tikzpicture}
\caption{As in line \eqref{eqn:betweenm} from the proof of Theorem \ref{thm:densityofQinR}, the real number $nx+1$ is between two consecutive integers $m$ and $m+1$.}
\label{fig:betweenm}
\end{figure}

\noindent The right-hand side of \eqref{eqn:betweenm} implies $nx<m$, so $x<m/n$. In fact,
\begin{align}\label{eqn:betweenrationals}
\frac{m-1}{n}\leq x<\frac{m}{n}.
\end{align}
See Figure \ref{fig:betweenrationals}.

\begin{figure}
\centering
\begin{tikzpicture}
	\draw[<->,semithick] (-1,0) -- (7,0);
	\draw (3,0) node {$\bullet$};
	\draw (2.9,-0.5) node {$x$};	
	\draw (3.33,0) node {$\bullet$};		
	\draw (3.33,-0.7) node {$\displaystyle  \frac{m}{n}$};
	\draw[-,semithick] (2.83,0.4) -- (3.33,0.4);
	\draw (2.83,0.4) node {$|$};
	\draw (3.33,0.4) node {$|$};	
	\draw (3.08,0.9) node {$1/n$};
\end{tikzpicture}
\caption{As in line \eqref{eqn:betweenrationals} from the proof of Theorem \ref{thm:densityofQinR}, the rational number $m/n$ is than $1/n$ away from the real number $x$.}
\label{fig:betweenrationals}
\end{figure}

Also, combining the left-hand side of \eqref{eqn:betweenm} with \eqref{eqn:densitymiddlestep} yields
\begin{align}
m\leq nx+1<ny.
\end{align}
Hence, $m<ny$ and so $m/n<y$. Therefore, as in Figure \ref{fig:rationalbetweenreals},
\begin{align}\label{eqn:rationalbetweenreals}
x<\frac{m}{n}<y.
\end{align}

\begin{figure}
\centering
\begin{tikzpicture}
	\draw[<->,semithick] (-1,0) -- (7,0);
	\draw (3,0) node {$\bullet$};
	\draw (2.9,-0.5) node {$x$};
	\draw (4,0) node {$\bullet$};
	\draw (4.1,-0.5) node {$y$};	
	\draw (3.33,0) node {$\bullet$};		
	\draw (3.33,-0.7) node {$\displaystyle  \frac{m}{n}$};
\end{tikzpicture}
\caption{As in line \eqref{eqn:rationalbetweenreals} at the end of the proof of Theorem \ref{thm:densityofQinR}, the rational number $m/n$ is between the real numbers $x$ and $y$.}
\label{fig:rationalbetweenreals}
\end{figure}
\end{proof}

Theorem \ref{thm:squarerootoftwo} shows that no rational number is the square root of 2, but that doesn't immediately mean the the square root of 2 is a real number. This is something we can prove using the completeness of the real line and its consequences.

\begin{theorem}\label{thm:squarerootoftworeal}
There is an $x\in\R$ where $x^2=2$.
\end{theorem}

The proof makes use of a carefully chosen set of real numbers that's bounded above whose supremum plays a key role along with a pair of contradictions. The supremum of this set exists by the the Axiom of Completeness \ref{ax:axiomofcompleteness} and the way the set is chosen allows us to show, using a pair of similar contradictions, that the square of its supremum can be neither larger nor smaller than 2. So by the trichotomy property of $\R$, the supremum must be 2 (see Axiom \ref{ax:orderedfield}). However, only one of the contradictions is derived here.

\begin{proof}[Half of the proof of Theorem \ref{thm:squarerootoftworeal}]
Let $S$ be the set of real numbers whose squares are less than 2. That is, $S=\{y\in\R:y^2<2\}$.

Since $1^2=1<2$ we have $1\in S$, so $S$ is nonempty. Since $1<7<y$ implies $2<49=7^2<y^2$, we have that $7$ is an upper bound for $S$. By the Axiom of Completeness \ref{ax:axiomofcompleteness}, there is a real number $u$ where $u=\sup{S}$. Since $u$ is an upper bound for $S$, we have $0<1\leq u$. In particular, $u$ is positive.

To establish a contradiction, suppose $u^2<2$. The goal is to find an element in $S$ that is greater than $u$. Since $u^2<2$, we have $2-u^2>0$. Since $u>0$, we have $2u+1>0$ as well. Thus,
\begin{align}
\frac{2-u^2}{2u+1}>0, 
\end{align}
which is enough wiggle room to work with. By Corollary \ref{cor:archimedeanproperty}, there is a positive integer $n$ large enough where both 
\begin{align}
0<\frac{1}{n}<\frac{2-u^2}{2u+1} \quad\textnormal{and}\quad n>1.
\end{align}
Thus, after some algebraic manipulation we have 
\begin{align}
\frac{2u+1}{n}<2-u^2,
\end{align}
and by choosing $n>1$ we have $1/n^2<1/n$. Combining these results we have 
\begin{align}
\left(u+\frac{1}{n}\right)^2&=u^2+\frac{2u}{n}+\frac{1}{n^2}\\
&<u^2+\frac{2u}{n}+\frac{1}{n}\\
&=u^2+\frac{2u+1}{n}\\
&<u^2+2-u^2.
\end{align}
This implies $(u+(1/n))^2<2$, therefore $u+(1/n)$ is in $S$ and greater than $u$. But this contradicts the assertion that $u$ is an upper bound for $S$, so we must have $u^2\geq 2$.

The case where $u^2>2$ is assumed leads to another contradiction following from a similar argument. So, the completion of the proof of Theorem \ref{thm:squarerootoftwo} is left as an exercise.
\end{proof}

Theorem \ref{thm:squarerootoftworeal} tells us that there is a positive real number whose square is 2 and this number is denoted by $\sqrt{2}$. Theorem \ref{thm:squarerootoftwo} tells us $\sqrt{2}$ is not a rational number, so it must be {\em irrational}. That is, $\sqrt{2}\in\R\backslash\Q$. We can use this to prove that between any two real numbers, there is an irrational number.

\begin{corollary}[Density of $\R\backslash\Q$ in $\R$]\label{cor:densityofirrationals}
Let $x,y\in\R$ where $x<y$. Then there is some $v\in\R\backslash\Q$ where $x<v<y$.
\end{corollary}

\begin{proof}
Suppose $x<y$ so that $x-\sqrt{2}<y-\sqrt{2}$. By the Density of $\Q$ in $\R$ (Theorem \ref{thm:densityofQinR}), there is a rational number $r$ such that
\begin{align}
x-\sqrt{2}<r<y-\sqrt{2}.
\end{align}
Therefore, $x<r+\sqrt{2}<y.$

Now let, $v=r+\sqrt{2}$. To complete the proof, we should show $v$ is irrational. This can be done by establishing a contradiction. We have $r\in\Q$, so suppose $v\in\Q$ as well. Then there are integers $m$ and $s$ along with positive integers $n$ and $t$ where $r=m/n$ and $v=s/t$. Then we have
\begin{align}
\sqrt{2}=v-r=\frac{ns-mt}{nt}.
\end{align}
Since $ns-mt$ is an integer and $nt$ is a positive integer, we have that $\sqrt{2}$ is rational, establishing a contradiction of Theorem \ref{thm:squarerootoftwo}. 
\end{proof}

The Axiom of Completeness \ref{ax:axiomofcompleteness} allows us to prove many results about the structure and properties of the real line. Some of which you may have seen or worked with before, others may be new. Still, more questions remain to be asked and answered: In addition to $\sqrt{2}$, what other kinds of real numbers are irrational? How many irrational numbers are there? What other properties does the real line have in store for us?

\vs
\section*{Exercises}
\setcounter{theorem}{0}

Exercises are for play: Do scratch work, draw stuff, and make mistakes---make {\em lots} of mistakes---before worrying about writing proofs. {\em Have fun!}


\xca Prove Theorem \ref{thm:equivalentinfimum}: Suppose $\ell\in\R$ is a lower bound for a set $S\subseteq\R$. Then $\ell=\inf{S}$ if and only if $\ell$ is the greatest lower bound of $S$. 

Hint: Carefully modify the proof of Theorem \ref{thm:equivalentaxiomofcompleteness}, which connects supremum and least upper bound.

\xca Prove for any prime number $p$ and every rational number $r$ that $r^2\neq p$. (So $\sqrt{p}$ is irrational). Modify proof of Theorem \ref{thm:squarerootoftwo}.

\xca Complete the proof of Theorem \ref{thm:squarerootoftworeal} for the case where $u=\sup{S}$ and it is assumed $u^2>2$. 

Hint: Carefully modifying the half of the proof already provided.
%

\xca Prove that if $x\in\R$, then $x\acl{\Q}$ and $x\acl{\R\backslash\Q}$. Thus, both the set of rational number $\Q$ and the set of irrational numbers $\R\backslash\Q$ are {\em dense} in $\R$.


\vs
\section{Arbitrarily close in Euclidean spaces}
\label{sec:arbitrarilycloseineuclideanspaces}

In the more general setting of points and sets in a Euclidean space $\R^m$, we have the following definition for {\em arbitrarily close}.

\begin{definition}\label{def:acl}
Let $\bfy$ be a point in $\R^m$ and let $B\subseteq\R^m$. The point $\bfy$ is said to be \textit{arbitrarily close} to the set $B$, and we write $\bfy \acl B$, if for every $\varepsilon > 0$ there is a point $\bfx_\varepsilon$ in $B$ such that
\begin{align}\label{eqn:aclinequality}
d_m(\bfx_\varepsilon,\bfy)=	 \|\bfx_\varepsilon -\bfy\|_m <\varepsilon.
\end{align}
The phrase ``$B$ is {\em arbitrarily close} to $\bfy$'' is defined and denoted in the same way. 
\end{definition}

\begin{example}\label{eg:acl}
In Figure \ref{fig:yaclBxepsilon}, $B$ is a solid rectangle containing three of its corners and two of its sides. The point $\bfy$ is the corner of the rectangle that is not in $B$ but is arbitrarily close to $B$.
\end{example}

\begin{figure}
\centering
\begin{tikzpicture} 
\draw[dashed, fill=blue!15] (-3,-1.41) rectangle (1.41,1.41);
	\draw (-0.8,0) node {$B$};
	\draw (-3.02,1.42) node {$\bullet$};
	\draw (-3.02,-1.42) node {$\bullet$};
	\draw (1.44,-1.42) node {$\bullet$};					
\begin{scope}[thick, dashed, red] 
\draw (1.45,1.43) circle (2.1cm); 
\draw (1.45,1.43) circle (1.4cm); 
\draw (1.45,1.43) circle (0.7cm); 
\end{scope}	
\draw[semithick] (-3,-1.41) -- (-3,1.41);
\draw[semithick] (-3,-1.41) -- (1.41,-1.41);
\draw[-,semithick, red] (1.48,1.48) -- (2.92,2.92);
\draw[red] (2.8,2.55) node {$\varepsilon$};
	\draw (-0.3,1.2) node {$\bullet$};
	\draw (-0.2,0.8) node {$\bfx_\varepsilon$};
	\draw (1.1,0.5) node {$\bullet$};	
	\draw (1.2,1.2) node {$\bullet$};
	\draw[fill=white] (1.41,1.41) circle (0.08cm);
	\draw (1.8,1.41) node {$\bfy$};	
\end{tikzpicture}
\caption{A point $\bfy$ arbitrarily close to a set $B$. The point $\bfx_\varepsilon$ is in $B$ and within a distance of $\varepsilon$ from $\bfy$ (within the larger red circle).}
\label{fig:yaclBxepsilon}
\end{figure}

\begin{remark}\label{rmk:smallaswelikerevisit}
As in the definition for arbitrarily close in the real line (Defintion \ref{def:aclreal}), we can think of the positive real number $\varepsilon$ as the amount of error or ``wiggle room'' we'd like to allow. In Figure \ref{fig:yaclBxepsilon}, just one radius $\varepsilon$ is drawn to keep things from getting too cluttered. However, since Definitions \ref{def:aclreal} and \ref{def:acl} allow for \textit{any} $\varepsilon>0$, we can take the error to be as small as we like. So, $\bfy \acl B$ means $B$ gets as close to $\bfy$ as we like, no matter how close. In order to prove $\bfy \acl B$, it suffices can respond to each $\varepsilon>0$ with a point $\bfx_\varepsilon$ which is in $B$ and within $\varepsilon$ of $\bfy$.
\end{remark}

As a first result, points in a set are arbitrarily close to the set. 

\begin{lemma}\label{lem:elementacl}
Let $\bfy\in\R^m$ and let $B\subseteq\R^m$. If $\bfy\in B$, then $\bfy\acl B$.
\end{lemma}

The idea of the proof is to choose the same point $\bfx_\varepsilon=\bfy$ for each $\varepsilon>0$.

\begin{proof}
Assume $\bfy\in B$ and let $\varepsilon>0$. Choosing $\bfx_\varepsilon=\bfy$ yields
\begin{align}
	d_m(\bfx_\varepsilon,\bfy)&=\|\bfx_\varepsilon -\bfy\|_m = \|\bfy -\bfy\|_m = 0<\varepsilon.
\end{align}
 Hence, $\bfy\acl B$.
\end{proof}

A natural question students have asked is: ``What does it mean when two points are arbitrarily close?'' Since the definition of arbitrarily close (Definition \ref{def:acl}) compares a point to a set, we need to be creative in order to properly to use it to answer this question. We can replace one of the points with a singleton. 

\begin{lemma}\label{lem:equalpoints}
Let $\bfx,\bfy\in\R^m$. We have $\bfx=\bfy$ if and only if $\bfx\acl\{\bfy\}$.
\end{lemma}

\begin{proof}
Suppose $\bfx=\bfy$. Then for every $\varepsilon>0$, we have $d_m(\bfx,\bfy)=0<\varepsilon$. Therefore, $\bfx \acl \{\bfy\}$.

Now, suppose $\bfx\neq \bfy$. Then $0 < d_m(\bfx,\bfy)/2 \leq d_m(\bfx,\bfy)$. Therefore, $\bfx$ is not arbitrarily close to $\{\bfy\}$.
\end{proof}

\begin{remark}\label{rmk:twopointsacl}
In any Euclidean space $\R^m$, it's appropriate to interpret Lemma \ref{lem:equalpoints} as saying two points are arbitrarily close if and only if they are equal. \textit{But this is not necessarily the case when we extend the definition of arbitrarily close to topological spaces!} If you're not familiar with topology yet, don't worry. I'm just pointing out that cool things happen in topology when other beautiful structures aside from Euclidean spaces are in play.
\end{remark}

\begin{prob}
Draw some figures to go with the proofs of Lemmas \ref{lem:elementacl} and \ref{lem:equalpoints}. Playing with examples in the plane $\R^2$ isn't a bad idea.
\end{prob}

To provide another perspective for notions of closeness, consider the terminology of {\em neighborhoods} found in topology.

\begin{definition}\label{def:neighborhood}
Let $\varepsilon>0$ and $\bfc\in\R^m$. The $\varepsilon$-{\em neighborhood} of $\bfc$, denoted by $V_\varepsilon(\bfc)$, is the set of points within $\varepsilon$ of $\bfc$. That is,
\begin{align}\label{eqn:neighborhood}
V_\varepsilon(\bfc)=\{\bfx\in \R^m: d_m(\bfx,\bfc)=\|\bfx-\bfc\|_m<\varepsilon\}.
\end{align}
\end{definition}

The value of $\varepsilon$ used here can still be thought of as an ``error'' or as a bound for the distance from $\bfc$ we want to allow. Also, the word {\em neighborhood} means the same thing as an ``$\varepsilon$-neighborhood of $\bfc$'' and is used when $\varepsilon$ and $\bfc$ need not be specified.

Given a fixed $m\in\N$ and its corresponding Euclidean space $\R^m$, what do $\varepsilon$-neighborhoods look like? See Figure \ref{fig:threeneighborhoods} for $\varepsilon$-neighborhoods of some fixed $\varepsilon>0$ centered at $\bfc_3$ in $\R^3$, $\bfc_2$ in $\R^2$, and $c_1$ in $\R$, respectively. 

\begin{figure}
\centering
\begin{tikzpicture}
\draw (-5,0) node {$V_\varepsilon(\bfc_3)\subseteq\R^3$};
  \shade[ball color = blue!15, opacity = 0.4] (0,0) circle (2cm);
  \draw (0,-0.02) node {$\bullet$}; 
  \draw (0,-0.35) node {$\bfc_3$}; 
  \draw (-2,0) arc (180:360:2 and 0.6);
  \draw[dashed] (2,0) arc (0:180:2 and 0.6);
  \draw[dashed] (0,0 ) -- node[above]{$\varepsilon$} (2,0);
\draw (-5,-5) node {$V_\varepsilon(\bfc_2)\subseteq\R^2$};  
  \draw[dashed, fill=blue!15] (0,-5) circle (2cm);
  \draw (0,-5) -- node[above]{$\varepsilon$} (2,-5);
  \draw (0,-5.02) node {$\bullet$};
  \draw (0,-5.35) node {$\bfc_2$};
\draw (-5,-8) node {$V_\varepsilon(c_1)\subseteq\R$};
	\draw[-,semithick] (-2,-8) -- (2,-8);
	\draw (1.98,-8) node {$)$};
	\draw (-1.98,-8) node {$($};
  \draw (0,-8.02) node {$\bullet$};	
	\draw (-2,-8.5) node {$c_1-\varepsilon$};
	\draw (0,-8.5) node {$c_1$};	
	\draw (2,-8.5) node {$c_1+\varepsilon$};
\end{tikzpicture}
\caption{Some $\varepsilon$-neighborhoods centered at $\bfc_3$ in $\R^3$, $\bfc_2$ in $\R^2$, and $c_1$ in $\R$, respectively, for some fixed $\varepsilon>0$.}
\label{fig:threeneighborhoods}
\end{figure}

\begin{itemize}
 \item $V_\varepsilon(\bfc_3)$ is the open sphere of radius $\varepsilon$ centered $\bfc_3$ and does not include points on its surface exactly $\varepsilon$ away from $\bfc_3$. 
 \item $V_\varepsilon(\bfc_2)$ is the open disk of radius $\varepsilon$ centered $\bfc_2$ and does not include points on the circle exactly $\varepsilon$ away from $\bfc_2$.
 \item $V_\varepsilon(c_1)$ is the open interval of length $2\varepsilon$ centered $c_1$ and does not include the endpoints $c_1-\varepsilon$ and $c_1+\varepsilon$.
\end{itemize}

\begin{remark}\label{rmk:aclvianeighborhoods}
As suggested at the start of this chapter, neighborhoods provide an important perspective for defining notions of arbitrarily close. For a point $\bfy$ and a set $B$ in a Euclidean space $\R^m$, we have $\bfy\acl{B}$ if and only if every $\varepsilon$-neighborhood of $\bfy$ intersects $B$. That is, $\bfy\acl{B}$ if and only if for every $\varepsilon>0$ we have
\begin{align}
V_\varepsilon(\bfy)\cap B &\neq \varnothing,
\end{align} 
where $\varnothing$ denotes the empty set. 
See Definitions \ref{def:acl} and \ref{def:neighborhood}.
\end{remark}

The following lemma highlights a special property of neighborhoods in the real line that turns out to be quite useful. In particular, the real line is ordered, so any neighborhood of a real number has only two directions to go in. This effect is established by the Trichotomy Property of the real line (see Axiom \ref{ax:orderedfield}) and the definition of absolute value (Definition \ref{def:absolutevalue}). Specifically, every pair real numbers $x$ and $y$ satisfy exactly one of the following: $x<y$, $y<x$, or $x=y$. The proof of the lemma is left to Exercise \ref{exer:neighborhoods}.

\begin{lemma}\label{lem:neighborhoods}
Let $c$ and $x$ be real numbers. For each $\varepsilon>0$, the following are equivalent:
\begin{align}\label{eqn:neighborhoods}
	x\in V_\varepsilon(c)& \,\Longleftrightarrow\,
	|x-c|<\varepsilon\\
	& \,\Longleftrightarrow\, -\varepsilon<x-c<\varepsilon \\
	& \,\Longleftrightarrow\, c-\varepsilon<x<c+\varepsilon.
\end{align}
\end{lemma}

Let's try turning things around. What if some point $\bfw\in\R^m$ is not arbitrarily close to $B$? Negation leads immediately to the following definition.

\begin{definition}\label{def:awf}
We say $\bfw$ is \textit{away from} $B$ and write $\bfw \awf B$ if there is some $\varepsilon_\bfw>0$ such that for every $\bfx$ in $B$ we have 
\begin{align}\label{eqn:awfinequality}
d_m(\bfx,\bfw)=\|\bfx -\bfw\|_m \geq\varepsilon_\bfw.
\end{align}
\end{definition}

\begin{figure}
\centering
\begin{tikzpicture} 
\draw[dashed, fill=blue!15] (-3,-1.41) rectangle (1.41,1.41);
	\draw[blue] (3,0) node {$\circ$};
	\draw (3,-0.4) node {$\bfw$};
	\draw (-0.8,0) node {$B$};				
\draw[semithick] (-3,-1.41) -- (-3,1.41);
\draw[semithick] (-3,-1.41) -- (1.41,-1.41);
\draw[semithick, dashed, blue] (3,0) circle (1cm);
\draw[-,semithick, blue] (3.05,0.05) -- (3.71,0.71);
\draw[blue] (3.6,0.15) node {$\varepsilon_\bfw$};
\end{tikzpicture}
\caption{A point $\bfw$ away from a set $B$ where all points in $B$ are more than some positive distance $\varepsilon_\bfw$ away from $\bfw$. Thus, $V_{\varepsilon_\bfw}(\bfw)\subseteq\R^m\backslash B$. See Definition \ref{def:awf}.}
\label{fig:wawfB}
\end{figure}

\begin{remark}\label{rmk:distancebetween}
To recap, we have $\bfy \acl B$ if there is no distance between $\bfy$ and $B$, while $\bfw \awf B$ if there is some distance between $\bfw$ and $B$. Equivalently, $\bfw\awf{B}$ if there is some $\varepsilon_\bfw$-neighborhood $V_{\varepsilon_\bfw}(\bfw)$ where we have $V_{\varepsilon_\bfw}(\bfw)\subseteq\R^m\backslash B$. So in order to prove $\bfw \awf B$, all we need is one fixed distance $\varepsilon_\bfw>0$ that keeps the all points in $B$ that far from $\bfw$ or more. See Figure \ref{fig:wawfB}.
\end{remark}

Theorem \ref{thm:densityofQinR}, Corollary \ref{cor:densityofirrationals}, and Lemma \ref{lem:neighborhoods} combine to produce a nice result on the relationships between rational numbers, irrational numbers, and $\varepsilon$-neighborhoods in the real line.

\begin{theorem}\label{thm:opencontainsrationalirrational}
Let $c\in\R$ and $\varepsilon>0$. Then there is a rational number and an irrational number in the $\varepsilon$-neighborhood
\begin{align}
V_\varepsilon(c)=(c-\varepsilon,c+\varepsilon).
\end{align}
\end{theorem}

\begin{proof}
Let $c\in\R$ and $\varepsilon>0$. Then $c-\varepsilon<c+\varepsilon$, so by Theorems \ref{thm:densityofQinR} and Corollary \ref{cor:densityofirrationals}, there is an $r\in\Q$ and a $v\in\R\backslash\Q$ where
\begin{align}
c-\varepsilon<r<c+\varepsilon \quad\textnormal{and}\quad  c-\varepsilon<v<c+\varepsilon.
\end{align}
By Lemma \ref{lem:neighborhoods} we have $r\in (c-\varepsilon,c+\varepsilon)$ and $v\in (c-\varepsilon,c+\varepsilon)$.
\end{proof}

Another property of real numbers that might seem intuitive is the notion that the only real number arbitrarily close to both positive and negative real numbers is zero. To this end, let $\R^+=(0,\infty)$ and $\R^-=(-\infty,0)$.

\begin{lemma}\label{lem:zeroacl}
A real number $\ell$ is equal to $0$ if and only if $\ell\acl\R^+$ and $\ell\acl\R^-$.
\end{lemma}

\begin{proof}
First, assume $\ell=0$ and let $\varepsilon>0$. Then $\varepsilon/2\in\R^+$, $-\varepsilon/2\in\R^-$, and
\begin{align}
-\varepsilon< &-\frac{\varepsilon}{2} < \ell=0 < \frac{\varepsilon}{2}<\varepsilon.
\end{align}
By Lemma \ref{lem:neighborhoods}, $0\acl\R^+$ and $0\acl\R^-$.

Next, assume $\ell>0$ and let $y< 0$. Then 
\begin{align}
y< 0 < \frac{\ell}{2}<\ell.
\end{align}
Since $\varepsilon_0=\ell/2>0$, it follows that $|\ell-y|>\ell/2$. Therefore $\ell\awf{\R^-}$. Similarly, no negative number is arbitrarily close to $\R^+$.
\end{proof}

To close out this chapter, we will often want to consider the set of points arbitrarily close to a given set.

\begin{definition}\label{def:closure}
Let $S\subseteq\R^m$. The {\em closure} of $S$, denoted by $\overline{S}$, is the set of points arbitrarily close to $S$. Thus,
\begin{align}
\overline{S} &= \{ \bfx \in \R^m: \bfx \acl S\}.
\end{align}
\end{definition}

\begin{example}\label{eg:aclclosure}
In Figure \ref{fig:aclclosure}, $\overline{B}$ is the closure of the rectangle $B$ from Example \ref{eg:acl}. $\overline{B}$ is a solid rectangle that contains its corners and sides, thus the corner $\bfy$ is in $\overline{B}$.

\begin{figure}
\centering
\begin{tikzpicture}  
\draw[fill=blue!15] (-3,-1.41) rectangle (1.41,1.41);
	\draw (1.44,1.41) node {$\bullet$};
	\draw (1.8,1.41) node {$\bfy$};
	\draw (-0.8,0) node {$\overline{B}$};
	\draw (-3.02,1.42) node {$\bullet$};
	\draw (-3.02,-1.42) node {$\bullet$};
	\draw (1.44,-1.42) node {$\bullet$};					
\begin{scope}[semithick, dashed, red] 
\draw (1.44,1.42) circle (2.1cm); 
\draw (1.44,1.42) circle (1.4cm); 
\draw (1.44,1.42) circle (0.7cm); 
\end{scope}	
\draw (-3,-1.41) -- (-3,1.41);
\draw (-3,-1.41) -- (1.41,-1.41);
\draw[-,semithick, red] (1.48,1.48) -- (2.92,2.92);
\draw[red] (2.8,2.55) node {$\varepsilon$};
	\draw (-0.3,1.2) node {$\bullet$};
	\draw (-0.2,0.8) node {$\bfx_\varepsilon$};
	\draw (1.1,0.5) node {$\bullet$};	
	\draw (1.2,1.2) node {$\bullet$};
\end{tikzpicture}
\caption{The closure $\overline{B}$ contains all points in and arbitrarily close to the rectangle $B$, including the corner $\bfy$ and the sides.}
\label{fig:aclclosure}
\end{figure}
\end{example}

\begin{example}\label{eg:closedneighborhoods}
There are subtle differences between $\varepsilon$-neighborhoods and their closures. Can you spot the differences between Figures \ref{fig:threeneighborhoods} and \ref{fig:closedneighborhoods}?

\begin{figure}
\centering
\begin{tikzpicture}
\draw (-5,0) node {$\overline{V_\varepsilon(\bfc_3)}\subseteq\R^3$};
  \shade[ball color = blue!15, opacity = 0.4] (0,0) circle (2cm);
  \draw (0,0) circle (2cm);
  \draw (0,-0.02) node {$\bullet$}; 
  \draw (0,-0.35) node {$\bfc_3$}; 
  \draw (-2,0) arc (180:360:2 and 0.6);
  \draw[dashed] (2,0) arc (0:180:2 and 0.6);
  \draw[dashed] (0,0 ) -- node[above]{$\varepsilon$} (2,0);
\draw (-5,-5) node {$\overline{V_\varepsilon(\bfc_2)}\subseteq\R^2$};  
  \draw[fill=blue!15] (0,-5) circle (2cm);
  \draw (0,-5) -- node[above]{$\varepsilon$} (2,-5);
  \draw (0,-5.02) node {$\bullet$};
  \draw (0,-5.35) node {$\bfc_2$};
\draw (-5,-8) node {$\overline{V_\varepsilon(c_1)}\subseteq\R$};
	\draw[-,semithick] (-2,-8) -- (2,-8);
	\draw (2,-8) node {$]$};
	\draw (-2,-8) node {$[$};
  \draw (0,-8.02) node {$\bullet$};	
	\draw (-2,-8.5) node {$c_1-\varepsilon$};
	\draw (0,-8.5) node {$c_1$};	
	\draw (2,-8.5) node {$c_1+\varepsilon$};
\end{tikzpicture}
\caption{Closures of $\varepsilon$-neighborhoods centered at $\bfc_3$ in $\R^3$, $\bfc_2$ in $\R^2$, and $c_1$ in $\R$, respectively, for some fixed $\varepsilon>0$. Which points are in these sets that are not in their counterparts in Figure \ref{fig:threeneighborhoods}?}
\label{fig:closedneighborhoods}
\end{figure}

\begin{itemize}
 \item $\overline{V_\varepsilon(\bfc_3)}$ is the closed sphere of radius $\varepsilon$ centered $\bfc_3$, including the points on the surface exactly $\varepsilon$ away from $\bfc_3$. 
 \item $\overline{V_\varepsilon(\bfc_2)}$ is the closed disk of radius $\varepsilon$ centered $\bfc_2$, including the points on the circle exactly $\varepsilon$ away from $\bfc_2$.
 \item $\overline{V_\varepsilon(c_1)}$ is the closed interval of length $2\varepsilon$ centered $c_1$, including the endpoints $c_1-\varepsilon$ and $c_1+\varepsilon$.
\end{itemize}
\end{example}

\begin{remark}\label{rmk:closureconvential}
Some of you may have seen a definition for closure in a class on real analysis, topology, or some other topic. Definition \ref{def:closure} is equivalent to more conventional definitions of closure, but the justification is left to Chapter \ref{ch:topologyofeuclideanspaces}. For now, closures give us a powerful tool for exploring properties of sets before the boss fight in Chapter \ref{ch:sequenceslimitsandcodas}: Defining limits and convergence for {\em sequences}. 
\end{remark}

\begin{example}\label{eg:closure}
Recall the sets $A,B,F,$ and $G$ from previous Examples \ref{eg:openinterval} and \ref{eg:twocountablesets} from Section \ref{sec:acl}. What are the closures of these sets? We have:
\begin{align}
	A&=\{2-(1/\sqrt{n}):n\in\N\}\quad
	\Longrightarrow \quad\overline{A}=A\cup\{2\},\notag\\
	B&=\{[2-(1/\sqrt{n})](-1)^n:n\in\N\}\quad
	\Longrightarrow \quad\overline{B}=B\cup\{-2,2\}\notag,\\
	F&=[0,3140]=\{x\in\R:0\leq x\leq 3140\}\quad
	\Longrightarrow \quad\overline{F}=[0,3140]=F, \quad\textnormal{and} \notag\\ 
	G&=(0,3140)=\{x\in\R:0<x<3140\}\quad
	\Longrightarrow \quad\overline{G}=[0,3140]=F.\notag
\end{align}

As usual, I think figures help. See Figure \ref{fig:setsandclosures}. 

\begin{figure}
\centering
\begin{tikzpicture}
\draw (-2,0) node {$A$};
\draw (4,0) node {$\circ$};
\draw (3.76,0) node {$...$};
\draw (3,-0.5) node {$1$};
\draw (4,-0.5) node {$2$};
\foreach \Point in {(3,0), (3.29,0), (3.42,0), (3.5,0)}
{
    \node at \Point {\textbullet};
}
\draw (-2,-1.5) node {$\overline{A}$};
\draw (4,-1.5) node {$\bullet$};
\draw (3.76,-1.5) node {$...$};
\draw (3,-2) node {$1$};
\draw (4,-2) node {$2$};
\foreach \Point in {(3,-1.5), (3.29,-1.5), (3.42,-1.5), (3.5,-1.5)}
{
    \node at \Point {\textbullet};
}
\draw (-2,-3) node {$B$};
\draw (0,-3) node {$\circ$};
\draw (4,-3) node {$\circ$};
\draw (3.76,-3) node {...};
\draw (0.28,-3) node {...};
\draw (0,-3.5) node {$-2$};
\draw (1,-3.5) node {$-1$};
\draw (4,-3.5) node {$2$};
\foreach \Point in {(1,-3), (3.29,-3), (0.58,-3), (3.5,-3)}
{
    \node at \Point {\textbullet};
}
\draw (-2,-4.5) node {$\overline{B}$};
\draw (0,-4.5) node {$\bullet$};
\draw (4,-4.5) node {$\bullet$};
\draw (3.76,-4.5) node {...};
\draw (0.28,-4.5) node {...};
\draw (0,-5) node {$-2$};
\draw (1,-5) node {$-1$};
\draw (4,-5) node {$2$};
\foreach \Point in {(1,-4.5), (3.29,-4.5), (0.58,-4.5), (3.5,-4.5)}
{
    \node at \Point {\textbullet};
}
\draw (-2,-6) node {$F$};
	\draw[-,semithick] (0,-6) -- (4,-6);
	\draw (0,-6) node {$[$};
	\draw (4,-6) node {$]$};
	\draw (0,-6.5) node {$0$};
	\draw (4,-6.5) node {$3140$};
\draw (-2,-7.5) node {$\overline{F}$};
	\draw[-,semithick] (0,-7.5) -- (4,-7.5);
	\draw (0,-7.5) node {$[$};
	\draw (4,-7.5) node {$]$};
	\draw (0,-8) node {$0$};
	\draw (4,-8) node {$3140$};
\draw (-2,-9) node {$G$};
	\draw[-,semithick] (0,-9) -- (4,-9);
	\draw (0.02,-9) node {$($};
	\draw (3.98,-9) node {$)$};
	\draw (0,-9.5) node {$0$};
	\draw (4,-9.5) node {$3140$};
\draw (-2,-10.5) node {$\overline{G}$};
	\draw[-,semithick] (0,-10.5) -- (4,-10.5);
	\draw (0,-10.5) node {$[$};
	\draw (4,-10.5) node {$]$};
	\draw (0,-11) node {$0$};
	\draw (4,-11) node {$3140$};	
\end{tikzpicture}
\caption{The sets of real numbers $A,B,F,$ and $G$ from Examples \ref{eg:openinterval} and \ref{eg:twocountablesets}, along with their closures. In each case, the closure contains the points in the original set along with points arbitrarily close to the original set.}
\label{fig:setsandclosures}
\end{figure}
\end{example}

Each of the sets in Example \ref{eg:closure} are bounded above by $3140$ and bounded below $-2$. See Figure \ref{fig:setsandclosures}. In Euclidean spaces, the notions of bounded above and bounded below lose meaning since there are infinitely many directions to move around in, not just the two we are limited to in the real line. Even so, {\em boundedness} is a property some sets can have. The idea is to use single bound for all directions.

\begin{definition}\label{def:boundedset} 
A set $S\subseteq\R^m$ is {\em bounded} if there is a $b\geq 0$ such that for all $\bfx$ in $S$ we have
\begin{align}
\|\bfx\|_m\leq b.
\end{align}
In this case, $b$ is called a {\em bound} for $S$ and we say $S$ is {\em bounded by} $b$.
\end{definition}

\begin{remark}\label{rmk:boundedabovebelow}
In the real line $\R$, a set $S$ is bounded if and only if it is bounded above and bounded below.
\end{remark}

This section wraps up with a trio of results about bounded sets. Lemmas \ref{lem:epsilonbound} and \ref{lem:orderacl} partially stem from the order inherent to the real line, while Lemma \ref{lem:closurebound} pertains to bounded sets in Euclidean space. All of these results make use of arbitrarily close in some way (Definitions \ref{def:aclreal} and \ref{def:acl}).

\begin{lemma}\label{lem:epsilonbound}
Suppose $x$, $a$, and $b$ are real numbers. Then 
\begin{enumerate}
	\item $x\leq b$ if and only if for every $\varepsilon>0$ we have $x<b+\varepsilon$.
	\item $x\geq a$ if and only if for every $\varepsilon>0$ we have $x>a-\varepsilon$.
\end{enumerate}
\end{lemma}

The proofs of statements (i) and (ii) in Lemma \ref{lem:epsilonbound} are similar enough that we  will only work on (i) here and leave (ii) as an exercise. Also, since $b<b+\varepsilon$ for every $\varepsilon>0$, one direction of (i) has a short proof. For the other direction, a contraposition argument yields the result.

\begin{proof}[Proof of (i) in Lemma \ref{lem:epsilonbound}]
First, suppose $x\leq b$ and let $\varepsilon>0$. Then 
\begin{align}
	x\leq b< b+\varepsilon.
\end{align}

For the other direction, assume $x>b$ to set up a contraposition argument. Then $x-b>0$ and so
\begin{align}
	\varepsilon_0=\frac{x-b}{2}>0.
\end{align}
Hence,
\begin{align}
	b+\varepsilon_0
	&=b+\frac{x-b}{2}
	=\frac{x+b}{2}
	<\frac{x+x}{2}
	=x. 
\end{align}
Therefore, there is some $\varepsilon_0>0$ where $x>b+\varepsilon_0$.
\end{proof} 

\begin{lemma}\label{lem:orderacl}
Suppose $S\subseteq \R$ and $y\acl{S}$.
\begin{enumerate}
	\item If $b$ is an upper bound for $S$, then $y\leq b$. 
	\item If $a$ is a lower bound for $S$, then $a\leq y$. 
\end{enumerate}
\end{lemma}

\begin{figure}
\centering
\begin{tikzpicture} 
\draw (-3,0) node {$S$};
	\draw (-0.5,0) node {$\circ$};
	\draw (-0.5,-0.5) node {$a$};
	\draw (4.5,0) node {$\circ$};
	\draw (4.5,-0.5) node {$b$};	
	\draw[-,semithick] (0,0) -- (2,0);
	\draw (0,0) node {$[$};
	\draw (2,0) node {$)$};
	\draw (2,-0.5) node {$y$};	
\draw (4,0) node {$\circ$};
\draw (3.76,0.02) node {...};
\foreach \Point in {(3,0), (3.5,0)}
{
    \node at \Point {\textbullet};
}	
\end{tikzpicture}
\caption{As in Lemma \ref{lem:orderacl}, here is a set of real numbers $S$ along with a lower bound $a$, an upper bound $b$, and a point $y$ where $y\acl{S}$.}
\label{fig:orderacl}
\end{figure}

The statements (i) and (ii) in Lemma \ref{lem:orderacl} are similar enough that we will only work on (i) here and leave (ii) as an exercise. See Figure \ref{fig:orderacl}.

\begin{proof}[Proof of (i) in Lemma \ref{lem:orderacl}]
To set up a contraposition argument, suppose $y\acl{S}$ and $y>b$. Then $\varepsilon_b=y-b>0$. By the definition of arbitrarily close (Definition \ref{def:acl}), there is some $x_b$ in $S$ where 
\begin{align}
	|x_b-y|<\varepsilon_b=y-b.
\end{align}
Therefore,
\begin{align}\label{eqn:orderacl}
	y-x_b&\leq|x_b-y|<y-b.
\end{align}
So, by subtracting $y$ then multiplying through by $-1$ in \eqref{eqn:orderacl} we have
\begin{align}
	x_b&>b.
\end{align}
Since $x_b$ is in $S$, $b$ is not an upper bound for $S$.
\end{proof}

The points in the closure of a bounded set are so close to the original set that there is no distance between them. As a result, a bounded set and its closure respect the same bounds. See Figure \ref{fig:Sboundedbyb}.

\begin{lemma}\label{lem:closurebound}
Suppose $S\subseteq\R^m$ is bounded by $b$. Then $\overline{S}$ is bounded by $b$ as well. That is, if $\bfy\acl{S}$, then 
\begin{align}
\|\bfy\|_m\leq b.
\end{align}
\end{lemma}

\begin{figure}
\centering
\begin{tikzpicture} 
\draw[dashed, fill=blue!15] (-1.41,-1.41) rectangle (1.41,1.41);
	\draw (0,-1) node {$S$};
\draw (0,0) circle (2.5cm);	
	\draw (2.5,0) --(0,0);  
	\draw (2,-0.3) node {$b$};  
	\draw (0,-0.02) node {$\bullet$};
	\draw (0,-0.3) node {$\mathbf{0}$};
\begin{scope}[thick, dashed, red] 
\draw (1.42,1.42) circle (0.2cm); 
\draw (1.42,1.42) circle (1.2cm); 
\draw (1.42,1.42) circle (0.7cm); 
\end{scope}	
	\draw[fill=white] (1.41,1.41) circle (0.08cm);
	\draw (1.8,1.41) node {$\bfy$};
\end{tikzpicture}
\caption{As in Lemma \ref{lem:closurebound} and Lekha Patil's proof, a set $S$ in the plane $\R^2$ bounded by  a positive real number $b$ along with a point $\bfy$ in $\R^2$ where $\bfy\acl{S}$ and $\|\bfy\|_m\leq b$.}
\label{fig:Sboundedbyb}
\end{figure}

Lemma \ref{lem:closurebound} was presented as an exercise to my undergraduate real analysis class in Summer 2021. One student, Lekha Patil, came up with an elegant proof. Her proof is provided below along with Figure \ref{fig:Sboundedbyb}. An alternate proof is provided thereafter.

\begin{proof}[Lekha Patil's proof of Lemma \ref{lem:closurebound}]
Suppose $\bfy\acl{S}$ and $b$ is a bound for $S$. Then for every point $\bfx$ in $S$ we have 
\begin{align}
	\|\bfy\|_m&=\|\bfy-\bfx+\bfx\|_m\\
	&\leq \|\bfy-\bfx\|_m+\|\bfx\|_m \quad\textnormal{(triangle inequality \eqref{eqn:triangleinequality})}\\
	&\leq \|\bfy-\bfx\|_m+b. \quad\textnormal{(Definition \ref{def:boundedset})}
\end{align}
Now let $\varepsilon>0$. Since $\bfy\acl{S}$, there is some $\bfx_\varepsilon$ in $S$ where 
\begin{align}
	\|\bfy\|_m&\leq \|\bfy-\bfx_\varepsilon\|_m+b<\varepsilon+b.
\end{align}
Since $\varepsilon>0$ is arbitrary (standing for {\em any} positive number, so Lemma \ref{lem:epsilonbound} applies), we have 
\begin{align}
	\|\bfy\|_m&\leq b.
\end{align}
Hence, $b$ is a bound for $\overline{S}$ as well.
\end{proof}

The line ``Since $\varepsilon>0$ is arbitrary'' highlights a key intuitive idea behind the definition of arbitrarily close, much like Lemma \ref{lem:equalpoints}: If the difference between two objects is less than {\em every} $\varepsilon>0$, then there's no difference at all. To me, this intuition is reinforced by Lemmas \ref{lem:epsilonbound} and \ref{lem:closurebound}.

An alternative proof of Lemma \ref{lem:closurebound} makes use of a contraposition argument and the definition of {\em away from} (Definition \ref{def:awf}). Basically, a point whose norm is greater than a bound for a set is away from the set. See Figure \ref{fig:wawayfromS}.

\begin{figure}
\centering
\begin{tikzpicture} 
\draw[dashed, fill=blue!15] (-1.41,-1.41) rectangle (1.41,1.41);
	\draw (0,-1) node {$S$};
\draw (0,0) circle (2.5cm);	
	\draw (2.5,0) --(0,0);  
	\draw (2,-0.3) node {$b$};    
	\draw (0,-0.02) node {$\bullet$};
	\draw (0,-0.3) node {$\mathbf{0}$};
\draw[thick, dashed, blue] (4,0) circle (1.5cm);
\draw[blue] (4.85,0.4) node {$\varepsilon_\bfw$};
\draw (4,-0.3) node {$\bfw$};  
 	\draw[blue] (4,0) node {$\circ$};
	\draw[-,semithick, blue] (4.05,0.05) -- (5.06,1.06);	
\end{tikzpicture}
\caption{In the second proof of Lemma \ref{lem:closurebound}, a set $S$ in the plane $\R^2$ bounded by a positive real number $b$ is away from any point $\bfw$ in $\R^2$ where $\|\bfw\|_m> b$.}
\label{fig:wawayfromS}
\end{figure}

\begin{proof}[Alternate proof of Lemma \ref{lem:closurebound}]
Suppose $b$ is a bound for $S$ and $\|\bfw\|_m>b$.
For every $\bfx$ in $S$ we have $\|\bfx\|_m\leq b$, so $-b\leq -\|\bfx\|_m$. Define 
\begin{align}
	\varepsilon_\bfw= \|\bfw\|_m-b>0.
\end{align}
Then we have
\begin{align}
	0< \varepsilon_\bfw&=\|\bfw\|_m-b\\ 
	&\leq \|\bfw\|_m-\|\bfx\|_m\\	
	&\leq\|\bfw-\bfx\|_m \quad\textnormal{(reverse tri. ineq.  \eqref{eqn:reversetriangleinequality})} 
\end{align}
Hence, $\bfw\awf{S}$ and $\bfw$ is not in $\overline{S}$.
\end{proof}

The following chapter explores properties of sequences through the lens of arbitrarily close (Definition \ref{def:acl}). In particular, a central tenet of this book is the way arbitrarily close allows us to parse the definitions of {\em convergence } and {\em limits} (Definition \ref{def:sequentiallimit}). 

\vs
\section*{Exercises}
\setcounter{theorem}{0}

Exercises are for play: Do scratch work, draw stuff, and make mistakes---make {\em lots} of mistakes---before worrying about writing proofs. {\em Have fun!}


\xca\label{exer:neighborhoods}
Prove Lemma \ref{lem:neighborhoods}. Hint: Use the definition of absolute value (Definition \ref{def:absolutevalue}) and Axiom \ref{ax:orderedfield}, especially the Trichotomy property.

\xca Prove part (ii) of Lemma \ref{lem:epsilonbound}. The proof follows from a careful consideration and modification of the proof of part (i).

\xca\label{exer:threeacl}
Pick a nonempty set $S$ in the plane $\R^2$. Draw $S$ and, depending on what you drew, try to draw three points as follows: One that is arbitrarily close to both $S$ and its complement  $\R^2\backslash S$, one that is away from $S$, and one that is away from  $\R^2\backslash S$. Next, draw the three following sets: The set of all points arbitrarily close to both $S$ and its complement  $\R^2\backslash S$, the set of all points away from $S$, and the set of all points away from $\R^2\backslash S$. Are there any points away from both $S$ and $\R^2\backslash S$?

\chapter[Sequences, Limits, and Codas]{Sequences, Limits, and Codas}
\label{ch:sequenceslimitsandcodas}

In my opinion, the technical definitions of {\em limit} and {\em convergence} (Definition \ref{def:sequentiallimit}) along with {\em continuity} (Definition \ref{def:continuity}) are the most difficult in undergraduate mathematics. As explored in this chapter, the definition of arbitrarily close (Definitions \ref{def:aclreal} and \ref{def:acl}) provides a stepping stone towards understanding these challenging concepts.

If you're interested, you might like to take a look at \cite{Barnes,Borzellino,LarsenSwinyard,Seager} for some other interesting approaches to teaching real analysis at the undergraduate level which are, in part, designed to address this difficulty. Also see \cite{Cornu,LarsenSwinyard} for a thorough discussion of the challenges that come with teaching convergence and limits.

This chapter explores limits and convergence for {\em sequences}, highlighted by Definition \ref{def:sequentiallimit}. We start with a formal definition of a sequence (Definition \ref{def:sequence}) and an investigation of what it means for points to be arbitrarily close to sequences.

\vs
\section{Sequences and staying close}
\label{sec:sequencesandstayingclose}

What are {\em sequences}? One way to think of a sequence is as an unending list of objects. More specifically, a sequence imbues two structures on a collection of objects: (i) The objects are listed in ordered (first, second, third...); and (ii) the list is infinitely long, even if some objects are repeated. 

Formally, a {\em sequence} is a function from the set of positive integers $\N$ into some nonempty set. The ordering of $\N$ in the domain becomes an ordering on the range of the function, generating the terms of the sequence. This ordering enables us to expand the notion of arbitrarily close to the more focused---and classically important---notions of {\em limits} and {\em convergence}.

\begin{definition}\label{def:sequence}
A {\em sequence} $(\bfx_n)$ is a function whose domain is $\N$. For a nonempty set $B\subseteq\R^m$, a {\em sequence of points in} $B$ is a function $f:\N\to B$ where the outputs---also called {\em terms}---are written
\begin{align}\label{eqn:sequencerange}
f(n)=\bfx_n
\end{align}
for each $n\in\N$. The {\em range} of a sequence is the set of terms given by:
\begin{align}\label{eqn:sequencerange}
f(\N)=\{\bfx_n\in B: n\in\N \textnormal{ and } \bfx_n=f(n)\}.
\end{align}
(The range $f(\N)$ does not repeat elements and does not have any particular ordering.)
\end{definition}

\begin{remark}\label{rmk:sequencerange}
The index $n$ of the term $\bfx_n$ is actually an input for the function $f$ which defines the sequence $(\bfx_n)$. Following convention, sequences in this book are written $(\bfx_n)$, suppressing the fact that $(\bfx_n)$ represents a function but keeping the ordering from $\N$ intact. Technically, the range $f(\N)$ does not repeat elements and does not have any particular ordering on its own. Throughout the book, figures for sequences of real numbers are sometimes graphs with both inputs (positive integers) on a horizontal axis and outputs (terms) on a vertical axis. Other times, figures for sequences are plots of the ranges, especially for sequences of points in the plane where the order is encoded in the indices of the terms. See Figures \ref{fig:twocountablesetssequences}, \ref{fig:sequenceangraph}, and \ref{fig:sequencebngraph}.
\end{remark}

\begin{notation}\label{not:abusesequencenotation}
Following another convention which is a mild abuse of notation, for a sequence $(\bfx_n)$ defined by $f:\N\to B$ we take the phrase ``$\bfy$ is arbitrarily close to the sequence $(\bfx_n)$'' to mean the point $\bfy$ is arbitrarily close to the range $f(\N)$. So, $\bfy\acl{(\bfx_n)}$ means $\bfy\acl{f(\N)}$. Similarly, $(\bfx_n)\subseteq\R^m$ means $f(\N)\subseteq\R^m$.  
\end{notation}

The following examples provide a first glimpse into the connection between {\em arbitrarily close} and {\em limits} of sequences.

\begin{example}\label{eg:twocountablesetssequences}
Consider the following sequences of real numbers $(a_n)$ and $(b_n)$ defined for each $n\in\N$ by
\begin{align}
a_n&=2-\left(\frac{1}{\sqrt{n}}\right) \qquad\textnormal{and}\qquad
b_n=\left[2-\left(\frac{1}{\sqrt{n}}\right)\right](-1)^n.
\end{align}

\begin{figure}
\centering
\begin{tikzpicture} 
\draw (-2,0) node {$(a_n)$};
\draw (4,0) node {$\circ$};
\draw (3.76,0) node {$...$};
\draw (2.85,-0.5) node {$a_1$};
\draw (3.31,-0.5) node {$a_2$};
\draw (4.05,-0.5) node {$2$};
\foreach \Point in {(3,0), (3.29,0), (3.42,0), (3.5,0)}
{
    \node at \Point {\textbullet};
}
\draw (-2,-1.5) node {$(b_n)$};
\draw (0,-1.5) node {$\circ$};
\draw (4,-1.5) node {$\circ$};
\draw (-0.1,-2) node {$-2$};
\draw (4.05,-2) node {$2$};
\draw (3.76,-1.5) node {...};
\draw (0.28,-1.5) node {...};
\foreach \Point in {(1,-1.5), (3.29,-1.5), (0.58,-1.5), (3.5,-1.5)}
{
    \node at \Point {\textbullet};
}
\draw (1.1,-2) node {$b_1$};
\draw (3.14,-2) node {$b_2$};
\draw (0.58,-2) node {$b_3$};
\draw (3.6,-2) node {$b_4$};
\end{tikzpicture}
\caption{The ranges of the sequences of real numbers $(a_n)$ and $(b_n)$ in Example \ref{eg:twocountablesetssequences} are the sets $A$ and $B$, respectively, from Example \ref{eg:twocountablesets}. The black dots, like $\bullet$, represent the terms of the sequences.}
\label{fig:twocountablesetssequences}
\end{figure}

\begin{figure}
\centering
\begin{tikzpicture}
\foreach \Point in {(0.25,0), (0.5,0.29), (0.75,0.42), (1,0.50), (1.25,0.55), (1.5,0.59), (1.75,0.62), (2,0.65), (2.25,0.67), (2.5,0.68), (2.75,0.70), (3,0.71), (3.25,0.72), (3.5,0.73), (3.75,0.74), (4,0.75)}
{
   \node at \Point {\textbullet};
}
\draw (-3,0.75) node {graph of};
\draw (-3,0) node {$(a_n)$};
	\draw (0,-1) node {$-$};
	\draw (0,0) node {$-$};		
	\draw (0,1) node {$-$};
	\draw (-0.5,-1) node {$0$};
	\draw (-0.5,0) node {$1$};
	\draw (-0.5,1) node {$2$};
  \draw[-To, dashed] (0,1) -- (4.5,1);		
  \draw[-To] (0,-1.3) -- (4.5,-1.3);
  \draw[-To] (0,-1.3) -- (0,1.5) node[above] {terms$\qquad$};
	\draw (1,-1.3) node {$|$};
	\draw (2,-1.3) node {$|$};
	\draw (3,-1.3) node {$|$};
	\draw (4,-1.3) node {$|$};
	\draw (1,-1.8) node {$4$};
	\draw (2,-1.8) node {$8$};
	\draw (3,-1.8) node {$12$};
	\draw (4,-1.8) node {$16$};
	\draw (5.5,-1.8) node {indices};
\end{tikzpicture}
\caption{A graph of the sequence $(a_n)$ from Example \ref{eg:twocountablesetssequences}. Technically, the black dots are {\em not} the terms of the sequence. Those would lie on the vertical axis but are not plotted here.}
\label{fig:sequenceangraph}
\end{figure}

\begin{figure}
\centering
\begin{tikzpicture}
\foreach \Point in {(0.25,-1.00), (0.5,1.29), (0.75,-1.42), (1,1.50), (1.25,-1.55), (1.5,1.59), (1.75,-1.62), (2,1.65), (2.25,-1.67), (2.5,1.68), (2.75,-1.70), (3,1.71), (3.25,-1.72), (3.5,1.73), (3.75,-1.74), (4,1.75)}
{
   \node at \Point {\textbullet};
}
\draw (-3,0.75) node {graph of};
\draw (-3,0) node {$(b_n)$};
	\draw (0,-2) node {$-$};
	\draw (0,-1) node {$-$};
	\draw (0,0) node {$-$};		
	\draw (0,1) node {$-$};
	\draw (0,2) node {$-$};	
	\draw (-0.5,-2) node {$-2$};
	\draw (-0.5,-1) node {$-1$};
	\draw (-0.5,0) node {$0$};
	\draw (-0.5,1) node {$1$};		
	\draw (-0.5,2) node {$2$};
  \draw[-To, dashed] (0,-2) -- (4.5,-2);		
  \draw[-To, dashed] (0,2) -- (4.5,2);
  \draw[-To] (0,-2.3) -- (4.5,-2.3);
  \draw[-To] (0,-2.3) -- (0,2.5) node[above] {terms$\qquad$};
	\draw (1,-2.3) node {$|$};
	\draw (2,-2.3) node {$|$};
	\draw (3,-2.3) node {$|$};
	\draw (4,-2.3) node {$|$};
	\draw (1,-2.8) node {$4$};
	\draw (2,-2.8) node {$8$};
	\draw (3,-2.8) node {$12$};
	\draw (4,-2.8) node {$16$};
	\draw (5.5,-2.8) node {indices};
\end{tikzpicture}
\caption{A graph of the sequence $(b_n)$ from Example \ref{eg:twocountablesetssequences}. The terms $b_n$ would be {\em heights} along the vertical axis, but they are not plotted here. Note the heights alternate between being close to $2$ and $-2$ as the indices increase.}
\label{fig:sequencebngraph}
\end{figure}
\end{example}

In the proofs following Example \ref{eg:twocountablesets}, we saw that $2\acl{A}$ and $2\acl{B}$. Hence, $2\acl{(a_n)}$ and $2\acl{(b_n)}$. 
So we can say both $(a_n)$ and $(b_n)$ {\em get arbitrarily close to} 2. Based on your intuition from calculus, does either sequence have 2 as their {\em limit?} In other words, does either sequence {\em converge} to 2? 

\begin{remark}\label{rmk:twosteplimit}
One way I like to parse the limit and convergence of a given sequence is to consider: (i) Does the sequence {\em get arbitrarily close} to a point? and (ii) Does the sequence {\em stay close} to that point?

Overall, my calculus intuition tells me $(a_n)$ {\em converges} to 2 while $(b_n)$ does not. To me, $(a_n)$ both gets close and stays close to 2. On the other hand, $(b_n)$ gets close to $2$ but does not stay close to 2. See Figures \ref{fig:twocountablesetssequences}, \ref{fig:sequenceangraph}, and \ref{fig:sequencebngraph}.
\end{remark}

\begin{figure}
\centering
\begin{tikzpicture}
\draw (-2,0) node {$(b_n)$};
\foreach \Point in {(1,0), (3.29,0), (0.58,0), (3.5,0)}
{
    \node at \Point {\textbullet};
}
\draw (1.1,-0.5) node {$b_1$};
\draw (3.14,-0.5) node {$b_2$};
\draw (0.58,-0.5) node {$b_3$};
\draw (3.6,-0.5) node {$b_4$};
\draw (0,0) node {$\circ$};
\draw (4,0) node {$\circ$};
\draw (3.76,0) node {...};
\draw (0.28,0) node {...};
\draw (-0.1,-0.5) node {$-2$};
\draw (4.05,-0.5) node {$2$};
\begin{scope}[thick, dashed, blue] 
\draw (4,0) circle (1.5cm); 
\end{scope}
\draw[-,semithick, blue] (4.05,0.02) -- (5.06,1.06);
\draw[blue] (4.7,0.3) node {$\varepsilon_0$};
\end{tikzpicture}
\caption{As in Remark \ref{rmk:twosteplimit}, the range of the sequence $(b_n)$ from Example \ref{eg:twocountablesetssequences} seems to get arbitrarily close to $2$ but does not stay close to $2$ in that an infinite number of the terms are $\varepsilon_0$ or more away from $2$.}
\label{fig:bnepsilon0}
\end{figure}

\begin{figure}
\centering
\begin{tikzpicture}
\draw (-2,0) node {$(a_n)$};
\draw (4,0) node {$\circ$};
\draw (4,-0.5) node {$2$};
\draw (3.76,0) node {$...$};
\draw (3,-0.5) node {$a_1$};
\foreach \Point in {(3,0), (3.29,0), (3.42,0), (3.5,0)}
{
    \node at \Point {\textbullet};
}
\draw (-2,-2.5) node {$(a_n)$};
\draw (4,-2.5) node {$\circ$};
\draw (4,-3) node {$2$};
\draw (3.76,-2.5) node {$...$};
\draw (3,-3) node {$a_1$};
\foreach \Point in {(3,-2.5), (3.29,-2.5), (3.42,-2.5), (3.5,-2.5)}
{
    \node at \Point {\textbullet};
}
\draw (-2,-4) node {$(a_n)$};
\draw (4,-4) node {$\circ$};
\draw (4,-4.65) node {$2$};
\draw (3.76,-4) node {$...$};
\draw (3,-4.65) node {$a_1$};
\foreach \Point in {(3,-4), (3.29,-4), (3.42,-4), (3.5,-4)}
{
    \node at \Point {\textbullet};
}
\begin{scope}[thick, dashed, red] 
\draw (4,-4) circle (0.38cm); 
\draw (4,-2.5) circle (0.85cm); 
\draw (4,0) circle (1.4cm); 
\end{scope}	
\draw[-,semithick, red] (4.05,0.05) -- (5,1);
\draw[red] (4.6,0.3) node {$\varepsilon$};
\end{tikzpicture}
\caption{The sequence $(a_n)$ from Example \ref{eg:twocountablesetssequences} gets arbitrarily close and stays close to $2$: For every $\varepsilon>0$, an infinite number of the $a_n$ terms are within $\varepsilon$ of $2$ {\em and} only a finite number of the  terms are more than $\varepsilon$ away from $2$.}
\label{fig:anepsilons}
\end{figure}

As proven below Example \ref{eg:twocountablesets}, $2\acl{B}$ since the $b_n$ where $n$ is even are near $2$, so the sequence $(b_n)$ gets arbitrarily close to $2$. However, an infinite number of the terms of $(b_n)$, specifically those where $n$ is odd, are more than $\varepsilon_0=3/2$ away from the point $2$. (These terms are bunching together near $-2$.) This behavior exhibits what I mean by saying the sequence $(b_n)$ does not stay close to $2$. See Figure \ref{fig:bnepsilon0}.

Something else happens with $(a_n)$. No matter how small we take $\varepsilon>0$ to be, only a finite number of the terms $a_n$ are more than $\varepsilon$ away from $2$. See Figure \ref{fig:anepsilons}. This is one way to interpret what I mean when I say $(a_n)$ gets close and stays close to $2$. It is also a characterization of when a sequence {\em converges} and has a {\em limit}. Dylan Alvarenga, a student from Fall 2021, used very similar terminology when describing how he saw sequences converge.

{\em So what do I mean by all this, exactly?} And can I prove what I'm saying here? We'll get there in Section \ref{sec:limitofasequence} once we have more mathematical tools at our disposal. But there is no need to dive into the definition of convergence and limit (Definition \ref{def:sequentiallimit}) just yet. Playing around with more examples will help pave the way to understanding the difficult definitions ahead.

For now, here is an example showing that ranges of sequences can be found within all sorts of sets. See Figure \ref{fig:openintervalsequence}.

\begin{figure}
\centering
\begin{tikzpicture}
\draw (-2,-1.5) node {$G$};
	\draw[-,semithick] (0,-1.5) -- (4,-1.5);
	\draw (0.02,-1.5) node {$($};
	\draw (3.98,-1.5) node {$)$};	
	\draw (2.75,-1.5) node {$\bullet$};
	\draw (3.25,-1.5) node {$\bullet$};
	\draw (3.75,-1.5) node {$\bullet$};			
	\draw (0,-2) node {$0$};	
	\draw (4.2,-1.5) node {$\ell$};
\begin{scope}[thick, dashed, red] 
\draw (4,-1.5) circle (.5cm); 
\draw (4,-1.5) circle (1cm); 
\draw (4,-1.5) circle (1.5cm); 
\end{scope}	
\end{tikzpicture}
\caption{The open interval $G=(0,3140)$ from Examples \ref{eg:openinterval}, \ref{eg:openintervalrevisit}, and \ref{eg:openintervalsequence} along with a few distances around $\ell=3140$ and points in $G$ within those distances of $\ell$, respectively.}
\label{fig:openintervalsequence}
\end{figure}

\begin{example}\label{eg:openintervalsequence}
Once again, consider the open interval $G=(0,3140)$ and the real number $\ell=3140$. In Examples \ref{eg:openinterval} and \ref{eg:openintervalrevisit}, we showed 
\begin{align}\label{eqn:openintervalsequence}
	3140 \acl G \qquad\textnormal{and}\qquad \sup{G}=3140.  
\end{align}
It turns out that a sequence $(x_n)$ of points in $G$ can be found that provides an infinite collection of elements which are as close to the right endpoint $\ell=3140$ as desired. See Figure \ref{fig:openintervalsequence}. In particular, $x_n=3140-(1/n)$ is in $G$ for each $n\in \N$. Each $x_n$ is represented by a $\bullet$ in the figure, which is not to scale. 

Moreover, for a given $\varepsilon>0$ (which establishes how close to $\ell$ we'd like to get) and thanks to the Corollary of the Archimedean Property \ref{cor:archimedeanproperty}, there is a positive integer $n_\varepsilon$ large enough to give us
\begin{align}
	\frac{1}{n_\varepsilon} <\varepsilon.
\end{align}
Now, by choosing the point $x_{n_\varepsilon}=3140-(1/n_\varepsilon)$ in $G$ we have
\begin{align}
	d_{\R}(x_{n_\varepsilon},\ell)=|x_{n_\varepsilon}-\ell|&=\left|3140-\frac{1}{n_\varepsilon}-3140\right| = \frac{1}{n_\varepsilon}  <\varepsilon.
\end{align}
Hence, $(x_n)\subseteq G$ and $3140\acl{(x_n)}$.
\end{example}

\begin{remark}\label{rmk:almostdefiningconvergence}
Actually, in Example \ref{eg:openintervalsequence}, any positive integer $n>1/\varepsilon$ suffices. Moreover, given $\varepsilon>0$, an infinite number of $x_n$ from the set $G$ are within $\varepsilon$ of $3140$ and only a finite number are $\varepsilon$ or more away from $3140$. As we will be discussed in Section \ref{sec:limitofasequence}, this means $(x_n)$ converges and the limit of $(x_n)$ is $2$.
\end{remark}

The following example provides a sequence that does not seem to stay close to any points at all.

\begin{example}\label{eg:alternatingtoinfinity}
Let $c_n=f(n)=(-1)^nn$ for each $n\in\N$. This sequence jumps between negative odd numbers and positive even numbers, like this:
\begin{align}
	c_1=-1,\quad c_2=2,\quad c_3=-3,\quad c_4=4,\ldots
\end{align}

\begin{figure}
\centering
\begin{tikzpicture} 
\draw (-5,0) node {$(c_n)$};
\foreach \Point in {(-0.5,0), (1,0), (-1.5,0), (2,0), (-2.5,0), (3,0)}
{
    \node at \Point {\textbullet};
}
\draw (-3,0) node {$...$};
\draw (3.5,0) node {$...$};
\draw (-0.5,-0.5) node {$c_1$};
\draw (1,-0.5) node {$c_2$};
\draw (-1.5,-0.5) node {$c_3$};
\draw (2,-0.5) node {$c_4$};
\draw (-2.5,-0.5) node {$c_5$};
\draw (3,-0.5) node {$c_6$};

\draw (-5,-1.5) node {$f(\N)$};
\foreach \Point in {(-0.5,-1.5), (1,0-1.5), (-1.5,-1.5), (2,-1.5), (-2.5,-1.5), (3,-1.5)}
{
    \node at \Point {\textbullet};
}
\draw (-3,-1.5) node {$...$};
\draw (3.5,-1.5) node {$...$};
\draw (-0.5,-2) node {$-1$};
\draw (1,-2) node {$2$};
\draw (-1.5,-2) node {$-3$};
\draw (2,-2) node {$4$};
\draw (-2.5,-2) node {$-5$};
\draw (3,-2) node {$6$};
\end{tikzpicture}
\caption{Two plots of the sequence $(c_n)$ from Example \ref{eg:alternatingtoinfinity}: The top features the terms in their positions along the order determined by their indices, but not their values; the second features the values of the range $f(\N)$ (so the values of the terms), but no order. What would a graph of the sequence $(c_n)$ look like?}
\label{fig:alternatingtoinfinity}
\end{figure}

In Figure \ref{fig:alternatingtoinfinity}, it looks like the terms $c_n$ do not stay close to any particular real number as the index $n$ increases. When we specify their values as integers, the ordering from $\N$ is no longer specified. The ordering the sequence $(c_n)$ inherits from $\N$ plays a significant role in determining that $(c_n)$ does not converge. 
\end{example}

Before getting to the definitions of {\em limit} and {\em convergence} for sequences are the focus of Section \ref{sec:limitofasequence} (Definition \ref{def:sequentiallimit}), let's consider sequences in $\R^2$ to get a sense of how sequences can behave in higher dimensions.

\begin{example}\label{eg:nolimit}
Consider the sequence in the plane $\R^2$ given by 
\begin{align}\label{eqn:nolimit}
\bfz_n=
\begin{cases}
\left[
	\begin{array}{c}
		2+(2/n)	\\
		1	
	\end{array}
\right],\quad\textnormal{if }n \textnormal{ is odd},\\ \\
\left[
	\begin{array}{c}
		-1	\\
		3-(2/n)
	\end{array}
\right],\quad\textnormal{if }n \textnormal{ is even}.
\end{cases}
\end{align}
Also, consider the points $\bfu$ and $\bfv$ given by
\begin{align}
	\bfu=
\left[
	\begin{array}{c}
	2\\
	1	
\end{array}
\right]
\qquad\textnormal{and}\qquad
	\bfv=
\left[
	\begin{array}{c}
	-1\\
	3
\end{array}
\right].
\end{align}
See Figure \ref{fig:nolimit} for a plot of the sequence $(\bfz_n)$ along with the points $\bfu$ and $\bfv$. It turns out we have both $\bfu \acl (\bfz_n)$ and $\bfv \acl (\bfz_n)$.

\begin{figure}
\centering
\begin{tikzpicture}
\draw (-3,2) node {$(\bfz_n)$};
\foreach \Point in {(4,1), (-1,2), (2.67,1), (-1,2.5)}
{
    \node at \Point {\textbullet};
}
	\draw (2,1) node {$\circ$};
	\draw (2,0.5) node {$\bfu$};
	\draw (2.29,1) node {$...$};
	\draw (4.1,0.48) node {$\bfz_1$};
	\draw (2.77,0.48) node {$\bfz_3$};
\foreach \Point in {(-1,2.68), (-1,2.78), (-1,2.88)}
{
    \node at \Point {$.$};
}
	\draw (-1,3) node {$\circ$};
	\draw (-0.53,3) node {$\bfv$};
	\draw (-0.5,1.98) node {$\bfz_2$};
	\draw (-0.5,2.48) node {$\bfz_4$};
\end{tikzpicture}
\caption{The terms $\bfz_n$ where $n$ is odd seem to be bunching together near $\bfu$ while the terms
where $n$ is even are bunching together near $\bfv$. See Example \ref{eg:nolimit}.}
\label{fig:nolimit}
\end{figure}
\end{example}

\begin{proof}[Proof of $\bfu \acl (\bfz_n)$ and $\bfv \acl (\bfz_n)$ in Example \ref{eg:nolimit}]

Let $\varepsilon>0$. By the Corollary of the Archimedean Property \ref{cor:archimedeanproperty}, there is an {\em odd} integer $j_\varepsilon$ that's large enough to give us\footnote{Don't see how these computations work out? If so, that's okay. Some steps were skipped. Fill them in!}
\begin{align}
	d_2(\bfz_{j_\varepsilon},\bfu)=\|\bfz_{j_\varepsilon}-\bfu\|_2&= \frac{2}{{j_\varepsilon}}<\varepsilon.
\end{align}
Again by the Corollary of the Archimedean Property \ref{cor:archimedeanproperty}, there is also an {\em even} integer $k_\varepsilon$ that's large enough to give us
\begin{align}
	d_2(\bfz_{k_\varepsilon},\bfv)=\|\bfz_{k_\varepsilon}-\bfv\|_2&= \frac{2}{k_\varepsilon}<\varepsilon.
\end{align}
Hence, $\bfu \acl (\bfz_n)$ and $\bfv \acl (\bfz_n)$.
\end{proof}

\begin{remark}\label{rmk:nohigherdimgraphs}
Unlike some of the previous examples of sequences of real numbers, I will not provide graphs for sequences in the plane $\R^2$ or Euclidean spaces $\R^m$ where $m\geq 2$. The graph of a sequence in the plane $\R^2$ would have three axes: one for the indices and two for the terms (which are vectors with two entries). Honestly, I don't know how to plot such graphs in a meaningful way for a PDF like this one.
\end{remark}

Does the sequence $(\bfz_n)$ have a {\em limit?} Should it?

\begin{example}\label{eg:2Dlimit}
Consider the sequence $(\bfx_n)$ in the plane $\R^2$ defined by 
\begin{align}\label{eqn:2Dlimit}
\bfx_n=
\left[
	\begin{array}{c}
		2-(1/\sqrt{n})\\
		\\	
		1+(1/n^2)	
	\end{array}
\right]\quad \textnormal{for each }n\in\N.  
\end{align}
Also, consider the point $\bfy$ given by
\begin{align}
	\bfy=
\left[
	\begin{array}{c}
		2\\
		1	
	\end{array}
\right].
\end{align}
See Figure \ref{fig:2Dlimit} for a plot of the sequence $(\bfx_n)$ along with $\bfy$.
\end{example}

\begin{figure}
\centering
\begin{tikzpicture}
\draw (-2,1.5) node {$(\bfx_n)$};
\foreach \Point in {(1,2), (1.29,1.25), (1.42,1.11), (1.5,1.06)}
{
    \node at \Point {\textbullet};
}
	\draw (2,1) node {$\circ$};
	\draw (2.5,1) node {$\bfy$};
	\draw (0.5,2) node {$\bfx_1$};
	\draw (0.79,1.25) node {$\bfx_2$};
\foreach \Point in {(1.65,1.03), (1.75,1.01), (1.85,1)}
{
    \node at \Point {$.$};
}
\end{tikzpicture}
\caption{The sequence $(\bfx_n)$ and point $\bfy$ from Example \ref{eg:2Dlimit}.}
\label{fig:2Dlimit}
\end{figure}

It turns out $\bfy \acl (\bfx_n)$---which will be proven shortly---and it looks to me like $(\bfx_n)$ gets close and stays close to $\bfy$ and only $\bfy$. But the coordinates converge at different rates, which is why the $\bfx_n$ do not lie on a straight line. It also affects the way I choose to prove $\bfy \acl (\bfx_n)$.

\begin{scratch}\label{scr:2Dlimit}
As in Scratch \ref{scr:startatend}, let's try starting at the end. For each $\varepsilon>0$, we want to end up with a term $\bfx_{n_\varepsilon}$ from the sequence $(\bfx_n)$ that is within $\varepsilon$ of $\bfy$. That is, we want a positive integer $n_\varepsilon$ where
\begin{align}\label{eqn:2Dlimitgoal}
	d_2(\bfx_{n_\varepsilon},\bfy)&=\|\bfx_{n_\varepsilon}-\bfy\|_2\\
	&=\sqrt{\left(\frac{-1}{\sqrt{n_\varepsilon}}\right)^2+\left(\frac{1}{n_\varepsilon^2}\right)^2}\\
	&=\sqrt{\frac{1}{n_\varepsilon}+\frac{1}{n_\varepsilon^4}}\\
	&<\varepsilon.
\end{align}
So, how do we for solve $n_\varepsilon$ in the inequality
\begin{align}\label{eqn:2Dlimitgoalrightmost}
\sqrt{\frac{1}{n_\varepsilon}+\frac{1}{n_\varepsilon^4}}
&<\varepsilon?
\end{align}
{\em We don't need to!} 

Any positive integer $n_\varepsilon$ that's large enough to ensure $\bfx_{n_\varepsilon}$ is within $\varepsilon$ of $\bfy$ will do. There's no need for $n_\varepsilon$ to be as small as possible or anything like that. With this in mind, we can try to find a simpler inequality to work with that will still get the job done. 

To find a simpler inequality than \eqref{eqn:2Dlimitgoalrightmost}, note that for every positive integer $n$ we have 
\begin{align}\label{eqn:comparingpowers}
	\frac{1}{n^4}\leq\frac{1}{n}.
\end{align}
This is due to the fact that $1/n$ is between $0$ and $1$ and taking the 4th power of a number between $0$ and $1$ keeps it the same or makes it smaller. I won't prove this here, but it follows from Theorem \ref{thm:inequalityproperties}.

Now, by adding $1/n$ to both sides of \eqref{eqn:comparingpowers} we get
\begin{align}\label{eqn:comparingpowersplus}
	\frac{1}{n}+\frac{1}{n^4}\leq \frac{2}{n}.
\end{align}
Also, since $0\leq x\leq y$ implies $\sqrt{x}\leq\sqrt{y}$ (another consequence of Theorem \ref{thm:inequalityproperties}), we have
\begin{align}\label{eqn:comparingroots}
	\sqrt{\frac{1}{n}+\frac{1}{n^4}}\leq\sqrt{\frac{2}{n}}.
\end{align}
So all we really need is a positive integer  $n_\varepsilon$ large enough to give us
\begin{align}\label{eqn:comparingrootsepsilon}
	\sqrt{\frac{2}{n_\varepsilon}}<\varepsilon.
\end{align}
This inequality is much more manageable than \eqref{eqn:2Dlimitgoalrightmost}, right? Solving for $n_\varepsilon$ in \eqref{eqn:comparingrootsepsilon} yields
\begin{align}\label{eqn:comparingroots}
	n_\varepsilon>\frac{2}{\varepsilon^2}.
\end{align}

We have enough for a proof.
\end{scratch} 

\begin{proof}[Proof of $\bfy\acl(\bfx_n)$ in Example \ref{eg:2Dlimit}]
Let $\varepsilon>0$. (Once again, we let $\varepsilon>0$ to allow for any distance or amount of ``error''  and set up a verification of the definition of arbitrarily close, Definition \ref{def:acl}.) By the Archimedean Property (Theorem \ref{thm:archimedeanproperty}), there is a positive integer $n_\varepsilon$ large enough so that
\begin{align}\label{eqn:comparingspowersproof}
	n_\varepsilon>\frac{2}{\varepsilon^2}.
\end{align}
Since $n_\varepsilon\geq 1$, we have 
\begin{align}
\frac{1}{n_\varepsilon}+\frac{1}{n_\varepsilon^4}\leq \frac{2}{n_\varepsilon}.
\end{align}
So, since $0\leq x\leq y$ implies $\sqrt{x}\leq\sqrt{y}$ we have
\begin{align}\label{eqn:comparingrootsproof}
	\sqrt{\frac{1}{n_\varepsilon}+\frac{1}{n_\varepsilon^4}}\leq\sqrt{\frac{2}{n_\varepsilon}}.
\end{align}
Now consider the term $\bfx_{n_\varepsilon}$ from the sequence $(\bfx_n)$. We have
\begin{align}\label{eqn:2Dlimitproof}
	d_2(\bfx_{n_\varepsilon},\bfy)&=\|\bfx_{n_\varepsilon}-\bfy\|_2\\
	&=\sqrt{\left(\frac{-1}{\sqrt{n_\varepsilon}}\right)^2+\left(\frac{1}{n_\varepsilon^2}\right)^2}\\
	&=\sqrt{\frac{1}{n_\varepsilon}+\frac{1}{n_\varepsilon^4}}\\
	&\leq \sqrt{\frac{2}{n_{\varepsilon}}}\\
	&<\varepsilon.
\end{align}
Hence, $\bfy\acl(\bfx_n)$.
\end{proof}

\begin{remark}\label{rmk:notoptimalindex}
The formula for the key index $n_\varepsilon$ provided by \eqref{eqn:comparingspowersproof} is not optimal, but it is good enough. For instance, in Figure \ref{fig:2Dlimitwith1}, $\varepsilon_1=1$ provides a radius around $\bfy$. Based on the formula \eqref{eqn:comparingspowersproof}, we can choose the index 
\begin{align}
n_1=3>\frac{2}{1^2}=2,
\end{align} 
which ensures $\bfx_{n_1}=\bfx_3$ is within $\varepsilon_1=1$ of $\bfy$. As Figure \ref{fig:2Dlimitwith1} indicates\footnote{In fact, $d_2(\bfx_2,\bfy)=\sqrt{9/16}=3/4<1=\varepsilon_1$.}, the point $\bfx_2$ is also within $\varepsilon_1=1$ of $\bfy$, but this doesn't matter. All we need is a single index whose term is as close to $\bfy$ as we like.
\end{remark}

\begin{figure}
\centering
\begin{tikzpicture}
\draw (-2,1.5) node {$(\bfx_n)$};
\foreach \Point in {(1,2), (1.29,1.25), (1.42,1.11), (1.5,1.06)}
{
    \node at \Point {\textbullet};
}
	\draw (2,1) node {$\circ$};
	\draw (2.3,0.9) node {$\bfy$};
	\draw (0.85,1.75) node {$\bfx_1$};
	\draw (1.6,1.45) node {$\bfx_2$};
\foreach \Point in {(1.65,1.03), (1.75,1.01), (1.85,1)}
{
    \node at \Point {$.$};
}
\begin{scope}[thick, dashed, red] 
\draw (2,1) circle (1cm); 
\end{scope}	
\draw[-,semithick, red] (2.05,1.05) -- (2.71,1.71);
\draw[red] (2.2,1.55) node {$1$};
\end{tikzpicture}
\caption{All of the terms of the sequence $(\bfx_n)$ are within a distance of $1$ away from the point $\bfy$, except for the first term $\bfx_1$. See Examples \ref{eg:2Dlimit} and \ref{eg:2Dlimitconverges}.}
\label{fig:2Dlimitwith1}
\end{figure}

\begin{prob}\label{prob:limitintuition}
To prepare for the next section, what does the notation ``$\lim_{n\to\infty}\bfx_n=\bfy$'' mean to you? Write what your intuition tells you and draw some things. What does your intuition say about the sequences in Examples \ref{eg:twocountablesetssequences}, \ref{eg:openintervalsequence}, \ref{eg:alternatingtoinfinity}, \ref{eg:nolimit}, and \ref{eg:2Dlimit}? Describe the similarities and differences you can see between these sequences. 
\end{prob}

\begin{remark}\label{rmk:convergesornot}
In Definition \ref{def:sequentiallimit} below, we build on the notion of arbitrarily close to define  when a sequence to {\em converges} to its {\em limit}. This definition is notoriously difficult to understand, leading me to present Section \ref{sec:sequencesandremainingclose} first where we discuss what sequences are and what they can be arbitrarily close to. See the works \cite{Barnes,Borzellino,LarsenSwinyard,Seager} for some other approaches to teaching real analysis at the undergraduate level which are, in part, designed to address this difficulty. Also see \cite{Cornu,LarsenSwinyard} for a thorough discussion of the challenges that come with the pedagogy of convergence and limits.

By covering arbitrarily close first, I have set us up to carefully parse the definition of limit and convergence for sequences in Definition \ref{def:sequentiallimit}. I hope it helps!
\end{remark}

Once we have Definition \ref{def:sequentiallimit}, we can formally state and prove some of the comments made throughout this section. Namely:
\begin{enumerate}
	\item The sequences $(a_n)$, $(x_n)$, and $(\bfx_n)$ converge. That is, their limits exist. See Examples \ref{eg:twocountablesetssequences}, \ref{eg:openintervalsequence}, and \ref{eg:2Dlimit}, respectively. 
	\item The sequences $(b_n)$, $(c_n)$, and $(\bfz_n)$ do not converge. That is, they {\em diverge} and their limits do not exist. See Examples \ref{eg:twocountablesetssequences}, \ref{eg:alternatingtoinfinity}, and \ref{eg:nolimit}, respectively.
\end{enumerate}

The notions of arbitrarily close (Definition \ref{def:acl}) and limit of sequence (Definition \ref{def:sequentiallimit} below) are deeply related. This relationship is the focus of the next section and is highlighted by a complete characterization via Theorem \ref{thm:exercise0}. As developed below, if we want to say some point is the limit of a sequence, it is not enough for the point to be arbitrarily close to the sequence. The sequence must stay close to the point {\em and only that point.}

\vs
\section*{Exercises}
\setcounter{theorem}{0}

Exercises are for play: Do scratch work, draw stuff, and make mistakes---make {\em lots} of mistakes---before worrying about writing proofs. {\em Have fun!}


\vs
\section{Limit of a sequence}
\label{sec:limitofasequence}

This sections focuses on one of the most difficult concepts in analysis: A formal definition for the {\em limit} and {\em convergence} of a sequence. My hope is that by first defining and exploring arbitrarily close in Chapter \ref{ch:kernelofanalysis} and sequences in Section \ref{sec:sequencesandstayingclose}, you are now better prepared for this classic challenge. Even so, it's tough! Please be patient and give yourself time to understand what's going on.

As we delve into the formal definitions, please keep in mind whatever properties you expect limits to have. One of the goals is to use formal definitions to capture our intuition about convergence.

A feature I'm aiming for when defining limit and convergence stems from ideas explored earlier in the book: A convergent sequence should get arbitrarily close and stay close to its limit. 

\begin{definition}\label{def:sequentiallimit}
Let $(\bfx_n)$ be a sequence of points in $\R^m$ and let $\bfy$ be a point in $\R^m$. The sequence $(\bfx_n)$ {\em converges to} $\bfy$ if for every distance $\varepsilon>0$ there is a positive integer $n_\varepsilon$ such that for every positive integer $n$ we have
\begin{align}
 n\geq n_\varepsilon\quad\implies\quad	d_m(\bfx_n,\bfy)=\|\bfx_n-\bfy\|_m &< \varepsilon.
\end{align}
In this case, $\bfy$ is called the {\em limit} of the sequence $(\bfx_n)$ and we write
\begin{align}
	\lim_{n\to\infty}\bfx_n=\bfy \qquad\textnormal{or}\qquad \textstyle{\lim_{n\to\infty}\bfx_n=\bfx}.
\end{align} 
Also, we say $(\bfx_n)$ {\em converges to} $\bfy$ and the positive integer $n_\varepsilon$ is called a {\em threshold} for this convergence. If a sequence does not converge to any point, we say the sequence {\em diverges}.
\end{definition}

\begin{remark}\label{rmk:aclvslimit}
The concepts arbitrarily close and convergence are deeply related, but they are not equivalent. For both arbitrarily close and convergence, $\varepsilon>0$ tells us how close we would like the terms of the sequence $(\bfx_n)$ to be to $\bfy$. The key difference is role played by the positive integer $n_\varepsilon$. In both cases, $n_\varepsilon$ is an index whose value depends on the given positive distance or ``error'' $\varepsilon$. With arbitrarily close, $n_\varepsilon$ ensures at least one term $\bfx_{n_\varepsilon}$ is within $\varepsilon$ of $\bfy$. With convergence, $n_\varepsilon$ is a threshold ensuring {\em all} terms $\bfx_n$ that follow $\bfx_{n_\varepsilon}$ are within $\varepsilon$ of $\bfy$ (encoded by indices $n\geq n_\varepsilon$). See Figure \ref{fig:openintervalsequenceconverges}. 

Thus, when $\lim_{n\to\infty}\bfx_n=\bfy$ and we specify a distance, the term $\bfx_{n_\varepsilon}$ gets close enough to $\bfy$ and the terms $\bfx_n$ with $n\geq n_\varepsilon$  stay close to $\bfy$. Also, only a finite number of the terms, those with index between $1$ and $n_\varepsilon-1$, can be more than $\varepsilon$ away from $\bfy$. 
\end{remark}

\begin{figure}
\centering
\begin{tikzpicture}	
\draw (-2,0) node {I};
\draw (5,0) node {$\circ$};
\draw (5,-0.3) node {$\ell$};
\draw (4.6,0.01) node {......};
\foreach \Point in {(0,0), (2.5,0), (3.33,0), (3.75, 0), (4,0), (4.17,0)}
{
    \node at \Point {\textbullet};
}
\draw (0,-0.4) node {$x_1$};
\draw (2.4,-0.4) node {$x_2$};

\draw (-2,-3) node {II};
\draw (5,-3) node {$\circ$};
\draw (5,-3.3) node {$\ell$};
\foreach \Point in {(0,-3), (2.5,-3), (3.33,-3)}
{
    \node at \Point {\textbullet};
}
\draw (0,-3.4) node {$x_1$};
\draw (2.4,-3.4) node {$x_2$};
\draw (3.3,-3.42) node {$x_{n_\varepsilon}$};
\begin{scope}[thick, dashed, red] 
\draw (5,-3) circle (2.3cm); 
\end{scope}	
\draw[-,semithick, red] (5.05,-2.95) -- (6.63,-1.37);
\draw[red] (6.1,-2.25) node {$\varepsilon$};

\draw (-2,-8) node {III};
\draw (5,-8) node {$\circ$};
\draw (5,-8.3) node {$\ell$};
\draw (4.6,-7.99) node {......};
\foreach \Point in {(3.33,-8), (3.75,-8), (4,-8), (4.17,-8)}
{
    \node at \Point {\textbullet};
}
\draw (3.3,-8.42) node {$x_{n_\varepsilon}$};
\draw (4,-8.42) node {$x_n$};
\begin{scope}[thick, dashed, red] 
\draw (5,-8) circle (2.3cm); 
\end{scope}	
\draw[-,semithick, red] (5.05,-7.95) -- (6.63,-6.37);
\draw[red] (6.1,-7.25) node {$\varepsilon$};

\end{tikzpicture}
\caption{A visual progression through the definition of limit and convergence (Definition \ref{def:sequentiallimit}) using the sequence $x_n=3140-(1/n)$ from Examples \ref{eg:openintervalsequence} and \ref{eg:openintervalsequenceconverges} whose limit is $\ell=3140$. Step I: Lay out the range of the sequence and the candidate for the limit $\ell$. Step II: Let $\varepsilon>0$ represent a distance around $\ell$ and find an index $n_\varepsilon$ ensuring $x_{n_\varepsilon}$ is within $\varepsilon$ of $\ell$. Step III: Ignore or discard the terms with indices $1,2,\ldots,n_\varepsilon-1$, then verify $n_\varepsilon$ is a threshold by checking the terms $x_n$ where $n\geq n_\varepsilon$ are all within $\varepsilon$ of $\ell$. If this process works {\em for every} $\varepsilon>0$, then $\lim_{n\to\infty}x_n=\ell$.}
\label{fig:openintervalsequenceconverges}
\end{figure}

It may help to consider quantified versions of the statements where ``$\forall$'' means ``for all'' and ``$\exists$'' means ``there exists'':\s
\begin{center}
\begin{tabular}{lc|cl}
\underline{$\bfy \acl{(\bfx_n)}$} & & & \underline{$\bfy =\lim_{n\to\infty}(\bfx_n)$}\\
 & & & \\
$\forall\, \varepsilon > 0,$ & & & $\forall\, \varepsilon > 0,$ \\
$\exists\, n_\varepsilon \in\N$ such that \hspace{25pt}& & & $\exists\, n_\varepsilon \in\N$ such that\\
\hspace{10pt}$d_m(\bfx_{n_\varepsilon},\bfy)<\varepsilon$. & & &\hspace{10pt}$n\geq n_\varepsilon  \, \Longrightarrow\, d_m(\bfx_n,\bfy)<\varepsilon$. 
\end{tabular}
\end{center}
\s

To me, the difference in the role played by the index $n_\varepsilon$ is the most important reason to introduce a formal definition of arbitrarily close before dealing with convergence. It allows us to break down {\em convergence} and {\em limits} for sequences into the smaller steps {\em arbitrarily close} and {\em staying close}. 

\begin{example}\label{eg:openintervalsequenceconverges}
Recall the sequence $(x_n)$ from Example \ref{eg:openintervalsequence} given by $x_n=3140-(1/n)$ for each $n\in\N$. We have $\lim_{n\to\infty}x_n=\ell=3140$. 
\end{example}

\begin{scratch}\label{scr:openintervalsequenceconverges}
To prove $\lim_{n\to\infty}x_n=\ell=3140$, Example \ref{eg:openintervalsequence} provides scratch work we can build on. Specifically, for a given $\varepsilon>0$, the choice for an index $n_\varepsilon$ where  $1/n_\varepsilon<\varepsilon$ (equivalently, $n_\varepsilon>1/\varepsilon$) provides us with a suitable threshold. For the red $\varepsilon>0$ in Figure \ref{fig:openintervalsequenceconverges}: In II, the term $x_{n_\varepsilon}=3140-1/n_\varepsilon$ is within $\varepsilon$ of $\ell=3140$; {\em moreover}, in III each $x_n$ where $n\geq n_\varepsilon$ (the ``$\bullet$'' and ``$......$'' to the right of $x_{n_\varepsilon}$) is also  within $\varepsilon$ of $\ell=3140$. Thus, II shows $(x_n)$ getting arbitrarily close to $3140$ while III shows $(x_n)$ staying close to $3140$.
\end{scratch}

Once again, by not specifying a particular value of $\varepsilon$, we are accounting for {\em all} positive distances at the same time.

\begin{proof}[Proof of $\lim_{n\to\infty}x_n=\ell=3140$ in Example \ref{eg:openintervalsequenceconverges}]
Let $\varepsilon>0$.
Choose a positive integer $n_\varepsilon$  where $n_\varepsilon>1/\varepsilon$. Then for every index $n\geq n_\varepsilon$ we have
\begin{align}
	\frac{1}{n}\leq \frac{1}{n_\varepsilon}<\varepsilon.
\end{align}
Therefore, $n_\varepsilon$ is a threshold since  for every $n\geq n_\varepsilon$ we have
\begin{align}
	d_{\R}(x_n,\ell)=|x_n-\ell|&=\left|3140-\frac{1}{n}-3140\right| = \frac{1}{n}\leq \frac{1}{n_\varepsilon} <\varepsilon.
\end{align}
Hence, $\lim_{n\to\infty}x_n=\ell=3140$.
\end{proof}

\begin{example}\label{eg:twocountablesetssequencesconverges}
Recall the sequence $(a_n)$ from Example \ref{eg:twocountablesetssequences} given by $a_n=2-(1/\sqrt{n})$ for each $n\in\N$. We have $\lim_{n\to\infty}a_n=2$. 
\end{example}

\begin{scratch}\label{scr:twocountablesetssequencesconverges}
In Example \ref{scr:startatend}, the choice of an index $n_\varepsilon>1/\varepsilon^2$ was developed to show $2\acl{A}$ and therefore $2\acl{(a_n)}$. But $n_\varepsilon$ also serves as a threshold we can use to prove $\lim_{n\to\infty}a_n=2$.
\end{scratch}

\begin{proof}
[Proof of $\lim_{n\to\infty}a_n=2$ in Example \ref{eg:twocountablesetssequencesconverges}]
Let $\varepsilon>0$. Choose a positive integer $n_\varepsilon$ such that 
\begin{align}
	n_\varepsilon>\frac{1}{\varepsilon^2}, \quad\textnormal{which implies}\quad
	\frac{1}{\sqrt{n_\varepsilon}}<\varepsilon.
\end{align}
Note $0\leq x\leq y$ implies $\sqrt{x}\leq \sqrt{y}$. So, for every index $n\geq n_\varepsilon$ we have 
\begin{align}
	\frac{1}{\sqrt{n}}&\leq\frac{1}{\sqrt{n_\varepsilon}}<\varepsilon.
\end{align}
Then $n_\varepsilon$ is a threshold since for every $n\geq n_\varepsilon$ we also have 
\begin{align}
d_\R(a_n,2)&=|a_n-2|=\left|2-\frac{1}{\sqrt{n}}-2 \right| =\frac{1}{\sqrt{n}}\leq\frac{1}{\sqrt{n_\varepsilon}}<\varepsilon.
\end{align}
Therefore, $\lim_{n\to\infty}a_n=2$.
\end{proof}

\begin{remark}\label{rmk:limitproofguide}
Scratch work is an important part of developing any proof, but it is vital to proving some point is the limit of a sequence (equivalently, a sequence converges). In fact, the scratch work for showing a point is arbitrarily close to a sequence often---but not always---leads to a suitable threshold that can be used to prove the point is actually the limit of the sequence. 

Here's a guide for developing scratch work and proving $\lim_{n\to\infty}\bfx_n=\bfy$ by verifying Definition \ref{def:sequentiallimit}:\s

\textbf{Scratch work for $\lim_{n\to\infty}\bfx_n=\bfy$:}
\begin{itemize}
	\item Start with scratch work. Do whatever makes sense for you. 
	\item Include a figure with the sequence $(\bfx_n)$ and the point $\bfy$ which serves as a candidate for the limit.
	\item Consider the inequality you want to end up with, typically:
	\begin{align}
		d_m(\bfx_n,\bfy)&=\|\bfx_n-\bfy\|_m<\varepsilon.
	\end{align}
	Use this inequality to find a formula for a potential threshold $n_\varepsilon$. The triangle inequality as it appears in \eqref{eqn:triangleinequality} and \eqref{eqn:addzero} is used quite often in scratch work and proofs involving limits and convergence.
	\item Consider all indices $n\geq n_\varepsilon$ and their distances $d_m(\bfx_n,\bfy)=\|\bfx_n-\bfy\|_m$ in order to ensure $n_\varepsilon$ is a threshold for the convergence of $(\bfx_n)$ to its limit $\bfy$. 
\end{itemize}	

\textbf{Proving $\lim_{n\to\infty}\bfx_n=\bfy$:} 
\begin{itemize}
	\item Early in the proof, perhaps the first step, write ``Let $\varepsilon>0$'' or something similar, indicating you are accounting for {\em all} positive distances at the same time.
	\item Define or choose an index $n_\varepsilon$ based on your scratch work.
	\item Verify $n_\varepsilon$ is truly a threshold and $\bfy$ is the limit by showing $n\geq n_\varepsilon$ implies
\begin{align}
	d_m(\bfx_n,\bfy)&=\|\bfx_n-\bfy\|_m<\varepsilon.
\end{align}
\end{itemize}
\end{remark} 
 
\begin{example}\label{eg:2Dlimitconverges}
Consider the sequence $(\bfx_n)$ and point $\bfy$ in the plane $\R^2$ from Example \ref{eg:2Dlimit} defined by 
\begin{align}\label{eqn:2Dlimitconverges}
\bfx_n=
\left[
	\begin{array}{c}
		\displaystyle 2-(1/\sqrt{n})\\
		\\	
		\displaystyle 1+(1/n^2)	
	\end{array}
\right]\quad \textnormal{for each }n\in\N,
\quad \textnormal{and} \quad
	\bfy=
\left[
	\begin{array}{c}
		2\\
		1	
	\end{array}
\right].
\end{align}
We have $\lim_{n\to\infty}\bfx_n=\bfy$.
\end{example} 
 
\begin{scratch}\label{scr:2Dlimitconverges}
For starters, see Figure \ref{fig:2Dlimitconverges} for another plot of the sequence $(\bfx_n)$ and point $\bfy$.

\begin{figure}
\centering
\begin{tikzpicture}	
\draw (-2,1.5) node {$(\bfx_n)$};
\foreach \Point in {(1,2), (1.29,1.25), (1.42,1.11), (1.5,1.06)}
{
    \node at \Point {\textbullet};
}
	\draw (2,1) node {$\circ$};
	\draw (2.5,1) node {$\bfy$};
	\draw (0.5,2) node {$\bfx_1$};
	\draw (0.79,1.25) node {$\bfx_2$};
\foreach \Point in {(1.65,1.03), (1.75,1.01), (1.85,1)}
{
    \node at \Point {$.$};
}
\end{tikzpicture}
\caption{A plot of the sequence $(\bfx_n)$ and point $\bfy$ from Examples \ref{eg:2Dlimit} and \ref{eg:2Dlimitconverges} where $\lim_{n\to\infty}\bfx_n=\bfy$.}
\label{fig:2Dlimitconverges}
\end{figure}
 
The bulk of the work to prove $\lim_{n\to\infty}\bfx_n=\bfy$ was done in Scratch \ref{scr:2Dlimit} where the goal was to show $\bfy\acl{(\bfx_n)}$. There, in response to $\varepsilon>0$, the choice of an index was determined to be any positive integer $n_\varepsilon$ large enough to give us
\begin{align}\label{eqn:comparingrootsrepeat}
	n_\varepsilon>\frac{2}{\varepsilon^2}\qquad \Longleftrightarrow \qquad \sqrt{\frac{2}{n_\varepsilon}}<\varepsilon.
\end{align} 
\end{scratch} 
 
\begin{proof}[Proof of $\lim_{n\to\infty}\bfx_n=\bfy$ in Example \ref{eg:2Dlimitconverges}]
Let $\varepsilon>0$. Choose a positive integer $n_\varepsilon$ where
\begin{align}\label{eqn:2Ddefiningthreshold}
	n_\varepsilon>\frac{2}{\varepsilon^2}, \quad\textnormal{which implies}\quad \sqrt{\frac{2}{n_\varepsilon}}<\varepsilon.
\end{align}
For each index $n\geq n_\varepsilon$ we also have $n\geq 1$, hence 
\begin{align}
	\frac{1}{n}+\frac{1}{n^4}\leq \frac{2}{n}.
\end{align}
Furthermore, since $0\leq x\leq y$ implies $\sqrt{x}\leq \sqrt{y}$, we have
\begin{align}\label{eqn:2Dcomparingrootsproof}
	\sqrt{\frac{1}{n}+\frac{1}{n^4}}						&\leq\sqrt{\frac{2}{n}}\leq \sqrt{\frac{2}{n_\varepsilon}}<\varepsilon.
\end{align}
Therefore, for every $n\geq n_\varepsilon$ we have 
\begin{align}\label{eqn:2Dlimitconvergesproof}
	d_2(\bfx_n,\bfy)&=\|\bfx_n-\bfy\|_2\\
	&=\sqrt{\left(\frac{-1}{\sqrt{n}}\right)^2+\left(\frac{1}{n^2}\right)^2}\qquad\textnormal{(by  \eqref{eqn:distance})}\\
	&=\sqrt{\frac{1}{n}+\frac{1}{n^4}}\\
	&\leq\sqrt{\frac{2}{n}}\qquad\textnormal{(by  \eqref{eqn:2Dcomparingrootsproof})}\\
	&\leq\sqrt{\frac{2}{n_\varepsilon}}\\
	&<\varepsilon.
\end{align}
Hence, $\lim_{n\to\infty}\bfx_n=\bfy$.
\end{proof} 
 
\begin{remark}\label{rmk:stilldetailoverconcision}
Revised versions of many of the steps in Scratch \ref{scr:2Dlimit} appear in the proof for $\lim_{n\to\infty}\bfx_n=\bfy$ in Example \ref{eg:2Dlimitconverges}, yet many were left out. I considered shortening the proof even further, but at this point I am still aiming for detail over concision. 
\end{remark}
 
The following theorem establishes the fundamental connection between the definitions of limit and arbitrarily close.

\begin{theorem}\label{thm:exercise0}
Let $\bfy\in\R^m$ and $S\subseteq \R^m$. Then $\bfy$ is arbitrarily close to $S$ if and only if there is a sequence $(\bfx_n)$ of points in $S$ whose limit is $\bfy$. 
\end{theorem}

\begin{figure}
\centering
\begin{tikzpicture}	 
\draw[dashed, fill=blue!15] (-3,-1.41) rectangle (1.41,1.41);
	\draw (-1,-0.7) node {$S$};
	\draw (-3.02,1.42) node {$\bullet$};
	\draw (-3.02,-1.42) node {$\bullet$};
	\draw (1.44,-1.42) node {$\bullet$};		
\begin{scope}[thick, dashed, red] 
\draw (1.45,1.43) circle (2.1cm); 
\draw (1.45,1.43) circle (1.4cm); 
\draw (1.45,1.43) circle (0.7cm); 
\end{scope}	
\draw (-3,-1.41) -- (-3,1.41);
\draw (-3,-1.41) -- (1.41,-1.41);
\draw[-,semithick, red] (1.48,1.48) -- (2.92,2.92);
\draw[red] (2.8,2.55) node {$\varepsilon$};
	\draw (-2,0.5) node {$\bullet$};
	\draw (-2,0.1) node {$\bfx_1$};	
	\draw (-1,0.5) node {$\bullet$};	
	\draw (-1,0.1) node {$\bfx_2$};	
	\draw (-0.1,0.6) node {$\bullet$};
	\draw (0.2,0.2) node {$\bfx_{n_\varepsilon}$};
	\draw (0.6,0.8) node {$\bullet$};
	\draw (1.05,1.05) node {$\bullet$};		
	\draw (1.2,1.2) node {$\cdot$};
	\draw (1.25,1.25) node {$\cdot$};
	\draw (1.3,1.3) node {$\cdot$};
\draw[fill=white] (1.41,1.41) circle (0.08cm);
	\draw (1.8,1.41) node {$\bfy$};	
\end{tikzpicture}
\caption{A set $S$ and a point $\bfy$ in the plane $\R^2$ where $\bfy\acl S$. By Theorem \ref{thm:exercise0}, there is a sequence $(\bfx_n)$ of points in $S$ where $\lim_{n\to\infty}\bfx_n=\bfy$.}
\label{fig:exercise0}
\end{figure}

The proof of Theorem \ref{thm:exercise0} is an important exercise for mathematicians who are dealing with the definition of convergence for the first time. I strongly encourage you to try it yourself before reading the proof provided here. Your understanding of the definitions for arbitrarily close and convergence (Definitions \ref{def:acl} and \ref{def:sequentiallimit}) will be strengthened by giving this proof a shot.

\begin{proof} Assume $\bfy$ is arbitrarily close to $S$ and note that $1/n>0$ for each positive integer $n$. So by the definition of arbitrarily close (Definition \ref{def:acl}), for each $n$ there is some $\bfx_n$ in $S$ where 
\begin{align}
	d_m(\bfx_n,\bfy)&=\|\bfx_n-\bfy\|_m <\frac{1}{n}.
\end{align}
Now, let $\varepsilon>0$. (By not specifying a particular value of $\varepsilon$, we are accounting for all positive distances at the same time.) By the Corollary of the Archimedean Property \ref{cor:archimedeanproperty}, there is a positive integer $n_\varepsilon$ large enough so that $1/n_\varepsilon<\varepsilon$. Then for every $n\geq n_\varepsilon$ we have
\begin{align}
	\frac{1}{n}&\leq \frac{1}{n_\varepsilon} < \varepsilon.
\end{align}
Hence, $n_\varepsilon$ is a threshold for the convergence of $(\bfx_n)$ to its limit $\bfy$  since for every $n\geq n_\varepsilon$ we have
\begin{align}
	d_m(\bfx_n,\bfy)&=\|\bfx_n-\bfy\|_m < \frac{1}{n}\leq \frac{1}{n_\varepsilon} < \varepsilon.
\end{align}
Therefore, $(\bfx_n)$ is a sequence of points in $S$ where $\lim_{n\to\infty}\bfx_n=\bfy$.

For the converse, assume there is a sequence $(\bfx_n)$ of points in $S$ whose limit is $\bfy$ and let $\varepsilon>0$. By the definition of limit and  convergence (Definition \ref{def:sequentiallimit}), there is a threshold $n_\varepsilon$ such that $\bfx_{n_\varepsilon}$ is in $S$ and $\|\bfx_{n_\varepsilon}-\bfy\|_m< \varepsilon$. Therefore, $\bfy\acl{S}$.
\end{proof}

\begin{remark}\label{rmk:infinitepointsfromacl}
In the above proof of Theorem \ref{thm:exercise0}, why does the definition of arbitrarily close---which furnishes just one point in response to a given positive distance---yield a full sequence of points with a potentially infinite set of distinct terms? There are infinitely many positive integers $n$ and each positive distance $1/n$ comes with its own point $\bfx_n$.
\end{remark}

\begin{remark}\label{rmk:aclnotlimit}
A word of caution: Theorem \ref{thm:exercise0} does not say that if a sequence is arbitrarily close to a given point, then the limit exists and is equal to the given point. 
\end{remark}

\begin{example}\label{eg:twocountablesetsdiverges}
Once again, consider the following sequence of real numbers $(b_n)$ from Example \ref{eg:twocountablesetssequences} defined for each $n\in\N$ by
\begin{align}
	b_n&=\left[2-\frac{1}{\sqrt{n}}\right](-1)^n.
\end{align}
In Remark \ref{rmk:twosteplimit}, I mention that my calculus intuition tells me $(b_n)$ does not converge to $2$, despite the fact that $2\acl(b_n)$. See Figure \ref{fig:twocountablesetsdiverges}.

\begin{figure}
\centering
\begin{tikzpicture}	
\draw (-2,0) node {$(b_n)$};
\foreach \Point in {(1,0), (3.29,0), (0.58,0), (3.5,0)}
{
    \node at \Point {\textbullet};
}
\draw (1.1,-0.5) node {$b_1$};
\draw (3.14,-0.5) node {$b_2$};
\draw (0.58,-0.5) node {$b_3$};
\draw (3.6,-0.5) node {$b_4$};
\draw (0,0) node {$\circ$};
\draw (4,0) node {$\circ$};
\draw (3.76,0) node {...};
\draw (0.28,0) node {...};
\draw (-0.1,-0.5) node {$-2$};
\draw (4.05,-0.5) node {$2$};
\begin{scope}[thick, dashed, blue] 
\draw (4,0) circle (1.5cm); 
\end{scope}
\draw[-,semithick, blue] (4.05,0.02) -- (5.06,1.06);
\draw[blue] (4.75,0.35) node {$\varepsilon_0$};
\end{tikzpicture}
\caption{The sequence $(b_n)$ from Examples \ref{eg:twocountablesetssequences} and \ref{fig:twocountablesetsdiverges} does not converge to $2$ (or anything else for that matter). Here, an infinite collection of the terms $b_n$ are $\varepsilon_0=3/2$ or more away from $2$.}
\label{fig:twocountablesetsdiverges}
\end{figure}

\begin{proof}[Proof of $\lim_{n\to\infty}b_n\neq 2$.]
An infinite number of the terms of $(b_n)$, specifically those whose index is odd, are more than $\varepsilon_0=3/2$ away from the point $2$. So, for any given index $n$, there is an odd integer $j_n\geq n$ that yields
\begin{align}
	d_\R(b_{j_n},2)&=|b_{j_n}-2|=4-\frac{1}{\sqrt{j_n}} >\frac{3}{2}=\varepsilon_0.
\end{align}  
Therefore, the distance $\varepsilon_0=3/2$ has no corresponding threshold. Hence, $(b_n)$ does not converge to $2$ and $\lim_{n\to\infty}b_n\neq 2$.
\end{proof}

The above proof only shows that $(b_n)$ does not converge to $2$. A similar proof shows that $(b_n)$ also does not converge to $-2$. The arguments rely on an idea mentioned in Remark \ref{rmk:aclvslimit} due to my former student Dylan Alvarenga: When a sequence converges and we have any $\varepsilon>0$, only a finite number of the terms can be more than $\varepsilon$ away from the limit.

But such an argument is not enough to say $(b_n)$ {\em diverges}. For that, we could show $(b_n)$ does not converge to any real number at all. However, we will try a different approach. In Example \ref{eg:divergentexamples}, we revisit the divergence of the sequence $(b_n)$, as well as the divergence of $(c_n)$ from Example \ref{eg:alternatingtoinfinity} and $(\bfz_n)$ from Example \ref{eg:nolimit}, once we have more mathematical tools at our disposal.
\end{example}

Theorem \ref{thm:exercise0} allows us to prove many facts about limits of sequences, including the following corollary. More results follow later in this chapter and throughout the book.

\begin{corollary}\label{cor:supremumlimit}
Suppose $S,T\subseteq\R$ where $u=\sup{S}$ and $v=\inf{T}$. Then there is a sequence $(x_n)$ of real numbers in $S$ whose limit is $u$ and there is a sequence $(w_n)$ of real numbers in $T$ whose limit is $v$.
\end{corollary}

\begin{proof}
The statement follows from Definition \ref{def:supremumacl} and Theorem \ref{thm:exercise0}. Since a supremum and an infimum are arbitrarily close to their respective sets, each is the limit of a sequence of points in their respective sets. 
\end{proof}

One more idea before moving on to the next section. Another way to think of the role the threshold $n_\varepsilon$ plays for us in Definition \ref{def:sequentiallimit} is that it allows us to say a property {\em eventually} holds.

\begin{definition}\label{def:eventually}
A statement or property $P(\cdot)$ is said to hold {\em eventually} or {\em for large enough} $n\in\N$ if there is a threshold $n_0$ such that  $n\geq n_0$ implies $P(n)$ is true. 
\end{definition}

So when  $\lim_{n\to\infty}\bfx_n=\bfy$, we can say for any $\varepsilon>0$, the terms $\bfx_n$ are eventually within $\varepsilon$ of $\bfy$.

With such a deep connection between arbitrarily close and limits in the context of sequences established in Theorem \ref{thm:exercise0}, it should hopefully come as no surprise that many of the concepts explored in calculus, analysis, and even topology can be discussed in terms of points arbitrarily close to sets.
%
%
%

\vs
\section*{Exercises}
\setcounter{theorem}{0}

Exercises are for play: Do scratch work, draw stuff, and make mistakes---make {\em lots} of mistakes---before worrying about writing proofs. {\em Have fun!}

\xca\label{exer:0} Provide a walkthrough for the proof of Theorem \ref{thm:exercise0}. (This  fundamental exercise connects the definition of arbitrarily close to the definition of limits for sequences.)

\xca Find the real numbers arbitrarily close to, but not in, each of the following sets (they're all subsets of the real line). Don't prove anything, but draw stuff!

\begin{itemize}
	\item[\textbf{a.}] $A=\N$
	\item[\textbf{b.}] $B=\{8-(-1)^n/n: n\in\N\}$
	\item[\textbf{c.}] $C=\{3(-1)^n+1/n:n\in\N\}$
	\item[\textbf{d.}] $D=\{.9,.99,.999,\ldots\}$ 	
	\item[\textbf{e.}] $E=(-1,1]$
	\item[\textbf{f.}] $F=[-1,1]$  
	\item[\textbf{g.}] $G=(-1,1]\cup\{-3+(1/n):n\in\N\}$
	\item[\textbf{h.}] $H=\{0,1,0,1/2,0,1/3,\ldots\}$  
\end{itemize}

\xca A sequence of real numbers $(x_n)$ is a said to be {\em strictly increasing} if $x_n<x_{n+1}$ for each $n\in\N$. 
\begin{enumerate}
	\item Prove that if $U\subseteq\R$, $\sup U$ exists, and $\sup U\notin U$, then there is a strictly increasing sequence $(x_n)$ of points in $U$ such that $\sup U=\lim{x_n}$.
	\item Find an example of a set $V\subseteq\R$ where $\sup{V}$ exists but there is no strictly increasing sequence of points in $V$ whose limit is $\sup{V}$.  
\end{enumerate}	

\xca Suppose $(x_n)\subseteq\R$ is a sequence where $\lim_{n\to\infty}x_n=0$ and let $z_n=x_n+3140+(-1)^n$ for each $n\in\N$. (Compare with Example \ref{eg:twocountablesetsdiverges}.) 
\begin{enumerate}
	\item Prove $3139\acl{(z_n)}$ and $3141\acl{(z_n)}$.
	\item Despite the results of part (i), prove 
	\begin{align}
		\lim_{n\to\infty}z_n\neq 3139 \quad\textnormal{and}\quad \lim_{n\to\infty}z_n\neq 3141.
	\end{align}
	\item Prove $(z_n)$ diverges by showing for every real number $c$ we have
	\begin{align}
		\lim_{n\to\infty}z_n\neq c.
	\end{align} (The upcoming result Divergence Criteria \ref{thm:divergencecriteria} would make short work of this proof, but the goal here is to get some practice working with the negation of Definition \ref{def:sequentiallimit}.)
\end{enumerate}

\vs
\section{Properties of limits}
\label{sec:propertiesoflimits}


Now that we have a formal definition for the limit and convergence of a sequence in Definition \ref{def:sequentiallimit}, we can use it to prove a bunch of ideas from calculus and multivariable calculus. Lots of scratch work and details are provided with the proofs. Much of the effort comes from the scratch work for determining appropriate thresholds.  

First, the limit of a sequence in a Euclidean space is unique.

\begin{theorem}\label{thm:limitsunique}
Every convergent sequence in $\R^m$ has a unique limit. 
\end{theorem}

\begin{scratch}
My idea is to show any two points $\bfy$ and $\bfz$ satisfying the definition of limit for the same sequence are arbitrarily close to each other, so they must be the same point (see Lemma \ref{lem:equalpoints}). The definition of limit and convergence for sequences (Definition \ref{def:sequentiallimit}) ensures that we can respond to {\em any positive distance} we like.

However, unlike the examples explored in Section \ref{sec:limitofasequence}, we do not have explicit formulas for the thresholds to work with. Instead, we are in a more general setting where we will both {\em use} assumptions that certain limits exist and {\em show} other related limits exist. Also, whenever $\varepsilon>0$, we have $\varepsilon/2>0$ as well. So given any $\varepsilon>0$, we can use the definition of limit to take both $\bfy$ and $\bfz$ to be within $\varepsilon/2$ of the sequence, each coming with their own threshold which we can use to create a particular index to complete the proof. The points $\bfy$ and $\bfz$ would then be within any $\varepsilon$ of each other thanks to the triangle inequality \eqref{eqn:triangleinequality} combined with one particular term in the sequence. 
\end{scratch}

\begin{proof}
Suppose $(\bfx_n)$ is a convergent sequence in $\R^m$ where
\begin{align}
	\bfy&=\lim_{n\to\infty}\bfx_n \qquad\textnormal{and}\qquad 
	\bfz=\lim_{n\to\infty}\bfx_n.
\end{align}
Let $\varepsilon>0$. Then $\varepsilon/2>0$ and by the definition of limit (Definition \ref{def:sequentiallimit}), there are two positive integers $j_{\varepsilon/2}$ and $k_{\varepsilon/2}$ where 
\begin{align}
	n &\geq j_{\varepsilon/2} \quad\implies\quad d_m(\bfx_n,\bfy)<\varepsilon/2  \qquad\textnormal{and} \label{eqn:limituniquey}\\
	n &\geq k_{\varepsilon/2} \quad\implies\quad d_m(\bfx_n,\bfz) <\varepsilon/2.\label{eqn:limituniquez}
\end{align}
Now consider the index $n_\varepsilon=\max\{j_{\varepsilon/2},k_{\varepsilon/2}\}$ (so $n_\varepsilon$ is the larger of the two). We have both $n_\varepsilon\geq j_{\varepsilon/2}$ and $n_\varepsilon\geq k_{\varepsilon/2}$, therefore
\begin{align}
	d_m(\bfy,\bfz)&\leq d_m(\bfy,\bfx_{n_\varepsilon})+d_m(\bfx_{n_\varepsilon},\bfz)
	\qquad\textnormal{(tri. ineq. \eqref{eqn:triangleinequality})}\\
	&< \frac{\varepsilon}{2} +\frac{\varepsilon}{2} \qquad\textnormal{(\eqref{eqn:limituniquey} and  \eqref{eqn:limituniquez})}\\
	&=\varepsilon.
\end{align}
Hence, $\bfy\acl\{\bfz\}$ and by Lemma \ref{lem:equalpoints}, $\bfy=\bfz$.
\end{proof}

\begin{remark}\label{rmk:notationsubscriptproof}
Does the above proof make sense? Without more detailed scratch work, it may be hard to tell why each step is there. However, you should be able to tell how each line of the proof follows from assertions and conclusions that come before. This is easier said than done and takes some time. So, please take your time. If you don't feel comfortable with this proof, try writing up a walkthrough. How would you write the proof?

Can you tell why I chose the subscripts for $n_\varepsilon$, $j_{\varepsilon/2}$, and $k_{\varepsilon/2}$ the way I did? Each of these positive integers is an index for the sequence $(\bfx_n)$ given in response to the distance $\varepsilon$ or $\varepsilon/2$, accordingly. These distances allow us to make use of the definition of arbitrarily close (Definition \ref{def:acl}) and the definition of limit and convergence for sequences (Definition \ref{def:sequentiallimit}).
\end{remark}

The following is a special case of the definition of bounded sets (Definition \ref{def:boundedset}) adapted for sequences.

\begin{definition}\label{def:boundedsequence}
A sequence $(\bfx_n)$ of points in $\R^m$ is {\em bounded} if its range is a bounded set. That is, $(\bfx_n)$ is bounded if there is a real number $b\geq 0$ such that for every index $n\in\N$ we have 
\begin{align}
	\|\bfx_n\|_m&<b.
\end{align}
In this case, we say $b$ is a {\em bound} for the sequence $(\bfx_n)$. If a sequence is not bounded, we say it is {\em unbounded}.
\end{definition}

\begin{theorem}\label{thm:convergentsequencebounded}
Every convergent sequence in $\R^m$ is bounded. 
\end{theorem}

\begin{scratch}\label{scr:convergentsequencebounded}
Consider the following pair of figures of convergent sequences. In Figure \ref{fig:convergentsequencebounded1}, the term $\bfx_1$ happens to have the largest norm out of all of the terms in the sequence. 

\begin{figure}
\centering
\begin{tikzpicture}  
\draw (0,0) circle (3.5cm);	
	\draw (0.05,0) -- (3.5,0);  
	\draw (1.75,-0.4) node {$b$};    
	\draw (0,-0.02) node {$\circ$};
	\draw (0,-0.4) node {$\mathbf{0}$};
\draw (1.45,1.43) node {$\circ$};
	\draw (1.8,1.43) node {$\bfy$};				
\begin{scope}[thick, dashed, blue] 
\draw (1.45,1.43) circle (0.7cm); 
\end{scope}	
\draw[-,semithick, blue] (1.49,1.49) -- (2,1.9);
\draw[blue] (1.55,1.85) node {$\varepsilon_0$};
	\draw (-3.5,0) node {$\bullet$};
	\draw (-3.1,-0.2) node {$\bfx_1$};	
	\draw (-1.75,0.3) node {$\bullet$};	
	\draw (-1.75,-0.1) node {$\bfx_2$};	
	\draw (-0.1,0.6) node {$\bullet$};
	\draw (0.1,0.63) node {$\cdot$};
	\draw (0.25,0.65) node {$\cdot$};
	\draw (0.4,0.67) node {$\cdot$};	
	\draw (0.6,0.7) node {$\bullet$};
	\draw (0.9,0.3) node {$\bfx_{n_0-1}$};	
	\draw (1.05,1.05) node {$\bullet$};
	\draw (1.5,1) node {$\bfx_{n_0}$};		
	\draw (1.18,1.18) node {$\cdot$};
	\draw (1.24,1.24) node {$\cdot$};
	\draw (1.3,1.3) node {$\cdot$};
\end{tikzpicture}
\caption{A convergent sequence $(\bfx_n)$ with limit $\bfy$ in $\R^m$ where the first term $\bfx_1$ has the largest norm. See Scratch Work \ref{scr:convergentsequencebounded} and the proof of Theorem \ref{thm:convergentsequencebounded}.}
\label{fig:convergentsequencebounded1}
\end{figure}

\begin{figure}
\centering
\begin{tikzpicture}  
\draw (0,0) circle (2.75cm);	
	\draw (2.75,0) --(0.05,0);  
	\draw (1.38,-0.4) node {$b$};    
	\draw (0,-0.02) node {$\circ$};
	\draw (0,-0.4) node {$\mathbf{0}$};
\draw (1.45,1.43) node {$\circ$};
	\draw (1.8,1.43) node {$\bfy$};				
\begin{scope}[thick, dashed, blue] 
\draw (1.45,1.43) circle (0.7cm); 
\end{scope}	
\draw[-,semithick, blue] (1.49,1.49) -- (2,1.9);
\draw[blue] (1.55,1.85) node {$\varepsilon_0$};
	\draw (-1.5,0.5) node {$\bullet$};
	\draw (-1.5,0.1) node {$\bfx_1$};	
	\draw (-0.8,0.5) node {$\bullet$};	
	\draw (-0.8,0.1) node {$\bfx_2$};	
	\draw (-0.1,0.6) node {$\bullet$};
	\draw (0.1,0.63) node {$\cdot$};
	\draw (0.25,0.65) node {$\cdot$};
	\draw (0.4,0.67) node {$\cdot$};	
	\draw (0.6,0.7) node {$\bullet$};
	\draw (0.9,0.3) node {$\bfx_{n_0-1}$};	
	\draw (1.05,1.05) node {$\bullet$};
	\draw (1.5,1) node {$\bfx_{n_0}$};		
	\draw (1.18,1.18) node {$\cdot$};
	\draw (1.24,1.24) node {$\cdot$};
	\draw (1.3,1.3) node {$\cdot$};
\end{tikzpicture}
\caption{A convergent sequence $(\bfx_n)$ with limit $\bfy$ in $\R^m$ where no particular term has the largest norm but $\|\bfy\|+\varepsilon_0$ for some positive $\varepsilon_0$ serves as a bound for $(\bfx_n)$. See Scratch Work \ref{scr:convergentsequencebounded} and the proof of Theorem \ref{thm:convergentsequencebounded}.}
\label{fig:convergentsequencebounded2}
\end{figure}

In Figure \ref{fig:convergentsequencebounded2}, no particular term in the sequence has the largest norm, but for some $\varepsilon_0>0$ the real number $\varepsilon_0+\|\bfy\|$ is a bound for the sequence.

In any case, when a sequence $(\bfx_n)$ converges to $\bfy$ in $\R^m$, we can choose any $\varepsilon_0>0$ and get a threshold $n_0$ to ensure all of the terms $\bfx_n$ where $n\geq n_0$ are within $\varepsilon_0$ of $\bfy$. From there, the real number $b$ defined by
\begin{align}
	b=\max\{\|\bfx_1\|_m,\|\bfx_2\|_m,\ldots,\|\bfx_{n_0-1}\|_m,\varepsilon_0+\|\bfy\|_m\}
\end{align}
is a bound for the sequence.
\end{scratch}

\begin{proof}
Suppose $(\bfx_n)$ is a sequence in $\R^m$ whose limit is $\bfy$. Let $\varepsilon_0=7$. Since $\bfy=\lim_{n\to\infty}\bfx_n$, there is a threshold $n_0$ where $n\geq n_0$ implies
\begin{align}\label{eqn:initialbound7}
	d_m(\bfx_n,\bfy)&=\|\bfx_n-\bfy\|_m<\varepsilon_0=7.
\end{align}
Therefore, for $n\geq n_0$,
\begin{align}
	\|\bfx_n\|_m&=\|\bfx_n\underbrace{-\bfy+\bfy}_{\textnormal{add } \mathbf{0}}\|_m\\
	&\leq\|\bfx_n-\bfy\|_m+\|\bfy\|_m\qquad\textnormal{(tri. ineq. \eqref{eqn:addzero})}\\
	&<7+\|\bfy\|_m. \qquad\textnormal{(\eqref{eqn:initialbound7})}
\end{align}
 
Now define $b$ by
\begin{align}
	b=\max\{\|\bfx_1\|_m,\|\bfx_2\|_m,\ldots,\|\bfx_{n_0-1}\|_m,7+\|\bfy\|_m\}.
\end{align}
Then $b\geq 0$ and for every index $n\in\N$ we have 
\begin{align}
	\|\bfx_n\|_m&\leq b.
\end{align}
Therefore, $(x_n)$ is bounded.
\end{proof}

You may recall ideas from calculus such as ``the limit of sums is the sum of the limits'' and others relating algebra to limits and convergence. We now have the mathematical tools to formally prove these ideas.

\begin{theorem}\label{thm:algebraiclimitseuclidean}
Suppose $c\in\R$ and suppose $(\bfa_n)$ and $(\bfb_n)$ are convergent sequences in $\R^m$ where $\lim_{n\to\infty}\bfa_n=\bfa$ and $\lim_{n\to\infty}\bfb_n=\bfb$.
Then
\begin{enumerate}
	\item $\displaystyle \lim_{n\to\infty}(\bfa_n+\bfb_n)
	=\lim_{n\to\infty}\bfa_n+\lim_{n\to\infty}\bfb_n=\bfa+\bfb, \qquad$ and\\
	\item $\displaystyle \lim_{n\to\infty} (c\,\bfa_n)
	=c\lim_{n\to\infty}\bfa_n=c\,\bfa$.
\end{enumerate}
\end{theorem}

\begin{remark}\label{rmk:algebraiclimitseuclidean}
To prove each statement in Theorem \ref{thm:algebraiclimitseuclidean}, we can verify the definition of limit (Definition \ref{def:sequentiallimit}) holds by considering an arbitrary $\varepsilon>0$ and finding a suitable threshold $n_\varepsilon$. This threshold can be shown to ensure the terms with indices $n\geq n_\varepsilon$ are within $\varepsilon$ of the proposed limit. 

As in the proof of Theorem \ref{thm:limitsunique} but unlike the examples explored in Section \ref{sec:limitofasequence}, we do not have explicit formulas to work with. We are in a more general setting where we will both {\em use} assumptions that certain limits exist along with their corresponding thresholds, and {\em show} other related limits exist by defining new thresholds as needed.
\end{remark}

\begin{scratch}\label{scr:algebraiclimitssum}
To derive some scratch work for directly proving 
\begin{align}
	\lim_{n\to\infty}(\bfa_n+\bfb_n)=\bfa+\bfb,
\end{align} let's start at the end. Given $\varepsilon>0$, we want to end up with 
\begin{align}\label{eqn:algebraiclimitssumgoal}
	\|(\bfa_n+\bfb_n)-(\bfa+\bfb)\|_m<\varepsilon.
\end{align}
We can assume
\begin{align}
	\lim_{n\to\infty}\bfa_n&=\bfa\qquad\textnormal{and}\qquad\lim_{n\to\infty}\bfb_n=\bfb,
\end{align}
so by Definition \ref{def:sequentiallimit}, given any $\varepsilon>0$ there are thresholds $j_\varepsilon$ and $k_\varepsilon$ where $n\geq j_\varepsilon$ implies
\begin{align}\label{eqn:initialbounda}
	\|\bfa_n-\bfa\|_m<\varepsilon
\end{align}
and $n\geq k_\varepsilon$ implies
\begin{align}\label{eqn:initialboundb}
	\|\bfb_n-\bfb\|_m<\varepsilon.
\end{align}
Combining \eqref{eqn:initialbounda} and \eqref{eqn:initialboundb} by taking $n$ large enough so that both $n\geq j_\varepsilon$ and $n\geq k_\varepsilon$, and applying triangle inequality \eqref{eqn:triangleinequality2}, gives us
\begin{align}
	\|(\bfa_n+\bfb_n)-(\bfa+\bfb)\|_m
	&=\|(\bfa_n-\bfa)+(\bfb_n-\bfb)\|_m\\
	&\leq\|\bfa_n-\bfa\|_m+\|\bfb_n-\bfb)\|_m\\
	&<\varepsilon+\varepsilon\\
	&=2\varepsilon.
\end{align}
This isn't quite our goal \eqref{eqn:algebraiclimitssumgoal}, but we can adapt: The definition of limit and convergence for sequences (Definition \ref{def:sequentiallimit}) ensures that we can respond to {\em any positive distance} we like with suitable thresholds. So, as in the proof of Theorem \ref{thm:limitsunique}, $\varepsilon/2>0$ whenever $\varepsilon>0$. We can take advantage of the assumptions 
\begin{align}
	\lim_{n\to\infty}\bfa_n&=\bfa\qquad\textnormal{and}\qquad\lim_{n\to\infty}\bfb_n=\bfb
\end{align}
by considering indices $n$ and thresholds $j_{\varepsilon/2}$ and $k_{\varepsilon/2}$ where both
\begin{align}
	n\geq j_{\varepsilon/2} &\quad\implies\quad \|\bfa_n-\bfa\|_m<\frac{\varepsilon}{2} \qquad\textnormal{and}
	\label{eqn:secondbounda}\\
	n\geq k_{\varepsilon/2} &\quad\implies\quad 	\|\bfb_n-\bfb\|_m<\frac{\varepsilon}{2}.\label{eqn:secondboundb}
\end{align}
See Figure \ref{fig:algebraiclimitssum}. Choosing $n_\varepsilon=\max\{j_{\varepsilon/2},k_{\varepsilon/2}\}$ (so $n_\varepsilon$ is the larger of the two) yields a sufficient threshold for the sums.
\end{scratch}

\begin{figure}
\centering
\begin{tikzpicture}	
\draw (-3,0) node {$(\bfa_n)$};
\draw (4,0) node {$\circ$};
\draw (4.15,-0.25) node {$\bfa$};
\draw (3.76,0.02) node {$...$};
\foreach \Point in {(2,0), (3,0), (3.5,0)}
{
    \node at \Point {\textbullet};
}
\draw (2,-0.4) node {$\bfa_1$};
\draw (3,-0.4) node {$\bfa_2$};
\draw (3.6,-0.3) node {$\bfa_{n_\varepsilon}$};
\begin{scope}[thick, dashed, red] 
\draw (4,0) circle (0.85cm); 
\end{scope}	
\draw[-,semithick, red] (4.05,0.05) -- (4.6,0.6);
\draw[red] (4.05,0.5) node {$\varepsilon/2$};

\draw (-3,-2) node {$(\bfb_n)$};
\draw (4,-2) node {$\circ$};
\draw (4.15,-2.25) node {$\bfb$};
\draw (3.76,-1.98) node {$...$};
\foreach \Point in {(1,-2), (3,-2), (3.5,-2)}
{
    \node at \Point {\textbullet};
}
\draw (1,-2.4) node {$\bfb_2$};
\draw (3,-2.4) node {$\bfb_1$};
\draw (3.6,-2.3) node {$\bfb_{n_\varepsilon}$};
\begin{scope}[thick, dashed, red] 
\draw (4,-2) circle (0.85cm); 
\end{scope}	
\draw[-,semithick, red] (4.05,-1.95) -- (4.6,-1.4);
\draw[red] (4.05,-1.5) node {$\varepsilon/2$};

\draw (-3,-5) node {$(\bfa_n+\bfb_n)$};
\draw (4,-5) node {$\circ$};
\draw (4.75,-5.15) node {$\bfa+\bfb$};
\draw (3.73,-4.98) node {$...$};
\foreach \Point in {(0,-5), (1,-5), (3,-5)}
{
    \node at \Point {\textbullet};
}
\draw (-0.3,-5.4) node {$\bfa_2+\bfb_2$};
\draw (1.3,-5.4) node {$\bfa_1+\bfb_1$};
\draw (3.3,-5.3) node {$\bfa_{n_\varepsilon}+\bfb_{n_\varepsilon}$};
\begin{scope}[thick, dashed, red]  
\draw (4,-5) circle (1.7cm);  
\end{scope}	
\draw[-,semithick, red] (4.05,-4.95) -- (5.22,-3.83);
\draw[red] (4.5,-4.1) node {$\varepsilon$};
\end{tikzpicture}
\caption{A figure to accompany the proof of Theorem \ref{thm:algebraiclimitseuclidean} showing $\lim_{n\to\infty}\bfa_n=\bfa$ and $\lim_{n\to\infty}\bfb_n=\bfb$ implies $\lim_{n\to\infty}(\bfa_n+\bfb_n)=\bfa+\bfb$.}
\label{fig:algebraiclimitssum}
\end{figure}

\begin{proof}[Proof of $\lim_{n\to\infty}(\bfa_n+\bfb_n)=\bfa+\bfb$ in Theorem \ref{thm:algebraiclimitseuclidean}]
Assume
\begin{align}\label{eqn:algebraiclimitssumassumption}
	\lim_{n\to\infty}\bfa_n&=\bfa\qquad\textnormal{and}\qquad\lim_{n\to\infty}\bfb_n=\bfb,
\end{align}
and let $\varepsilon>0$. Then $\varepsilon/2>0$ and there are thresholds $j_{\varepsilon/2}$ and $k_{\varepsilon/2}$ where 
\begin{align}
	n\geq j_{\varepsilon/2} &\implies\|\bfa_n-\bfa\|_m<\frac{\varepsilon}{2} \qquad\textnormal{and}
	\label{eqn:boundaproof}\\
	n\geq k_{\varepsilon/2} &\implies	\|\bfb_n-\bfb\|_m<\frac{\varepsilon}{2}.\label{eqn:boundbproof}
\end{align}
Define $n_\varepsilon=\max\{j_{\varepsilon/2},k_{\varepsilon/2}\}$ (so $n_\varepsilon$ is the larger of the two). Then every index $n$ where $n\geq n_\varepsilon$ is large enough to give us both $n\geq j_{\varepsilon/2}$ and $n\geq k_{\varepsilon/2}$. So by the triangle inequality \eqref{eqn:triangleinequality2}, \eqref{eqn:boundaproof}, and \eqref{eqn:boundbproof} we have $n\geq n_\varepsilon$ implies
\begin{align}
	\|(\bfa_n+\bfb_n)-(\bfa+\bfb)\|_m
	&=\|(\bfa_n-\bfa)+(\bfb_n-\bfb)\|_m\\
	&\leq\|\bfa_n-\bfa\|_m+\|\bfb_n-\bfb\|_m\\
	&<\frac{\varepsilon}{2}+\frac{\varepsilon}{2}\\
	&=\varepsilon.
\end{align}
Therefore, $\lim_{n\to\infty}(\bfa_n+\bfb_n)=\bfa+\bfb$.
\end{proof}

Next, let's prove $\lim_{n\to\infty}c\,\bfa_n=c\,\bfa$ as in Theorem \ref{thm:algebraiclimitseuclidean}.

\begin{scratch}
Once again, let's start at the end. Given $\varepsilon>0$, we want to end up with 
\begin{align}\label{eqn:algebraiclimitsscalargoal}
	\|c\,\bfa_n-c\,\bfa\|_m<\varepsilon.
\end{align}
By \eqref{eqn:scalarfactor} we can consider
\begin{align}\label{eqn:algebraiclimitsscalargoalrevised}
	\|c\,\bfa_n-c\,\bfa\|_m&=|c|\|\bfa_n-\bfa\|_m<\varepsilon.
\end{align}
So if $c\neq 0$, then $|c|\neq 0$ and we can divide both sides of the rightmost inequality in \eqref{eqn:algebraiclimitsscalargoalrevised} to get
\begin{align}\label{eqn:algebraiclimitsscalargoalfinal}
	\|\bfa_n-\bfa\|_m<\frac{\varepsilon}{|c|}.
\end{align}
So, by assuming $\lim_{n\to\infty}\bfa_n=\bfa$,  Definition \ref{def:sequentiallimit} ensures a threshold $n_{\varepsilon/|c|}$ can be found in response to the positive distance $\varepsilon/|c|$ that will suffice, as long as $c\neq 0$.  
\end{scratch}

\begin{proof}[Proof of $\lim_{n\to\infty}c\,\bfa_n=c\,\bfa$ in Theorem \ref{thm:algebraiclimitseuclidean}] This proof has two cases: (i) $c=0$ and (ii) $c\neq 0$.

\underline{Case (i) $c=0$}: Suppose $c=0$ and $\varepsilon>0$. Define $n_\varepsilon=7$. Then for every index $n\geq n_\varepsilon=7$ we have
\begin{align}
	\|c\,\bfa_n-c\,\bfa\|_m&=\|\mathbf{0}-\mathbf{0}\|_m=0<\varepsilon.
\end{align}
Therefore, $\lim_{n\to\infty}c\,\bfa_n=c\,\bfa=\mathbf{0}$.

\underline{Case (ii) $c\neq 0$}: Suppose $c\neq 0$, $\lim_{n\to\infty}\bfa_n=\bfa$, and $\varepsilon>0$. Then $\varepsilon/|c|>0$ and by the definition of limit (Definition \ref{def:sequentiallimit}), there is a threshold $n_{\varepsilon/|c|}$ such that 
\begin{align}\label{eqn:algebraiclimitsscalarpenultimateproof} 
	n\geq n_{\varepsilon/|c|} \quad\implies\quad\|\bfa_n-\bfa\|_m<\frac{\varepsilon}{|c|}.
\end{align}
By \eqref{eqn:scalarfactor} and \eqref{eqn:algebraiclimitsscalarpenultimateproof},
for all indices $n\geq n_\varepsilon$ we have
\begin{align}
	\|c\,\bfa_n-c\,\bfa\|_m
	&=|c|\|\bfa_n-\bfa\|_m\\
	&<|c|\frac{\varepsilon}{|c|} \\
	&=\varepsilon.
\end{align}	
Therefore, $\lim_{n\to\infty}c\,\bfa_n=c\,\bfa.$
\end{proof}

There are more algebraic properties for limits when we restrict our attention to the real line. Do you remember hearing statements like ``the limit of a product is the product of the limits'' and ``the limit of a quotient is the quotient of the limits''?

\begin{theorem}\label{thm:algebraiclimitsreal}
Suppose $(a_n)$ and $(b_n)$ are convergent sequences in the real line $\R$ where $\lim_{n\to\infty}a_n=a$ and $\lim_{n\to\infty}b_n=b$. Then
\begin{align}
	\lim_{n\to\infty}(a_nb_n)&=\left(\lim_{n\to\infty}a_n\right)\left(\lim_{n\to\infty}b_n\right)=ab.
\end{align}	
If we further assume $b\neq 0$ and $b_n\neq 0$ for any index $n$, then 
\begin{align}
	\lim_{n\to\infty}\frac{a_n}{b_n}	&=\frac{\displaystyle \lim_{n\to\infty}a_n}{\displaystyle\lim_{n\to\infty}b_n}		=\frac{a}{b}.
\end{align}
\end{theorem}\s

\begin{scratch}
To derive some scratch work for showing \begin{align}
	\lim_{n\to\infty}a_nb_n=ab
\end{align} 
in Theorem \ref{thm:algebraiclimitsreal}, let's start at the end. Given $\varepsilon>0$, we want to end up with 
\begin{align}\label{eqn:algebraiclimitsproductgoal}
	|a_nb_n-ab|<\varepsilon.
\end{align}

In order to get \eqref{eqn:algebraiclimitsproductgoal}, we can try to play around with expressions involving both $|a_n-a|$ and $|b_n-b|$ since they appear when we assume $\lim_{n\to\infty}a_n=a$ and $\lim_{n\to\infty}b_n=b$. 

There are plenty of tools at our disposal to help us here, but the key step of the proof below may seem to come out of nowhere. So, try to consider it to be a technique someone figured out a long time ago that we can now use to our advantage. This key step is the choice of a particular version of zero given by 
\begin{align}
	0=-ab_n+ab_n.
\end{align}
Adding this especially nice version of zero along with the triangle inequality \eqref{eqn:addzero} yields the following string of ineqaulaities:
\begin{align}\label{eqn:limitproductscratch}
	|a_nb_n-ab|&=|a_nb_n\underbrace{-ab_n+ab_n}_{\textnormal{add }0}-ab|\\
	&\leq |a_nb_n-ab_n|+|ab_n-ab| \qquad\textnormal{(tri. ineq. \eqref{eqn:addzero})}\\
	&=|b_n||a_n-a|+|a||b_n-b|.
\end{align}
From there, assuming $\lim_{n\to\infty}a_n=a$ and $\lim_{n\to\infty}b_n=b$ allows us to respond to a given $\varepsilon>0$ with thresholds $j_\varepsilon$ and $k_\varepsilon$ where 
\begin{align}\label{eqn:limitproductaepsilon}
	n &\geq j_\varepsilon \quad\implies\quad |a_n-a|<\varepsilon\qquad \textnormal{and}\\
\label{eqn:limitproductbepsilon}
	n &\geq k_\varepsilon \quad\implies\quad |b_n-b|<\varepsilon.
\end{align}
Combining \eqref{eqn:limitproductscratch} through \eqref{eqn:limitproductbepsilon} yields
\begin{align}\label{eqn:limitproductscratchepsilon}
	|a_nb_n-ab|&\leq |b_n||a_n-a|+|a||b_n-b|\\
	&<|b_n|\varepsilon+|a|\varepsilon\\
	&=\varepsilon(|b_n|+|a|).
\end{align}

From here, we can use a common bound on $a$ and the sequence $(b_n)$ to help find a suitable threshold for $(a_nb_n)$. By Theorem \ref{thm:convergentsequencebounded}, the convergence of $(a_n)$ and $(b_n)$ ensures they are bounded by some real numbers $u\geq 0$ and $v\geq 0$, respectively. We can consider the sum $q=u+v$, which is a bound for both $(a_n)$ and $(b_n)$. The limit $a$ is arbitrarily close to the sequence $(a_n)$ by Theorem \ref{thm:exercise0}, so by Lemma \ref{lem:closurebound} $a$ respects that same bound and we have $|a|\leq u\leq q$. (A similar statement holds for $b$, but we won't need it.) Next, by choosing $u>0$ or $v>0$ to ensure $q>0$, we can then consider the positive distance $\varepsilon/2q$. The thresholds $j_{\varepsilon/2q}$ for $(a_n)$ and $k_{\varepsilon/2q}$ for $(b_n)$ combine to create a threshold for $(a_nb_n)$, for instance we can use $n_\varepsilon=j_{\varepsilon/2q}+k_{\varepsilon/2q}$. 

We have the pieces. Time for a proof.
\end{scratch}

\begin{proof}[Proving $\lim_{n\to\infty}a_nb_n=ab$ in Theorem \ref{thm:algebraiclimitsreal}]
Assume $\lim_{n\to\infty}a_n=a$ and $\lim_{n\to\infty}b_n=b$. By Theorem \ref{thm:convergentsequencebounded}, $(a_n)$ and $(b_n)$ are both bounded by a some positive real number $q$. The limit $a$ is arbitrarily close to the sequence $(a_n)$ by Theorem \ref{thm:exercise0}, so  we have $|a|\leq q$ by Lemma \ref{lem:closurebound}. Thus,
\begin{align}\label{eqn:ab_nboundq}
	|a|\leq q \quad\textnormal{and}\quad |b_n|\leq q \quad\textnormal{for every index }n\in\N.
\end{align}
Now, let $\varepsilon>0$. Since $2q>0$, we also have $\varepsilon/2q>0$. By the definition of limit and convergence for sequences (Definition \ref{def:sequentiallimit}), there are thresholds $j_{\varepsilon/2q}$ and $k_{\varepsilon/2q}$ where 
\begin{align}\label{eqn:limitproductaq}
	n &\geq j_{\varepsilon/2q} \quad\implies\quad |a_n-a|<\frac{\varepsilon}{2q}\qquad \textnormal{and}\\
\label{eqn:limitproductbq}
	n &\geq k_{\varepsilon/2q} \quad\implies\quad |b_n-b|<\frac{\varepsilon}{2q}.
\end{align}
Define $n_\varepsilon=j_{\varepsilon/2q}+k_{\varepsilon/2q}$. Then for any index $n\geq n_\varepsilon$ we have both $n\geq j_{\varepsilon/2q}$ and $n\geq k_{\varepsilon/2q}$. Hence, 
\begin{align}\label{eqn:limitproductproof}
	|a_nb_n-ab|&=|a_nb_n\underbrace{-ab_n+ab_n}_{\textnormal{add }0}-ab|\\
	&\leq |a_nb_n-ab_n|+|ab_n-ab| \qquad\textnormal{(tri. ineq. \eqref{eqn:addzero})}\\
	&=|b_n||a_n-a|+|a||b_n-b| \\
	&\leq q|a_n-a|+q|b_n-b| \qquad\textnormal{(by \eqref{eqn:ab_nboundq})}\\
	&<q\left(\frac{\varepsilon}{2q}\right)+q\left(\frac{\varepsilon}{2q}\right)
\qquad\textnormal{(by \eqref{eqn:limitproductaq} and \eqref{eqn:limitproductbq})}\\
	&=\varepsilon.	
\end{align}
Therefore, $\lim_{n\to\infty}a_nb_n=ab$.
\end{proof}

The next result takes advantage of the deep connection between limits and arbitrarily close in Theorem \ref{thm:exercise0} and what we have proven regarding the order inherent to the real line in Lemma \ref{lem:orderacl}.

\begin{corollary}\label{cor:orderlimits}
Suppose $a$ and $b$ are real numbers and $(x_n)$ is a convergent sequence in the real line $\R$. 
\begin{enumerate}
	\item If $x_n\leq b$ for every index $n\in\N$, then $\lim_{n\to\infty}x_n=\ell\leq b$. 
	\item If $x_n\geq a$ for every index $n\in\N$, then $\lim_{n\to\infty}x_n=\ell\geq a$. 
\end{enumerate}
\end{corollary}

\begin{figure}
\centering
\begin{tikzpicture}
\draw (-3,-1.5) node {$(x_n)$};
\draw (-0.5,-1.5) node {$\circ$};
\draw (4,-1.5) node {$\circ$};
\draw (4.5,-1.5) node {$\circ$};
\draw (3.76,-1.5) node {...};
\draw (-0.5,-2) node {$a$};
\draw (1,-2) node {$x_1$};
\draw (3,-2) node {$x_2$};
\draw (4,-2) node {$\ell$};
\draw (4.5,-2) node {$b$};
\foreach \Point in {(1,-1.5), (3,-1.5),  (3.5,-1.5)}
{
    \node at \Point {\textbullet};
}
\end{tikzpicture}
\caption{A plot of a convergent sequence $(x_n)$ with limit $\ell$, lower bound $a$, and upper bound $b$ as in Corollary \ref{cor:orderlimits}.}
\end{figure}

\begin{proof}
Suppose $\lim_{n\to\infty}x_n=\ell$. Then by Theorem \ref{thm:exercise0}, $\ell\acl{(x_n)}$.

Suppose $x_n\leq b$ for every index $n\in\N$. Then $b$ is an upper bound for $(x_n)$ and by part (i) of Lemma \ref{lem:orderacl}, we have $\ell\leq b$.

Now suppose $x_n\geq a$ for every index $n\in\N$. Then $a$ is a lower bound for $(x_n)$ and by part (ii) of Lemma \ref{lem:orderacl}, we have $\ell\geq a$.
\end{proof}

To motivate another property of convergent sequences from calculus, consider the following example along with Figures \ref{fig:squeezesequencegraphs} and \ref{fig:squeezesequenceranges}.

\begin{example}\label{eg:squeezesequence}
Consider the sequence of real numbers $(y_n)$ defined for each $n\in\N$ by
\begin{align}
	y_n &= \frac{\sin{\sqrt{n^2+1}}}{n}.
\end{align}
Then $\lim_{n\to\infty}y_n=0$. See Figures \ref{fig:squeezesequencegraphs} and \ref{fig:squeezesequenceranges}.
\end{example}

\begin{remark}
A proof for Example \ref{eg:squeezesequence} using the definition of limit and convergence (Definition \ref{def:sequentiallimit}) would follow from the fact from trigonometry that for every real number $x$ we have
\begin{align}
	-1\leq \sin{x}\leq 1
\end{align}
combined with choosing a positive integer $n_\varepsilon>1/\varepsilon$ as a threshold in response to the distance $\varepsilon>0$. However, a proof using the Squeeze Theorem \ref{thm:squeezesequence} is in order so we can see how to put it to use.
\end{remark}

\begin{figure}
\centering
\begin{tikzpicture}
\foreach \Point in {(0.25,1), (0.5,0.5), (0.75,0.33), (1,0.25), (1.25,0.2), (1.5,0.17), (1.75,0.14), (2,0.13), (2.25,0.11), (2.5,0.10), (2.75,0.09), (3,0.08), (3.25,0.08), (3.5,0.07), (3.75,0.07), (4,0.06)}
{
   \node[red] at \Point {\textbullet};
}
\foreach \Point in {(0.25,-1), (0.5,-0.5), (0.75,-0.33), (1,-0.25), (1.25,-0.2), (1.5,-0.17), (1.75,-0.14), (2,-0.13), (2.25,-0.11), (2.5,-0.10), (2.75,-0.09), (3,-0.08), (3.25,-0.08), (3.5,-0.07), (3.75,-0.07), (4,-0.06)}
{
   \node[blue!60] at \Point {\textbullet};
}
\foreach \Point in {(0.25,0.84), (0.5,0.45), (0.75,0.05), (1,-0.19), (1.25,-0.19), (1.5,-0.05), (1.75,-0.09), (2,0.12), (2.25,0.05), (2.5,-0.05), (2.75,-0.09), (3,-0.04), (3.25,0.03), (3.5,0.07), (3.75,0.04), (4,-0.02)}
{
   \node at \Point {\textbullet};
}
\draw (-3,1.5) node {graphs of};
\draw[red] (-3,0.75) node {$(z_n)$};
\draw (-3,0) node {$(y_n)$};
\draw[blue] (-3,-0.75) node {$(x_n)$};
	\draw (0,-1) node {$-$};
	\draw (0,0) node {$-$};		
	\draw (0,1) node {$-$};
	\draw (-0.5,-1) node {$-1$};
	\draw (-0.5,0) node {$0$};
	\draw (-0.5,1) node {$1$};
  \draw[-To, dashed] (0,0) -- (4.5,0);		
  \draw[-To] (0,-1.3) -- (4.5,-1.3);
  \draw[-To] (0,-1.3) -- (0,1.5) node[above] {terms$\qquad$};
	\draw (1,-1.3) node {$|$};
	\draw (2,-1.3) node {$|$};
	\draw (3,-1.3) node {$|$};
	\draw (4,-1.3) node {$|$};
	\draw (1,-1.8) node {$4$};
	\draw (2,-1.8) node {$8$};
	\draw (3,-1.8) node {$12$};
	\draw (4,-1.8) node {$16$};
	\draw (5.5,-1.8) node {indices};
\end{tikzpicture}
\caption{Graphs of the sequences $(x_n), (y_n),$ and $(z_n)$ where $x_n=-1/n$, $y_n=\sin{(\sqrt{n^2+1}})/n$, and $z_n=1/n$ for each $n\in\N$. Note the term (i.e., height) $y_n$ is between $x_n$ and $z_n$ for each index $n$. See the Squeeze Theorem \ref{thm:squeezesequence}, Example \ref{eg:squeezesequence}, and Figure \ref{fig:squeezesequenceranges}.}
\label{fig:squeezesequencegraphs}
\end{figure}

\begin{figure}
\centering
\begin{tikzpicture}
\draw[blue!80] (-6,0) node {$(x_n)$};
\draw (0,0) node {$\circ$};
\draw (0,-0.5) node {$0$};
\draw[blue!80] (-0.5,0) node {$...$};
\foreach \Point in {(-4,0), (-2,0), (-1.33,0), (-1,0)}
{
    \node[blue!80] at \Point {\textbullet};
}
\draw[blue!80] (-4,-0.5) node {$x_1$};
\draw[blue!80] (-2,-0.5) node {$x_2$};
\draw[blue!80] (-1.4,-0.5) node {$x_3$};
\draw[blue!80] (-0.9,-0.5) node {$x_4$};

\draw (-6,-1.5) node {$(y_n)$};
\draw (0,-1.5) node {$\circ$};
\draw (0,-2) node {$0$};
\draw (-0.3,-1.5) node {$...$};
\foreach \Point in {(3.37,-1.5), (1.82,-1.5), (0.19,-1.5), (-0.76,-1.5)}
{
    \node at \Point {\textbullet};
}
\draw (3.37,-2) node {$y_1$};
\draw (1.82,-2) node {$y_2$};
\draw (0.4,-2) node {$y_3$};
\draw (-0.76,-2) node {$y_4$};

\draw[red] (-6,-3) node {$(z_n)$};
\draw (0,-3) node {$\circ$};
\draw (0,-3.5) node {$0$};
\draw[red] (0.5,-3) node {$...$};
\foreach \Point in {(4,-3), (2,-3), (1.33,-3), (1,-3)}
{
    \node[red] at \Point {\textbullet};
}
\draw[red] (4,-3.5) node {$z_1$};
\draw[red] (2,-3.5) node {$z_2$};
\draw[red] (1.4,-3.5) node {$z_3$};
\draw[red] (0.9,-3.5) node {$z_4$};

\draw (-6,-4.5) node {{\em all}};
\draw (0,-4.5) node {$\circ$};
\draw (0,-5) node {$0$};
\draw (-0.3,-4.5) node {$...$};
\draw[red] (0.5,-4.5) node {$...$};

\foreach \Point in {(-4,-4.5), (-2,-4.5), (-1.33,-4.5), (-1,-4.5)}
{
    \node[blue!80] at \Point {\textbullet};
}
\draw[blue!80] (-4,-5) node {$x_1$};
\draw[blue!80] (-2,-5) node {$x_2$};
\draw[blue!80] (-1.5,-5) node {$x_3$};
\draw[blue!80] (-1,-5) node {$x_4$};

\foreach \Point in {(3.37,-4.5), (1.82,-4.5), (0.19,-4.5), (-0.76,-4.5)}
{
    \node at \Point {\textbullet};
}
\draw (3.37,-5) node {$y_1$};
\draw (1.82,-5) node {$y_2$};
\draw (0.4,-5) node {$y_3$};
\draw (-0.5,-5) node {$y_4$};

\foreach \Point in {(4,-4.5), (2,-4.5), (1.33,-4.5), (1,-4.5)}
{
    \node[red] at \Point {\textbullet};
}
\draw[red] (4,-5) node {$z_1$};
\draw[red] (2.2,-5) node {$z_2$};
\draw[red] (1.3,-5) node {$z_3$};
\draw[red] (0.9,-5) node {$z_4$};
\end{tikzpicture}
\caption{Ranges of the sequences $(x_n), (y_n),$ and $(z_n)$ where $x_n=-1/n$, $y_n=\sin{(\sqrt{n^2+1})}/n$, and $z_n=1/n$ for each $n\in\N$.  See the Squeeze Theorem \ref{thm:squeezesequence} and Example \ref{eg:squeezesequence}. Between Figures \ref{fig:squeezesequencegraphs} and \ref{fig:squeezesequenceranges}, which best showcases the ``squeezing'' that gives the Squeeze Theorem \ref{thm:squeezesequence} its name?}
\label{fig:squeezesequenceranges}
\end{figure}

\begin{theorem}[Squeeze Theorem]\label{thm:squeezesequence}
Suppose $(x_n)$, $(y_n)$, and $(z_n)$ are sequences in the real line $\R$ where 
\begin{enumerate}
	\item for each $n\in\N$ we have $x_n\leq y_n\leq z_n$, and
	\item $\displaystyle \lim_{n\to\infty}x_n=\lim_{n\to\infty}z_n=\ell$.
\end{enumerate}
Then $(y_n)$ converges and $\displaystyle \lim_{n\to\infty}y_n=\ell$.
\end{theorem}

\begin{scratch}
Let's try to prove Theorem \ref{thm:squeezesequence} directly from the assumptions and the definition of limit and convergence for sequences (Definition \ref{def:sequentiallimit}). Also, see Figures \ref{fig:squeezesequencegraphs} and \ref{fig:squeezesequenceranges}.

The goal is to find a threshold $n_\varepsilon\in\N$ where $n\geq n_\varepsilon$ implies 
\begin{align}\label{eqn:squeezeabsolute}
	|y_n-\ell|<\varepsilon. 
\end{align}
How does \eqref{eqn:squeezeabsolute} relate to the assumption $x_n\leq y_n\leq z_n$? We can subtract $\ell$ and get
\begin{align}
	x_n-\ell\leq y_n-\ell\leq z_n-\ell,
\end{align}
but this expression does not involve absolute values like the definition of limit and convergence for sequences in the real line (Definition \ref{def:sequentiallimit}). To connect with the assumption that $(x_n)$ and $(z_n)$ converge, note we also have 
\begin{align}
	-|x_n-\ell|&\leq x_n-\ell\leq y_n-\ell\leq z_n-\ell\leq |z_n-\ell|.
\end{align}
From here, each sequence $(x_n)$ and $(z_n)$ has a corresponding threshold we can use. These two thresholds allow us to define a threshold for $(y_n)$ and squeeze the terms from both sides via properties of inequalities and Lemma \ref{lem:neighborhoods} where inequalities with and without absolute values are related to one another.
\end{scratch}

\begin{proof}
Assume $x_n\leq y_n\leq z_n$ for each $n\in\N$ and
\begin{align}
	\lim_{n\to\infty}x_n&=\lim_{n\to\infty}z_n=\ell.
\end{align}
Let $\varepsilon>0$. Since $(x_n)$ and $(z_n)$ converge to $\ell$, there are thresholds $j_\varepsilon$ and $k_\varepsilon$ where $n\geq j_\varepsilon$ and $n\geq k_\varepsilon$ imply
\begin{align}
	|x_n-\ell|<\varepsilon\qquad\textnormal{and}\qquad	|z_n-\ell|<\varepsilon,
\end{align}
respectively. Now define $n_\varepsilon=\max\{j_\varepsilon,k_\varepsilon\}$. Then for every $n\geq n_\varepsilon$ we have both $n\geq j_\varepsilon$ and $n\geq k_\varepsilon$. Therefore, by Lemma \ref{lem:neighborhoods} and other properties of inequalities we have
\begin{align}
	-\varepsilon &<-|x_n-\ell|\leq x_n-\ell\leq y_n-\ell\leq z_n-\ell\leq |z_n-\ell|<\varepsilon.
\end{align}
In particular, we have
\begin{align}
	-\varepsilon &< y_n-\ell <\varepsilon.
\end{align}
So, $n_\varepsilon$ is a threshold for $(y_n)$ since by Lemma \ref{lem:neighborhoods} we have
\begin{align}
	|y_n-\ell|<\varepsilon. 
\end{align}
Therefore, $(y_n)$ converges and $\lim_{n\to\infty}y_n=\ell$.
\end{proof}

\begin{proof}[Proof for Example \ref{eg:squeezesequence}]
A classic result from trigonometry helps us here. For every real number $\theta$ we have 
\begin{align}
	-1\leq\sin{\theta}&\leq 1.
\end{align}
Therefore, for each $n\in\N$ we have
\begin{align}
	z_n=-\frac{1}{n}\leq\frac{\sin{\sqrt{n^2+1}}}{n}&\leq \frac{1}{n}=x_n.
\end{align} 
We also have 
\begin{align}
	\lim_{n\to\infty}\left(-\frac{1}{n}\right)=\lim_{n\to\infty}\left(\frac{1}{n}\right)=0.
\end{align}
So, by the Squeeze Theorem \ref{thm:squeezesequence}, we have $\lim_{n\to\infty}y_n=0$.
\end{proof}
%
%

\vs
\section*{Exercises}
\setcounter{theorem}{0}

Exercises are for play: Do scratch work, draw stuff, and make mistakes---make {\em lots} of mistakes---before worrying about writing proofs. {\em Have fun!}
%

\vs
\section{Ensuring convergence}
\label{sec:ensuringconvergence}

Section \ref{sec:propertiesoflimits} features results where the convergence of some given sequences is assumed and leads to the convergence of other related sequences. This section explores other ways to ensure the convergence of a sequence by considering various properties the sequence or a related set might have.

\begin{definition}\label{def:monotonesequence}
A sequence of real numbers $(x_n)$ is {\em increasing} if for every $n\in\N$ we have
\begin{align}
	x_n\leq x_{n+1}.
\end{align}
If for every $n\in\N$ we have
\begin{align}
	x_n< x_{n+1}, 
\end{align}
then $(x_n)$ is {\em strictly increasing}.

A sequence of real numbers $(y_n)$ is {\em decreasing} if for every $n\in\N$ we have
\begin{align}
	y_n\geq y_{n+1}.
\end{align}
If for every $n\in\N$ we have
\begin{align}
	y_n> y_{n+1}, 
\end{align}
then $(y_n)$ is {\em strictly decreasing}.

A sequence of real numbers is {\em monotone} if it is increasing or decreasing.
\end{definition}

\begin{figure}
\centering
\begin{tikzpicture}
\foreach \Point in {(0.25,1), (0.5,0.71), (0.75,0.58), (1,0.50), (1.25,0.45), (1.5,0.41), (1.75,0.38), (2,0.35), (2.25,0.33), (2.5,0.32), (2.75,0.30), (3,0.29), (3.25,0.28), (3.5,0.27), (3.75,0.26), (4,0.25)}
{
   \node[red] at \Point {\textbullet};
}
\foreach \Point in {(0.25,-1), (0.5,-0.71), (0.75,-0.58), (1,-0.50), (1.25,-0.45), (1.5,-0.41), (1.75,-0.38), (2,-0.35), (2.25,-0.33), (2.5,-0.32), (2.75,-0.30), (3,-0.29), (3.25,-0.28), (3.5,-0.27), (3.75,-0.26), (4,-0.25)}
{
   \node[blue!60] at \Point {\textbullet};
}
\foreach \Point in {(0.25,0), (0.5,0), (0.75,0), (1,0), (1.25,0), (1.5,0), (1.75,0), (2,0), (2.25,0), (2.5,0), (2.75,0), (3,0), (3.25,0), (3.5,0), (3.75,0), (4,0)}
{
   \node at \Point {\textbullet};
}
\draw (-3,1.5) node {graphs of};
\draw[red] (-3,0.75) node {$(b_n)$};
\draw (-3,0) node {$(c_n)$};
\draw[blue] (-3,-0.75) node {$(a_n)$};
	\draw (0,-1) node {$-$};
	\draw (0,0) node {$-$};		
	\draw (0,1) node {$-$};
	\draw (-0.5,-1) node {$9$};
	\draw (-0.5,0) node {$10$};
	\draw (-0.5,1) node {$11$};
  \draw[-To, dashed] (0,0) -- (4.5,0);		
  \draw[-To] (0,-1.3) -- (4.5,-1.3);
  \draw[-To] (0,-1.3) -- (0,1.5) node[above] {terms$\qquad$};
	\draw (1,-1.3) node {$|$};
	\draw (2,-1.3) node {$|$};
	\draw (3,-1.3) node {$|$};
	\draw (4,-1.3) node {$|$};
	\draw (1,-1.8) node {$4$};
	\draw (2,-1.8) node {$8$};
	\draw (3,-1.8) node {$12$};
	\draw (4,-1.8) node {$16$};
	\draw (5.5,-1.8) node {indices};
\end{tikzpicture}
\caption{Graphs of the sequences $(a_n)$, $(b_n)$, and $(c_n)$ from Example \ref{eg:monotonesequences}. All three sequences are monotone and converge to 10.}
\label{fig:monotonesequencesgraphs}
\end{figure}

\begin{figure}
\centering
\begin{tikzpicture}
\draw[blue=60!] (-4,0) node {$(a_n)$};
\draw (0,0) node {$\circ$};
\draw (0,-0.5) node {$10$};
\draw[blue=60!] (-0.25,0) node {$...$};
\draw[blue=60!] (-1.10,-0.5) node {$a_1$};
\draw[blue=60!] (-0.65,-0.5) node {$a_2$};
\foreach \Point in {(-1,0), (-0.71,0), (-0.58,0), (-0.50,0)}
{
    \node[blue=60!] at \Point {\textbullet};
}
\draw[red] (-4,-1.5) node {$(b_n)$};
\draw (0,-1.5) node {$\circ$};
\draw (0,-2) node {$10$};
\draw[red] (0.25,-1.5) node {$...$};
\draw[red] (1.15,-2) node {$b_1$};
\draw[red] (0.71,-2) node {$b_2$};
\foreach \Point in {(1,-1.5), (0.71,-1.5), (0.58,-1.5), (0.50,-1.5)}
{
    \node[red] at \Point {\textbullet};
}
\draw (-4,-3) node {$(c_n)$};
\draw (0,-3) node {\textbullet};
\draw (0,-3.5) node {$10=c_n$};
\end{tikzpicture}
\caption{Ranges of the sequences $(a_n)$, $(b_n)$, and $(c_n)$ from Example \ref{eg:monotonesequences}. All three sequences are monotone and converge to 10.}
\label{fig:monotonesequences}
\end{figure}

\begin{example}\label{eg:monotonesequences}
Consider the sequences of real numbers $(a_n)$, $(b_n)$, and $(c_n)$ defined for each $n\in\N$ by
\begin{align}
	a_n&=10-\frac{1}{\sqrt{n}},\quad 
	b_n=10+\frac{1}{\sqrt{n}},
	\quad\textnormal{and}\quad
	c_n=10.
\end{align}
See Figures \ref{fig:monotonesequencesgraphs} and \ref{fig:monotonesequences}. For every $n\in\N$ we have
\begin{align}
	-\frac{1}{\sqrt{n}}\leq -\frac{1}{\sqrt{n+1}}
	\qquad\textnormal{and}\qquad 
	\frac{1}{\sqrt{n}}\geq \frac{1}{\sqrt{n+1}}.
\end{align}
Hence, $(a_n)$ is increasing while $(b_n)$ is decreasing.

For every $n\in\N$ we have both
\begin{align}
	c_n&=10\leq 10=c_{n+1}\qquad\textnormal{and}\qquad c_n=10\geq 10=c_{n+1}.
\end{align}
So, despite how strange it may sound, $(c_n)$ is both increasing and decreasing. The same is true for all constant sequences of real numbers.
\end{example}

The following lemma gives us an alternative way to think about monotonicity: Instead of considering consecutive terms like $x_n$ and $x_{n+1}$, we can compare two terms based on the order of their indices.

\begin{lemma}\label{lem:monotonesequence}
A sequence of real numbers $(x_n)$ is increasing if and only if for every pair of positive integers $j$ and $k$ we have
\begin{align}
	j<k \qquad \Longrightarrow \qquad x_j\leq x_k.
\end{align}
Likewise, a sequence of real numbers $(y_n)$ is decreasing if and only if for every pair of positive integers $j$ and $k$ we have
\begin{align}
	j<k \qquad \Longrightarrow \qquad y_j\geq y_k.
\end{align}
\end{lemma}

\begin{remark}
The proofs of the two cases in Lemma \ref{lem:monotonesequence} can be very similar, so I will prove the case for increasing sequences here but leave the case for decreasing sequences as an exercise. An induction argument helps with one of the implications.
\end{remark}

\begin{proof}[Proving the increasing case in Lemma \ref{lem:monotonesequence}]
Suppose for every pair of positive integers $j$ and $k$ we have
\begin{align}
	j<k \qquad \Longrightarrow \qquad x_j\leq x_k.
\end{align}
Since $n<n+1$ for every positive integer $n$, we have
\begin{align}
	x_n\leq x_{n+1}.
\end{align}
Hence, $(x_n)$ is increasing.

Now suppose $(x_n)$ is increasing and fix a positive integer $j$. To establish a base case for an induction argument, we have 
\begin{align}
	x_j\leq x_{j+1}
\end{align}
by the definition of an increasing sequence (Definition \ref{def:monotonesequence}).

To establish an inductive case for the same fixed $j$, suppose $k$ is a positive integer where $j<k$ and we have
\begin{align}
	x_j\leq x_k.
\end{align}
By the definition of an increasing sequence (Definition \ref{def:monotonesequence}), we have 
\begin{align}
	x_k\leq x_{k+1}
\end{align}
and since $j<k<k+1$, we also have
\begin{align}
	x_j\leq x_k \leq x_{k+1}.
\end{align}
Furthermore, since $j$ represents an arbitrary positive integer, for every pair of positive integers $j$ and $k$ we have
\begin{align}
	j<k \qquad \Longrightarrow \qquad x_j\leq x_k.
\end{align}
\end{proof}

Monotonicity and boundedness combine to ensure convergence. See Figure \ref{fig:monotoneboundedsequence}.

\begin{theorem}[Monotone and Bounded Convergence]
\label{thm:monotoneboundedsequence}
If $(x_n)$ is a monotone and bounded sequence of real numbers, then $(x_n)$ converges.  Furthermore, if $(x_n)$ is increasing and bounded, then
\begin{align}
	\lim_{n\to\infty}x_n=\sup\{x_n:n\in\N\}.
\end{align}
If $(x_n)$ is decreasing and bounded, then
\begin{align}
	\lim_{n\to\infty}x_n=\inf\{x_n:n\in\N\}.
\end{align}
\end{theorem}

\begin{scratch}\label{scr:monotoneboundedsequence}
The proofs of the two cases in Theorem \ref{thm:monotoneboundedsequence} can be very similar, so I will explore the case for increasing sequences here but leave the case for decreasing sequences as an exercise.

A subtlety worth noting is that the candidate for the limit of an increasing sequence---the supremum---does not exist unless the sequence is bounded above. This concern deserves attention. In general, we need to be careful and ensure the tools and concepts we use are justified. 

By assuming the sequence $(x_n)$ is bounded, its supremum $u$ is assured to exist by the Axiom of Completeness \ref{ax:axiomofcompleteness}. So, we can use $u$ as a candidate for the limit, as follows. Our goal is now to find a threshold $n_\varepsilon$ where 
\begin{align}\label{eqn:monotoneboundedsequencegoal}
	n\geq n_\varepsilon \quad\implies\quad |x_n-u|<\varepsilon.
\end{align}
Since the supremum of a sequence is arbitrarily close to the sequence, for every $\varepsilon>0$ there is an index $n_\varepsilon$ where 
\begin{align}
	|x_{n_\varepsilon}-u|<\varepsilon.
\end{align} 
From there, the assumption that $(x_n)$ is increasing and bounded above by $u$ combines with properties of absolute value and inequalities to give us our goal \eqref{eqn:monotoneboundedsequencegoal}.
\end{scratch}

\begin{figure}
\centering
\begin{tikzpicture}	
\draw (5,0) node {$\circ$};
\draw (5,-0.4) node {$u$};
\draw (4.6,0.01) node {.....};
\foreach \Point in {(0,0), (2.5,0), (3.33,0), (3.75, 0), (4,0), (4.17,0)}
{
    \node at \Point {\textbullet};
}
\draw (0,-0.5) node {$x_1$};
\draw (2.4,-0.5) node {$x_2$};
\draw (3.3,-0.53) node {$x_{n_\varepsilon}$};
\begin{scope}[thick, dashed, red] 
\draw (5,0) circle (2.3cm); 
\draw (5,0) circle (1.4cm); 
\draw (5,0) circle (0.7cm);
\end{scope}	
\draw[-,semithick, red] (5.05,0.05) -- (6.63,1.63);
\draw[red] (6.5,1.2) node {$\varepsilon$};
\end{tikzpicture}
\caption{An increasing sequence of real numbers $(x_n)$ bounded above by its supremum $u$ where the index $n_\varepsilon$ is a threshold responding to some distance $\varepsilon>0$. See Scratch Work \ref{scr:monotoneboundedsequence} for the proof of the increasing (and bounded) case of the Monotone and Bounded Convergence Theorem \ref{thm:monotoneboundedsequence}.}
\label{fig:monotoneboundedsequence}
\end{figure}

\begin{proof}[Proving the increasing case in Theorem \ref{thm:monotoneboundedsequence}]
Suppose $(x_n)$ is an increasing and bounded sequence of real numbers.  (See Figure \ref{fig:monotoneboundedsequence}.) Then the Axiom of Completeness \ref{ax:axiomofcompleteness} ensures the existence of the supremum
\begin{align}
	u=\sup\{x_n:n\in\N\}.
\end{align}
By the definition of supremum (Definition \ref{def:supremumacl}), $u$ is an upper bound for the range of $(x_n)$. Hence, for every $k\in\N$ we have 
\begin{align}\label{eqn:uupperbound}
	x_k\leq u.
\end{align}
Also by the definition of supremum (Definition \ref{def:supremumacl}) we have $u\acl (x_n)$. So by the definition of arbitrarily close (Definition \ref{def:acl}), for every $\varepsilon>0$ there is an index $n_\varepsilon$ where the term $x_{n_\varepsilon}$ satisfies
\begin{align}\label{eqn:uepsilon}
	|x_{n_\varepsilon}-u|<\varepsilon.
\end{align}
Inequality \eqref{eqn:uupperbound} ensures $0\leq u-x_k$ for every $k\in\N$, so by the definition of absolute value (Definition \ref{def:absolutevalue}) we have 
\begin{align}\label{eqn:uxkabsolutevalue}
	|x_k-u|=u-x_k.
\end{align}
Since $(x_n)$ is increasing, by Lemma \ref{lem:monotonesequence} for every positive integer $n$ where $n\geq n_\varepsilon$ we have 
\begin{align}\label{eqn:xnepsilonxn}
	x_{n_\varepsilon}\leq x_n.
\end{align}
So, by \eqref{eqn:uupperbound}, \eqref{eqn:uepsilon}, \eqref{eqn:uxkabsolutevalue},
\eqref{eqn:xnepsilonxn}, and properies of inequalities, for every positive integer $n$ where $n\geq n_\varepsilon$ we have 
\begin{align}
	|x_n-u|=u-x_n\leq u-x_{n_\varepsilon}=|x_{n_\varepsilon}-u|<\varepsilon.
\end{align}
Therefore, $n_\varepsilon$ is a threshold for the convergence $(x_n)$ and 
\begin{align}
	\lim_{n\to\infty}x_n&=u=\sup\{x_n:n\in\N\}.
\end{align}
\end{proof}

Out of necessity, the context of Theorem \ref{thm:monotoneboundedsequence} is limited to sequences in the real line $\R$ due to the inequalities (i.e., the order of the real line). 

The following theorem shows the convergence of a sequence in a Euclidean space $\R^m$ ensures the convergence of its components in $\R$, and vice versa. See Figure \ref{fig:componentwisesequences}.

\begin{figure}
\centering
\begin{tikzpicture}	
\draw (-2,1.5) node {$(\bfx_n)$};
\foreach \Point in {(1,2), (1.29,1.25), (1.42,1.11), (1.5,1.06)}
{
    \node at \Point {\textbullet};
}
	\draw (2,1) node {$\circ$};
	\draw (2.5,1) node {$\bfx$};
	\draw (0.5,2) node {$\bfx_1$};
	\draw (0.79,1.25) node {$\bfx_2$};
\foreach \Point in {(1.65,1.03), (1.75,1.01), (1.85,1)}
{
    \node at \Point {$.$};
}
\draw[blue=60!] (-2,0) node {$(x_{1,n})$};
\draw (2,0) node {$\circ$};
\draw (2,-0.5) node {$x_1$};
\draw[blue=60!] (1.75,0) node {$...$};
\draw[blue=60!] (1,-0.5) node {$x_{1,1}$};
\foreach \Point in {(1,0), (1.29,0), (1.42,0), (1.50,0)}
{
    \node[blue=60!] at \Point {\textbullet};
}
\draw[red] (-2,-1.5) node {$(x_{2,n})$};
\draw (2,-1.5) node {$\circ$};
\draw (2,-2) node {$x_2$};
\draw[red] (2.12,-1.5) node {$.$};
\draw[red] (3,-2) node {$x_{2,1}$};
\foreach \Point in {(3,-1.5), (2.25,-1.5)}
{
    \node[red] at \Point {\textbullet};
}
\end{tikzpicture}
\caption{A plot of the sequence $(\bfx_n)$ and point $\bfx$ in the plane $\R^2$ from Theorem \ref{thm:componentwisesequences} and \ref{eg:2Dlimitconverges} where $\lim_{n\to\infty}\bfx_n=\bfx$ along with their components in their own horizontal copies of the real line $\R$. The horizontal components $(x_{1,n})$ increase towards $x_1$ while the vertical components $(x_{2,n})$ decrease towards $x_2$.}
\label{fig:componentwisesequences}
\end{figure}

\begin{theorem}\label{thm:componentwisesequences}
Suppose $\bfx$ is a point and $(\bfx_n)$ is a sequence of points in $\R^m$ where for each index $n$ we have
\begin{align}
\bfx
&=\left[
	\begin{array}{c}
	x_1\\
	x_2\\
	\vdots\\
	x_m
\end{array}
\right]
\qquad\textnormal{and}\qquad
\bfx_n
=\left[
	\begin{array}{c}
	x_{1,n}\\
	x_{2,n}\\
	\vdots\\
	x_{m,n}
\end{array}
\right].
\end{align}
Then 
\begin{align}
	\lim_{n\to\infty}\bfx_n&=\bfx
\end{align}
if and only if for every $k=1,2,\ldots,m$ we have
\begin{align}
	\lim_{n\to\infty}x_{k,n}&=x_k.
\end{align}
\end{theorem}

\begin{remark}
Theorem \ref{thm:componentwisesequences} yields the following equations when either of its hypotheses are satisfied:
\begin{align}\label{eqn:swaplimit}
\lim_{n\to\infty}\bfx_n
&=\lim_{n\to\infty}
\left[
	\begin{array}{c}
	x_{1,n}\\
	x_{2,n}\\		
	\vdots\\
	x_{m,n}
	\end{array}
\right]	
=\left[
	\begin{array}{c}
	\displaystyle\lim_{n\to\infty}x_{1,n}\\
	\displaystyle\lim_{n\to\infty}x_{2,n}\\
	\vdots\\
	\displaystyle\lim_{n\to\infty}x_{m,n}
\end{array}
\right]	
=
\left[
	\begin{array}{c}
	x_1\\
	x_2\\
	\vdots\\
	x_m
	\end{array}
\right]
=
\bfx.
\end{align}
Note how in \eqref{eqn:swaplimit}, the limit symbol can be thought of as moving in and out of the brackets. Also, when we have the convergence
\begin{align}
	\lim_{n\to\infty}x_{k,n}&=x_k
\end{align}
for every $k=1,2,\ldots,m$, we can say $(\bfx_n)$ {\em converges componentwise}.
\end{remark}

\begin{scratch}
When $\lim_{n\to\infty}\bfx_n=\bfx$, there is a threshold $n_\varepsilon$ that ensures $\bfx_n$ is within $\varepsilon$ of $\bfx$ in all directions at the same time, including each component's direction. So, the same threshold $n_\varepsilon$ suffices for the  convergence in each component. See Figure \ref{fig:boxinadisk}.

On the other hand, when $\lim_{n\to\infty}x_{k,n}=x_k$ for each $k=1,\ldots,m$, then given any positive distance there are $m$ thresholds, one for each component. So given $\varepsilon>0$, each of the $m$ thresholds can be adapted to respond to a suitable proportion of $\varepsilon$. From there, the maximum of the set of $m$ adapted thresholds serves as a threshold to ensure $\lim_{n\to\infty}\bfx_n=\bfx$.
\end{scratch}

\begin{figure}
\centering
\begin{tikzpicture}
\draw[dashed,blue,fill=blue!15] (0,0) ellipse (2cm and 2cm);
	\draw (0,0) node {\textbullet};
	\draw (0,-0.4) node {$\bfx$};
	\draw[line width=1.5pt,-stealth](0,0)--(1,1.73);
	\draw (1.35,1.8) node {$\bfx_n$};
	\draw (1.15,0.6) node {$\bfx_n-\bfx$};
	\draw[dashed,black] (0,0) rectangle (1,1.73);
	\draw (0.6,-0.3) node {$u_1$};
	\draw (-0.4,0.86) node {$u_2$};		
	\draw[blue] (0,0) -- (-2,0);
	\draw[blue] (-1,-0.4) node {$\varepsilon$};
\end{tikzpicture}
\caption{In the plane $\R^2$, any point $\bfx_n$ within a positive distance $\varepsilon$ of the point $\bfx$ creates a vector $\bfx_n-\bfx$ whose  components $u_1=x_{1,n}-x_{1}$ and $u_2=x_{2,n}-x_{2}$ have lengths (absolute values) strictly less than $\varepsilon$. See the proof of Theorem \ref{thm:componentwisesequences}.}
\label{fig:boxinadisk}
\end{figure}

\begin{proof}[Proof of Theorem \ref{thm:componentwisesequences}]
Suppose $\lim_{n\to\infty}\bfx_n=\bfx$ where the convergence is in $\R^m$. Let $\varepsilon>0$. By the definition of limit and convergence (Definition \ref{def:sequentiallimit}), there is a threshold $n_\varepsilon$ where $n\geq n_\varepsilon$ ensures 
\begin{align}
	d_m(\bfx_n,\bfx)=\|\bfx_n-\bfx\|_m<\varepsilon.
\end{align}
The threshold $n_\varepsilon$ for the sequence $(\bfx_n)$ also serves as a suitable threshold for each component sequence $(x_{k,n})$. Indeed, since $0\leq x\leq y$ implies $\sqrt{x}\leq \sqrt{y}$, we have for each $k=1,\ldots,m$ and every $n\geq n_\varepsilon$ that
\begin{align}
	d_\R(x_{k,n},x_k)
	&=|x_{k,n}-x_k|\\
	&=\sqrt{(x_{k,n}-x_k)^2}\\
	&\leq\sqrt{\sum_{j=1}^m(x_{j,n}-x_j)^2}\\
	&=\|\bfx_n-\bfx\|_m\\
	&<\varepsilon.
\end{align}
See Figure \ref{fig:boxinadisk}. Hence, for every $k=1,2,\ldots,m$ we have
\begin{align}
	\lim_{n\to\infty}x_{k,n}&=x_k.
\end{align}

Now suppose for every $k=1,2,\ldots,m$ we have
\begin{align}
	\lim_{n\to\infty}x_{k,n}&=x_k.
\end{align}
Let $\varepsilon>0$ and note $\varepsilon/\sqrt{m}>0$. In response, there is a threshold $n_k$ for each $k=1,2,\ldots,m$ where for every $n\geq n_k$ we have
\begin{align}
	d_\R(x_{k,n},x_k)
	&=|x_{k,n}-x_k|
	<\frac{\varepsilon}{\sqrt{m}}.
	\label{eqn:componentbound}
\end{align}
Define $n_\varepsilon=\max\{n_1,n_2,\ldots,n_m\}$. Then for every $n\geq n_\varepsilon$ we have $n\geq n_k$ for each $k=1,2,\ldots,m$, thus inequality \eqref{eqn:componentbound} holds for each $k$. Therefore, since $0\leq x< y$ implies $\sqrt{x}< \sqrt{y}$, we also have
\begin{align}
	d_m(\bfx_n,\bfx)
	&=\|\bfx_n-\bfx\|_m\\
	&=\sqrt{\sum_{j=1}^m(x_{k,n}-x_k)^2}\\
	&<\sqrt{\sum_{j=1}^m \left(\frac{\varepsilon}{\sqrt{m}}\right)^2}\\
	&=\sqrt{m\left(\frac{\varepsilon^2}{m}\right)}\\
	&=\varepsilon.
\end{align}
Hence, $\lim_{n\to\infty}\bfx_n=\bfx$.
\end{proof}

{\em Subsequences} have already made an appearance, but a formal definition is in order so we can start proving results about their convergence.

\begin{definition}\label{def:subsequence}
Let $(\bfx_n)$ be a sequence of points in $\R^m$ and let $(n_k)$  be a strictly increasing sequence of positive integers. That is, 
\begin{align}
	n_1<n_2<n_3<\cdots
\end{align}
The sequence $(\bfx_{n_k})$ is called a {\em subsequence} of $(\bfx_n)$. 
\end{definition}

\begin{remark}\label{rmk:subsequencenotation}
The notation used to define subsequences in Definition \ref{def:subsequence} can be confusing, especially with the double subscripts. However, it ensures subsequences comprise only terms from the original sequence and the terms stay in order.
\end{remark}
 
\begin{definition}\label{def:subsequentiallimits}
A point $\bfy$ is a \textit{subsequential limit} of a sequence $(\bfx_n)$ if $(\bfx_n)$ has a subsequence whose limit is $\bfy$. The set of subsequential limits of $(\bfx_n)$ is denoted by $\Slim((\bfx_n))$.
\end{definition}

\begin{example}\label{eg:nolimitsubsequences}
Consider the sequence $(\bfz_n)$ from Example \ref{eg:nolimit} in $\R^2$ given by 
\begin{align}\label{eqn:nolimitsubsequences}
\bfz_n=
\begin{cases}
\left[
	\begin{array}{c}
		2+(2/n)	\\
		1	
	\end{array}
\right],\quad\textnormal{if }n \textnormal{ is odd},\\ \\
\left[
	\begin{array}{c}
		-1	\\
		3-(2/n)
	\end{array}
\right],\quad\textnormal{if }n \textnormal{ is even}.
\end{cases}
\end{align}

\begin{figure}
\centering
\begin{tikzpicture}
\draw (-3,2) node {$(\bfz_n)$};
\foreach \Point in {(4,1), (-1,2), (2.67,1), (-1,2.5)}
{
    \node at \Point {\textbullet};
}
	\draw (2,1) node {$\circ$};
	\draw (2,0.5) node {$\bfu$};
	\draw (2.29,1) node {$...$};
	\draw (4.1,0.48) node {$\bfz_1$};
	\draw (2.77,0.48) node {$\bfz_3$};
\foreach \Point in {(-1,2.68), (-1,2.78), (-1,2.88)}
{
    \node at \Point {$.$};
}
	\draw (-1,3) node {$\circ$};
	\draw (-0.53,3) node {$\bfv$};
	\draw (-0.5,1.98) node {$\bfz_2$};
	\draw (-0.5,2.48) node {$\bfz_4$};
\end{tikzpicture}
\caption{A plot of the sequence $(\bfz_n)$ along with $\bfu$ and $\bfv$ from Example \ref{eg:nolimitsubsequences}.}
\label{fig:nolimitcopy}
\end{figure}
See Figure \ref{fig:nolimitcopy}. Also, consider the points $\bfu$ and $\bfv$ given by
\begin{align}
	\bfu=
\left[
	\begin{array}{c}
	2\\
	1	
\end{array}
\right]
\qquad\textnormal{and}\qquad
	\bfv=
\left[
	\begin{array}{c}
	-1\\
	3
\end{array}
\right],
\end{align}
as well as the subsequences $(\bfz_{2k-1}), (\bfz_{2k})$, and $(\bfz_{3k})$. We have 
\begin{align}
\bfz_{2k-1}&=
\left[
	\begin{array}{c}
		2+(2/(2k-1))	\\
		1	
	\end{array}
\right],\quad
\bfz_{2k}=
\left[
	\begin{array}{c}
		-1	\\
		3-(1/k)
	\end{array}
\right]
\quad\textnormal{and},\\	
\bfz_{3k}&=
\begin{cases}
\left[
	\begin{array}{c}
		2+(2/3k)	\\
		1	
	\end{array}
\right],\quad\textnormal{if }k \textnormal{ is odd},\\ \\
\left[
	\begin{array}{c}
		-1	\\
		3-(2/3k)
	\end{array}
\right],\quad\textnormal{if }k \textnormal{ is even}.
\end{cases}
\end{align}
See Figure \ref{fig:nolimitsubsequences}. The subsequence of odd indices $(\bfz_{2k-1})$ converges to $\bfu$ with a threshold $s_\varepsilon$ satisfying 
\begin{align}
	s_\varepsilon	&>\frac{1}{\varepsilon}+\frac{1}{2}.
\end{align}
The subsequence of even indices $(\bfz_{2k})$ converges to $\bfv$ with a threshold $t_\varepsilon$ satisfying 
\begin{align}
	t_\varepsilon	&>\frac{1}{\varepsilon}.
\end{align}
The subsequence $(\bfz_{3k})$ diverges, as will be proven later. 
\end{example}

\begin{figure}
\centering
\begin{tikzpicture}
	\draw (2,1) node {$\circ$};
	\draw (2,0.5) node {$\bfu$};
	\draw (-1,3) node {$\circ$};
	\draw (-0.53,3) node {$\bfv$};
\begin{scope}[red]
\foreach \Point in {(4,1), (2.67,1)}
{
    \node at \Point {\textbullet};
}

	\draw (2.29,1) node {$...$};
	\draw (4.1,0.48) node {$\bfz_1$};
	\draw (2.77,0.48) node {$\bfz_3$};
	\draw (3,1.75) node {$(\bfz_{2k-1})$};
\end{scope}
\begin{scope}[blue]
\foreach \Point in {(-1,2.68), (-1,2.78), (-1,2.88)}
{
    \node at \Point {$.$};
}
\foreach \Point in {(-1,2), (-1,2.5)}
{
    \node at \Point {\textbullet};
}
	\draw (-0.5,1.98) node {$\bfz_2$};
	\draw (-0.5,2.48) node {$\bfz_4$};
	\draw (-2,2.4) node {$(\bfz_{2k})$};
\end{scope}
\end{tikzpicture}
\caption{A plot of the subsequences $(\bfz_{2k-1})$ in red and $(\bfz_{2k})$ in blue, along with $\bfu$ and $\bfv$ from Example \ref{eg:nolimitsubsequences}.}
\label{fig:nolimitsubsequences}
\end{figure}

Ultimately, we have $\Slim{(\bfz_n)}=\{\bfu,\bfv\}$. The thresholds $s_\varepsilon$ and $t_\varepsilon$ can be used to show $\bfu$ and $\bfv$ belong to $\Slim{(\bfz_n)}$ by proving
\begin{align}
	\lim_{k\to\infty}\bfz_{2k-1}&=\bfu\quad\textnormal{and}\quad
\lim_{k\to\infty}\bfz_{2k}=\bfv,\quad\textnormal{respectively}.
\end{align}
But why are no other points in $\Slim{(\bfz_n)}$? Every point in $\R^2$ aside from $\bfu$ and $\bfv$ is either away from $(\bfz_n)$ or is one term from $(\bfz_n)$ and away from all the others. This takes some effort to prove and is left as an exercise.

The following theorem solidifies what I believe is an intuitive idea.

\begin{theorem}\label{thm:allsubsequences}
Every subsequence of a convergent sequence in $\R^m$ converges to the same limit as the original sequence.
\end{theorem}

\begin{proof}
Suppose $(\bfx_n)$ is a convergent sequence of points in $\R^m$ whose limit is $\bfy$, and suppose $(\bfx_{n_k})$ is a subsequence of $(\bfx_n)$. Let $\varepsilon>0$. By the definition of limit and convergence (Definition \ref{def:sequentiallimit}), there is a threshold $j_\varepsilon$ where $n\geq j_\varepsilon$ implies
\begin{align}
	d_m(\bfx_n,\bfy)=\|\bfx_n-\bfy\|_m
	<\varepsilon.
\end{align}
Then $j_\varepsilon$ is also a suitable threshold for the subsequence $(\bfx_{n_k})$, as follows. Since the indices of $(\bfx_{n_k})$ form a strictly increasing sequence by definition of subsequence (Definition \ref{def:subsequence}), we have $n_k\geq k$ for every positive integer $k$. Hence, we have $n_k\geq k\geq j_\varepsilon$ implies
\begin{align}
	d_m(\bfx_{n_k},\bfy)=\|\bfx_{n_k}-\bfy\|_m <\varepsilon.
\end{align}
Therefore, $(\bfx_{n_k})$ converges to $\bfy$.
\end{proof}

The following corollary of Theorem \ref{thm:allsubsequences} formalizes an idea from calculus.

\begin{corollary}\label{cor:powersofconstant}
Suppose $0\leq c<1$. Then $\lim_{n\to\infty}=0$.
\end{corollary}

\begin{remark}
Corollary \ref{cor:powersofconstant} does not follow from Corollary \ref{cor:orderlimits} which provides bounds on where limits could be (between $0$ and $1$) but not enough information to determine the limit precisely. Theorem \ref{thm:allsubsequences} provides a way to do so by taking advantage of a particular subsequence as well as properties unique to zero.
\end{remark}

\begin{proof}
Suppose $0\leq c<1$. Then for every positive integer $n$ we have 
\begin{align}\label{eqn:powersofconstant}
	c^n\geq c^n\cdot c=c^{n+1}.
\end{align}
So by Corollary \ref{cor:orderlimits}, $(c^n)$ is a decreasing sequence bounded below by $0$. By the Monotone and Bounded Convergence Theorem \ref{thm:monotoneboundedsequence}, $(c^n)$ converges to some real number $\ell$. Then the subsequence $(c^{2n})$ also converges to $\ell$ by Theorem \ref{thm:allsubsequences}. By Theorem \ref{thm:algebraiclimitsreal} we have 
\begin{align}
	\ell&=\lim_{n\to\infty}c^{2n}
	=\lim_{n\to\infty}(c^n\cdot c^n)				=\left(\lim_{n\to\infty} c^n\right)\left(\lim_{n\to\infty}c^n\right)
	=\ell^2.
\end{align}
Hence, either $\ell=0$ or $\ell=1$. However, by Lemma \ref{lem:monotonesequence} and  Corollary \ref{cor:orderlimits} we have 
\begin{align}
	0\leq \ell=\lim_{n\to\infty}c^n\leq c<1.
\end{align}
So, $0\leq \ell <1$ and we must have $\lim_{n\to\infty}c^n=\ell=0$.
\end{proof}

The contrapositions of Theorems \ref{thm:convergentsequencebounded} and \ref{thm:allsubsequences} provide conditions to ensure divergence. As such, the proof of the following theorem is omitted.

\begin{theorem}[Divergence Criteria]
\label{thm:divergencecriteria}
Suppose $(\bfx_n)\subseteq\R^m$.
\begin{enumerate}
	\item If $(\bfx_n)$ is unbounded, then $(\bfx_n)$ diverges.
	\item If $(\bfx_n)$ has subsequences with different limits, then $(\bfx_n)$ diverges.
	\item If $(\bfx_n)$ has a divergent subsequence, then $(\bfx_n)$ diverges.
\end{enumerate}
\end{theorem}

\begin{example}\label{eg:divergentexamples}
The sequences $(b_n)$, $(c_n)$, and $(\bfz_n)$ from Examples \ref{eg:twocountablesetssequences}, \ref{eg:alternatingtoinfinity}, and \ref{eg:nolimit} diverge.
\end{example}

\begin{proof}[Proofs for Example \ref{eg:divergentexamples}] 
Since $c_n=(-1)^nn$ for each positive integer $n$, we have $|c_n|=n$. So by the Archimedean Property \ref{thm:archimedeanproperty}, $(c_n)$ is unbounded. Hence, by part (i) of the Divergence Criteria \ref{thm:divergencecriteria}, $(c_n)$ diverges.

The subsequences $(b_{2k-1})$ and $(b_{2k})$ satisfy
\begin{align}
	b_{2k-1}&=-\left(2-\frac{1}{\sqrt{2k-1}}\right)\qquad\textnormal{and}\qquad
	b_{2k}=2-\frac{1}{\sqrt{2k}}
\end{align}
for every positive integer $k$. Hence, 
\begin{align}
	\lim_{k\to\infty}b_{2k-1}&=-2\qquad\textnormal{and}\qquad
	\lim_{k\to\infty}b_{2k}=2.
\end{align}
Therefore, $(b_n)$ diverges by part (ii) of the Divergence Criteria \ref{thm:divergencecriteria}.

A shown in Example \ref{eg:nolimitsubsequences}, the subsequences $(\bfz_{2k-1})$ and $(\bfz_{2k})$ satisfy 
\begin{align}
	\lim_{k\to\infty}\bfz_{2k-1}&=\bfu=
\left[
	\begin{array}{c}
	2\\
	1	
\end{array}
\right]
\qquad\textnormal{and}\qquad
	\lim_{k\to\infty}\bfz_{2k}=\bfv=
\left[
	\begin{array}{c}
	-1\\
	3
\end{array}
\right].
\end{align}
Therefore, $(\bfz_n)$ diverges by part (ii) of the Divergence Criteria \ref{thm:divergencecriteria}.
\end{proof}

The following section further develops and proves some significant results about ensuring convergence of sequences.

\vs
\section*{Exercises}
\setcounter{theorem}{0}

Exercises are for play: Do scratch work, draw stuff, and make mistakes---make {\em lots} of mistakes---before worrying about writing proofs. {\em Have fun!}

%

\vs
\section{Some big theorems}
\label{sec:somebigtheorems}

A fundamental consequence---perhaps {\em the} fundamental consequence---of the completeness of the real line and Euclidean spaces is that completeness ensures the existence of suitable candidates for limits of sequences. This section explores more ways the convergence of sequences and existence of important points can be ensured by building upon results we have proven so far.

To motivate the first result in this section, Example \ref{eg:emptyintersection} shows us that intersections of nonempty and overlapping sets can be empty.

\begin{example}\label{eg:emptyintersection}
Consider the intervals defined for each $n\in\N$ by
\begin{align}
(0,1/n]=\{x\in\R:0<x\leq 1/n\}.
\end{align}
We have $\bigcap_{n=1}^\infty(0,1/n]=\varnothing$ (the intersection of all these intervals is empty). 
\end{example}

\begin{scratch}\label{scr:emptyintersection}
To prove the intersection in Example \ref{eg:emptyintersection} is empty, we should show that no real number is in the intersection. Since a given point needs to be in \textit{every} set in order to be in the intersection, we only need to find \textit{one} set where a given point doesn't belong. 

If $x\leq 0$, then $x$ is not in any of the intervals, so it's not in the intersection. But what if $x>0$? Things are trickier, but we can handle it. For instance, $x=1/100$ is in the first 100 intervals, but not the 101st since $1/101<1/100$ and we have $1/100\notin(0,1/101]$. It might help to note 
\begin{align}
(0,1/101]&=\{y\in\R:0<y\leq 1/101\}.
\end{align}
The Corollary of the Archimedean Property \ref{cor:archimedeanproperty} allows us to apply this type of argument for any positive number. 
\end{scratch}

Time for a proof, which can be done in two cases.

\begin{proof}[Proof of Example \ref{eg:emptyintersection}]

\noindent\underline{Case (i)}: Suppose $x\leq 0$. Then $x\notin (0,1]$. So,
\begin{align}
x\notin\bigcap_{n=1}^\infty(0,1/n].
\end{align}
\noindent\underline{Case (ii)}: Suppose $x>0$. By Corollary \ref{cor:archimedeanproperty}, there is an $n_x\in\N$ where $0<1/n_x<x$, so $x$ is too large to be in the interval $(0,1/n_x]$. Hence, 
\begin{align}
x\notin\bigcap_{n=1}^\infty(0,1/n].
\end{align}
So, whether $x\leq 0$ or $x>0$, we have $x$ is not in the intersection. Therefore,
\begin{align}
	\bigcap_{n=1}^\infty(0,1/n]=\varnothing.
\end{align}
\end{proof}

Example \ref{eg:emptyintersection} holds despite the fact that the intervals are \textit{nested}: Each interval contains the next one. 

\begin{definition}\label{def:nested}
A sequence $(S_n)$ of sets in $\R^m$ is {\em nested} if for every positive integer $n$ we have $S_n\supseteq S_{n+1}$.
\end{definition}

If we add a couple of conditions to the nested property, we can ensure the intersection of intervals is nonempty.

\begin{theorem}[NCBI Property\footnote{NCBI stands for ``nested, closed, bounded intervals''.}]\label{thm:nestedintervals}
Every nested sequence of closed and bounded intervals has a nonempty intersection.
\end{theorem}

To show the intersection is nonempty, we need to find just one point that is in all of the intervals. The Axiom of Completeness \ref{ax:axiomofcompleteness} ensures of the existence of such a point, as long as we have the right kind of set to invoke the axiom. See Figure \ref{fig:nestedclosedboundedintervals}.

\begin{proof}
Suppose for each positive integer $n$ we have 
\begin{align}
	[a_n,b_n]&=\{x\in\R:a_n\leq x\leq b_n\}\quad\textnormal{and}\label{eqn:intervalsclosedbounded}\\
	[a_1,b_1]&\supseteq	[a_2,b_2] \supseteq	[a_3,b_3] \supseteq \cdots.\label{eqn:intervalsnested}
\end{align}
Line \eqref{eqn:intervalsclosedbounded} ensures the intervals are closed and bounded while line \eqref{eqn:intervalsnested} ensures they form a nested sequence. 

\begin{figure}
\centering
\begin{tikzpicture}
\draw (-2,0) node {$[a_1,b_1]$};
	\draw[-,semithick] (0,0) -- (6,0);
	\draw (0,0) node {$[$};
	\draw (6,0) node {$]$};
	\draw (0,-0.5) node {$a_1$};
	\draw (6,-0.5) node {$b_1$};
\draw (-2,-1.5) node {$[a_2,b_2]$};
	\draw[-,semithick] (2,-1.5) -- (5,-1.5);
	\draw (2,-1.5) node {$[$};
	\draw (5,-1.5) node {$]$};
	\draw (2,-2) node {$a_2$};
	\draw (5,-2) node {$b_2$};
\draw (-2,-3) node {$[a_3,b_3]$};	
	\draw[-,semithick] (3,-3) -- (5,-3);
	\draw (3,-3) node {$[$};
	\draw (5,-3) node {$]$};
	\draw (3,-3.5) node {$a_3$};
	\draw (5,-3.5) node {$b_3$}; 
\draw (-2,-4.5) node {$[a_4,b_4]$};	
	\draw[-,semithick] (3.5,-4.5) -- (4.5,-4.5);
	\draw (3.5,-4.5) node {$[$};
	\draw (4.5,-4.5) node {$]$};
	\draw (3.5,-5) node {$a_4$};
	\draw (4.5,-5) node {$b_4$};
\draw (-2,-5.75) node {$\vdots$};	
\draw (4,-5.75) node {$\vdots$};
\draw (-2,-7) node {$L$};
\draw (4.25,-7) node {$\circ$};
\draw (4.25,-7.45) node {$u$};
\draw (3.88,-7) node {$...$};
\draw (0,-7.5) node {$a_1$};
\draw (2,-7.5) node {$a_2$};
\draw (3,-7.5) node {$a_3$};
\draw (3.5,-7.5) node {$a_4$};
\foreach \Point in {(0,-7), (2,-7), (3,-7), (3.5,-7)}
{
    \node at \Point {\textbullet};
}
\end{tikzpicture}
\caption{A sequence of nested, closed, bounded intervals to accompany the NCBI Property (Theorem \ref{thm:nestedintervals}) along with the set $L$ comprising the left endpoints of the intervals and its supremum $u$. This supremum $u$ is in the intersection of the intervals but not necessarily in $L$.}
\label{fig:nestedclosedboundedintervals}
\end{figure}

Let $L$ denote the set of left-endpoints of intervals $[a_n,b_n]$. So,
\begin{align}
	L&=\{a_n:n\in\N\}.
\end{align}
Since the first interval $[a_1,b_1]$ contains all the others, we have $a_n\leq b_1$ for every $n\in\N$. Thus, $L$ is bounded above by $b_1$. Since $a_1\in L$, $L$ is nonempty. By the Axiom of Completeness \ref{ax:axiomofcompleteness}, $u=\sup L$ exists. See Figure \ref{fig:nestedclosedboundedintervals}.

It turns out every $b_n$ is an upper bound for $L$.  If not, there would be some $a_k\in L$ where $b_n<a_k$, but this would contradict the nested property in line \eqref{eqn:intervalsnested}. Now, since $u=\sup L$ is the least upper bound of $L$ by Theorem \ref{thm:equivalentaxiomofcompleteness}, we have $u\leq b_n$ for every $n\in\N$. Since $u=\sup L$ is an upper bound for $L$ by Definition \ref{def:supremumacl}, we also have $a_n\leq u$ for every $n\in\N$. Hence, for every $n\in\N$ we have
\begin{align}
	a_n\leq u \leq b_n.
\end{align}
Therefore, $u\in [a_n,b_n]$ for every $n\in\N$ and so
\begin{align}
	\bigcap_{n=1}^\infty[a_n,b_n]\neq\varnothing.
\end{align}
\end{proof}

The NCBI Property (Theorem \ref{thm:nestedintervals}) generalizes to higher dimensions when we consider {\em boxes} in place of intervals: In the real line $\R$, a box is an interval; in the plane $\R^2$, a box is a rectangle; in $\R^3$, a box is a rectangular parallelepiped. The result appears as the NCBB\footnote{NCBB stands for ``nested, closed, bounded boxes''.} Property (Theorem \ref{thm:nestedboxes}) below.

\begin{definition}\label{def:boxes}
A set $B\subseteq\R^m$ is a {\em box} if for each $j=1,\ldots,m$ there is an interval $I_j$ where   
\begin{align}
B&=I_1\times I_2\times\cdots\times I_m\\
&=\left\{\bfx =
	\left[
		\begin{array}{c}
		x_1\\
		x_2\\
		\vdots\\
		x_m
		\end{array}
	\right]: x_j\in I_j\textnormal{ for each } j=1,\ldots,m
\right\}.
\end{align}
Furthermore, $B$ is a {\em closed box} if each $I_j$ is a closed interval. Similarly, $B$ is an {\em open box} if each $I_j$ is an open interval.  
\end{definition}

\begin{notation}\label{not:sequenceofboxes}
Since the NCBB Property (Theorem \ref{thm:nestedboxes}) deals with a sequence of boxes, some carefully chosen notation will help us be precise. Given a sequence of boxes $(B_n)$, for each positive integer $n$ and each $j=1,\ldots,m$, let $I_{j,n}$ be an interval where 
\begin{align} 
B_n&=I_{1,n}\times I_{2,n}\times\cdots\times I_{m,n}\\
&=\left\{\bfx =
	\left[
		\begin{array}{c}
		x_1\\
		x_2\\
		\vdots\\
		x_m
		\end{array}
	\right]: x_j\in I_{j,n}\textnormal{ for each } j=1,\ldots,m
\right\}.
\end{align}
The intersection of the boxes amounts to the cross product of the intersections of their component intervals:
\begin{align} 
\bigcap_{n=1}^\infty B_n&=
\bigcap_{n=1}^\infty I_{1,n}\times 
\bigcap_{n=1}^\infty I_{2,n}
\times\cdots\times 
\bigcap_{n=1}^\infty I_{m,n}\\
&=\left\{\bfx =
	\left[
		\begin{array}{c}
		x_1\\
		x_2\\
		\vdots\\
		x_m
		\end{array}
	\right]: x_j\in \bigcap_{n=1}^\infty I_{j,n}\textnormal{ for each } j=1,\ldots,m
\right\}.
\end{align}
Similarly, containment in $\R^m$ amounts to containment in $\R$ of each component. That is,
\begin{align}
	B_n&\supseteq B_{n+1} \quad\Longleftrightarrow\quad
	I_{j,n}\supseteq I_{j,n+1} \textnormal{ for each } j=1,\ldots,m,
\end{align}
so $(B_n)$ is nested if and only if $(I_{j,n})$ is nested for each $j=1,\ldots,m$.
\end{notation}

\begin{theorem}[NCBB Property]\label{thm:nestedboxes}
Every nested sequence of closed and bounded boxes in $\R^m$ has a nonempty intersection.
\end{theorem}

\begin{scratch}\label{scr:nestedboxes}
My approach is motivated by Theorem \ref{thm:componentwisesequences} where convergence of a sequence in a Euclidean space $\R^m$ is ensured by the convergence in the real line $\R$ of each of its components (and vice versa). Here, the NCBI Property (Theorem \ref{thm:nestedintervals}) applies to the closed intervals that define the closed boxes in question, allowing us to find a point in the intersection of these boxes by controlling their component intervals. 
\end{scratch}

\begin{proof}
Suppose $(B_n)$ is a nested sequence of closed and bounded boxes in $\R^m$ whose component intervals are given by $I_{j,n}$ for each positive integer $n$ and each $j=1,\ldots,m$. Then for each $j=1,\ldots,m$, the  sequence $(I_{j,n})$ is a nested sequence of closed and bounded intervals. So by the NCBI Property (Theorem \ref{thm:nestedintervals}), we have $\cap_{n=1}^\infty I_{j,n}$ is nonempty for each $j=1,\ldots,m$. Thus, for each $j=1,\ldots,m$ there is a real number $u_j$ in $\cap_{n=1}^\infty I_{j,n}$. Now define
\begin{align}
	\bfu &=
	\left[
		\begin{array}{c}
		u_1\\
		u_2\\
		\vdots\\
		u_m
		\end{array}
	\right].
\end{align}
Then $\bfu\in\cap_{n=1}^\infty B_n$, so 
$\cap_{n=1}^\infty B_n$ is nonempty.
\end{proof}

\begin{prob}\label{prob:disksinsidesquares}
Draw a big disk along with a smaller square contained completely inside the big disk. Then, draw a disk small enough to be contained completely within the square. Can you see how this process can be repeated indefinitely, no matter how small the disks and squares are? 

This problem is designed to help you connect the definition of limit and convergence of sequences (Definition \ref{def:sequentiallimit}), which involves spheres and disks in the form of $\varepsilon$-neighborhoods, to boxes in various Euclidean spaces. Boxes and convergence combine in the proof of the Bolzano-Weierstrass Theorem \ref{thm:bolzanoweierstrass}. Compare and contrast Figures \ref{fig:threeneighborhoods}, \ref{fig:bolzanoweierstrassscratch}, and \ref{fig:bolzanoweierstrassproof1}.
\end{prob}

The following theorem is a major result in analysis which guarantees the existence of a convergent subsequence for a given bounded sequence. The proof is considerably longer than the proofs we've seen up to this point, so give yourself time to go over it in detail.

\begin{theorem}[Bolzano-Weierstrass Theorem]\label{thm:bolzanoweierstrass}
Every bounded sequence in $\R^m$ has a convergent subsequence.
\end{theorem}

\begin{scratch}
The Bolzano-Weierstrass Theorem requires a bit of work to prove. All we have to start with is a bounded sequence in $\R^m$, so we need to ensure the existence of a suitable subsequence as well as a suitable candidate for the limit of the subsequence. From there, we need to prove the candidate really is the limit of the subsequence. The proof below takes advantage of results we have built up so far. In particular, the NCBB Property (Theorem \ref{thm:nestedboxes}) ensures the existence of a point which serves as a suitable candidate for the limit, but that will follow from a carefully constructed subsequence that stems from considering a carefully chosen sequence of boxes constructed in a recursive manner: After an initial set of steps, the process is repeated indefinitely. 

The idea is to take a box that's big enough to contain the bounded sequence, then recursively bisect boxes along each of their sides to produce smaller and smaller boxes. In the real line $\R$, bisecting an interval produces two intervals with half the original length. In the plane $\R^2$, bisecting a square along each of its sides produces $2^2=4$ squares with half the original side length. (See Figures \ref{fig:bolzanoweierstrassscratch} and \ref{fig:bolzanoweierstrassproof1}.) In a Euclidean space $\R^m$, bisecting a box along each of its sides produces $2^m$ boxes whose side lengths are half the original side lengths, respectively.

At each step, at least one of the smaller boxes must contain an infinite number of the terms of the sequence. From there, we can choose one such smaller box and one term in the box to add to a subsequence, then repeat. By bisecting with each new step, we ensure our chosen subsequence converges. This effect reminds me of the Squeeze Theorem \ref{thm:squeezesequence}, but we make use of other results instead.   
\end{scratch}

\begin{figure}
\centering
\begin{tikzpicture}	 
\draw[fill=blue!15] (-4,-4) rectangle (4,4);
	\draw[fill=white] (0,0) circle (0.08cm);
	\draw (0,-0.4) node {$\mathbf{0}$};
	\draw (-2,2.5) node {$\bullet$};
	\draw (-2,2.1) node {$\bfx_1$};	
	\draw (-1,2.5) node {$\bullet$};	
	\draw (-1,2.1) node {$\bfx_3$};	
	\draw (-0.1,2.6) node {$\bullet$};
	\draw (-0.1,2.2) node {$\bfx_5$};
	\draw (0.6,2.8) node {$\bullet$};
	\draw (1.05,3.05) node {$\bullet$};		
	\draw (1.2,3.2) node {$\cdot$};
	\draw (1.25,3.25) node {$\cdot$};
	\draw (1.3,3.3) node {$\cdot$};
\draw[fill=white] (1.41,3.41) circle (0.08cm);
	\draw (1.8,3.41) node {$\bfu$};	
		\draw (-2,-2.5) node {$\bullet$};
	\draw (-2,-2.1) node {$\bfx_2$};	
	\draw (-1,-2.5) node {$\bullet$};	
	\draw (-1,-2.1) node {$\bfx_4$};	
	\draw (-0.1,-2.6) node {$\bullet$};
	\draw (-0.1,-2.2) node {$\bfx_6$};
	\draw (0.6,-2.8) node {$\bullet$};
	\draw (1.05,-3.05) node {$\bullet$};		
	\draw (1.2,-3.2) node {$\cdot$};
	\draw (1.25,-3.25) node {$\cdot$};
	\draw (1.3,-3.3) node {$\cdot$};
\draw[fill=white] (1.41,-3.41) circle (0.08cm);
	\draw (1.8,-3.41) node {$\bfv$};
\end{tikzpicture}
\caption{A bounded sequence $(\bfx_n)$ in the plane $\R^2$ contained in a closed square with the origin $\mathbf{0}$ in the center. The points $\bfu$ and $\bfv$ are arbitrarily close to the sequence $(\bfx_n)$.}
\label{fig:bolzanoweierstrassscratch}
\end{figure}

\begin{proof}[Proof of the Bolzano-Weierstrass Theorem \ref{thm:bolzanoweierstrass}]
Suppose $(\bfx_n)$ is a bounded sequence of points in $\R^m$. Let $B_1$ be a closed box in $\R^m$ defined by 
\begin{align}
	B_1&=[-b,b]\times\cdots\times[-b,b]
\end{align}
where $b>0$ is large enough to contain the range of the sequence $(\bfx_n)$, similar to Figure \ref{fig:bolzanoweierstrassscratch}.

Choose a term $\bfx_{n_1}$ to serve as the first term of our desired subsequence. Next, consider the $2^m$ distinct closed boxes whose component intervals are of the form $[0,b]$ or $[-b,0]$ and whose union is $B_1$. At least one of these smaller boxes contains an infinite number of the terms of $(\bfx_n)$, so choose one such box and name it $B_2$. Note that the length of each side of $B_2$ is $b$, half the length of each side of $B_1$. From there, choose a second term $\bfx_{n_2}$ from the sequence where $n_1<n_2$ and $\bfx_{n_2}$ is in $B_2$. See Figure \ref{fig:bolzanoweierstrassproof1}.

\begin{figure}
\centering
\begin{tikzpicture}	 
\draw[fill=blue!15] (-4,-4) rectangle (4,4);
\draw[fill=blue!15] (0,0) rectangle (4,4);
	\draw[fill=white] (0,0) circle (0.08cm);
	\draw (0,-0.4) node {$\mathbf{0}$};
	\draw (3.5,-3.5) node {$B_1$};
	\draw (3.5,0.5) node {$B_2$};	
	\draw (-2,2.5) node {$\bullet$};
	\draw (-1,2.5) node {$\bullet$};	
	\draw (-0.1,2.6) node {$\bullet$};
	\draw (0.6,2.8) node {$\bullet$};
	\draw (0.6,2.4) node {$\bfx_{n_2}$};		
	\draw (1.05,3.05) node {$\bullet$};		
	\draw (1.2,3.2) node {$\cdot$};
	\draw (1.25,3.25) node {$\cdot$};
	\draw (1.3,3.3) node {$\cdot$};
\draw[fill=white] (1.41,3.41) circle (0.08cm);
	\draw (1.8,3.41) node {$\bfu$};	
		\draw (-2,-2.5) node {$\bullet$};
	\draw (-2,-2.1) node {$\bfx_{n_1}$};	
	\draw (-1,-2.5) node {$\bullet$};	
	\draw (-0.1,-2.6) node {$\bullet$};
	\draw (0.6,-2.8) node {$\bullet$};
	\draw (1.05,-3.05) node {$\bullet$};		
	\draw (1.2,-3.2) node {$\cdot$};
	\draw (1.25,-3.25) node {$\cdot$};
	\draw (1.3,-3.3) node {$\cdot$};
\draw[fill=white] (1.41,-3.41) circle (0.08cm);
	\draw (1.8,-3.41) node {$\bfv$};
\end{tikzpicture}
\caption{A bounded sequence $(\bfx_n)$ in the plane $\R^2$ contained in a closed square $B_1$ with the origin $\mathbf{0}$ in the center. The terms $\bfx_{n_1}$ and $\bfx_{n_2}$ are chosen so that $n_1<n_2$, $\bfx_{n_1}$ is in $B_1$, and $\bfx_{n_2}$ is in $B_2$. The points $\bfu$ and $\bfv$ are arbitrarily close to the sequence $(\bfx_n)$.}
\label{fig:bolzanoweierstrassproof1}
\end{figure}

Proceeding recursively, suppose $k>2$ is a positive integer for which a closed box $B_{k-1}$ and term $\bfx_{n_{k-1}}$ have been chosen where $n_{k-2}<n_{k-1}$, $\bfx_{n_{k-1}}$ is in $B_{k-1}$, $B_{k-2}\subseteq B_{k-1}$, and the length of each side of $B_{k-1}$ is $b/2^k$ (half the length of each side of $B_{k-2}$). Bisect each side of $B_{k-1}$ to produce $2^m$ distinct closed boxes whose union is $B_{k-1}$. At least one of these closed boxes contains an infinite number of the terms of the sequence $(\bfx_n)$, so choose one such box and rename it $B_k$. Note that the length of each side of $B_k$ is $b/2^{k-2}$, half the length of each side of $B_{k-1}$. From there, choose a $k$-th term $\bfx_{n_k}$ from the sequence where $n_{k-1}<n_k$ and $\bfx_{n_k}$ is in $B_k$.  

Our recursive process yields a subsequence $(\bfx_{n_k})$ and a sequence of closed boxes $(B_k)$ where 
\begin{align}
	B_{k-1}\supseteq B_k
\end{align}
and the length of the sides of $B_k$ is $b/2^{k-2}$ for each positive integer $k$. Thus, $(B_k)$ is a nested sequence of closed and bounded boxes. By the NCBB Property (Theorem \ref{thm:nestedboxes}), there is point $\bfu$ in $\cap_{k=1}^\infty B_k$.

It remains to show $\bfu$ is the limit of the subsequence $(\bfx_{n_k})$. Since our construction makes use of boxes instead of the disks or spheres defined by neighborhoods, consider the componentwise breakdown of $\bfu$ and $(\bfx_{n_k})$ given for each positive integer $k$ by
\begin{align}
	\bfu &=
	\left[
		\begin{array}{c}
		u_1\\
		u_2\\
		\vdots\\
		u_m
		\end{array}
	\right]\qquad\textnormal{and}\qquad
	\bfx_{n_k} =
	\left[
		\begin{array}{c}
		x_{1,n_k}\\
		x_{2,n_k}\\
		\vdots\\
		x_{m,n_k}
		\end{array}
	\right].
\end{align} 
Both $\bfu$ and $\bfx_{n_k}$ are in $B_k$ and the lengths of the sides of $B_k$ are $b/2^{k-2}$ for each positive integer $k$, so we have 
\begin{align}
	|x_{j,n_k}-u_j|\leq\frac{b}{2^{k-2}}
\end{align}
for each $j=1,\ldots,m$ and every positive integer $k$. By Corollary \ref{cor:powersofconstant} and part (ii) of Theorem \ref{thm:algebraiclimitseuclidean} we have
\begin{align}
	\lim_{k\to\infty}\frac{b}{2^{k-2}}=0.
\end{align}

Let $\varepsilon>0$. By the definition of limit and convergence for sequences (Definition \ref{def:sequentiallimit}), there is a threshold $k_\varepsilon$ where $k\geq k_\varepsilon$ implies 
\begin{align}
	|x_{j,n_k}-u_j|\leq\frac{b}{2^{k-2}}<\varepsilon
\end{align}
for each $j=1,\ldots,m$. Therefore, the index $k_\varepsilon$ is a threshold for every component sequence $(x_{j,n_k})$ at the same time and we have
\begin{align}
	\lim_{k\to\infty}x_{j,n_k}=u_j
\end{align}
for each $j=1,\ldots,m$. Since the components of the subsequence $(\bfx_{n_k})$ converge, by Theorem \ref{thm:componentwisesequences} we finally have 
\begin{align}
	\lim_{k\to\infty}\bfx_{n_k}&=\bfu.
\end{align}
\end{proof}

The Bolzano-Weierstrass Theorem \ref{thm:bolzanoweierstrass} ensures the existence of a limit even without a particular candidate in mind or readily available. This is quite different from many of the results pertaining to limits and convergence discussed earlier in this chapter. In fact, the definition of limit and convergence for sequences (Definition \ref{def:sequentiallimit}) explicitly relies on a candidate for the limit. The Cauchy Criterion (Theorem \ref{thm:cauchycriterion}) is a powerful result which ensures the existence of a limit even when we do not have a candidate. 

\begin{definition}\label{def:cauchysequence}
Let $(\bfx_n)$ be a sequence of points in $\R^m$. The sequence $(\bfx_n)$ is {\em Cauchy} if for every $\varepsilon>0$ there is a positive integer $n_\varepsilon$ such that for positive integers $n$ and $k$ we have
\begin{align}
 n,k\geq n_\varepsilon\quad\implies\quad	d_m(\bfx_n,\bfx_k)=\|\bfx_n-\bfx_k\|_m &< \varepsilon.
\end{align}
Also, the positive integer $n_\varepsilon$ is called a {\em threshold} and its value depends on $\varepsilon$.
\end{definition}

\begin{remark}\label{rmk:limitvscauchy}
The concepts of convergence and Cauchy are equivalent. For both, $\varepsilon>0$ tells us how close we would like pairs of objects to be while the positive integer $n_\varepsilon$ is a threshold that ensures the objects are within $\varepsilon$ of each other. The key difference is the objects taken into consideration: Convergence compares terms of a sequence with the limit or a candidate for the limit; Cauchy compares pairs of terms of a sequence and does not consider a candidate for the limit at all.

Consider quantified versions of the statements where ``$\forall$'' means ``for all'' and ``$\exists$'' means ``there exists'':\s
\begin{center}
\begin{tabular}{lc|cl}
\underline{$(\bfx_n)$ is Cauchy} & & & \underline{$\bfy =\lim_{n\to\infty}(\bfx_n)$}\\
 & & & \\
$\forall\, \varepsilon > 0,$ & & & $\forall\, \varepsilon > 0,$ \\
$\exists\, n_\varepsilon \in\N$ such that \hspace{25pt}& & & $\exists\, n_\varepsilon \in\N$ such that\\
\hspace{10pt}$n,k\geq n_\varepsilon  \, \Longrightarrow\, d_m(\bfx_n,\bfx_k)<\varepsilon$.& & &\hspace{10pt}$n\geq n_\varepsilon  \, \Longrightarrow\, d_m(\bfx_n,\bfy)<\varepsilon$. 
\end{tabular}
\end{center}
\end{remark}

To me, when a sequence is both convergent and Cauchy, the terms of the sequence get close and stay close to the limit and to each other. To prove they are equivalent, an additonal property of Cauchy sequences will help.

\begin{lemma}\label{lem:cauchybounded}
Every Cauchy sequence in $\R^m$ is bounded. 
\end{lemma}

\begin{scratch}\label{scr:cauchysequencebounded}
The proof of Lemma \ref{lem:cauchybounded} is very similar to the proof of Theorem \ref{thm:convergentsequencebounded} which says convergent sequences are bounded. Both amount to solidifying the idea that, eventually, the terms of the sequence are as close together as we like.

In Figure \ref{fig:cauchysequencebounded}, no particular term in the sequence has the largest norm, but for some $\varepsilon_0>0$ there is a threshold $n_0$ that produces the real number $\varepsilon_0+\|\bfx_{n_0}\|_m$ which is a bound for the sequence. However, this may not be true in general since a term whose index is less than $n_0$ may have a larger norm.

\begin{figure}
\centering
\begin{tikzpicture} 
\draw (0,0) circle (3.11cm);	
	\draw (0.05,0) --(3.11,0);  
	\draw (1.55,-0.4) node {$b$};    
	\draw (0,-0.02) node {$\circ$};
	\draw (0,-0.4) node {$\mathbf{0}$};
\draw (1.45,2.43) node {$\circ$};				
\begin{scope}[thick, dashed, blue] 
\draw (1.05,2.05) circle (0.8cm); 
\end{scope}	
\draw[-,semithick, blue] (1.08,2.08) -- (0.28,2.08);
\draw[blue] (0.68,2.35) node {$\varepsilon_0$};
	\draw (-1.5,-1.5) node {$\bullet$};
	\draw (-1.5,-1.9) node {$\bfx_1$};	
	\draw (-0.8,1.6) node {$\bullet$};	
	\draw (-0.8,1.2) node {$\bfx_2$};	
	\draw (-0.1,0.8) node {$\bullet$};
	\draw (-0.1,0.4) node {$\bfx_3$};
	\draw (0.1,0.83) node {$\cdot$};
	\draw (0.25,0.85) node {$\cdot$};
	\draw (0.4,0.87) node {$\cdot$};	
	\draw (0.6,0.9) node {$\bullet$};
	\draw (0.9,0.5) node {$\bfx_{n_0-1}$};	
	\draw (1.05,2.05) node {$\bullet$};
	\draw (1.05,1.75) node {$\bfx_{n_0}$};		
	\draw (1.18,2.18) node {$\cdot$};
	\draw (1.24,2.24) node {$\cdot$};
	\draw (1.3,2.3) node {$\cdot$};
\end{tikzpicture}
\caption{A Cauchy sequence $(\bfx_n)$ in the plane $\R^2$. By Lemma \ref{lem:cauchybounded}, $(\bfx_n)$ must be bounded by some nonnegative real number $b$. Note how the $\circ$ near the term $\bfx_{n_0}$ looks like it could represent the limit of the sequence $(\bfx_n)$, but the definition of a Cauchy sequence makes no mention of a limit at all. See Definition \ref{def:cauchysequence}, Scratch Work \ref{scr:cauchysequencebounded}, and the proof of Theorem \ref{thm:convergentsequencebounded}.}
\label{fig:cauchysequencebounded}
\end{figure}
\end{scratch}

\begin{proof}
Suppose $(\bfx_n)$ is a Cauchy sequence in $\R^m$. Let $\varepsilon_0=11$. By the definition of a Cauchy sequence \ref{def:cauchysequence}, there is a threshold $n_0$ where $n\geq n_0$ implies
\begin{align}\label{eqn:initialboundcauchy}
	d_m(\bfx_n,\bfx_{n_0})&=\|\bfx_n-\bfx_{n_0}\|_m<\varepsilon_0=11.
\end{align}
Hence, for $n\geq n_0$,
\begin{align}
	\|\bfx_n\|_m&=\|\bfx_n\underbrace{-\bfx_{n_0}+\bfx_{n_0}}_{\textnormal{add } \mathbf{0}}\|_m\\
	&\leq\|\bfx_n-\bfx_{n_0}\|_m+\|\bfx_{n_0}\|_m\qquad\textnormal{(tri. ineq. \eqref{eqn:addzero})}\\
	&<11+\|\bfx_{n_0}\|_m. \qquad\textnormal{(\eqref{eqn:initialboundcauchy})}
\end{align}
 
Now define $b$ by
\begin{align}
	b=\max\{\|\bfx_1\|_m,\|\bfx_2\|_m,\ldots,\|\bfx_{n_0-1}\|_m,11+\|\bfx_{n_0}\|_m\}.
\end{align}
Then $b\geq 0$ and for every index $n\in\N$ we have 
\begin{align}
	\|\bfx_n\|_m&<b.
\end{align}
Therefore, $(x_n)$ is bounded.
\end{proof}

The following Cauchy Criterion (Theorem \ref{thm:cauchycriterion}) is another major result in analysis that depends heavily on the completeness of the real line and Euclidean spaces. It provides a characterization of the existence of limits and the convergence of sequences, all without having an explicit candidates for the limits in mind. 

\begin{theorem}[Cauchy Criterion]\label{thm:cauchycriterion}
Suppose $(\bfx_n)$ is a sequence of points in $\R^m$. Then $(\bfx_n)$ converges if and only if $(\bfx_n)$ is Cauchy.
\end{theorem}

\begin{scratch}\label{scr:cauchycriterion}
The proof is lopsided: One direction follows from the definitions plus the triangle inequality \eqref{eqn:addzero}, the other sets up and takes advantage of the Bolzano-Weierstrass Theorem \ref{thm:bolzanoweierstrass} to ensure the existence of a candidate for the limit before the definitions and the triangle inequality \eqref{eqn:addzero} are used.
\end{scratch}

\begin{proof}
First, suppose $(\bfx_n)$ converges to $\bfy$ in $\R^m$. Let $\varepsilon>0$. By the definition of limit and convergence (Definition \ref{def:sequentiallimit}), there  is a threshold $n_{\varepsilon/2}$ where $n\geq n_{\varepsilon/2}$ ensures
\begin{align}
	d_m(\bfx_n,\bfy)&=\|\bfx_n-\bfy\|_m<\frac{\varepsilon}{2}.
\end{align}
Suppose $n$ and $k$ are positive integers where $n,k\geq n_{\varepsilon/2}$. Then we have 
\begin{align}
	d_m(\bfx_n,\bfx_k)&=\|\bfx_n-\bfx_k\|_m\\
	&=\|\bfx_n\underbrace{-\bfy+\bfy}_{\textnormal{add zero}}-\bfx_k\|_m\\
	&\leq \|\bfx_n-\bfy\|_m+\|\bfy-\bfx_k\|_m\\
	&<\frac{\varepsilon}{2}+\frac{\varepsilon}{2}\\
	&=\varepsilon.
\end{align}
Therefore, $n_{\varepsilon/2}$ is a suitable threshold and $(\bfx_n)$ is Cauchy.

Next, suppose $(\bfx_n)$ is Cauchy. In order to find a suitable candidate for the limit, note that by Lemma \ref{lem:cauchybounded}, $(\bfx_n)$ is bounded. So by the Bolzano-Weierstrass Theorem \ref{thm:bolzanoweierstrass}, there is a point $\bfy$ in $\R^m$ and subsequence $(\bfx_{n_k})$ whose limit is $\bfy$.

To show $\bfy$ is the limit of the original sequence $(\bfx_n)$, let $\varepsilon>0$. By the definition of a Cauchy sequence (Definition \ref{def:cauchysequence}), there is a threshold $j_{\varepsilon/2}$ such that $n,j\geq j_{\varepsilon/2}$ ensures 
\begin{align}
	d_m(\bfx_n,\bfx_j)&=\|\bfx_n-\bfx_j\|_m<\frac{\varepsilon}{2}.
\end{align}
Since the subsequence $(\bfx_{n_k})$ converges to $\bfy$, by the definition of limit and convergence for sequences (Definition \ref{def:sequentiallimit}), there is a threshold $k_{\varepsilon/2}$ such that $n_k\geq k\geq k_{\varepsilon/2}$ ensures 
\begin{align}
	d_m(\bfx_{n_k},\bfy)&=\|\bfx_{n_k}-\bfy\|_m<\frac{\varepsilon}{2}.
\end{align}
Now, there is a positive integer $k_0$ large enough so that both $n_{k_0}\geq k_0\geq k_{\varepsilon/2}$ and $n_{k_0}$ is an index for the subsequence where $n_{k_0}\geq j_{\varepsilon/2}$. Then for every $n\geq n_{k_0}$ we have 
\begin{align}
	d_m(\bfx_n,\bfy)&=\|\bfx_n-\bfy\|_m\\
	&=\|\bfx_n\underbrace{-\bfx_{n_{k_0}}+\bfx_{n_{k_0}}}_{\textnormal{add zero}}-\bfy\|_m\\
	&\leq \|\bfx_n-\bfx_{n_{k_0}}\|_m+\|\bfx_{n_{k_0}}-\bfy\|_m\\
	&<\frac{\varepsilon}{2}+\frac{\varepsilon}{2}\\
	&=\varepsilon.
\end{align}
Therefore, $n_{k_0}$ is a suitable threshold and $(\bfx_n)$ converges to $\bfy$.
\end{proof}

The next result solidifies the notion that if two sequences approach each other and one is Cauchy, then they both converge to the same limit.

\begin{corollary}\label{cor:samelimit}
Suppose $(\bfa_n)$ and $(\bfb_n)$ are sequences of points in $\R^m$ where $\lim_{n\to\infty}d_m(\bfa_n,\bfb_n)=0$ and suppose $(\bfa_n)$ is Cauchy. Then $(\bfa_n)$ and $(\bfb_n)$ converge to the same limit.
\end{corollary}

\begin{scratch}
By the Cauchy Criterion \ref{thm:cauchycriterion}, if $(\bfa_n)$ is Cauchy then $(\bfa_n)$ converges to some limit $\bfa$. From there, the condition of having the sequences approach each other, namely $\lim_{n\to\infty}d_m(\bfa_n,\bfb_n)=0$, can be used to show $(\bfb_n)$ converges to $\bfa$ as well. Some of the techniques used in the proof should look familiar.
\end{scratch}

\begin{proof}
Suppose $(\bfa_n)$ is Cauchy. Then by the Cauchy Criterion \ref{thm:cauchycriterion}, $(\bfa_n)$ converges to some limit $\bfa$. 

Furthermore, suppose $\lim_{n\to\infty}d_m(\bfa_n,\bfb_n)=0$ and let $\varepsilon>0$. Then we also have $\varepsilon/2>0$ and by two applications of the definition of limit and convergence for sequences (Definition \ref{def:sequentiallimit}), there are thresholds $j_{\varepsilon/2}$ and $k_{\varepsilon/2}$ where
\begin{align}
	n&\geq j_{\varepsilon/2}\quad\implies\quad d_m(\bfa_n,\bfa)=\|\bfa_n-\bfa\|_m <\frac{\varepsilon}{2}\qquad\textnormal{and}\\
	n&\geq k_{\varepsilon/2}\quad\implies\quad |d_m(\bfa_n,\bfb_n)-0|=\|\bfa_n-\bfb_n\|_m <\frac{\varepsilon}{2}.
\end{align}
Define $n_\varepsilon=\max\{j_{\varepsilon/2},k_{\varepsilon/2}\}$. Then for all $n\geq n_\varepsilon$ we have both $n\geq j_{\varepsilon/2}$ and $n\geq k_{\varepsilon/2}$. Therefore, $n\geq n_\varepsilon$ implies
\begin{align}
	d_m(\bfb_n,\bfa)&=\|\bfb_n-\bfa\|_m\\
	&=\|\bfb_n\underbrace{-\bfa_n+\bfa_n}_{\textnormal{add zero}}-\bfa\|_m\\
	&\leq \|\bfb_n-\bfa_n\|_m+\|\bfa_n-\bfa\|_m\\
	&<\frac{\varepsilon}{2}+\frac{\varepsilon}{2}\\
	&=\varepsilon.
\end{align}
Hence, $n_\varepsilon$ is a threshold for the convergence of $(\bfb_n)$ to $\bfa$, and 
\begin{align}
	\lim_{n\to\infty}\bfa_n=\bfa=\lim_{n\to\infty}\bfb_n.
\end{align}
\end{proof}


\vs
\section*{Exercises}
\setcounter{theorem}{0}

Exercises are for play: Do scratch work, draw stuff, and make mistakes---make {\em lots} of mistakes---before worrying about writing proofs. {\em Have fun!}

\chapter[Topology of Euclidean Spaces]{Topology of Euclidean Spaces}
\label{ch:topologyofeuclideanspaces}

The definition of arbitrarily close in Definition \ref{def:acl} allows us to explore ways a point can be arbitrarily close to a given set, whatever form such a set might take. An exploration of how a given point relates to either a given set or its {\em complement} through the lens of arbitrarily close leads to fundamental aspects of the usual topology on a Euclidean space $\R^m$. The definition of arbitrarily close is essentially topological in nature: We have $\bfy \acl{B}$ if and only if every neighborhood of $\bfy$ intersects $B$. See Figure \ref{fig:yaclBagain}.

\vs
\section{A closed-minded approach to topology}
\label{sec:closedmindedtopology}

Let's start with a toy problem.

\begin{prob}\label{prob:allpointsacl}
Draw a set $W$ in the plane $\R^2$ and determine the set of points arbitrarily close to $W$. HINT: Make use of Lemma \ref{lem:elementacl}. Next, determine and draw the set of points away from $W$.
\end{prob}

The set of points arbitrarily close to a given set gives rise to classic topological concepts: {\em closure} and {\em closed sets}. The definition of closure was already provided in Definition \ref{def:closure}; it's repeated here for convenience. Also, for a reminder about notation and terminology regarding $\varepsilon$-neighborhoods such as $V_\varepsilon(\bfx)$. See Section \ref{sec:arbitrarilycloseineuclideanspaces}, especially Definition \ref{def:neighborhood}, Figure \ref{fig:threeneighborhoods}, and Remark \ref{rmk:aclvianeighborhoods}).

\begin{definition}\label{def:closureclosed}
Let $B\subseteq\R^m$. The {\em closure} of $B$, denoted by $\overline{B}$, is the set of points arbitrarily close to $B$. Thus,
\begin{align}
\overline{B} = \{ \bfx \in \R^m: \bfx \acl B\} = \{\bfx\in \R^m: \, \forall\, \varepsilon>0, V_\varepsilon(\bfx)\cap B\neq\varnothing\}. 
\end{align}
A set $F\subseteq \R^m$ is {\em closed} if it contains all points arbitrarily close to it (that is, if $\overline{F}\subseteq F$). See Figures \ref{fig:yaclBagain} and \ref{fig:aclclosureagain}.
\end{definition}

Note that the empty set $\varnothing$ is vacuously closed.

\begin{figure}
\centering
\begin{tikzpicture} 
\draw[dashed, fill=blue!15] (-3,-1.41) rectangle (1.41,1.41);
	\draw (-0.8,0) node {$B$};
	\draw (-3.02,1.42) node {$\bullet$};
	\draw (-3.02,-1.42) node {$\bullet$};
	\draw (1.44,-1.42) node {$\bullet$};		
\begin{scope}[semithick, dashed, red] 
\draw (1.45,1.43) circle (2.1cm); 
\draw (1.45,1.43) circle (1.4cm); 
\draw (1.45,1.43) circle (0.7cm); 
\end{scope}	
\draw[semithick] (-3,-1.41) -- (-3,1.41);
\draw[semithick] (-3,-1.41) -- (1.41,-1.41);
	\draw[fill=white] (1.41,1.41) circle (0.08cm);
	\draw (1.8,1.41) node {$\bfy$};
\end{tikzpicture}
\caption{A point $\bfy$ and a set $B$ in the plane $\R^2$ where $\bfy\acl{B}$. As such, every neighborhood of $\bfy$ intersects $B$. Also, $B$ is not closed since $\bfy$ is not in $B$. See Definition \ref{def:closureclosed}.}
\label{fig:yaclBagain}
\end{figure}

\begin{remark}\label{rmk:classicclosure}
The definitions of closure and closed sets in Definition \ref{def:closureclosed} are not the standard ones found in other texts. However, they are equivalent. One benefit of using our Definition \ref{def:closureclosed}---as well as many definitions in this chapter and throughout the book---is how it follows directly from the definition of arbitrarily close (Definition \ref{def:acl}). Classic approaches such as the one in \cite[Definition 3.2.7, p.90]{Abbott} require the definitions of more complicated ideas such as limits of sequences or similar concepts  like {\em accumulation points} before defining closure and closed. (See Definition \ref{def:pointsaclset}.)
\end{remark}

\begin{figure}
\centering
\begin{tikzpicture}  
\draw[semithick, fill=blue!15] (-3,-1.41) rectangle (1.41,1.41);
	\draw (1.44,1.41) node {$\bullet$};
	\draw (1.8,1.41) node {$\bfy$};
	\draw (-0.8,0) node {$F$};
	\draw (-3.02,1.42) node {$\bullet$};
	\draw (-3.02,-1.42) node {$\bullet$};
	\draw (1.44,-1.42) node {$\bullet$};					
\begin{scope}[semithick, dashed, red] 
\draw (1.44,1.42) circle (2.1cm); 
\draw (1.44,1.42) circle (1.4cm); 
\draw (1.44,1.42) circle (0.7cm); 
\end{scope}	
\draw (-3,-1.41) -- (-3,1.41);
\draw (-3,-1.41) -- (1.41,-1.41);
\draw[-,semithick, red] (1.48,1.48) -- (2.92,2.92);
\draw[red] (2.8,2.55) node {$\varepsilon$};
\end{tikzpicture}
\caption{The set $F$ contains all points in and arbitrarily close to $F$, including the corner $\bfy$ and the sides of the rectangle. As such, $F$ is closed. See Definition \ref{def:closureclosed}.}
\label{fig:aclclosureagain}
\end{figure}

\begin{prob}\label{prob:isyoursetclosed}
What is the closure $\overline{W}$ of your set $W$ from Problem \ref{prob:allpointsacl}? Is your set $W$ closed? One way to check is to make sure all of the points outside of $W$ are away from $W$.
\end{prob}

The first result in this section is a fundamental property of closed sets.

\begin{lemma}\label{lem:intersectionclosed}
The intersection of any collection of closed sets is closed.
\end{lemma}

\begin{scratch}\label{scr:intersectionclosed}
We need to be careful here. Lemma \ref{lem:intersectionclosed} refers to {\em any} collection of closed sets, no matter how large. So, I let $A$ stand for any nonempty index set which could be finite, countably infinite, or uncountable. Aside from this sublety, the result follows from the definitions of a closed set (Definition \ref{def:closureclosed}), arbitrarily close (Definition \ref{def:acl}), and neighborhood
(Definition \ref{def:neighborhood}), as well as properties of intersections. It may help to consider to revisit Remark \ref{rmk:aclvianeighborhoods} as well.
\end{scratch}

\begin{proof}
Let $\{F_\alpha:\alpha\in A\}$ denote a collection of closed sets in $\R^m$ with a nonempty index set $A$. If any $F_\alpha$ is empty, then the intersection $\cap_{\alpha\in A} F_\alpha$ is empty as well and is therefore closed. 

Next, suppose $F_\alpha$ is nonempty for every $\alpha\in A$ and $\bfy$ is arbitrarily close to $\cap_{\alpha\in A} F_\alpha$. Let $\varepsilon>0$ and consider the $\varepsilon$-neighborhood $V_\varepsilon(\bfy)$. By the definition of a closed set (Definition \ref{def:closureclosed}) and the version of arbitrarily close in terms of neighborhoods as in Remark \ref{rmk:aclvianeighborhoods}, there is some $\bfx \in (\cap_{\alpha\in A} F_\alpha)\cap V_\varepsilon(\bfy)$ which means $\bfx \in F_\alpha\cap V_\varepsilon(\bfy)$ for each $\alpha$. Since the distance $\varepsilon$ was chosen arbitrarily and each $F_\alpha$ is closed, it follows that $\bfy\in F_\alpha$ for each $\alpha$; thus $\bfy\in\cap_{\alpha\in A} F_\alpha$. Therefore, $\cap_{\alpha\in A} F_\alpha$ is closed.
\end{proof}

Another lemma regarding closed sets indicates closures are themselves closed. Its proof is left as an exercise.

\begin{lemma}\label{lem:closureisclosed}
For any set $S\subseteq\R^m$, the closure $\overline{S}$ is a closed set. 
\end{lemma}

There is another immediate corollary of Lemma \ref{lem:intersectionclosed} stemming from the fact that the coda of a sequence of points in $\R^m$ is defined to be an intersection of closed sets. See Definition \ref{def:sequencecoda}.

\begin{corollary}\label{cor:sequencecodaclosed}
The coda of every sequence of points in $\R^m$ is closed.
\end{corollary}

A classic way to define closed sets in analysis stems from considering the limits of sequences whose terms are in the set. The following theorem provides a first characterization along these lines, but a more classic characterization which is very similar comes from considering what are called {\em accumulation points} (see Definition \ref{def:pointsaclset}). 

\begin{theorem}\label{thm:closedlimits}
A set $F\subseteq\R^m$ is closed if and only if $F$ contains the limits of all convergent sequences of points in $F$.
\end{theorem}

\begin{scratch}\label{scr:closedlimits}
The proof follows from the definition of a closed set (Definition \ref{def:closureclosed}) along with the fundamental connection between the definition of arbitrarily close (Definition \ref{def:acl}) and the definition of limit and convergence for sequences (Definition \ref{def:sequentiallimit}) provided by Theorem \ref{thm:exercise0}.
\end{scratch}

\begin{proof}
Suppose $F$ is a closed subset of $\R^m$ and let $(\bfx_n)$ be a convergent sequence of points in $\R^m$ with limit $\bfx$. By Theorem \ref{thm:exercise0}, we have $\bfx\acl{(\bfx_n)}$. 
Since every term in $(\bfx_n)$ is in $F$, we also have $\bfx\acl{F}$. By the definition of a closed set (Definition \ref{def:closureclosed}), the limit $\bfx$ is in $F$.

Now suppose $F$ contains the limits of all convergent sequences of points in $F$ and suppose $\bfy\acl{F}$. By Theorem \ref{thm:exercise0}, there is a sequence $(\bfy_n)$ of points in $F$ whose limit is $\bfy$. So, $F$ contains $\bfy$ and therefore, $F$ is closed.
\end{proof}

The {\em complements} of the closed sets in a Euclidean space $\R^m$ form the fundamental objects in the mathematical subject area known as {\em topology}. These complements are called {\em open} sets. See Definitions \ref{def:complement}, \ref{def:open}, and \ref{def:topology}.

\begin{definition}\label{def:complement}
Let $B$ be a subset of some set $X$. The \textit{complement} of $B$ (with respect to $X$) is the set of points in $X$ but not $B$. Thus, the complement of $B$ is given by
\begin{align}
	X\backslash B&=\{x\in X:x\notin B\}.
\end{align}
\end{definition}

To reinforce a key concept in Euclidean spaces, any point $\bfz$ not in a closed set $F$ is away from $F$ (see Definitions \ref{def:closureclosed} and \ref{def:awf}). Hence, there is an $\varepsilon_\bfz$-neighborhood $V_{\varepsilon_\bfz}(\bfz)$ which does not intersect $F$. This is precisely a characterizing property---therefore a defining property---of points in an {\em open} set (cf. \cite[Definition 3.2.1, p.88]{Abbott}). See Figure \ref{fig:openset}.

\begin{definition}\label{def:open}
A set $U\subseteq \R^m$ is {\em open} if every point in $U$ has a neighborhood contained in $U$.
\end{definition} 

\begin{figure}
\centering
\begin{tikzpicture}  
\draw[dashed, fill=blue!15] (-3,-1.41) rectangle (1.41,1.41);					
\begin{scope}[semithick, dashed, red] 
\draw (1.44,1.42) circle (2.1cm); 
\draw (1.44,1.42) circle (1.4cm); 
\draw (1.44,1.42) circle (0.7cm); 
\end{scope}	
	\draw (-2.2,-0.5) node {$\bullet$};
	\draw (-2.2,-0.8) node {$\bfa$};
	\draw[dashed] (-2.2,-0.5) circle (0.65cm);	
	\draw (-2.7,1.08) node {$\bullet$};
	\draw (-2.7,0.6) node {$\bfb$};
	\draw[dashed] (-2.7,1.1) circle (0.23cm);
	\draw[fill=white] (1.41,1.41) circle (0.08cm);
	\draw (1.8,1.41) node {$\bfy$};
	\draw (-0.8,0) node {$U$};
	\draw[fill=white] (-3,1.41) circle (0.08cm);
	\draw[fill=white] (-3,-1.41) circle (0.08cm);
	\draw[fill=white] (1.41,-1.41) circle (0.08cm);		
\end{tikzpicture}
\caption{The set $U$ contains a neighborhood around each of its points, such as $\bfa$ and $\bfb$. So, $U$ is an open set (see Definition \ref{def:open}). Note the neighborhood centered at $\bfb$ is necessarily smaller than the neighborhood centered at $\bfa$.  The point $\bfy$ is arbitrarily close to $U$, but no neighborhood of $\bfy$ is contained in $U$ and $\bfy$ is not in $U$.}
\label{fig:openset}
\end{figure}

\begin{remark}\label{rmk:openclassic}
Equivalently, $U$ is open if for every $\bfa\in U$ there is an $\varepsilon_\bfa>0$ such that $V_{\varepsilon_\bfa}(\bfa)\subseteq U$ (see Figure \ref{fig:openset}). In the language of the negation of arbitrarily close, a set $U$ is open if every point of $U$ is away from $\R^m\backslash U$ (see Definition \ref{def:awf}). In other words, all points in an open set are {\em interior points} (see Definition \ref{def:pointsaclset} below).
\end{remark}

The fundamental connection between open and closed sets readily follows (cf. \cite[Theorem 3.2.13, p.92]{Abbott}).

\begin{theorem}\label{thm:closedopen}
A set $U\subseteq \R^m$ is open if and only if its complement $\R^m\backslash U$ is closed.
\end{theorem}

A proof very similar to the one presented here was created by Rasha Issa as she prepared for a final exam in the summer of 2019. In particular, she used the language of arbitrarily close and preferred this approach over the one used in \cite[Theorem 3.2.13, p.92]{Abbott}.

\begin{proof}[Rasha Issa's proof of Theorem \ref{thm:closedopen}.]
Assume $U$ is open and suppose $\bfy$ is arbitrarily close to $\R^m\backslash U$. By way of contradiction, assume $\bfy\in U$. Since $U$ is open, there is an $\varepsilon_\bfy$-neighborhood of $\bfy$ contained in $U$. Hence, every $\bfx$ in $\R^m\backslash U$ lies outside of this $\varepsilon_\bfy$-neighborhood of $\bfy$. Thus, $\bfx$ is at least a positive distance $\varepsilon_\bfy$ away from $\bfy$. Hence, $\bfy$ is away from $\R^m\backslash U$, a contradiction. Therefore, $\R^m\backslash U$ is closed.

For the converse, assume $\R^m\backslash U$ is closed and let $\bfz\in U$. Since $\R^m\backslash U$ contains all points arbitrarily close to $\R^m\backslash U$, $\bfz$ is away from $\R^m\backslash U$. So there must be some $\varepsilon_\bfz>0$ where $V_{\varepsilon_{\bfz}}(\bfz)\subseteq U$. Therefore, $U$ is open.
\end{proof}

\begin{prob}\label{prob:isyourcomplementopen}
Is the complement of your set $W$ from Problem \ref{prob:allpointsacl} open?
\end{prob}

The following illustrates Theorem \ref{thm:closedopen} in the real line $\R$ by revisiting the closed interval $F$ studied throughout Chapter \ref{ch:kernelofanalysis}.

\begin{example}
Consider the closed interval $F=[0,3140]$ and let $w=3141$. Also, consider the positive distance $\varepsilon_w=1/10$. Then for every $x\in F$,
\begin{align}
	|x-3141|\geq\frac{1}{10}=\varepsilon_w.
\end{align} 
As a result, the neighborhood $V_{\varepsilon_w}(w)=(3141-1/10,3141+1/10)$
contains $w$ and is away from $F$ in the sense that 
\begin{align}
	F\cap V_{\varepsilon_w}(w) =[0,3140]\cap \left(3141-\frac{1}{10},3141+\frac{1}{10}\right) =\varnothing.
\end{align}

Now consider the complement
\begin{align}
	\R\backslash F&=(-\infty,0)\cup (3140,\infty).
\end{align} 
Every element in $\R\backslash F$ comes with an $\varepsilon$-neighborhood which is also contained in  $\R\backslash F$. Specifically, for each $z\in \R\backslash F$, define $\varepsilon_z$ to be the shorter of the distances between $z$ and the endpoints of $F=[0,3140]$. Since $z\neq 0$ and $z\neq 3140$, we can use 
\begin{align}\label{eqn:chooseepsilonz}
	\varepsilon_z=\min\{|z-0|,|z-3140|\}>0.
\end{align}
Then (perhaps after drawing a figure) we have
\begin{align}
	F\cap V_{\varepsilon_z}(z)= [0,3140]\cap V_{\varepsilon_z}(z) =\varnothing.
\end{align}
Hence, $V_{\varepsilon_z}(z)\subseteq \R\backslash F$. 

Therefore, the complement $\R\backslash F$ is open since every element in the complement of the closed interval $F$ comes with an $\varepsilon$-neighborhood contained in $\R\backslash F$. (See Definition \ref{def:open}.) Moreover, by Theorem \ref{thm:closedopen}, the interval $F=[0,3140]$ is closed. (Why is this last statement not redundant?)
\end{example}

The word {\em topology} describes both a mathematical topic and a particular mathematical object. The topic is massive and connects many other topics in beautiful and endless ways. The mathematical object is a particular collection of subsets of a given set.

\begin{definition}\label{def:topology}
Let $X$ be a set and let $\mathcal{T}$ be a collection of subsets of $X$. Then $\mathcal{T}$ is a {\em topology on} $X$ if the following properties hold:
\begin{enumerate}
	\item The empty set $\varnothing$ and the set $X$ are in $\mathcal{T}$.
	\item The intersection of any finite number of sets in $\mathcal{T}$ is a set in $\mathcal{T}$.
	\item The union of any collection of sets in $\mathcal{T}$ is a set in $\mathcal{T}$.
\end{enumerate}
When a set $X$ is paired with a topology on $X$, we call $X$ a {\em topological space}.
\end{definition}

\begin{theorem}\label{thm:opensetstopology}
The collection of all open subsets of $\R^m$ is a topology on $\R^m$. That is,
\begin{enumerate}
	\item The empty set $\varnothing$ and the set $\R^m$ are open.
	\item The intersection of any finite number of open sets is open. 
	\item The union of any collection of open sets is open.
\end{enumerate} 
\end{theorem}

\begin{remark}\label{rmk:topologyterminologyopen}
For those of you who have seen topology before, Theorem \ref{thm:opensetstopology} may sound like a tautology. After all, in a topology class, open sets are defined to be the sets in a topology. However, our definition for open sets (Definition \ref{def:open}) precedes the definition of topology (Definition \ref{def:topology}), so the proof of Theorem \ref{thm:opensetstopology} amounts to verifying the collection of all open subsets of $\R^m$ satisfies the three properties defining a topology.
\end{remark}

\begin{scratch}\label{scr:opensetstopology}
All of the results follow from a careful application of the definitions. 
\end{scratch}

\begin{proof} Let $\mathcal{T}$ denote the collection of all open subsets of $\R^m$.\s   

\noindent \underline{Proof of (i)}: Consider the empty set $\varnothing$. Then $\varnothing$ vacuously satisfies the definition of an open set (Definition \ref{def:open}) since it has no points in need of a neighborhood. Hence, $\varnothing$ is in $\mathcal{T}$.

Now consider the set $\R^m$ itself. 
Since $\R^m$ contains all $\varepsilon$-neighborhoods of all points in $\R^m$, we have $\R^m$ is open. (For instance, $\R^m$ contains the $17$-neighborhood of $\bfx$ for every $\bfx\in\R^m$). Hence, $\R^m$ is in $\mathcal{T}$.\s

\noindent \underline{Proof of (ii)}: Suppose $U_1,U_2,\ldots,U_n$ are open sets in $\R^m$ and let 
\begin{align}
	\bfx\in\bigcap_{j=1}^n U_j.
\end{align}
So, $\bfx$ is in the open set $U_j$ for each $j=1,\ldots,n$. By the definition of open (Definition \ref{def:open}), for each $j=1,\ldots,n$ there is an $\varepsilon_j>0$ such that the $\varepsilon_j$-neighborhood of $\bfx$, $V_{\varepsilon_j}(\bfx)$, is contained in $U_j$. Since we are considering a finite number of open sets, the smallest of these neighborhoods has positive radius $\varepsilon_0=\min\{\varepsilon_1,\ldots,\varepsilon_n\}$ and is contained in each of the $\varepsilon_j$-neighborhoods of $\bfx$. That is,
\begin{align}
	V_{\varepsilon_0}(\bfx)\subseteq \bigcap_{j=1}^n V_{\varepsilon_j}(\bfx)\subseteq \bigcap_{j=1}^n U_j.
\end{align}
Hence, the intersection of any finite number of open sets in $\R^m$ is an open set, and so $\bigcap_{j=1}^n U_j$ is in $\mathcal{T}$.\s

\noindent \underline{Proof of (iii)}:
Suppose $\{U_\alpha:\alpha\in A\}$ is a collection of open sets in $\R^m$ with nonempty index set $A$ and let 
\begin{align}
	\bfx\in\bigcup_{\alpha\in A} U_\alpha.
\end{align}
Then there must be some index $\alpha_\bfx$ in $A$ where $\bfx$ is in the open set $U_{\alpha_\bfx}$. By the definition of open (Definition \ref{def:open}), there is an $\varepsilon_\bfx>0$ where 
\begin{align}
	V_{\varepsilon_\bfx}(\bfx)\subseteq U_{\alpha_\bfx}\subseteq \bigcup_{\alpha\in A} U_\alpha.
\end{align}
Therefore, the union of any collection of open sets in $\R^m$ is an open set.

Since all three conditions defining a topology are satisfied by $\mathcal{T}$, we have $\mathcal{T}$ is a topology on $\R^m$. (See Definition \ref{def:topology}.)
\end{proof}

\begin{example}\label{eg:reallinetopology}
The standard topology on the real line $\R$ is the collection of all open intervals and all unions of open intervals.
\end{example}

Example \ref{eg:reallinetopology} is just one piece of a more powerful statement regarding the standard topology on the real line. Its proof is left as a challenging exercise and makes use of the following definition.

\begin{definition}\label{def:pairwisedisjoint}
A collection of sets $\mathcal{S}$ is {\em pairwise disjoint} if for every pair of sets $A,B\in \mathcal{S}$ where $A\neq B$ we have $A\cap B=\varnothing$.
\end{definition}

\begin{theorem}\label{thm:reallineopensets}
Every open subset of the real line is the union of a pairwise disjoint countable collection of open intervals. That is, for every open set $U\subseteq\R$, there is a sequence of open intervals $(I_n)$ where the collection $\{I_n:n\in\N\}$ is pairwise disjoint and 
\begin{align}
	U=\bigcap_{n=1}^\infty I_n.
\end{align}
\end{theorem}

An analogy of Theorem \ref{thm:opensetstopology} holds for collections of closed sets, one of which has already been stated in Lemma \ref{lem:intersectionclosed}.

\begin{theorem}\label{thm:closedtrio}
The following properties regarding closed sets in $\R^m$ hold:
\begin{enumerate}
	\item The empty set $\varnothing$ and the set $\R^m$ are closed.
	\item The union of any finite number of closed sets is closed. 
	\item The intersection of any collection of closed sets is closed.
\end{enumerate}
\end{theorem}

\begin{scratch}\label{scr:closedtriodemorgans}
Part (iii) of Theorem \ref{thm:closedtrio} is a rephrased version of Lemma \ref{lem:intersectionclosed}. Parts (i) and (ii) follow from Theorem \ref{thm:closedopen} and results from set theory on the relationships between complements, intersections, and unions known as {\em De Morgan's Laws}. These results are stated but not proven below. From there, the details of the proof of Theorem \ref{thm:closedtrio} are left as an exercise. 
\end{scratch}

\begin{theorem}[De Morgan's Laws] 
\label{thm:demorgans}
Suppose $A$ and $B$ are subsets of some set $X$. Then: 
\begin{enumerate}
	\item $X\backslash(A\cap B)=(X\backslash A)\cup(X\backslash B)$; and
	\item $X\backslash(A\cup B)=(X\backslash A)\cap (X\backslash B)$. 	
\end{enumerate}
Suppose $\mathcal{C}$ is a collection of subsets of some set $X$.
Then:
\begin{enumerate}
	\item $\displaystyle X\backslash\left(\bigcap_{S\,\in\,\mathcal{C}} S\right)=\bigcup_{S\,\in\,\mathcal{C}} (X\backslash S)$; and
	\item $\displaystyle X\backslash\left(\bigcup_{S\,\in\,\mathcal{C}} S\right)=\bigcap_{S\,\in\,\mathcal{C}} (X\backslash S)$. 
\end{enumerate}
\end{theorem}

The following section explores the classic topological notion of {\em connectedness} using an unconventional definition that stems from the concept of arbitrarily close.

%

\vs
\section*{Exercises}
\setcounter{theorem}{0}

Exercises are for play: Do scratch work, draw stuff, and make mistakes---make {\em lots} of mistakes---before worrying about writing proofs. {\em Have fun!}

\xca\label{exer:closedinterval}
Consider an closed interval $[a,b]=\{x\in\R:a\leq x\leq b\}$ where $a$ and $b$ are real numbers satisfying $a<b$. Is $[a,b]$ closed according to Definition \ref{def:closureclosed}? Draw figures and prove your answer. Why isn't this trivial?

\xca\label{exer:openinterval}
Consider an open interval $(a,b)=\{x\in\R:a<x<b\}$ where $a$ and $b$ are real numbers satisfying $a<b$. Is $(a,b)$ open according to Definition \ref{def:open}? Draw figures and prove your answer. Why isn't this trivial? 

\xca\label{exer:neitherinterval}
Consider an interval $(a,b]=\{x\in\R:a<x\leq b\}$ where $a$ and $b$ are real numbers satisfying $a<b$. Prove $(a,b]$ is neither closed nor open.

\xca\label{exer:reallineclopen}
Consider the real line $\R=(-\infty,\infty)$. Prove $\R$ is both open and closed.

\xca\label{exer:closureisclosed}
Prove Corollary \ref{lem:closureisclosed}: The closure of a set is itself a closed set. 

\xca Prove the closure of a given set in $\R^m$ is the smallest closed set containing the given set in the following sense: Given a set $S\subseteq\R^m$, every closed set which contains $S$ also contains the closure $\overline{S}$.

\xca\label{exer:neitherrectangle}
Consider a rectangle $B$ in the plane $\R^2$ such as in Figure \ref{fig:yaclBagain} which contains some of its sides and corners, but not all of them. Prove $B$ is neither open nor closed.

\xca\label{exer:demorgans}
Prove De Morgan's Laws (Theorem \ref{thm:demorgans}).

\xca\label{exer:reallineopensets}
Prove Theorem \ref{thm:reallineopensets}.

\xca\label{exer:closedtrio}
Prove Theorem \ref{thm:closedtrio}.

\vs
\section{Connected sets}
\label{sec:connectedsets}

The notion of {\em connectedness} is yet another classic topic in analysis and topology which lends itself to a description in terms of arbitrarily close (Definition \ref{def:acl}). Intuitively, a set $E$ is {\em connected} if it comes in one piece with no separate chunks. For the purpose of writing proofs, a more technical definition is in order: A set $E$ is {\em connected} if every partition of $E$ features a point in one set arbitrarily close to another set. See Figure \ref{fig:connectedinplane}.

\begin{definition}\label{def:connected}
A set $E\subseteq\R^m$ is {\em connected} if for any pair of nonempty sets $A$ and $B$ where $A\cup B=E$, there is a point $\bfx$ in $A$ where $\bfx\acl B$ or there is a point $\bfy$ in $B$ where $\bfy \acl A$. On the other hand, a set $E$ is {\em  disconnected} if it is not connected, that is, there are nonempty sets $A$ and $B$ where $A\cup B=E$, every point in $A$ is a away from $B$, and every point in $B$ is away from $A$. 
\end{definition}

\begin{figure}
\centering
\begin{tikzpicture} 
\draw[semithick, fill=blue!15] (-1,-1) rectangle (1,1);
	\draw (0,0) node {$S$};
\draw[dashed, fill=blue!15] (2,2) circle (1.41cm); 	
	\draw (2,2) node {$D$};	
\draw[fill] (1,1) circle (0.08cm);
	\draw (0.6,0.6) node {$\bfw$};
\begin{scope}[semithick, dashed, red] 
	\draw (1,1) circle (0.3cm);
	\draw (1,1) circle (1cm); 
	\draw (1,1) circle (1.7cm); 
\end{scope}	
\end{tikzpicture}
\caption{The closed square $S$ and the open disk $D$ form Example \ref{eg:connectedinplane} form the union $E=S\cup D$, which is connected. See Definition \ref{def:connected}. The point $\bfw$ (the upper right corner of $S$) is in $S$ and arbitrarily close to $D$.}
\label{fig:connectedinplane}
\end{figure}

\begin{example}\label{eg:connectedinplane}
In the plane $\R^2$, consider the closed square $S$ and the open disk $D$ as in Figure \ref{fig:connectedinplane} which are given by
\begin{align}
	S&=\left\{
	\bfx\in\R^2:\bfx=
	\left[
	\begin{array}{c}
	x\\
	y
\end{array}
\right],
	-1\leq x\leq 1, \textnormal{ and } -1\leq y\leq 1
	\right\}, \qquad\textnormal{and}\\
	D&=
	V_{\sqrt{2}}(\bfx_0)=
	\left\{
	\bfx\in\R^2: d_2(\bfx,\bfx_0)<\sqrt{2}
	\textnormal{ where }
	\bfx_0=\left[
	\begin{array}{c}
	2\\
	2
\end{array}
\right]
	\right\}
.
\end{align}
The set $E=S\cup D$ is connected (which is not proven here). In particular, the sets $S$ and $D$ do not form a separation of $E$ since the point $\bfw=\left[
	\begin{array}{c}
	1\\
	1
\end{array}
\right]$ is in both $S$ (its coordinates check out) and $\overline{D}$ where 
\begin{align}
	\overline{D}&=\left\{
	\bfx\in\R^2: d_2(\bfx,\bfx_0)\leq\sqrt{2}
	\textnormal{ where }
	\bfx_0=\left[
	\begin{array}{c}
	2\\
	2
\end{array}
\right]
	\right\}
	=\overline{V_{\sqrt{2}}(\bfx_0)}.
\end{align}
To check, we have $d_2(\bfx_0,\bfw)=\sqrt{(2-1)^2+(2-1)^2}=\sqrt{2}$. Hence, $\bfw$ is indeed in $\overline{D}$ and so $\bfw\in S$ while $\bfw\acl{D}$.
\end{example}

\begin{remark}\label{rmk:connectedclassic}
In other notation, $E$ is connected if for any nonempty sets $A$ and $B$ with $A\cup B=E$ there is some $\bfx\in A\cap \overline{B}$ or some $\bfy\in B\cap \overline{A}$. Also, the empty set $\varnothing$ is vacuously connected since it is not a union of nonempty sets.

Alternate---and more common---approaches first define disconnected then take its negation to define connected. 
\begin{itemize}
	\item[(i)] A set $F\subseteq\R^m$ is {\em disconnected} if there is a pair of nonempty sets $A$ and $B$  where $A\cup B=F, A\cap\overline{B}=\varnothing$, and $B\cap\overline{A}=\varnothing$. Such a pairing of nonempty sets $A$ and $B$ is called a {\em separation} of $F$. From there, a set $E\subseteq\R^m$ is said to be connected if it is not disconnected. (Cf. \cite[Definition 3.4.4, p.104]{Abbott}.)
	\item[(ii)] A set $F\subseteq\R$ is {\em disconnected} if there is a pair of nonempty open sets $U$ and $V$  where $F\subseteq U\cup V$, $U\cap V=\varnothing$, $U\cap F\neq \varnothing$, and $V\cap F\neq \varnothing$. Such a pairing of sets $U$ and $V$ is called a {\em separation} of $F$. From there, a set $E\subseteq\R^m$ is said to be connected if it is not disconnected. 
\end{itemize}

The fact that these two approaches to defining connectedness are equivalent to Definition \ref{def:connected} is left as an exercise. Personally, I find them to be clunky and unsatisfying: Both approaches define connectedness as the negation of some other property. But they are more common, so it's important to mention them here. This led me to wonder what connectedness means on its own terms, in turn leading to Definition \ref{def:connected}.
\end{remark}

\begin{figure}
\centering
\begin{tikzpicture} 
\draw[dashed, fill=blue!15] (-1,-1) rectangle (1,1);
	\draw (0,0) node {$S^o$};
\draw[dashed, fill=blue!15] (2,2) circle (1.41cm); 	
	\draw (2,2) node {$D$};	
\draw[fill=white] (1,1) circle (0.08cm);
	\draw (0.6,0.6) node {$\bfw$};
\begin{scope}[semithick, dashed, red] 
	\draw (1,1) circle (0.3cm);
	\draw (1,1) circle (1cm); 
	\draw (1,1) circle (1.7cm); 
\end{scope}	
\end{tikzpicture}
\caption{The open square $S^o$ and the open disk $D$ form Example \ref{eg:disconnectedinplane} form the union $U=S\cup D$, which is disconnected. See Definition \ref{def:connected}. The point $\bfw$ (the upper right corner of $S^o$) is {\em not} in $S^o$ but is arbitrarily close to both $S^o$ and $D$.}
\label{fig:disconnectedinplane}
\end{figure}

\begin{example}\label{eg:disconnectedinplane}
Now consider the open square $S^o$ given by
\begin{align}
	S^o&=\left\{
	\bfx\in\R^2:\bfx=
	\left[
	\begin{array}{c}
	x\\
	y
\end{array}
\right], 
	-1< x<1, \textnormal{ and } -1< y < 1
	\right\}
\end{align}
and once again consider the open disk $D$ from Example \ref{eg:connectedinplane}. It turns out the set $U=S^o\cup D$ is disconnected since both both $\overline{S^o}\cap D=\varnothing$ and $S^o\cap\overline{D}=\varnothing$. See Figure \ref{fig:disconnectedinplane}.

Furthermore, $S^o$ and $D$ form a separation of $U$ (according to both versions of a separation in Remark \ref{rmk:connectedclassic}). In this case, the point $\bfw=\left[
	\begin{array}{c}
	1\\
	1
\end{array}
\right]$ is the only point in the plane where we have both $\bfw\acl{S^o}$ and $\bfw\acl{D}$, but $\bfw$ is in neither $S^o$ nor $D$. 
\end{example}

Thanks to the deep connection between limits of convergent sequences and the notion of arbitrarily close established in Theorem \ref{thm:exercise0}, there is a characterization of connectedness in terms of limits.

\begin{corollary}\label{cor:connectedacl}
A set $E\subseteq\R^m$ is connected if and only if for any pair of nonempty sets $A$ and $B$ where $A\cup B=E$, there is a sequence of points in $A$ whose limit is in $B$, or vice versa. 
\end{corollary}

\begin{scratch}\label{scr:connectedacl}
The result follows from Definition \ref{def:connected} and both directions of Theorem \ref{thm:exercise0} (the fundamental connection between arbitrarily close and limits of sequences).
\end{scratch}

\begin{proof}
First, suppose $E$ is connected. Then, without loss of generality and by Definition \ref{def:connected}, for any pair of nonempty sets $A$ and $B$ where $A\cup B=E$, there is a point $\bfy$ in $B$ where $\bfy\acl{A}$. By Theorem \ref{thm:exercise0}, there is a sequence $(\bfx_n)$ of points in $A$ whose limit is $\bfy$.

Now suppose $A$ and $B$ are nonempty sets where $A\cup B=E$. Also, without loss of generality, suppose there is a sequence $(\bfx_n)$ of points in $A$ whose limit is $\bfy$ and where $\bfy$ is in $B$. By Theorem \ref{thm:exercise0}, $\bfy\acl{B}$. So, by the definition of a connected set (Definition \ref{def:connected}), $E$ is connected.
\end{proof}

Next up, consider the classic example of the {\em topologist's sine curve} thought of as a subset of the plane. 

\begin{example}\label{eg:topologistsine}
Let $G$ be the graph of the function $g:\R^+\to\R$ given by
\begin{align}
	g(x)=\sin\left(\frac{1}{x}\right).
\end{align}
That is, 
\begin{align}
	G&=\left\{
	\bfx\in\R^2:\bfx=
	\left[
	\begin{array}{c}
	x\\
	y
\end{array}
\right]
	\textnormal{where } x>0 \textnormal{ and } y=\sin\left(\frac{1}{x}\right)
	\right\}.
\end{align}
Try using free online software such as Desmos, GeoGebra, or WolframAlpha to plot $G$. Also, let $L$ be the line segment given by 
\begin{align}
	L&=\left\{
	\bfx\in\R^2:\bfx=
	\left[
	\begin{array}{c}
	x\\
	y
\end{array}
\right]
	\textnormal{where } x=0 \textnormal{ and } -1\leq y\leq 1
	\right\}.\label{eqn:segmentL}
\end{align}
Even though $G\cap L=\varnothing$, the set $E=G\cup L$ is connected. 
The proof that $E=G\cup L$ is connected will be handled later when we have more tools at our disposal. For now, we can prove every point in $L$ is the limit of a convergent sequence of points $G$. 
\end{example}

\begin{proof}[Partial proof of Example \ref{eg:topologistsine}]
Let $\bfp=\left[
	\begin{array}{c}
	0\\
	y_0
\end{array}
\right]$ be a point in $L$. Then \eqref{eqn:segmentL} ensures we have $-1\leq y_0\leq 1$, so there is some $x_0>0$ where $\sin(1/x_0)=y_0$. 

Now consider the sequence of positive real numbers defined by
\begin{align}
	a_n=\frac{1}{x_0+2\pi n}
\end{align}
for each positive integer $n$. Then, thanks to the periodicity of the sine function, we have 
\begin{align}
	g(a_n)&=\sin\left(\frac{1}{a_n}\right)=\sin(x_0+2\pi n)=y_0.
\end{align}
From there, consider the sequence $(\bfx_n)$ in the plane defined by 
\begin{align}
\bfx_n&=\left[
	\begin{array}{c}
	a_n\\
	g(a_n)
	\end{array}
	\right]
=\left[
	\begin{array}{c}
	a_n\\
	y_0
	\end{array}
\right]
\end{align}
for each positive integer $n$ and $(\bfx_n)$ is a sequence of points in $G$. Since
\begin{align}
	\lim_{n\to\infty}a_n&=\lim_{n\to\infty}\left(\frac{1}{x_0+2\pi n}\right)=0 \qquad\textnormal{and}\\
	\lim_{n\to\infty}g(a_n)&=\lim_{n\to\infty}\sin\left(\frac{1}{a_n}\right)=\lim_{n\to\infty}y_0=y_0,
\end{align}
Theorem \ref{thm:componentwisesequences} (regarding componentwise convergence) applies and tells us 
\begin{align}
	\lim_{n\to\infty}\bfx_n
	&=\left[
		\begin{array}{c}
			\displaystyle\lim_{n\to\infty}a_n\\
			\displaystyle\lim_{n\to\infty}y_0
		\end{array}
	\right]
	=\left[
		\begin{array}{c}
			0\\
			y_0
		\end{array}
	\right]
	=\bfp.
\end{align}
So, every point in $L$ is the limit of a sequence of points in $G$.
\end{proof}

Theorem \ref{thm:realconnected} is yet another special feature of the real line $\R$: Intervals and singletons are the only connected subsets of the real line. 

In the case of intervals, their characterization in Lemma \ref{lem:intervalcharacterization} provides a helpful perspective to consider as when trying to prove Theorem \ref{thm:realconnected}. The proof of Lemma \ref{lem:intervalcharacterization} amounts to checking the definition of an interval in Definition \ref{def:intervals}, so it is omitted. After all, intervals are defined to be subsets of the real line comprising the points between endpoints and possibly the endpoints themselves. See Figure \ref{fig:intervalplots}.

\begin{lemma}\label{lem:intervalcharacterization}
A subset $I$ of the real line $\R$ is an interval if and only if whenever $x\in I,y\in I$ and $x<z<y$, then $z\in I$ as well.
\end{lemma}

\begin{theorem}\label{thm:realconnected}
A nonempty subset of the real line $\R$ is connected if and only if the subset is an interval or a singleton.
\end{theorem}

The proof of the implication stating connected implies the set is either a singleton or an interval is handled via contraposition. But first, the proof that singletons and intervals are connected is handled directly with a pair of cases. 

\begin{proof}
Suppose $E$ is a nonempty subset of the real line $\R$.

\underline{Case (i)}: Suppose is a singleton where $E=\{c\}$. Then any nonempty subsets $A$ and $B$ of $E$ must also be the same singleton, that is
\begin{align}
	A&=B=\{c\}.
\end{align}
Since $c\in A$ and $c\in B$, by Lemma \ref{lem:elementacl} we have $c\acl{A}$ and $c\acl{B}$. Therefore, $E$ is connected (see Definition \ref{def:connected}).

\underline{Case (ii)}: Suppose $E$ is an interval with nonempty subsets $A$ and $B$ where $E=A\cup B$ and, without loss of generality, there are real numbers $a_0\in A$ and $b_0\in B$ where $a_0<b_0$. 

Now consider the closed interval $I_0=[a_0,b_0]$ which is a subset of $E$ since $E$ is an interval. Following a bisection method much like the proof of the NCBI Property \ref{thm:nestedintervals}, the midpoint $(a_0+b_0)/2$ is in $E$. Therefore, $(a_0+b_0)/2$ is in $A$ or $B$. Define $I_1=[a_1,b_1]$ by either $I_1=[a_0,(a_0+b_0)/2]$ or $I_1=[(a_0+b_0)/2,b_0]$, chosen so that $a_1\in A$ and $b_1\in B$. Proceeding recursively, define a sequence of closed and bounded intervals $I_n$ where: the midpoint of $I_n$ is an endpoint of $I_{n+1}$; $I_{n+1}$ is chosen so its left endpoint $a_{n+1}$ is in $A$, and its right endpoint $b_{n+1}$ is in $B$. Then the sequence of intervals $(I_n)$ is nested, so by the NCBI Property \ref{thm:nestedintervals}, there is a point 
$x$ where
\begin{align}
	x\in \bigcap_{n=1}^\infty I_n\subseteq E.
\end{align}

Since the $I_n$ are constructed via bisection, their lengths are successively cut in half. So, for every $n\in\N$ we have 
\begin{align}
	a_n \leq x \leq b_n\qquad \textnormal{and}\qquad |a_n-b_n|=\frac{|a_0-b_0|}{2^n}.
\end{align}
By Theorem \ref{thm:algebraiclimitseuclidean} and Corollary \ref{cor:powersofconstant}, we have
\begin{align}
	\lim_{n\to\infty}\frac{|a_0-b_0|}{2^n}=0.
\end{align}

Now let $\varepsilon>0$. Since $a_n\leq x\leq b_n$, properties inequalities and the definition of limit and convergence (Definition \ref{def:sequentiallimit}), there is a threshold $n_\varepsilon$ where for every $n\geq n_\varepsilon$ we have both
\begin{align}
	|a_n-x|&\leq |a_n-b_n|=\frac{|a_0-b_0|}{2^n}<\varepsilon \quad\textnormal{and}\\
	|b_n-x|&\leq |a_n-b_n|=\frac{|a_0-b_0|}{2^n}<\varepsilon.
\end{align} 
Since $a_n\in A$ and $b_n\in B$ for every $n\in\N$, we have $x\acl{A}$ and $x\acl{B}$. Since $x\in E$ and $E=A\cup B$, we have $x\in A$ or $x\in B$. Therefore, by Definition \ref{def:connected}, $E$ is connected.

Next, to show a nonempty connected subset of the real line is connected, let's argue via contraposition. Suppose $E$ contains at least two points and is not an interval (otherwise, $E$ contains just one point and is thus a singleton). By Lemma \ref{lem:intervalcharacterization} and without loss of generality, there are $x\in E$ and $y\in E$ where $x<y$ as well as $z\in (x,y)$ where $z\notin E$. Define 
\begin{align}
	L_z=(-\infty,z)\cap E\quad\textnormal{and}\quad R_z=(z,\infty)\cap E. 
\end{align}
Then $E=L_z\cup R_z$, $x\in L_z$, and $y\in R_z$. For every $\ell\in L_z$ and every $r\in R_z$ we have
\begin{align}
 \ell<z<r.
\end{align}
So, for every $\ell\in L_z$ and every $r\in R_z$ we have both 
\begin{align}
	|\ell-r|&>|\ell-z|>0 \quad\textnormal{and}\\
	|\ell-r|&>|r-z|>0.
\end{align}
Hence, every $\ell\in L_z$ is away from $R_z$ and every $r\in R_z$ is away from $L_z$. Therefore, $E$ is disconnected (see Definition \ref{def:connected}).
\end{proof}

To conclude the section, the following definition provides a specific meaning for the concept of having two {\em sets} arbitrarily close to one another. Former students Jeffrey Robbins, Lekha Patil, and Ryan Aniceto each thought of equivalent versions of this definition. 

\begin{definition}\label{def:twosetsacl}
Suppose $A,B\subseteq\R^m$. The sets $A$ and $B$ are said to be {\em arbitrarily close} if their closures intersect, thus $\overline{A}\cap\overline{B}\neq\varnothing$. In this case we write $A\acl{B}$. 

Equivalently, $A\acl{B}$ if there is a point $\bfy$ where both $\bfy\acl{A}$ and $\bfy\acl{B}$. 
\end{definition}

The notions of connectedness (Definition \ref{def:connected}) and having two sets arbitrarily close to one another (Definition \ref{def:twosetsacl}) are related, but not equivalent.

\begin{example}\label{eg:setsacldisconnected}
The open sets $S^o$ and $D$ in Example \ref{eg:disconnectedinplane} are arbitrarily close as in Definition \ref{def:twosetsacl} ($S^o\acl{D}$) since the point $\bfw$ is arbitrarily close to both sets as in Definition \ref{def:acl}. That is, $\bfw\acl{D}$ as explained in Example \ref{eg:connectedinplane} while $\bfw\acl{S^o}$ by similar reasoning: We have $\bfw\in{S}$ where $S$ is as in Example \ref{eg:connectedinplane}, and since $\bfw \in S$ and $S=\overline{S^o}$. However, the union $E^o=S^o\cup D$ is not connected as explained in Example \ref{eg:disconnectedinplane}.
\end{example}

\begin{remark}\label{rmk:variousacl}
Note Definition \ref{def:twosetsacl} makes use of  the notation ``$\acl$'' in two subtly different ways: When comparing two {\em sets} using $A\acl{B}$, Definition \ref{def:twosetsacl} provides a reasonable perspective which explicitly relies on what it means for {\em points} to be arbitrarily close to sets with both $\bfy\acl{A}$ and $\bfy\acl{B}$ as in Definitions \ref{def:aclreal} or \ref{def:acl}.

Additionally, a perspective on having two {\em points} arbitrarily close to one another is provided by Lemma \ref{lem:equalpoints}: In $\R^m$, two points are arbitrarily close to one another if and only if they are the same point. In the more general setting of topological spaces, this is not necessarily the case.
\end{remark}

Example \ref{eg:setsacldisconnected} shows us two sets can be arbitrarily close while having a disconnected union. Lemma \ref{lem:elementacl}---points in a set are arbitrarily close to the set---allows us to codify the relationship between connectedness (Definition \ref{def:connected})  and having pairs of sets arbitrarily close to one another (Definition \ref{def:twosetsacl}).

\begin{theorem}\label{thm:connectedimpliesacl}
Suppose $E\subseteq\R^m$ is connected. Then for every pair of nonempty sets $A$ and $B$ where $E=A\cup B$ we have $A\acl{B}$.
\end{theorem} 

\begin{proof}
Suppose $E$ is connected and $A$ and $B$ are nonempty where $E=A\cup B$. By the definition of connectedness (Definition \ref{def:connected}) and without loss of generality, there is a point $\bfy\in A$ such that $\bfy\acl{B}$. Since $\bfy\in A$, we also have $\bfy\acl{A}$ by Lemma \ref{lem:elementacl}. Therefore, $A\acl{B}$.
\end{proof}

The next section explores {\em compactness}, another important classic topic in analysis and topology with a difficult definition.

\vs
\section*{Exercises}
\setcounter{theorem}{0}

Exercises are for play: Do scratch work, draw stuff, and make mistakes---make {\em lots} of mistakes---before worrying about writing proofs. {\em Have fun!}
%

\vs
\section{Compact sets}
\label{sec:compactsets}

{\em Compactness} is a topological property with implications across analysis, but its definition can be difficult to understand and appreciate at first. As my former student Ryan Aniceto says:

\begin{quote}
Compactness is the next best thing to finiteness.
\end{quote}

This section aims to make sense of Ryan's notion in a mathematically concrete way. To help motivate the formal definition of compactness (Definition \ref{def:compact}), let's first see what we can say about finite sets in the real line $\R$. 

\begin{definition}\label{def:finite}
A set is {\em finite} if it is empty or if it has $n_0$ elements for some positive integer $n_0$.
\end{definition}

\begin{proposition}\label{prop:finitereal}
Suppose $S$ is a nonempty and finite set of real numbers. Then
\begin{enumerate}
	\item $S$ is bounded.
	\item Both $\max{S}$ and $\min{S}$ exist.
	\item $S$ is closed.
	\item Every sequence of real numbers in $S$ has a constant subsequence.
\end{enumerate}
\end{proposition}

For a nonempty finite set of real numbers, we can always list the elements from least to greatest.

\begin{proof}
Let $S$ be a nonempty and finite set of real numbers $S$ with $n_0$ elements. Without loss of generality, we have
\begin{align}
	S&=\{s_1,s_2,\ldots,s_{n_0}\} \quad\textnormal{where}\quad s_1<s_2<\cdots <s_{n_0}.
\end{align}
Hence, $s_1$ is a lower bound for $S$ which is in $S$ and $s_{n_0}$ is an upper bound for $S$ which is in $S$. So, $S$ is bounded with $\min{S}=s_1$ and $\max{S}=s_{n_0}$.

To see why $S$ is closed, note that by Lemma \ref{lem:elementacl}, each of the elements in $S$ is arbitrarily close to $S$. Also, every other real number is away from $S$. To that end, suppose $x$ is a real number that is not in $S$. Then for every $k=1,\ldots,n_0$ we have
\begin{align}
	x&\neq s_k\qquad \textnormal{and so}\qquad d_{\R}(x,s_k)=|x-s_k|>0. 
\end{align}
Now let $\varepsilon_x=\min\{|x-s_k|:k=1,\ldots,n_0\}$. Since there are only finite distances to consider, we have $\varepsilon_x>0$. Also, for every $k=1,\ldots,n_0$ we have
\begin{align}
	d_{\R}(x,s_k)=|x-s_k|\geq \varepsilon_x>0. 
\end{align}
Therefore, $x\awf{S}$. Since $S$ contains all points arbitrarily close to $S$, $S$ is closed. 

Finally, suppose $(x_n)$ is a sequence of points in $S$. Since there are infinitely many terms of the sequence $(x_n)$ but only a finite number of real numbers in $S$, at least one of the real numbers in $S$ must be repeated an infinite number of times. Hence, for some index $j_0\in\{1,2,\ldots,n_0\}$ and its corresponding element $s_{j_0}\in S$, there is a constant subsequence $(x_{n_k})$ such that for every index $n_k$ we have $x_{n_k}=s_{j_0}$.
\end{proof}

Each of the properties in Proposition \ref{prop:finitereal} fails to hold in general for infinite subsets of the real line and Euclidean spaces. For an example of a set of real numbers where none of these properties hold, consider the set of rational numbers $\Q$. 

\begin{example}\label{eg:rationalsnotcompact}
The set of rational numbers $\Q$ is unbounded. Moreover, neither $\sup{\Q}$ nor $\inf{\Q}$ exist, so neither $\max{\Q}$ and $\min{\Q}$ exist. $\Q$ is not closed since, for instance, $\sqrt{2}$ is arbitrarily close to $\Q$ but not in $\Q$. Also, some sequences of rational numbers have no constant subsequences. For instance, consider the sequence $(c_n)$ defined by $c_n=n$ for each positive integer $n$: Each $c_n=n$ is a rational number, but each one appears in any given subsequence of $(c_n)$ at most once and cannot be repeated. Therefore, given any subsequence of $(c_n)$, no term is repeated an infinite number of times. Hence, $(c_n)$ has no constant subsequences.
\end{example}

Compactness is a way to impose conditions on a set---whether finite or infinite---which ensures some properties hold which are similar to those of finite sets in Proposition \ref{prop:finitereal}. Still, I feel the formal definition for compactness (Definition \ref{def:compact}) could use some further motivation. To that end, consider closed and bounded intervals.

\begin{proposition}\label{prop:closedboundedintervals}
Suppose $a,b\in\R$ with $a<b$. Then the interval $I=[a,b]$ satisfies the following properties:
\begin{enumerate}
	\item $I$ is bounded.
	\item Both $\max{I}$ and $\min{I}$ exist.
	\item $I$ is closed.
	\item Every sequence of real numbers in $I$ has a convergent subsequence whose limit is in $I$.
\end{enumerate}
\end{proposition}

\begin{remark}\label{rmk:closedboundedintervals}
The only difference between the properties for finite sets of real numbers in Proposition \ref{prop:finitereal} and closed and bounded intervals in Proposition \ref{prop:closedboundedintervals} lies with the fourth property, respectively. 
\end{remark}

\begin{proof}
For the interval $I=[a,b]$, $a$ is a lower bound for $I$ which is in $I$ and $b$ is an upper bound for $I$ which is in $I$. Hence, $I$ is bounded with $\min{I}=a$ and $\max{I}=b$. 

The proof that $I$ is closed is an important exercise, so I won't provide it here. See Exercise \ref{exer:closedinterval}.

This leaves property (iv). Suppose $(x_n)$ is a sequence of real numbers in $I$. Since $I$ is bounded, $(x_n)$ is bounded as well. By the Bolzano-Weierstrass Theorem \ref{thm:bolzanoweierstrass}, $(x_n)$ has a convergent subsequence $(x_{n_k})$ with limit $\ell$. By Theorem \ref{thm:exercise0}, $\ell$ is arbitrarily close to $(x_{n_k})$, so $\ell$ is arbitrarily close to $I$ as well. Since $I$ is closed, $\ell$ is in $I$. 
\end{proof}

\begin{remark}\label{rmk:compactviafinite}
How far can we push the analogy of finiteness to infinite sets? Compactness is one way to answer this question, and it uses collections open sets to imbue a modicum of finiteness on infinite sets. Loosely speaking, compactness replaces the idea of having a finite number of elements in a set with the notion of approximating the set with a finite number of open sets in a peculiar way. 
\end{remark}

Please be patient. The definition of compactness takes a while to build.

\begin{definition}\label{def:opencover}
Let $S$ be a subset of $\R^m$. An {\em open cover} for $S$ is a collection $\mathcal{U}$ where 
\begin{enumerate}
	\item every object in $\mathcal{U}$ is an open set; and
	\item $\displaystyle S \subseteq \bigcup_{U\in\,\mathcal{U}}U$, in which case we say $\mathcal{U}$ {\em covers} $S$.
\end{enumerate}
\end{definition}
 
\begin{example}\label{eg:opencoversornot}
Consider the set of real numbers $F$ given by
\begin{align}
	F&=\{0\}\cup\left\{\frac{1}{n}:n\in\N\right\}.
\end{align}
See Figure \ref{fig:opencoversornot}. Also consider the open sets 
\begin{align}
	V_1&=\left(\frac{3}{4},\frac{5}{4}\right)\qquad \textnormal{and}\qquad 
	V_n=(a_n,b_n)
\end{align}
where $V_n$ is the open interval defined for each $n\geq 2$ by taking $a_n$ to be the midpoint between $1/n$ and $1/(n+1)$ and $b_n$ to be the midpoint between $1/(n-1)$ and $1/n$. From there, consider the collection of open sets $\mathcal{V}$ given by 
\begin{align}
	\mathcal{V}&=\left\{V_n:n\in\N\right\}=\{V_1,V_2,\ldots\}. 
\end{align}
See Figure \ref{fig:opencoversornot}.

We have that $\mathcal{V}$ is {\em not} an open cover for $F$ since $0$ is in $F$ but $0$ is not in any of the $V_n$, thus $\mathcal{V}$ does not cover $F$ since 
\begin{align}
	F\nsubseteq \bigcup_{n=1}^\infty V_n =\bigcup_{U\in\,\mathcal{V}}U.
\end{align}

However, by including one more open set that contains $0$ to the collection, we get an open cover $\mathcal{U}$ for the set $F$. To that end, define
\begin{align}
	V_0&=\left(-\frac{1}{3},\frac{1}{3}\right)\qquad\textnormal{and}\qquad
	\mathcal{U}=\mathcal{V}\cup\{V_0\}=\{V_0,V_1,V_2,\ldots\}.
\end{align}
Then $0\in V_0$ and $1/n\in V_n$ for each $n\in\N$. Therefore,
\begin{align}
	F\subseteq \bigcup_{n=0}^\infty V_n =\bigcup_{U\in\,\mathcal{U}}U.
\end{align}
Since $V_0$ and all of the $V_n$ are open, $\mathcal{U}$ is indeed an open cover for $F$.
\end{example}

\begin{figure}
\centering
\begin{tikzpicture}
\draw (-3,0) node {$F$};
\draw (0.5,0) node {$...$};
\foreach \Point in {(0,0), 
 (1.25,0), (1.67,0), (2.5,0), (5,0)}
{
    \node at \Point {\textbullet};
}
\draw (0,-0.5) node {$0$};
\draw (1.1,-0.5) node {$1/4$};
\draw (1.82,-0.5) node {$1/3$};
\draw (2.65,-0.5) node {$1/2$};
\draw (5,-0.5) node {$1$};

\draw (-3,-1.5) node {$\mathcal{V}$};
\draw (0.5,-1.5) node {$...$};
\foreach \Point in {(0,-1.5), (1.25,-1.5), (1.67,-1.5), (2.5,-1.5), (5,-1.5)}
{
    \node at \Point {\textbullet};
}
\draw[-,semithick] (3.75,-1.5) -- (6.25,-1.5);
\draw (3.78,-1.5) node {$($};
\draw (6.21,-1.5) node {$)$};
\draw (5,-2) node {$V_1$};
\draw[-,semithick] (2.09,-1.5) -- (3.75,-1.5);
\draw (2.13,-1.5) node {$($};
\draw (3.71,-1.5) node {$)$};
\draw (2.92,-2) node {$V_2$};
\draw[-,semithick] (1.46,-1.5) -- (2.09,-1.5);
\draw (1.50,-1.5) node {$($};
\draw (2.07,-1.5) node {$)$};
\draw (1.78,-2) node {$V_3$};
\draw[-,semithick] (1.12,-1.5) -- (1.46,-1.5);
\draw (1.16,-1.5) node {$($};
\draw (1.42,-1.5) node {$)$};
\draw (1.05,-2) node {$...$};

\draw (-3,-3) node {$\mathcal{W}$};
\foreach \Point in {(0,-3), (1.25,-3), (1.67,-3), (2.5,-3), (5,-3)}
{
    \node at \Point {\textbullet};
}
\draw[-,semithick] (3.75,-3) -- (6.25,-3);
\draw (3.78,-3) node {$($};
\draw (6.21,-3) node {$)$};
\draw (5,-3.5) node {$V_1$};
\draw[-,semithick] (2.09,-3) -- (3.75,-3);
\draw (2.13,-3) node {$($};
\draw (3.71,-3) node {$)$};
\draw (2.92,-3.5) node {$V_2$};
\draw[-,semithick] (1.46,-3) -- (2.09,-3);
\draw (1.50,-3) node {$($};
\draw (2.07,-3) node {$)$};
\draw (1.78,-3.5) node {$V_3$};
\draw[-,semithick] (-1.67,-3) -- (1.67,-3);
\draw (-1.66,-3) node {$($};
\draw (1.63,-3) node {$)$};
\draw (0,-3.5) node {$V_0$};
\end{tikzpicture}
\caption{The set of real numbers $F$ along with collections of open sets $\mathcal{V}$ and $\mathcal{W}$ from Example \ref{eg:opencoversornot} and Remark \ref{rmk:opencoversornot}. $\mathcal{V}$ is not a cover for $F$, but $\mathcal{W}$ is.}
\label{fig:opencoversornot}
\end{figure}

\begin{remark}\label{rmk:opencoversornot}
Note that in Example \ref{eg:opencoversornot}, the collection $\mathcal{V}$ comprises an infinite number of open sets but did not have enough for their union to contain $F$, so $\mathcal{V}$ is not an open cover for $F$. By adding the single open set $V_0$, we obtained the collection $\mathcal{U}$ which is an open cover for $F$. 

Furthermore, after adding $V_0$ to the collection, we only needed a finite number of open sets to contain $F$. More specifically, consider the finite collection
\begin{align}
	\mathcal{W}&=\{V_0,V_1,V_2,V_3\}.
\end{align}
See Figure \ref{fig:opencoversornot} once again. Then $1/n\in V_n$ for $n=1,2,3$, $0\in V_0$, and $1/n \in V_0$ for every $n\geq 4$. Hence,
\begin{align}
	F\subseteq \bigcup_{n=0}^3 V_n =\bigcup_{U\in\,\mathcal{W}}U.
\end{align}
Therefore, $\mathcal{W}$ is a finite subcollection of $\mathcal{U}$ whose objects are open sets with a union that still contains $F$. That is, $\mathcal{W}$ is itself an open cover for $F$, so $\mathcal{W}$ is a {\em finite subcover} (see Definition \ref{def:finitesubcover}).
\end{remark}

\begin{definition}\label{def:finitesubcover}
Let $S\subseteq\R^m$ and let $\mathcal{U}$ be an open cover for $S$. A subcollection $\mathcal{W}$ (of $\mathcal{U}$) is a {\em finite subcover}  if there are a finite number of open sets $U_1,\ldots,U_{n_0}$ such that
\begin{enumerate}
	\item $\displaystyle \mathcal{W}=\left\{U_1,\ldots,U_{n_0}\right\}\subseteq \mathcal{U}$, and
	\item $\displaystyle S \subseteq \bigcup_{n=1}^{n_0}U_n=\bigcup_{U\in\,\mathcal{W}}U$. (That is, $\mathcal{W}$ covers $S$.)
\end{enumerate}
\end{definition}

\begin{remark}
In Example \ref{eg:opencoversornot}, the collection $\mathcal{W}$ provides a way for us to represent the infinite collection of points in $F$ with a finite number of open sets:  Every real number in $F$ can be represented by the open set in the collection $\mathcal{W}$ containing the real number. Specifically, $V_n$ represents $1/n$ for each $n=1,2,3$ while $V_0$ simultaneously represents $0$ and all of the $1/n$ where $n\geq 4$. This is an interpretation of the ``next best thing to finiteness'' idea and leads us to the following definition for compactness.
\end{remark}

\begin{definition}\label{def:compact}
A set $K\subseteq\R^m$ is {\em compact} if every open cover for $K$ has a finite subcover.
\end{definition}

\begin{remark}\label{rmk:compactbeyondexample}
The definition of compactness (Definition \ref{def:compact}) goes beyond Example \ref{eg:opencoversornot} in a couple of important ways: (i) $K$ is not necessarily a subset of the real line; and (ii) {\em every} open cover for $K$ has a finite subcover (not just one open cover with a finite subcover). This is not an easy definition to process!
\end{remark}

\begin{example}\label{eg:notcompact}
Consider the set of positive integers $\N$ and the set of reciprocals of positive integers $S$ given by
\begin{align}
	S&=\left\{\frac{1}{n}:n\in\N\right\}.
\end{align}
Neither $\N$ nor $S$ is compact. To prove this, it suffices for each set to find a single open cover with no finite subcover.
\end{example}

\begin{proof}[Proof of Example \ref{eg:notcompact}]
For the set of positive integers $\N$, consider the collection $\mathcal{U}=\{U_1,U_2,\ldots\}$ of open intervals given for each $n$ in $\N$ by
\begin{align}
 U_n&=\left(n-1,n+1\right) \quad\textnormal{for}\quad 
\end{align}
Then for every $n\in\N$, $n$ is in $U_n$ and only $U_n$. Now let $\mathcal{V}$ be any finite subcollection of $\mathcal{U}$. Then there is some positive integer $k_0$ where 
\begin{align}
	\bigcup_{U\in\,\mathcal{V}}U\subseteq\bigcup_{n=1}^{k_0}U_n=(0,k_0+1).
\end{align}
So, $k_0+1\in \N$ but $k_0+1\notin\cup_{U\in\,\mathcal{V}}U$. Therefore, $\mathcal{V}$ is not an open cover for $\N$. Since $\mathcal{V}$ represents an arbitrary finite subcollection of $\mathcal{U}$, we have that $\mathcal{U}$ has no finite subcover. (Try drawing something for $\mathcal{U}$ and various $\mathcal{V}$.)

For the set $S$, consider the collection of open sets $\mathcal{W}=\{W_1,W_2,\ldots\}$ given by
\begin{align}
	W_1=\left(\frac{1}{2},2\right) \quad \textnormal{and}\quad  W_n&=\left(\frac{1}{n+1},\frac{1}{n-1}\right) \quad\textnormal{for}\quad n\geq 2.
\end{align}
Then for every $n\in\N$, $1/n$ is in $W_n$ and only $W_n$. Now let $\mathcal{T}$ be any finite subcollection of $\mathcal{W}$. Then there is some positive integer $j_0$ where 
\begin{align}
	\bigcup_{W\in\,\mathcal{T}}W\subseteq\bigcup_{n=1}^{j_0}W_n=\left(\frac{1}{j_0+1},2\right).
\end{align}
So, $1/(j_0+2)\in S$ but $1/(j_0+2)\notin\cup_{W\in\,\mathcal{T}}W$. Therefore, $\mathcal{T}$ is not an open cover for $S$. Since $\mathcal{T}$ represents an arbitrary finite subcollection of $\mathcal{W}$, we have that $\mathcal{W}$ has no finite subcover. Try drawing something for $\mathcal{W}$ and various $\mathcal{T}$.)
\end{proof}

\begin{remark}\label{rmk:notcompact}
The open covers with no finite subcovers in the proof of Example \ref{eg:notcompact} exploit features of the underlying sets: $\mathcal{U}$ takes advantage of the fact that $\N$ is unbounded while $\mathcal{W}$ takes advantage of the fact that $S$ is not closed since $0$ is arbitrarily close to $S$ but not in $S$.   
\end{remark}

So, what does it mean for a set to be {\em compact}? And how can we check whether or not a given set is compact? Definition \ref{def:compact} is difficult enough to understand, let alone work with. Thankfully, the Heine-Borel Theorem \ref{thm:heineborel} establishes various equivalent forms of compactness and allows us to make use of other, perhaps more manageable perspectives. At the moment, we are not quite ready to prove the Heine-Borel Theorem \ref{thm:heineborel}. For now, the depths of the relationships between the definitions of compactness, closed, and bounded are hinted at in the following lemmas.

\begin{lemma}\label{lem:compactimpliesbounded}
Every compact set in $\R^m$ is bounded.
\end{lemma}

\begin{proof}
To argue via contraposition, suppose $S$ is an unbounded set of points in $\R^m$. Also, consider the collection of open sets $\mathcal{V}=\{V_n(\mathbf{0}):n\in\N\}$ where for each positive integer $n$ we have 
\begin{align}
	V_n(\mathbf{0})=\{\bfx\in\R^m:\|\bfx\|_m<n\}
\end{align}
Then for every point $\bfx$ in $S$, by the Archimedean Property \ref{thm:archimedeanproperty} there is a positive integer $n_\bfx$ large enough so that $\|\bfx\|_m<n_\bfx$. Therefore, $\bfx$ is in $V_{n_\bfx}(\mathbf{0})$ and so $\mathcal{V}$ is an open cover for $S$.

Now suppose $\mathcal{W}$ is a finite subcollection of the open sets in $\mathcal{V}$. Then there is a positive integer $k_0$ where we have
\begin{align}
	\bigcup_{V\in\,\mathcal{W}}V\subseteq\bigcup_{n=1}^{k_0}V_n(\mathbf{0})=V_{k_0}(\mathbf{0})
\end{align}
Since $S$ is unbounded, there is a point $\bfx_0$ in $S$ where $\|\bfx_0\|_m>k_0$. As such, 
\begin{align}
	\bfx_0\notin V_{k_0}(\mathbf{0}) \quad\implies\quad \bfx_0\notin\bigcup_{V\in\,\mathcal{W}}V.
\end{align}
So, $\mathcal{W}$ is not an open cover for $S$. Since $\mathcal{W}$ is an arbitrary finite subcollection of $\mathcal{V}$, we have $S$ is not compact.
\end{proof}

\begin{lemma}\label{lem:compactimpliesclosed}
Every compact set in $\R^m$ is closed. 
\end{lemma}

\begin{proof}
To argue via contraposition, suppose $T$ is a subset of $\R^m$ that is not closed. Then there is a point $\bfy$ in $\R^m$ such that $\bfy\acl{T}$ but $\bfy$ is not in $T$. Now, for each positive integer $n$, consider the open set $\R^m\backslash \overline{V_{1/n}(\bfy)}$  where
\begin{align}
	\R^m\backslash \overline{V_{1/n}(\bfy)}&=\left\{\bfx\in\R^m:d_m(\bfx,\bfy)=\|\bfx-\bfy\|_m>1/n\right\}.
\end{align}
(That is, $\R^m\backslash \overline{V_{1/n}(\bfy)}$ is the complement of the closed $1/n$-neighborhood of $\bfy$, so $\R^m\backslash \overline{V_{1/n}(\bfy)}$ is open by Theorem \ref{thm:closedopen}.) Since $\bfy$ is not in $T$, for every point $\bft$ in $T$ we have 
\begin{align}
	d_m(\bft,\bfy)=\|\bft-\bfy\|_m>0.
\end{align}
So, by the Corollary of the Archimedean Property (Corollary \ref{cor:archimedeanproperty}), there is some positive integer $n_\bft$ large enough so that 
\begin{align}
	\frac{1}{n_\bft}<d_m(\bft,\bfy)=\|\bft-\bfy\|_m.
\end{align}
Thus, $\bft$ is in $\R^m\backslash \overline{V_{1/n_\bft}(\bfy)}$ and therefore 
\begin{align}
	\mathcal{V}=\left\{\R^m\backslash \overline{V_{1/n}(\bfy)}:n\in\N\right\}
\end{align}
is an open cover for $T$. 

Now suppose $\mathcal{W}$ is a finite subcollection of the open sets in $\mathcal{V}$. Then there is a positive integer $k_0$ where we have
\begin{align}
	\bigcup_{V\in\,\mathcal{W}}V\subseteq\bigcup_{n=1}^{k_0}(\R^m\backslash \overline{V_{1/n}(\bfy)})=\R^m\backslash \overline{V_{1/k_0}(\bfy)}.
\end{align}
Since $\bfy\acl{T}$, there is a point $\bft_0$ in $T$ such that 
\begin{align}
	d_m(\bft_0,\bfy)=\|\bft_0-\bfy\|_m<\frac{1}{k_0}.
\end{align}
Therefore,
\begin{align}
	\bft_0\in V_{1/k_0}(\mathbf{0})
	\quad\implies\quad
	\bft_0\notin \R^m\backslash \overline{V_{1/k_0}(\bfy)}
	\quad\implies\quad
	\bft_0\notin\cup_{V\in\,\mathcal{W}}V.
\end{align}
Hence, $\mathcal{W}$ is not an open cover for $T$. Since $\mathcal{W}$ is an arbitrary finite subcollection of $\mathcal{V}$, we have $T$ is not compact.
\end{proof}

\begin{lemma}\label{thm:closedandcompact}
Every closed subset of a compact set in $\R^m$ is compact.
\end{lemma}

\begin{proof}
Suppose $K$ is a compact set in $\R^m$, $F$ is a closed set in $\R^m$, and $F\subseteq K$. By Lemma \ref{lem:compactimpliesclosed}, $K$ is also a closed set in $\R^m$. By Theorem \ref{thm:closedopen}, $\R^m\backslash F$ is an open set.

Now let $\mathcal{U}$ be an open cover for $F$. The collection of open sets $\mathcal{U}\cup\{\R^m\backslash F\}$ is an open cover for $K$ since 
\begin{align}
	F\subseteq \bigcup_{U\in\,\mathcal{U}}U
\end{align}
implies
\begin{align}
	K\subseteq \R^m =F\cup(\R^m\backslash F) \subseteq  \left(\bigcup_{U\in\,\mathcal{U}}U\right)\cup (\R^m\backslash F).
\end{align}
By the definition of compact (Definition \ref{def:compact}), there is a finite subcover $\mathcal{V}=\{U_1,\ldots,U_{n_0}\}$ of $\mathcal{U}\cup\{\R^m\backslash F\}$ where
\begin{align}
	K\subseteq \bigcup_{n=1}^{n_0}U_n.
\end{align}
If $\R^m\backslash F=U_j$ for some $j=1,\ldots,n_0$, then $\mathcal{V}\backslash\{\R^m\backslash F\}$ is a finite subcover of $\mathcal{U}$ which covers $F$. Otherwise, $\mathcal{V}$ is a finite subcover of $\mathcal{U}$ which covers $F$. Either way, since $\mathcal{U}$ is an arbitrary open cover for $F$, we have $F$ is compact.
\end{proof}


Codas provide another context in which we can think about compactness.

\begin{definition}\label{def:codacompact}
A set $F\subseteq\R^m$ is said to be {\em coda-compact} if every sequence of points in $F$ has nonempty coda contained in $F$.
\end{definition}
Thus, $F$ is coda-compact if and only if for every sequence $(\bfx_n)$ of points in $F$ we have $\Coda((\bfx_n))$ is nonempty and a subset of $F$.

The following version of the classic Heine-Borel Theorem lists a variety of equivalent forms of compactness in the context of Euclidean spaces.

\begin{theorem}[Heine-Borel]\label{thm:heineborel}
A set $K\subseteq \R^m$ is compact if and only if the following equivalent conditions hold:
\begin{enumerate}
	\item Every open cover of $K$ has a finite subcover. \textnormal{(Topological Compactness)}
	\item Every sequence in $K$ has a convergent subsequence whose limit is in $K$. \textnormal{(Sequential Compactness)} 
	\item $K$ is closed and bounded. 
	\item $K$ is coda-compact.
\end{enumerate}
\end{theorem}

\vs
\section*{Exercises}
\setcounter{theorem}{0}

Exercises are for play: Do scratch work, draw stuff, and make mistakes---make {\em lots} of mistakes---before worrying about writing proofs. {\em Have fun!}


\vs
\section{Other topological properties}
\label{sec:othertopologicalproperties}

There are lots of ways a point can compare to a set. The following notions are classic concepts in analysis (and topology) which compare points and sets. In this textbook, these concepts are defined through the lenses of {\em arbitrarily close} and {\em away from} (Definitions \ref{def:acl} and \ref{def:awf}), so I am yet again straying from convention. Please keep in mind, I am not making these decisions to buck convention lightly. The theorems, lemmas, exercises, etc., are designed to take the intuition and technical aspects gained from considering {\em arbitrarily close} and {\em away from}, and connect them to the classic---and, sometimes, more awkward---definitions.

\begin{definition}\label{def:pointsaclset}
Let $\bfy\in\R^m$ and $B\subseteq\R^m$. Then:
\begin{enumerate}
	\item $\bfy$ is a {\em boundary point} of $B$ if $\bfy\acl B$ and $\bfy\acl (\R^m\backslash B)$;
	\item $\bfy$ is an {\em interior point} of $B$ if $\bfy\awf (\R^m\backslash B)$;
	\item $\bfy$ is an {\em exterior point} of $B$ if $\bfy\awf{B}$;
	\item $\bfy$ is a {\em accumulation point} of $B$ if $\bfy\acl (B\backslash\{\bfy\})$; $\quad$ and 
	\item $\bfy$ is an {\em isolated point} of $B$ if $\bfy\in B$ but $\bfy\awf (B\backslash\{\bfy\})$.
\end{enumerate}
\end{definition}

See Figures \ref{fig:pointsaclB} and \ref{fig:pointsaclE}.

\begin{remark}\label{rmk:pointaclset}
In other words and stated more loosely, Definition \ref{def:pointsaclset} says:
\begin{enumerate}
	\item $\bfy$ is a boundary point of $B$ if $\bfy$ is near both $B$ and the complement of $B$;
	\item $\bfy$ is an interior point of $B$ if $\bfy$ is away from the complement of $B$;
	\item $\bfy$ is an exterior point of $B$ if $\bfy$ is away from $B$;
	\item $\bfy$ is an accumulation point of $B$ if $\bfy$ is near points in $B$ other than $\bfy$; $\quad$ and 
	\item $\bfy$ is an isolated point of $B$ if $\bfy$ is in $B$ but away from the other points in $B$.
\end{enumerate}
\end{remark}

\begin{figure}
\centering
\begin{tikzpicture} 
\draw[dashed, fill=blue!15] (-3,-1.41) rectangle (1.41,1.41);		
\draw[semithick] (-3,-1.41) -- (-3,1.41);			
\begin{scope}[semithick, dashed, red]  
\draw (1.44,1.42) circle (1.4cm); 
\draw (1.44,1.42) circle (1.05cm);
\draw (1.44,1.42) circle (0.7cm); 
\end{scope}	
	\draw (-2.2,-0.5) node {$\bullet$};
	\draw (-2.2,-0.8) node {$\bfa$};
	\draw[dashed] (-2.2,-0.5) circle (0.65cm);	
	\draw (-2.7,1.08) node {$\bullet$};
	\draw (-2.7,0.6) node {$\bfb$};
	\draw[dashed] (-2.7,1.1) circle (0.23cm);
	\draw[fill=white] (1.41,1.41) circle (0.08cm);
	\draw (1.8,1.41) node {$\bfy$};
	\draw (-0.8,0) node {$R$};
	\draw[fill=black] (-3,1.41) circle (0.08cm);
	\draw[fill=black] (-3,-1.41) circle (0.08cm);
	\draw[fill=white] (1.41,-1.41) circle (0.08cm);
\foreach \Point in {(2,-0.5), (2.58,-0.5), (2.85,-0.5), (3,-0.5),  (3.1,-0.5)}
{
    \node at \Point {\textbullet};
}	
	\draw[fill=white] (4,1.41) circle (0.08cm);
	\draw (4,1.16) node {$\bfw$};
\draw[dashed,blue] (4,1.41) circle (0.45cm); 	
	\draw[fill=white] (4,-0.48) circle (0.08cm);
	\draw (4.4,-0.5) node {$\bfz$};
	\draw (3.6,-0.5) node {$\cdots$};
\draw[dashed,blue] (2,-0.5) circle (0.45cm); 
	\draw (2,-0.75) node {$\bfx_1$};	
	\draw (2.67,-0.8) node {$\bfx_2$};
\draw (1,-2.5) node {$B=R\cup\{\bfx_n:n\in\N\}$};
\end{tikzpicture}
\caption{A set $B$ in the plane $\R^2$ which is the union of a rectangle $R$ and the range of a convergent sequence $(\bfx_n)$. Here we have: $\bfx_1$, $\bfx_2$, $\bfy$, and $\bfz$ are boundary points of $B$; $\bfa$ and $\bfb$ are interior points of $B$; $\bfw$ is an exterior point of $B$; $\bfa$, $\bfb$, $\bfy$, and $\bfz$ are accumulation points of $B$; and $\bfx_1$ and $\bfx_2$ are isolated points of $B$. See Definition \ref{def:pointsaclset}.}
\label{fig:pointsaclB}
\end{figure}

\begin{figure}
\centering
\begin{tikzpicture} 
\draw (-2,0) node {$E$};
	\draw[-,semithick] (0,0) -- (1,0);
	\draw (0,0) node {$[$};
	\draw (0,-0.5) node {$0$};
	\draw (0.98,0) node {$)$};
	\draw (1,-0.5) node {$1$};	
	\draw (2,0) node {\textbullet};
	\draw (2,-0.5) node {$2$};	
\draw (3,0) node {$\circ$};
\draw (2.9,-0.5) node {$3$}; 
\draw (3.25,0) node {...};
\foreach \Point in {(3.5,0), (4,0)}
{
    \node at \Point {\textbullet};
}
\draw (3.5,-0.5) node {$3.5$}; 
\draw (4.1,-0.5) node {$4$}; 
\end{tikzpicture}
\caption{A set $E$ in the real line $\R$ which is the union of an interval and the range of a convergent sequence. See Definition \ref{def:pointsaclset} and Example \ref{eg:pointsaclset}.}
\label{fig:pointsaclE}
\end{figure}

\begin{example}\label{eg:pointsaclset}
To get a better idea of what's going on with Definition \ref{def:pointsaclset} in the real line $\R$, consider the set of real numbers $E$ given by 
\begin{align}
	E=[0,1)\cup\{2\}\cup\{3+(1/n):n\in\N\}
\end{align}	 
See Figure \ref{fig:pointsaclE}. Then:
\begin{enumerate}
	\item $0,1,2,3$, and $4$ are some of the boundary points of $E$;
	\item $1/2$ is one of the interior points of $E$;
	\item Negative real numbers, real numbers greater than $4$, and $2.5$ are some of the exterior points of $E$;
	\item $0,1/2,1$, and $3$ are some of the accumulation points of $E$; and
	\item $2$ and each $3+(1/n)$ are isolated points of $E$.
\end{enumerate}
\end{example}

%

The following proposition shows Definition \ref{def:pointsaclset} to be equivalent to a classic counterpart. The proof is left as an exercise as they follow readily from the definitions of arbitrarily close (Definition \ref{def:acl}) and away from (Definition \ref{def:awf}) as well as the corresponding statements.  See \cite[Chapter 3]{Abbott} for some classic definitions on the real line.

\begin{proposition}\label{prop:pointsaclset}
Let $\bfy\in\R^m$ and $B\subseteq\R^m$. Then:
\begin{enumerate}
	\item $\bfy$ is a boundary point of $B$ if and only if $\bfy$ is in the closure of $B$ but not its interior. See \textnormal{Definitions \ref{def:closureclosed} and \ref{def:setsaclset}}.
	\item $\bfy$ is an interior point of $B$ if and only if a neighborhood of $\bfy$ is contained in $B$. Cf. \cite[Excercise 3.2.14, p.95]{Abbott}.
	\item $\bfy$ is an exterior point of $B$ if and only if a neighborhood of $\bfy$ is contained in $\R^m\backslash B$.
	\item $\bfy$ is an accumulation point of $B$ if and only if every neighborhood of $\bfy$ contains points of $B$ other than $\bfy$. Cf. \cite[Definition 3.2.4, p.89]{Abbott}.
	\item $\bfy$ is an isolated point of $B$ if and only if $\bfy$ is in $B$ and a neighborhood of $\bfy$ has no other points of $B$. Cf. \cite[Definition 3.2.6, p.90]{Abbott}.	
\end{enumerate}
\end{proposition}

\begin{remark}\label{rmk:neighborhoodcontainment}
Take a look at Figure \ref{fig:pointsaclB} again. The interior points $\bfa$ and $\bfb$ of the set $B$ have neighborhoods contained in $B$, while the exterior point $\bfw$ has a neighborhood contained in $\R^m\backslash B$. The existence of neighborhoods around the points in a set which are themselves contained in the set is the characterizing feature of open sets. See Definition \ref{def:open} and the following corollary whose proof is omitted.
\end{remark}

\begin{corollary}\label{cor:openasinteriorpoints}
A set $U\subseteq\R^m$ is open if and only if every point of $U$ is an interior point.
\end{corollary}

In many other analysis texts, accumulation points are also called ``cluster points'' or ``limit points''. By definition, however, not all limits are accumulation points. For example, the limit of a constant sequence is the constant, but the range of a constant sequence has no accumulation points. 

Along similar lines, a point arbitrarily close to a set is not necessarily an accumulation point of the set. Such points are exactly the isolated points of the sets. Still, the concepts of arbitrarily close and accumulation points are deeply related. The foundation of their relationship is illuminated by the following corollary of Theorem \ref{thm:exercise0}. (Cf. \cite[Theorem 3.2.5, p.89]{Abbott}.)

\begin{corollary}\label{cor:accumulationlimit}
A point $\bfx$ is an accumulation point of $A\subseteq\R^m$ if and only if $\bfx=\lim_{n\to\infty}\bfa_n$ for some sequence $(\bfa_n)$ of points in $A$ with $\bfx\neq \bfa_n$ for every $n\in\N$.  
\end{corollary}

\begin{proof}
The statement is a special case of Theorem \ref{thm:exercise0}. Points arbitrarily close to a set are limits of sequences contained in the set.
\end{proof}

A change of perspective with Definition \ref{def:pointsaclset}, which relates points to a given set, leads to the following definition relating a given set to other sets which highlight some facet of the given set. In particular, the {\em boundary} and {\em interior} of a set are more classic topological concepts that readily lends themselves to alternative---and perhaps more intuitive---definitions via arbitrarily close. For comparison, see \cite[Definition 2.13, p.65]{AdamsFranzosa} but also the discussions on page 9 where the phrase ``arbitrarily close'' is used but not formally defined, and page 65 regarding ``points that lie close to both the inside and the outside of the set''. 

\begin{definition}\label{def:setsaclset}
Let $B\subseteq\R^m$.
\begin{itemize}
	\item[(i)] The {\em boundary} of $B$, denoted by $\partial B$, is the set of points arbitrarily close to both $B$ and $\R^m\backslash B$.
	\item[(ii)] The {\em interior} of $B$, denoted by $B^o$, is the set of points away from $\R^m\backslash B$.
	\item[(iii)] The {\em exterior} of $B$, denoted by $B^e$, is the set of points away from $B$.
	\item[(iv)] The set of accumulation points of $B$ is denoted by $B'$.
	\item[(v)] The set of isolated points of $B$ is denoted by $I_B$.	
\end{itemize}
\end{definition}
See Figure \ref{fig:setsaclB} for an example of set $B$ in the plane $\R^2$---which is the union of a solid rectangle $R$ and the range of a convergent sequence $(\bfx_n)$---along with its closure $\overline{B}$, interior $B^o$, boundary $\partial B$, and set of accumulation points $B'$. What do the exterior $B^e$ and set of isolated points $I_B$ look like?

The following corollary is an analog of Remark \ref{rmk:neighborhoodcontainment} which follows immediately from of Proposition \ref{prop:pointsaclset} along with Definitions \ref{def:pointsaclset} and \ref{def:setsaclset}. As such, the proof is omitted.

\begin{corollary}\label{cor:interiorexterior}
The following statements hold regarding any set $B\subseteq\R^m$:
\begin{enumerate}
	\item The interior of $B$, $B^o$, is the set of points in $B$ which have a neighborhood contained in $B$.
	\item The exterior of $B$, $B^e$, is the set of points in $\R^m\backslash B$ which have a neighborhood in $\R^m\backslash B$.
\end{enumerate}
\end{corollary}

\begin{figure}
\centering
\begin{tikzpicture} 
\draw (-5,0) node {$B$};
\draw[dashed, fill=blue!15] (-3,-1.41) rectangle (1.41,1.41);		
\draw[semithick]  (-3,-1.41) -- (-3,1.41);			
	\draw[fill=white] (1.41,1.41) circle (0.08cm);
	\draw (1.8,1.41) node {$\bfy$};
	\draw (-0.8,0) node {$R$};
	\draw[fill=black] (-3,1.41) circle (0.08cm);
	\draw[fill=black] (-3,-1.41) circle (0.08cm);
	\draw[fill=white] (1.41,-1.41) circle (0.08cm);
\foreach \Point in {(2,-0.5), (2.58,-0.5), (2.85,-0.5), (3,-0.5),  (3.1,-0.5)}
{
    \node at \Point {\textbullet};
}	
	\draw[fill=white] (4,-0.48) circle (0.08cm);
	\draw (4.4,-0.5) node {$\bfz$};
	\draw (3.6,-0.5) node {$\cdots$};
	\draw (2,-0.8) node {$\bfx_1$};	
	\draw (2.67,-0.8) node {$\bfx_2$};

\draw (-5,-4) node {$\overline{B}$};
\draw[semithick, fill=blue!15] (-3,-5.41) rectangle (1.41,-2.59);				
	\draw[fill=black] (1.41,-2.59) circle (0.08cm);
	\draw (1.8,-2.59) node {$\bfy$};
	\draw (-0.8,-4) node {$\overline{R}$};
	\draw[fill=black] (-3,-2.59) circle (0.08cm);
	\draw[fill=black] (-3,-5.41) circle (0.08cm);
	\draw[fill=black] (1.41,-5.41) circle (0.08cm);
\foreach \Point in {(2,-4.5), (2.58,-4.5), (2.85,-4.5), (3,-4.5),  (3.1,-4.5)}
{
    \node at \Point {\textbullet};
}	
	\draw[fill=black] (4,-4.48) circle (0.08cm);
	\draw (4.4,-4.5) node {$\bfz$};
	\draw (3.6,-4.5) node {$\cdots$};
	\draw (2,-4.8) node {$\bfx_1$};	
	\draw (2.67,-4.8) node {$\bfx_2$};

\draw (-5,-8) node {$B^o$};
\draw[dashed, fill=blue!15] (-3,-9.41) rectangle (1.41,-6.59);					
	\draw[fill=white] (1.41,-6.59) circle (0.08cm);
	\draw (1.8,-6.59) node {$\bfy$};
	\draw (-0.8,-8) node {$R^o$};
	\draw[fill=white] (-3,-6.59) circle (0.08cm);
	\draw[fill=white] (-3,-9.41) circle (0.08cm);
	\draw[fill=white] (1.41,-9.41) circle (0.08cm);

\draw (-5,-12) node {$\partial B$};
\draw[semithick] (-3,-13.41) rectangle (1.41,-10.59);			
	\draw[fill=black] (1.41,-10.59) circle (0.08cm);
	\draw (1.8,-10.59) node {$\bfy$};
	\draw (-0.8,-12) node {$\partial R$};
	\draw[fill=black] (-3,-10.59) circle (0.08cm);
	\draw[fill=black] (-3,-13.41) circle (0.08cm);
	\draw[fill=black] (1.41,-13.41) circle (0.08cm);
\foreach \Point in {(2,-12.5), (2.58,-12.5), (2.85,-12.5), (3,-12.5),  (3.1,-12.5)}
{
    \node at \Point {\textbullet};
}	
	\draw[fill=black] (4,-12.48) circle (0.08cm);
	\draw (4.4,-12.5) node {$\bfz$};
	\draw (3.6,-12.5) node {$\cdots$};
	\draw (2,-12.8) node {$\bfx_1$};	
	\draw (2.67,-12.8) node {$\bfx_2$};
	
\draw (-5,-16) node {$B'$};
\draw[semithick, fill=blue!15] (-3,-17.41) rectangle (1.41,-14.59);				
	\draw[fill=black] (1.41,-14.59) circle (0.08cm);
	\draw (1.8,-14.59) node {$\bfy$};
	\draw (-0.8,-16) node {$R'$};
	\draw[fill=black] (-3,-14.59) circle (0.08cm);
	\draw[fill=black] (-3,-17.41) circle (0.08cm);
	\draw[fill=black] (1.41,-17.41) circle (0.08cm);
	\draw[fill=black] (4,-16.48) circle (0.08cm);
	\draw (4.4,-16.5) node {$\bfz$};
\end{tikzpicture}
\caption{A set $B$ in the plane $\R^2$ along with its closure $\overline{B}$, interior $B^o$, boundary $\partial B$, and accumulation points $B'$. Compare with Figure \ref{fig:pointsaclB}.}
\label{fig:setsaclB}
\end{figure}

\begin{example}\label{eg:setsaclset}
Let's revisit the set 
\begin{align}
	E=[0,1)\cup\{2\}\cup\{3+(1/n):n\in\N\}\subseteq\R
\end{align}
from Figure \ref{fig:pointsaclE} and Example \ref{eg:pointsaclset}. We have:
\begin{itemize}
	\item[(i)] $\partial E = \{0,1,2,3\}\cup\{3+(1/n):n\in\N\}$;
	\item[(ii)] $E^o=(0,1)$;	
	\item[(iii)] $\displaystyle 
	E^e=(-\infty,0)\cup (1,2) 
	\cup(2,3)
	\cup\left[\bigcup_{n=1}^{\infty}\left(3+\frac{1}{n+1},3+\frac{1}{n}\right)\right] 
	\cup(4,\infty)
	$; 
	\item[(iv)] $E' = [0,1]\cup\{3\}$;$\quad$ and 
	\item[(v)] $I_E = \{2\}\cup\{3+(1/n):n\in\N\}$.
\end{itemize}
\end{example}



As with the connection between Definition \ref{def:pointsaclset} and their classic versions in Proposition \ref{prop:pointsaclset}, the sets defined in Definition \ref{def:setsaclset} have their own classic equivalences presented here in Proposition \ref{prop:setsaclset}. Along these lines, the phrase ``$B^o$ is the set of interior points of $B$'' is the same phrase used in the classic definition for the interior of given set, the difference being the definition used for ``interior point''. 

\begin{proposition}\label{prop:setsaclset}
Let $B\subseteq\R$. Then we have: \textnormal{(i)} $\overline{B} = B\cup B'$; \textnormal{(ii)} $I_B=B\backslash B'$; and \textnormal{(iii)} $\partial B = \overline{B}\backslash B^o$.
\end{proposition}

The section concludes with another definition which uses arbitrarily close to identify a way two sets can relate to one another.

\begin{definition}\label{def:dense}
Given two sets $A,B\subseteq\R^m$, we say $A$ is {\em dense with respect to} $B$ if every point in $B$ is arbitrarily close to $A$. In this case, we have $B\subseteq \overline{A}$. If we also have $A\subseteq B$, then we say $A$ is {\em dense in} $B$.
\end{definition}

\begin{remark}\label{rmk:densityresults}
A nice way to interpret the meaning of density as in Definition \ref{def:dense} is as follows: Thinking of $B$ as a set of points we'd like to approximate and $A$ is dense with respect to $B$, then we can use the points in $A$ as approximations for the points in $B$.

We have already statements involving density. Here's a summary involving rational and irrational real numbers:
\begin{enumerate}
	\item $\Q$ is dense in $\R$: Every real number is approximately a rational number. See Theorem \ref{thm:densityofQinR}.
	\item $\R\backslash\Q$ is dense in $\R$: Every real number is approximately an irrational number. See Corollary \ref{cor:densityofirrationals}.
	\item $\Q$ is dense with respect to $\R\backslash\Q$: Every rational number is approximately irrationals. This follows immediately from (ii).
	\item $\R\backslash\Q$ is dense with respect to $\Q$: Every irrational number is approximately rational. This follows immediately from (i).
\end{enumerate}
\end{remark}

The next section explores the ways in which functions {\em preserve closeness}. For instance, what kind of function will take a point arbitrarily close to a subset of the domain and preserve that closeness in their images? This type of question leads to parallel approaches for continuity and limits of functions.

\vs
\section*{Exercises}
\setcounter{theorem}{0}

Exercises are for play: Do scratch work, draw stuff, and make mistakes---make {\em lots} of mistakes---before worrying about writing proofs. {\em Have fun!}

\xca Prove Proposition \ref{prop:pointsaclset}. 

\xca Prove Corollary \ref{cor:openasinteriorpoints}.

\xca Prove Corollary \ref{cor:interiorexterior}.

\xca Prove Proposition \ref{prop:setsaclset}.

\xca Prove the interior of a given set in $\R^m$ is the largest open set contained in the given set in the following sense: Given a set $S\subseteq\R^m$, every open set which is contained in $S$ is also contained in the interior $S^o$. 

\xca Give examples of nonempty sets with the following properties:
\begin{enumerate}
	\item A set $A$ where both $A^o=\varnothing$ and $A^e=\varnothing$.
	\item A set $B$ where $B'=\overline{B}$.
	\item A set $C$ where $C$ is infinite, $C=I_C$, and $C'=\varnothing$.
	\item A set $D$ which is open and yet $\overline{D}\subseteq D$.
\end{enumerate}

\xca Given any set $B\subseteq\R^m$, prove the trio of sets $B^o$, $B^e$, and $\partial B$ are pairwise disjoint and $B^o\cup B^e\cup \partial B =\R^m$.

\chapter[Continuity]{Continuity}
\label{ch:continuity}

Functions transform points, sequences, and sets in all kinds of interesting ways. Our focus will be on functions that map one Euclidean space to another, often the real line to the real line.

Notions such as continuity, limits, convergence, and codas allow us to explore to the structure and behavior of functions. The first section provides formal definitions for functions and their basic structures before delving into the deeper material.

\vs
\section{Functions and images}
\label{sec:functionsandimages}

The following definitions and breakdowns of functions are probably familiar in some ways. They will allow us to carefully set up and prove classic results from calculus.

\begin{definition}\label{def:function}
A {\em function} $f$ is a relation between two nonempty sets $A$ and $B$ where every element of $A$ is associated with exactly one element of $B$. We write $f:A\to B$ and say $f$ is a {\em function from $A$ to $B$} or $f$ {\em maps $A$ to $B$}. The set $A$ is called the {\em domain} of $f$ and $B$ is the {\em codomain} of $f$.
\end{definition}

Along these lines, the {\em range} of a function is a special subset of the codomain. In general, the range is a subset of the codomain and they are not necessarily the same set.

\begin{definition}\label{def:rangeimage}
When $f:A\to B$, the {\em range} of $f$ is the subset of $B$ whose elements are associated with at least one element in $A$. Given $x\in A$, the {\em image} of $x$, often denoted by $f(x)$, is the element of $B$ associated with $x$ by $f$. Given a subset $S$ of the domain $A$, the {\em image} of $S$, often denoted by $f(S)$, is the subset of $B$ whose elements are associated with at least one element in $S$. In particular, $f(A)$ denotes the range of $f$.
\end{definition}

\begin{notation}\label{not:function}
The term {\em image} is used as a catch-all to refer to the {\em outputs} we get when plugging various objects---{\em inputs}---into a function. The following notation and terminology may be used whenever $f:A\to B$.
\begin{enumerate}
\item \underline{Images of points are points}: For $x\in A$ with $f(x)=y$, $x$ is an {\em input} and $y$ is its {\em output} or {\em image} in the range $f(A)$.
\item \underline{Images of sequencess are sequences}: For a sequence $(x_n)$ in the domain $A$, its {\em output} or {\em image} is the sequence $(f(x_n))$ in the range $f(A)$.
\item \underline{Images of sets are sets}: For a subset $S$ of the domain $A$, its {\em image} $f(S)$ is the subset of the range $f(A)$ given by 
\begin{align}
f(S)&=\{y\in B:f(x)=y \textnormal{ for some } x\in S\}\subseteq f(A).
\end{align}
\end{enumerate}
It is not necessarily the case that every point in the codomain is in the range (i.e., not every point in the codomain is necessarily an output). For instance, when $f:\R\to\R$ is given by $f(x)=x^2$, we have codomain $B=\R$ with range $f(\R)=[0,\infty)$. So, $-1\in B$ but there is no input $x\in A=\R$ where $f(x)=-1$. Hence, $-1\notin f(\R)$.

Also, for functions that map one Euclidean space to another and not just the real line to the real line, boldface variables like $\bfx$ and $\bfy$ are used to represent inputs and outputs.
\end{notation}

Let's play with some examples. Try to keep in mind whatever intuition you may have from calculus about limits and continuity.

\begin{remark}\label{rmk:graphsclarification}
The {\em graphs} of functions provided throughout the text---along with their domains and ranges---are designed to help fuel our discussion and drive us towards a formal definition for {\em continuity} (Definition \ref{def:continuity}). 

To be clear, by {\em graph} of a function from the real line to the real line, I mean a single plot that combines the domain and range, technically plotted in the plane. I believe there is value at looking at all three types figures  for a given function, namely its graph, domain, and range.
\end{remark}
 
Also, I'm assuming everyone reading this is familiar with some notion of continuity from calculus. My plan is to use that familiarity to help us construct the definition of continuity in Definition \ref{def:continuity} in a meaningful way.\footnote{After you get to Definition \ref{def:continuity}, I'd love to hear from you. Did my approach work?} It will take a while to set up.

As a first set of examples, consider the trio of functions $f,g,h:[0,4]\to\R$ with common domain $D=[0,4]$ in Examples \ref{eg:linef}, \ref{eg:splitg}, and \ref{eg:splith}. The real number $c=2$ is in the domain $D$ and its outputs $f(2)=1,g(2)=2$, and $h(2)=2$ play a central role for these particular examples. Let's also see how these functions transform their common domain $D=[0,4]$ into their respective ranges. 
{\em What do these functions look like?}

\begin{figure}
\centering
\begin{tikzpicture}
\draw (-2,2) node {graph};
\draw (-2,1.5) node {of $f$};
	\draw (-0.5,2) node {$2$};
	\draw (0,2) node {$-$};
	\draw (-0.5,1) node {$1$};
	\draw (0,1) node {$\bullet$};
	\draw (-0.5,-0.5) node {$0$};		
  \draw[-To] (-1, 0) -- (4.2, 0) node[right] {$x$};
  \draw[-To] (0, -1) -- (0, 4.2) node[above] {$y$};
  \draw[scale=1, domain=0:4, smooth, variable=\x, blue] plot ({\x}, {\x/2});
  \draw[blue] (4,2.5) node {$f$};
  \draw (2,0) node {$\bullet$};
  \draw (4,0) node {$|$};  
  \draw[blue] (2,1) node {$\bullet$};
  \draw (2,-.5) node {$2$};
  \draw (4,-.5) node {$4$};
\draw (-2,-2) node {$D$};
	\draw[-,semithick] (0,-2) -- (4,-2);
	\draw (0,-2) node {$[$};
	\draw (2,-2) node {$\bullet$};
	\draw (4,-2) node {$]$};
	\draw (0,-2.5) node {$0$};
	\draw (2,-2.5) node {$2$};
	\draw (4,-2.5) node {$4$};
\draw (-2,-3.5) node {$f(D)$};	
	\draw[-,semithick] (0,-3.5) -- (2,-3.5);
	\draw (0,-3.5) node {$[$};
	\draw (1,-3.5) node {$\bullet$};
	\draw (2,-3.5) node {$]$};
	\draw (0,-4) node {$0$};
	\draw (1,-4) node {$1$};
	\draw (2,-4) node {$2$};
\end{tikzpicture}
\caption{The graph of the function $f:[0,4]\to\R$ in Example \ref{eg:linef} which transforms the connected domain $D=[0,4]$ into the connected range $f(D)=[0,2]$.}
\label{fig:linef}
\end{figure}

\begin{example}\label{eg:linef}
Figure \ref{fig:linef} features the graph, domain, and range of the function $f:[0,4]\to\R$ given by
\begin{align}
	f(x)&=x/2.
\end{align}	
Note $f(2)=1$ and the range is $f([0,4])=[0,2]$.

Based on my intuition from calculus\footnote{From way back at the start of my senior year of high school in 1996.}, it looks to me like $f$ is continuous at $c=2$. Actually, it looks like $f$ is continuous at every $c\in D=[0,4]$, but for now my focus is on $c=2$ specifically.
\end{example}

\begin{figure}
\centering
\begin{tikzpicture} 
\draw (-2,2) node {graph};
\draw (-2,1.5) node {of $g$};
	\draw (-0.5,3) node {$3$};
	\draw (0,3) node {$-$};	
	\draw (-0.5,2) node {$2$};
	\draw (0,2) node {$\bullet$};
	\draw (-0.5,1) node {$1$};
	\draw (0,1) node {$-$};
	\draw (-0.5,-0.5) node {$0$};		
  \draw[-To] (-1, 0) -- (4.2, 0) node[right] {$x$};
  \draw[-To] (0, -1) -- (0, 4.2) node[above] {$y$};
  \draw[scale=1, domain=0:2, smooth, variable=\x, blue] plot ({\x}, {\x/2});
   \draw[scale=1, domain=2:4, smooth, variable=\x, blue] plot ({\x}, {1+\x/2});
  \draw[blue] (4,2.5) node {$g$};
  \draw (2,0) node {$\bullet$};
  \draw (4,0) node {$|$};  
  \draw[blue] (2,2) node {$\bullet$};
  \draw[blue, fill=white] (2,1) circle (0.08cm);
  \draw (2,-.5) node {$2$};
  \draw (4,-.5) node {$4$};
\draw (-2,-2) node {$D$};
	\draw[-,semithick] (0,-2) -- (4,-2);
	\draw (0,-2) node {$[$};
	\draw (2,-2) node {$\bullet$};
	\draw (4,-2) node {$]$};
	\draw (0,-2.5) node {$0$};
	\draw (2,-2.5) node {$2$};
	\draw (4,-2.5) node {$4$};
\draw (-2,-3.5) node {$g(D)$};	
	\draw[-,semithick] (0,-3.5) -- (1,-3.5);
	\draw[-,semithick] (2,-3.5) -- (3,-3.5);
	\draw (0,-3.5) node {$[$};
	\draw (1,-3.5) node {$)$};
	\draw (2,-3.5) node {$[$};
	\draw (2,-3.5) node {$\bullet$};
	\draw (3,-3.5) node {$]$};
	\draw (0,-4) node {$0$};
	\draw (1,-4) node {$1$};
	\draw (2,-4) node {$2$};
	\draw (3,-4) node {$3$};
\end{tikzpicture}
\caption{The graph of the function $g:[0,4]\to\R$ in Example \ref{eg:splitg} which transforms the connected domain $D=[0,4]$ into the {\em disconnected} range $g([0,4])=[0,1)\cup[2,3]$.}
\label{fig:splitg}
\end{figure}

\begin{example}\label{eg:splitg}
Figure \ref{fig:splitg} features the graph, domain, and range of the function $g:[0,4]\to\R$ given by
\begin{align}
	g(x)&=
	\begin{cases}
	x/2,& \textnormal{ if } 0\leq x< 2,\\
	1+(x/2),& \textnormal{ if } 2\leq x\leq 4.
	\end{cases}
\end{align}	
Note $g(2)=2$ and the range is $g([0,4])=[0,1)\cup[2,3]$.

My intuition tells me $g$ is discontinuous at $c=2$. The graph of $g$ (in blue) has a definitive break, a jump in the height from $2$ to $3$ at $c=2$.

This break and jump can also be seen by comparing the domain $D$ to its image $g(D)$ (the range). The set $D=[0,4]$ is an interval so it is connected by Theorem \ref{thm:realconnected}. On the other hand, its image $g(D)=[0,1)\cup[2,3]$ is a disconnected union of two intervals with a definitive gap. Continuous functions should not create gaps like this, right? But how can we capture that behavior in our definitions?
\end{example}

\begin{figure}
\centering
\begin{tikzpicture} 
\draw (-2,2) node {graph};
\draw (-2,1.5) node {of $h$};
	\draw (-0.5,3) node {$3$};
	\draw (0,3) node {$-$};	
	\draw (-0.5,2) node {$2$};
	\draw (0,2) node {$\bullet$};
	\draw (-0.5,0.5) node {$1/2$};
	\draw (-0.5,-0.5) node {$0$};		
  \draw[-To] (-1, 0) -- (4.2, 0) node[right] {$x$};
  \draw[-To] (0, -1) -- (0, 4.2) node[above] {$y$};
  \draw[scale=1, domain=0:1.77, smooth, variable=\x, blue] plot ({\x}, {1/(2-\x)});
   \draw[scale=1, domain=2:4, smooth, variable=\x, blue] plot ({\x}, {1+\x/2});
  \draw[blue] (4,2.5) node {$h$};
  \draw (2,0) node {$\bullet$};
  \draw (4,0) node {$|$};  
  \draw[blue] (2,2) node {$\bullet$};
  \draw (2,-.5) node {$2$};
  \draw (4,-.5) node {$4$};
\draw (-2,-2) node {$D$};
	\draw[-,semithick] (0,-2) -- (4,-2);
	\draw (0,-2) node {$[$};
	\draw (2,-2) node {$\bullet$};
	\draw (4,-2) node {$]$};
	\draw (0,-2.5) node {$0$};
	\draw (2,-2.5) node {$2$};
	\draw (4,-2.5) node {$4$};
\draw (-2,-3.5) node {$h(D)$};	
	\draw[-To,semithick] (0.5,-3.5) -- (6,-3.5);
	\draw (0.5,-3.5) node {$[$};
	\draw (2,-3.5) node {$\bullet$};
	\draw (0.5,-4) node {$1/2$};
	\draw (2,-4) node {$2$};
	\draw (6,-4) node {$\infty$};
\end{tikzpicture}
\caption{The graph of the function $h:[0,4]\to\R$ in Example \ref{eg:splith} which transforms the connected domain $D=[0,4]$ into the connected range $h([0,4])=[1/2,\infty)$, despite the break apparent in the graph at $c=2$.}
\label{fig:splith}
\end{figure}

\begin{example}\label{eg:splith}
Figure \ref{fig:splith} features the graph, domain, and range of the function $h:[0,4]\to\R$ given by
\begin{align}
	h(x)&=
	\begin{cases}
	\frac{1}{2-x},& \textnormal{ if } 0\leq x< 2,\\
	1+(x/2),& \textnormal{ if } 2\leq x\leq 4.
	\end{cases}
\end{align}	
Note $h(2)=2$ and the range is $h([0,4])=[1/2,\infty)$.

My intuition is telling me $h$ is discontinuous at $c=2$ since, like $g$, the graph of $h$ (in blue) has a definitive break at $c=2$. To the right of $c$, it looks like the heights of $h$ are ``approaching'' $h(2)=2$, but from the right, it looks like the heights of $h$ are ``tending to infinity''.

Unlike $g$, the break in $h$ is {\em not} readily seen  by comparing the domain $D$ to its image $h(D)$ (the range). In this case, both the set $D=[0,4]$ and its image $h(D)=[1/2,\infty)$ are connected intervals. {\em But $h$ is discontinuous at}  $c=2$, right? The graph of $h$ makes it seem that way, but how can we prove it? We still need a formal definition for continuity, and we need this definition to be able to prove things that intuition tells us. Things like, ``continuous functions have no break or jumps'' and ``continuous functions preserve connectedness''.
\end{example}

Even if the separate plots of the domain $D$ and range $h(D)$ are not giving us enough visual information to indicate $h$ is discontinuous, we can adjust our approach and keep playing around. For instance, we can consider {\em all} subsets and sequences in the domain $D$, then explore how $h$ transforms them in the range. Sequences and how they are transformed by functions are explored in Section \ref{sec:preservingcloseness} and throughout the rest of the book.

Continuing with Examples \ref{eg:linef}, \ref{eg:splitg}, and \ref{eg:splith}, some interesting behavior happens with the outputs when the sets of inputs are near $c=2$. See Figure \ref{fig:varioussets}. Investigating the images of various subsets of the domain will help us understand when they are sure to be arbitrarily close to $f(c)=f(2)$.

\begin{figure}
\centering
\begin{tikzpicture} 
\draw (-2,1.5) node {$D$};
	\draw[-,semithick] (0,1.5) -- (4,1.5);
	\draw (0,1.5) node {$[$};
	\draw (2,1.5) node {$\bullet$};
	\draw (4,1.5) node {$]$};
	\draw (0,1) node {$0$};
	\draw (2,1) node {$2$};
	\draw (4,1) node {$4$};
\draw (-2,0) node {$E_1$};
	\draw[-,semithick] (0,0) -- (2,0);
	\draw (0,0) node {$[$};
	\draw (2,0) node {$\bullet$};
	\draw (2,0) node {$)$};
	\draw (0,-0.5) node {$0$};
	\draw (2,-0.5) node {$2$};
\draw (-2,-1.5) node {$E_2$};	
	\draw[-,semithick] (2,-1.5) -- (4,-1.5);
	\draw (2,-1.5) node {$[$};
	\draw (2,-1.5) node {$\bullet$};
	\draw (4,-1.5) node {$]$};
	\draw (2,-2) node {$2$};
	\draw (4,-2) node {$4$};
\draw (-2,-3) node {$E_3$};	
	\draw[-,semithick] (1,-3) -- (3,-3);
	\draw (1.02,-3) node {$($};
	\draw (2,-3) node {$\bullet$};
	\draw (2.98,-3) node {$)$};
	\draw (1.02,-3.5) node {$1$};
	\draw (2,-3.5) node {$2$};
	\draw (2.98,-3.5) node {$3$};
\draw (-2,-4.5) node {$E_4$};	
	\draw[-,semithick] (1.75,-4.5) -- (2,-4.5);
	\draw (1.77,-4.5) node {$($};
	\draw (2,-4.51) node {$\bullet$};
	\draw (2,-4.5) node {$)$};
	\draw (1.5,-5) node {$7/4$};
	\draw (2.2,-5) node {$2$};
\end{tikzpicture}
\caption{Various subsets of the domain $D=[0,4]$ from Example \ref{eg:imagesofvarioussets}. Each of these sets is connected and arbitrarily close to $c=2$.}
\label{fig:varioussets}
\end{figure}

\begin{example}\label{eg:imagesofvarioussets}
Let's revisit the functions $f,g,$ and $h$ from Examples \ref{eg:linef}, \ref{eg:splitg}, and \ref{eg:splith} by restricting our attention to various subsets of the domain $[0,2]$ which are arbitrarily close to $c=2$ and see what we get with their images. To that end, first draw the following sets---each is arbitrarily close to $c=2$---then determine and draw their images under $f,g,$ and $h$:
\begin{multicols}{2}
\begin{itemize}
	\item[(i)] $E_1=[0,2)$
	\item[(ii)] $E_2=[2,4]$	
	\item[(iii)] $E_3=(1,3)$
	\item[(iv)] $E_4=(7/4,2)$
\end{itemize}
\end{multicols}

See Figure \ref{fig:varioussets}. Recall that for $\delta>0$, the $\delta$-neighborhood of $c$ is
\begin{align}
V_\delta(c)=(c-\delta,c+\delta).
\end{align}  
For instance, $E_3$ is a $\delta$-neighborhood of $c=2$ with $\delta=1$. We have
\begin{align}
V_{1}(2)=(2-1,2+1)=(1,3)=E_3.
\end{align}
\end{example}

\begin{figure}
\centering
\begin{tikzpicture} 
\draw (-2,2) node {graph};
\draw (-2,1.5) node {of $f$};
	\draw (-0.5,2) node {$2$};
	\draw (0,2) node {$-$};
	\draw (-0.5,1) node {$1$};
	\draw (0,1) node {$\bullet$};
	\draw (-0.5,-0.5) node {$0$};		
  \draw[-To] (-0.5,0) -- (4.2,0) node[right] {$x$};
  \draw[-To] (0,-0.5) -- (0,4.2) node[above] {$y$};
  \draw[scale=1, domain=0:4, smooth, variable=\x, blue] plot ({\x}, {\x/2});
  \draw[blue] (4,2.5) node {$f$};
  \draw (2,0) node {$\bullet$};
  \draw (4,0) node {$|$};  
  \draw[blue] (2,1) node {$\bullet$};
  \draw (2,-.5) node {$2$};
  \draw (4,-.5) node {$4$};
\draw[-] (-3,-1) -- (7,-1);

\draw (-2,-1.5) node {$E_1$};
	\draw[-,semithick] (0,-1.5) -- (2,-1.5);
	\draw (0,-1.5) node {$[$};
	\draw (2,-1.5) node {$\bullet$};
	\draw (2,-1.5) node {$)$};
	\draw (0,-2) node {$0$};
	\draw (2,-2) node {$2$};
\draw (-2,-3) node {$f(E_1)$};
	\draw[-,semithick] (0,-3) -- (1,-3);
	\draw (0,-3) node {$[$};
	\draw (1,-3) node {$\bullet$};
	\draw (1,-3) node {$)$};
	\draw (0,-3.5) node {$0$};
	\draw (1,-3.5) node {$1$};	
\draw[-] (-3,-4) -- (7,-4);

\draw (-2,-4.5) node {$E_2$};	
	\draw[-,semithick] (2,-4.5) -- (4,-4.5);
	\draw (2,-4.5) node {$[$};
	\draw (2,-4.5) node {$\bullet$};
	\draw (4,-4.5) node {$]$};
	\draw (2,-5) node {$2$};
	\draw (4,-5) node {$4$};
\draw (-2,-6) node {$f(E_2)$};	
	\draw[-,semithick] (1,-6) -- (2,-6);
	\draw (1,-6) node {$[$};
	\draw (1,-6) node {$\bullet$};
	\draw (2,-6) node {$]$};
	\draw (1,-6.5) node {$1$};
	\draw (2,-6.5) node {$2$};
\draw[-] (-3,-7) -- (7,-7);

\draw (-2,-7.5) node {$E_3$};	
	\draw[-,semithick] (1,-7.5) -- (3,-7.5);
	\draw (1.02,-7.5) node {$($};
	\draw (2,-7.5) node {$\bullet$};
	\draw (2.98,-7.5) node {$)$};
	\draw (1.02,-8) node {$1$};
	\draw (2,-8) node {$2$};
	\draw (3,-8) node {$3$};	
\draw (-2,-9) node {$f(E_3)$};	
	\draw[-,semithick] (0.5,-9) -- (1.5,-9);
	\draw (0.52,-9) node {$($};
	\draw (1,-9) node {$\bullet$};
	\draw (1.48,-9) node {$)$};
	\draw (0.3,-9.5) node {$1/2$};
	\draw (1,-9.5) node {$1$};
	\draw (1.7,-9.5) node {$3/2$};
\draw[-] (-3,-10) -- (7,-10);	
	
\draw (-2,-10.5) node {$E_4$};	
	\draw[-,semithick] (1.75,-10.5) -- (2,-10.5);
	\draw (1.77,-10.5) node {$($};
	\draw (2,-10.5) node {$\bullet$};
	\draw (2,-10.5) node {$)$};
	\draw (1.45,-11) node {$7/4$};
	\draw (2.1,-11) node {$2$};
\draw (-2,-12) node {$f(E_4)$};	
	\draw[-,semithick] (0.875,-12) -- (1,-12);
	\draw (0.895,-12) node {$($};
	\draw (1,-12) node {$\bullet$};
	\draw (1,-12) node {$)$};
	\draw (0.5,-12.5) node {$7/8$};
	\draw (1.1,-12.5) node {$1$};
\draw[-] (-3,-13) -- (7,-13);	
\end{tikzpicture}
\caption{Images under $f$ of the subsets of $D=[0,4]$ from Example \ref{eg:imagesofvarioussets}. Each image is connected and arbitrarily close to $f(2)=1$. See Example \ref{eg:linefmultiple}.}
\label{fig:linefmultiple}
\end{figure}

\begin{example}\label{eg:linefmultiple}
For the function $f$ we have the image $f(2)=1$ along with 
\begin{enumerate}
	\item $f(E_1)=f([0,2))=[0,1)$,
	\item $f(E_2)=f([2,4])=[1,2]$,	
	\item $f(E_3)=f((1,3))=(1/2,3/2)$, and 
	\item $f(E_4)=f((7/4,2))=(7/8,1)$.
\end{enumerate}
Figure \ref{fig:linefmultiple} features plots of these images under $f$ along with $f(2)=1$.

So, in the domain, the point $c=2$ is arbitrarily close to each of the $E_k$. The function $f$ maps the point $2$ and the sets $E_k$ to a point $f(2)=1$ and sets $f(E_k)$ in the range in a way that {\em preserves closeness}. That is, at least for $k=1,2,3,4$, we have $f$ ensures
\begin{align}\label{eqn:firstpreservecloseness}
2\acl{E_k} \quad\Longrightarrow\quad f(2)\acl{f(E_k)}.
\end{align}
\end{example}

What about $g$ and $h$?

\begin{figure}
\centering
\begin{tikzpicture} 
\draw (-2,2) node {graph};
\draw (-2,1.5) node {of $g$};
	\draw (-0.5,3) node {$3$};
	\draw (0,3) node {$-$};	
	\draw (-0.5,2) node {$2$};
	\draw (0,2) node {$\bullet$};
	\draw (-0.5,1) node {$1$};
	\draw (0,1) node {$-$};
	\draw (-0.5,-0.5) node {$0$};		
  \draw[-To] (-0.5,0) -- (4.2,0) node[right] {$x$};
  \draw[-To] (0,-0.5) -- (0,4.2) node[above] {$y$};
  \draw[scale=1, domain=0:2, smooth, variable=\x, blue] plot ({\x}, {\x/2});
   \draw[scale=1, domain=2:4, smooth, variable=\x, blue] plot ({\x}, {1+\x/2});
  \draw[blue] (4,2.5) node {$g$};
  \draw (2,0) node {$\bullet$};
  \draw (4,0) node {$|$};  
  \draw[blue] (2,2) node {$\bullet$};
  \draw[blue, fill=white] (2,1) circle (0.08cm);  
  \draw (2,-.5) node {$2$};
  \draw (4,-.5) node {$4$};
\draw[-] (-3,-1) -- (7,-1);  

\draw (-2,-1.5) node {$E_1$};
	\draw[-,semithick] (0,-1.5) -- (2,-1.5);
	\draw (0,-1.5) node {$[$};
	\draw (2,-1.5) node {$\bullet$};
	\draw (2,-1.5) node {$)$};
	\draw (0,-2) node {$0$};
	\draw (2,-2) node {$2$};
\draw (-2,-3) node {$g(E_1)$};
	\draw[-,semithick] (0,-3) -- (1,-3);
	\draw (0,-3) node {$[$};
	\draw (1,-3) node {$)$};
	\draw (2,-3) node {$\bullet$};	
	\draw (0,-3.5) node {$0$};
	\draw (1,-3.5) node {$1$};		
	\draw (2,-3.5) node {$2$};	
\draw[-] (-3,-4) -- (7,-4);

\draw (-2,-4.5) node {$E_2$};	
	\draw[-,semithick] (2,-4.5) -- (4,-4.5);
	\draw (2,-4.5) node {$[$};
	\draw (2,-4.5) node {$\bullet$};
	\draw (4,-4.5) node {$]$};
	\draw (2,-5) node {$2$};
	\draw (4,-5) node {$4$};
\draw (-2,-6) node {$g(E_2)$};	
	\draw[-,semithick] (2,-6) -- (3,-6);
	\draw (2,-6) node {$[$};
	\draw (2,-6) node {$\bullet$};
	\draw (3,-6) node {$]$};
	\draw (2,-6.5) node {$2$};
	\draw (3,-6.5) node {$3$};
\draw[-] (-3,-7) -- (7,-7);

\draw (-2,-7.5) node {$E_3$};	
	\draw[-,semithick] (1,-7.5) -- (3,-7.5);
	\draw (1.02,-7.5) node {$($};
	\draw (2,-7.5) node {$\bullet$};
	\draw (2.98,-7.5) node {$)$};
	\draw (1.02,-8) node {$1$};
	\draw (2,-8) node {$2$};
	\draw (3,-8) node {$3$};	
\draw (-2,-9) node {$g(E_3)$};	
	\draw[-,semithick] (0.5,-9) -- (1,-9);
	\draw (0.52,-9) node {$($};
	\draw (0.98,-9) node {$)$};
	\draw[-,semithick] (2,-9) -- (2.5,-9);	
	\draw (2,-9) node {$\bullet$};
	\draw (2,-9) node {$[$};
	\draw (2.48,-9) node {$)$};
	\draw (0.32,-9.5) node {$1/2$};
	\draw (1.02,-9.5) node {$1$};
	\draw (2,-9.5) node {$2$};
	\draw (2.72,-9.5) node {$5/2$};
\draw[-] (-3,-10) -- (7,-10);	
	
\draw (-2,-10.5) node {$E_4$};	
	\draw[-,semithick] (1.75,-10.5) -- (2,-10.5);
	\draw (1.77,-10.5) node {$($};
	\draw (2,-10.5) node {$\bullet$};
	\draw (2,-10.5) node {$)$};
	\draw (1.45,-11) node {$7/4$};
	\draw (2.1,-11) node {$2$};
\draw (-2,-12) node {$g(E_4)$};	
	\draw[-,semithick] (0.875,-12) -- (1,-12);
	\draw (0.895,-12) node {$($};
	\draw (2,-12) node {$\bullet$};
	\draw (1,-12) node {$)$};
	\draw (0.5,-12.5) node {$7/8$};
	\draw (1.1,-12.5) node {$1$};
	\draw (2,-12.5) node {$2$};
\draw[-] (-3,-13) -- (7,-13);	
\end{tikzpicture}
\caption{Images under $g$ of the subsets of $D=[0,4]$ from Example \ref{eg:imagesofvarioussets}. Two of the images are arbitrarily close to $g(2)=2$, the other two are not. See Example \ref{eg:splitgmultiple}.}
\label{fig:splitgmultiple}
\end{figure}

\begin{example}\label{eg:splitgmultiple}
For the function $g$ we have
\begin{enumerate}
	\item $g(E_1)=g([0,2))=[0,1)$,
	\item $g(E_2)=g([2,4])=[2,3]$,	
	\item $g(E_3)=g((1,3))=(1/2,1)\cup [2,5/2)$, and
	\item $g(E_4)=g((7/4,2))=(7/8,1)$.
\end{enumerate}
See Figure \ref{fig:splitgmultiple} for plots of these images under $g$ along with $g(2)=2$.

What do you notice? Let's investigate the images of the $E_k$ one at a time. Even though $2$ is arbitrarily close to $E_1$ in the domain, in the range $g(2)=2$ is away from $g(E_1)=[0,1)$. That is, 
\begin{align}\label{eqn:splitgE1}
2\acl{E_1} \quad\textnormal{but}\quad g(2)\awf{g(E_1)}.
\end{align}
So, $g$ does not preserve the closeness of $2$ and $E_1$ like $f$ does.

On other hand, $g(2)=2$ is in both $g(E_2)$ and $g(E_3)$, so $g(2)$ is arbitrarily close to both of these images by Lemma \ref{lem:elementacl}. However, in the case of $g(E_3)$, $g$ did not preserve the connectedness of the interval $E_3=(1,3)$ since
$g(E_3)=(1/2,1)\cup [2,5/2)$ is disconnected.

The case for $E_4$ is like $E_1$. We have
\begin{align}\label{eqn:splitgE4}
2\acl{E_4} \quad\textnormal{but}\quad g(2)\awf{g(E_4)}.
\end{align}

We can prove $g(2)\awf{g(E_1)}$ and $g(2)\awf{g(E_4)}$ simultaneously using the definition of {\em away from} (Definition \ref{def:awf}).
\begin{proof}[Proof that $g(2)\awf{g(E_1)}$ and $g(2)\awf{g(E_4)}$]
Every output $g(x)$ in either $g(E_1)$ or $g(E_4)$ satisfies $g(x)<1\leq 2$. Hence, when $x\in E_1$ or $x\in E_4$ we have
\begin{align}
|g(x)-g(2)|&=|g(x)-2|=2-g(x)>1. 
\end{align}
Hence, $g(2)\awf{g(E_1)}$ and $g(2)\awf{g(E_4)}$.
\end{proof}
\end{example}

Now what about $h$?

\begin{figure}
\centering
\begin{tikzpicture} 
\draw (-2,2) node {graph};
\draw (-2,1.5) node {of $h$};
	\draw (-0.5,3) node {$3$};
	\draw (0,3) node {$-$};	
	\draw (-0.5,2) node {$2$};
	\draw (0,2) node {$\bullet$};
	\draw (-0.5,0.5) node {$1/2$};
	\draw (-0.5,-0.5) node {$0$};		
  \draw[-To] (-0.5,0) -- (4.2,0) node[right] {$x$};
  \draw[-To] (0,-0.5) -- (0,4.2) node[above] {$y$};
  \draw[scale=1, domain=0:1.77, smooth, variable=\x, blue] plot ({\x}, {1/(2-\x)});
   \draw[scale=1, domain=2:4, smooth, variable=\x, blue] plot ({\x}, {1+\x/2});
  \draw[blue] (4,2.5) node {$h$};
  \draw (2,0) node {$\bullet$};
  \draw (4,0) node {$|$};  
  \draw[blue] (2,2) node {$\bullet$};
  \draw (2,-.5) node {$2$};
  \draw (4,-.5) node {$4$};
\draw[-] (-3,-1) -- (7,-1);  

\draw (-2,-1.5) node {$E_1$};
	\draw[-,semithick] (0,-1.5) -- (2,-1.5);
	\draw (0,-1.5) node {$[$};
	\draw (2,-1.5) node {$\bullet$};
	\draw (2,-1.5) node {$)$};
	\draw (0,-2) node {$0$};
	\draw (2,-2) node {$2$};
\draw (-2,-3) node {$h(E_1)$};
	\draw[-To,semithick] (0.5,-3) -- (6,-3);
	\draw (0.5,-3) node {$[$};
	\draw (2,-3) node {$\bullet$};
	\draw (0.5,-3.5) node {$1/2$};
	\draw (2,-3.5) node {$2$};
	\draw (6,-3.5) node {$\infty$};	
\draw[-] (-3,-4) -- (7,-4);

\draw (-2,-4.5) node {$E_2$};	
	\draw[-,semithick] (2,-4.5) -- (4,-4.5);
	\draw (2,-4.5) node {$[$};
	\draw (2,-4.5) node {$\bullet$};
	\draw (4,-4.5) node {$]$};
	\draw (2,-5) node {$2$};
	\draw (4,-5) node {$4$};
\draw (-2,-6) node {$h(E_2)$};	
	\draw[-,semithick] (2,-6) -- (3,-6);
	\draw (2,-6) node {$[$};
	\draw (2,-6) node {$\bullet$};
	\draw (3,-6) node {$]$};
	\draw (2,-6.5) node {$2$};	
	\draw (3,-6.5) node {$3$};
\draw[-] (-3,-7) -- (7,-7);

\draw (-2,-7.5) node {$E_3$};	
	\draw[-,semithick] (1,-7.5) -- (3,-7.5);
	\draw (1.02,-7.5) node {$($};
	\draw (2,-7.5) node {$\bullet$};
	\draw (2.98,-7.5) node {$)$};
	\draw (1.02,-8) node {$1$};
	\draw (2,-8) node {$2$};
	\draw (3,-8) node {$3$};	
\draw (-2,-9) node {$h(E_3)$};	
	\draw[-To,semithick] (1,-9) -- (6,-9);
	\draw (1.02,-9) node {$($};
	\draw (2,-9) node {$\bullet$};
	\draw (0.98,-9.5) node {$1$};
	\draw (2,-9.5) node {$2$};
	\draw (6,-9.5) node {$\infty$};
\draw[-] (-3,-10) -- (7,-10);	
	
\draw (-2,-10.5) node {$E_4$};	
	\draw[-,semithick] (1.75,-10.5) -- (2,-10.5);
	\draw (1.77,-10.5) node {$($};
	\draw (2,-10.5) node {$\bullet$};
	\draw (2,-10.5) node {$)$};
	\draw (1.45,-11) node {$7/4$};
	\draw (2.1,-11) node {$2$};
\draw (-2,-12) node {$h(E_4)$};	
	\draw[-To,semithick] (4,-12) -- (6,-12);
	\draw (2,-12) node {$\bullet$};
	\draw (4.02,-12) node {$($};	
	\draw (2,-12.5) node {$2$};
	\draw (4,-12.5) node {$4$};
	\draw (6,-12.5) node {$\infty$};
\draw[-] (-3,-13) -- (7,-13);	
\end{tikzpicture}
\caption{Images under $h$ of the subsets of $D=[0,4]$ from Example \ref{eg:imagesofvarioussets}. Of the images, only $h(E_4)=(4,\infty)$ is away from $h(2)=2$. See Example \ref{eg:splithmultiple}.}
\label{fig:splithmultiple}
\end{figure}

\begin{example}\label{eg:splithmultiple}
Finally, for the function $h$ we have
\begin{enumerate}
	\item $h(E_1)=h([0,2))=[1/2,\infty)$,
	\item $h(E_2)=h([2,4])=[2,3]$,
	\item $h(E_3)=h((1,3))=(1,\infty)$, and 
	\item $h(E_4)=h((7/4,2))=(4,\infty)$.
\end{enumerate}
See Figure \ref{fig:splithmultiple} for plots of these images under $h$ along with $h(2)=2$.

What do you notice this time? It looks like $h$ might preserves the connectedness of all of the $E_k$ intervals since all of the images $h(E_k)$ are intervals, too. (See Theorem \ref{thm:realconnected}.)

But the images $h(E_4)=(4,\infty)$ and $h(2)=2$ (the $\bullet$) in Figure  \ref{fig:splithmultiple} show us $h$ does not preserve the closeness between $c=2$ and $E_4=(7/2,2)$. Thus,
\begin{align}\label{eqn:splithE4}
2\acl{E_4} \quad\textnormal{but}\quad h(2)\awf{h(E_4)}.
\end{align}

Once again, we can prove $h(2)\awf{h(E_4)}$ using the definition of {\em away from} (Definition \ref{def:acl}).
\begin{proof}[Proof that $h(2)\awf{h(E_4)}$]
Every output $h(x)$ in $h(E_4)$ satisfies $h(x)>4>2$. Hence, when $x\in E_4$ we have
\begin{align}
|h(x)-h(2)|&=|h(x)-2|= h(x)-2> 4-2=2.
\end{align}
Therefore, $h(2)\awf{g(E_4)}$.
\end{proof}
\end{example}

So, what are properties functions can have which we can define and work with in order to prove connected sets are mapped to connected sets? Do these properties align with the intuition we gained about continuous functions from calculus? A notion of preserving closeness in terms of arbitrarily close  provides such a property and is the focus of Section \ref{sec:preservingcloseness}.

To recap, in Figure \ref{fig:splithmultiple}, the graph of $h$ captures---to me, anyway---the notion that $h$ does not preserve closeness at $c=2$. However, the figures with separate domains and ranges do not reveal this feature until we consider the relationship between the smaller interval $E_4=(7/4,2)$ and $c=2$ as well as the relationship between their images $h(E_4)=(4,\infty)$ and $h(2)=2$.

If the graphs seem to convey the preservation of connectedness more readily than the separate domains and ranges, why bother with domains and ranges? Keep in mind, a goal is to produce mathematical definitions we can use to prove facts that solidify the intuition coming from graphs and our experiences with calculus. The graphs show the behavior we want to explore while the domains and ranges are leading us to classic definitions and rigorous mathematics. 

Also, graphs are not as easy to draw when our functions transforms points and images in higher dimensions! For instance, the graph of a function from $\R^3$ into $\R^3$ would be 6-dimensional and I have no idea how to plot a 6-dimensional image on 2-dimensional page. But in some cases have managed to create pairs of 2-dimensional figures, one for its domain and another for its range. This is done in Section \ref{sec:segueintolinearalgebra} where higher-dimensional functions are explored in the context of linear algebra.

\vs
\section*{Exercises}
\setcounter{theorem}{0}

Exercises are for play: Do scratch work, draw stuff, and make mistakes---make {\em lots} of mistakes---before worrying about writing proofs. {\em Have fun!}
%
%


\vs
\section{Preserving closeness}
\label{sec:preservingcloseness}

To motivate the definition of {\em preserving closeness} in Definition \ref{def:preservecloseness}, let's recap Section \ref{sec:functionsandimages}: The functions $g$ and $h$ defined in Examples \ref{eg:splitg} and \ref{eg:splith} each take the point $c=2$ and the set $E_4=(7/4,2)$---which are arbitrarily close in the domain---and map them to images that are away from each other in the range.

On the other hand, function $f$ defined in Examples \ref{eg:linef} seems to map points and sets arbitrarily close in the domain to images  arbitrarily close in the range. This has not been proven yet, but the speculation motivates the following definition for the {\em preservation of closeness}.

\begin{definition}\label{def:preservecloseness}
Let $D\subseteq\R^k$, $\bfc\in D$, and $f:D\to\R^m$. We say $f$ \textit{preserves closeness at} $\bfc$ if for every subset $E$ of the domain $D$ we have
\begin{align}
	\bfc \acl E &\quad\Longrightarrow\quad f(\bfc)\acl f(E). 
\end{align}
If $f$ preserves closeness at every point in its domain $D$, we say $f$ \textit{preserves closeness}.
\end{definition}

So, whether $E$ is an interval, the range of a sequence, or something else, if $f$ preserves closeness at $\bfc$, then  anytime $E$ is arbitrarily close to $\bfc$ in the domain we also have $f(E)$ is arbitrarily close to $f(\bfc)$ in the range. 

\begin{example}\label{eg:notpreservingcloseness}
The functions $g$ and $h$ from Examples \ref{eg:splitg} and \ref{eg:splith} do not preserve closeness at $c=2$. Recall that $E_4=(7/4,2)$. We have:
\begin{enumerate}
	\item $2 \acl{E_4}$ but $g(2)\awf{g(E_4)}$; and
	\item $2 \acl{E_4}$ but $h(2)\awf{h(E_4)}$.
\end{enumerate}

\end{example}

What about the case when we believe a function $f$ preserves closeness? It may be more difficult since we would need to prove  for {\em any} point $\bfc$ in the domain and {\em any} subset $E$ of the domain where $\bfc\acl{E}$, we end up with $f(\bfc)\acl{f(E)}$. 

Honestly, I do not have a nice way to do this by  making direct use of Definition \ref{def:preservecloseness}. However, I have multiple ways to do it indirectly, including some classic approaches to {\em continuity}. The first makes use of sequences in both the domain and range, thanks to the fundamental relationship between arbitrarily close and limits of sequences in Theorem \ref{thm:exercise0}. This method also connects the notion of $f$ preserving closeness at $\bfc$ to the informal phrasing you may have  heard in a calculus class: $\bfx$ ``approaches'' $\bfc$ implies 
$f(\bfx)$ ``approaches'' $f(\bfc)$. 

So, how does all this work?

\begin{example}\label{eg:threesequences}
Consider the following trio of sequences $(x_n),(y_n)$, and $(z_n)$ defined for each positive integer $n$ by 
\begin{align}
	x_n&=2+\frac{2}{n},\quad
	y_n=2-\frac{2}{n},\quad\textnormal{and}\quad
	z_n=2+\frac{2(-1)^n}{n}.
\end{align}
See Figure \ref{fig:threesequences}. Note that all of these sequences are contained domain $D=[0,4]$ from Examples \ref{eg:linef}, \ref{eg:splitg}, and \ref{eg:splith}.

\begin{figure}
\centering
\begin{tikzpicture} 
\draw (-2,0) node {$(x_n)$};
\draw (2,0) node {$\circ$};
\draw (2.24,0) node {$...$};
\draw (4,-0.5) node {$4$};
\draw (3,-0.5) node {$3$};
\draw (2,-0.5) node {$2$};
\foreach \Point in {(4,0), (3,0), (2.67,0), (2.5,0)}
{
    \node at \Point {\textbullet};
}
\draw (-2,-1.5) node {$(y_n)$};
\draw (2,-1.5) node {$\circ$};
\draw (1.76,-1.5) node {...};
\draw (2,-2) node {$2$};
\draw (1,-2) node {$1$};
\draw (0,-2) node {$0$};
\foreach \Point in {(0,-1.5), (1,-1.5), (1.33,-1.5), (1.5,-1.5)}
{
    \node at \Point {\textbullet};
}
\draw (-2,-3) node {$(z_n)$};
\draw (2,-3) node {$\circ$};
\draw (1.76,-3) node {...};
\draw (2.24,-3) node {...};
\draw (0,-3.5) node {$0$};
\draw (3,-3.5) node {$3$};
\draw (1.15,-3.52) node {$4/3$};
\draw (2,-3.5) node  {$2$};
\foreach \Point in {(0,-3), (3,-3), (1.33,-3), (2.5, -3)}
{
    \node at \Point {\textbullet};
}
\end{tikzpicture}
\caption{The sequences $(x_n),(y_n)$, and $(z_n)$ from Example \ref{eg:threesequences} are contained in the set $D=[0,4]$ and each converges to 2.}
\label{fig:threesequences}
\end{figure}
\end{example}

Each of the sequences $(x_n),(y_n)$, and $(z_n)$ converges to 2. I'll prove the case for $(z_n)$ but skip the scratch work for the threshold $n_\varepsilon$.

\begin{proof}[Proof that $\lim_{n\to\infty}z_n=2$]
Let $\varepsilon>0$. By the Archimedean Property \ref{thm:archimedeanproperty}, we can choose $n_\varepsilon\in\N$ large enough to ensure
\begin{align}
	n_\varepsilon>\frac{2}{\varepsilon} \qquad\textnormal{and therefore}\qquad \frac{2}{n_\varepsilon}<\varepsilon.
\end{align}
So, $n_\varepsilon$ is a threshold since for all $n\geq n_\varepsilon$ we have
\begin{align}
	|z_n-2|&=\left|2+\frac{2(-1)^n}{n}-2 \right|=\left|\frac{2(-1)^n}{n} \right|=\frac{2}{n}\leq \frac{2}{n_\varepsilon}<\varepsilon.
\end{align}
Hence, $\lim_{n\to\infty}z_n=2$.
\end{proof}

Proofs showing $\lim_{n\to\infty}x_n=2$ and $\lim_{n\to\infty}y_n=2$ are nearly identical to the proof of $\lim_{n\to\infty}z_n=2$. In fact, we can use the same choice for a threshold in each case, namely $n_\varepsilon\in\N$ large enough to ensure $n_\varepsilon>2/\varepsilon$. {\em Try it yourself!} My guess is your scratch work will lead you to the same choice for $n_\varepsilon$. If I'm wrong, please tell me what you did instead! I'm interested in seeing what you come up with.

Let's see how the functions $f,g$, and $h$ from Examples \ref{eg:linef}, \ref{eg:splitg}, and \ref{eg:splith} transform the sequences  
$(x_n),(y_n)$, and $(z_n)$ from Example \ref{eg:threesequences}. 

\begin{figure}
\centering
\begin{tikzpicture} 
\draw (-2,2) node {graph};
\draw (-2,1.5) node {of $f$};
	\draw (-0.5,2) node {$2$};
	\draw (0,2) node {$-$};
	\draw (-0.5,1) node {$1$};
	\draw (0,1) node {$\bullet$};
	\draw (-0.5,-0.5) node {$0$};		
  \draw[-To] (-0.5,0) -- (4.2,0) node[right] {$x$};
  \draw[-To] (0,-0.5) -- (0,4.2) node[above] {$y$};
  \draw[scale=1, domain=0:4, smooth, variable=\x, blue] plot ({\x}, {\x/2});
  \draw[blue] (4,2.5) node {$f$};
  \draw (2,0) node {$\bullet$};
  \draw (4,0) node {$|$};  
  \draw[blue] (2,1) node {$\bullet$};
  \draw (2,-.5) node {$2$};
  \draw (4,-.5) node {$4$};
\draw[-] (-3,-1) -- (7,-1);

\draw (-2,-1.5) node {$(x_n)$};
\draw (2,-1.5) node {$\circ$};
\draw (2.24,-1.5) node {$...$};
\draw (4,-2) node {$4$};
\draw (3,-2) node {$3$};
\draw (2,-2) node {$2$};
\foreach \Point in {(4,-1.5), (3,-1.5), (2.67,-1.5), (2.5,-1.5)}
{
    \node at \Point {\textbullet};
}
\draw (-2,-3) node {$f(x_n)$};
\draw (1,-3) node {$\circ$};
\draw (1.24,-3) node {$...$};
\draw (2,-3.5) node {$2$};
\draw (1,-3.5) node {$1$};
\foreach \Point in {(2,-3), (1.5,-3)}
{
    \node at \Point {\textbullet};
}	
\draw[-] (-3,-4) -- (7,-4);

\draw (-2,-4.5) node {$(y_n)$};
\draw (2,-4.5) node {$\circ$};
\draw (1.76,-4.5) node {...};
\draw (2,-5) node {$2$};
\draw (1,-5) node {$1$};
\draw (0,-5) node {$0$};
\foreach \Point in {(0,-4.5), (1,-4.5), (1.33,-4.5), (1.5,-4.5)}
{
    \node at \Point {\textbullet};
}
\draw (-2,-6) node {$f(y_n)$};
\draw (1,-6) node {$\circ$};
\draw (0.76,-6) node {...};
\draw (1,-6.5) node {$1$};
\draw (0,-6.5) node {$0$};
\foreach \Point in {(0,-6), (0.5,-6)}
{
    \node at \Point {\textbullet};
}
\draw[-] (-3,-7) -- (7,-7);

\draw (-2,-7.5) node {$(z_n)$};
\draw (2,-7.5) node {$\circ$};
\draw (1.76,-7.5) node {...};
\draw (2.24,-7.5) node {...};
\draw (0,-8) node {$0$};
\draw (3,-8) node {$3$};
\draw (1.15,-8.02) node {$4/3$};
\draw (2,-8) node  {$2$};
\foreach \Point in {(0,-7.5), (3,-7.5), (1.33,-7.5), (2.5, -7.5)}
{
    \node at \Point {\textbullet};
}	
\draw (-2,-9) node {$f(z_n)$};
\draw (1,-9) node {$\circ$};
\draw (0.76,-9) node {...};
\draw (1.24,-9) node {...};
\draw (0,-9.5) node {$0$};
\draw (1.7,-9.52) node {$3/2$};
\draw (1,-9.5) node  {$1$};
\foreach \Point in {(0,-9), (1.5,-9)}
{
    \node at \Point {\textbullet};
}
\draw[-] (-3,-10) -- (7,-10);	
\end{tikzpicture}
\caption{The images of the sequences $(x_n),(y_n)$, and $(z_n)$ from Example \ref{eg:threesequences} under the function $f$. See Example \ref{eg:fthreesequences}. The sequences $(f(x_n)),(f(y_n))$, and $(f(z_n))$ each converge to $f(2)=1$. Note $f(2)=1$ is indicated with a $\bullet$ in the graph of $f$ but with a $\circ$ with the sequences $(f(y_n))$ and $(f(z_n))$. For $(f(x_n))$, we happen to have  $f(x_1)=f(4)=2$.}
\label{fig:fthreesequences}
\end{figure}

\begin{example}\label{eg:fthreesequences}
As in Example \ref{eg:linef}, define $f:[0,4]\to\R$ define by
\begin{align}
	f(x) &= x/2,
\end{align}
and define $(x_n),(y_n)$, and $(z_n)$ as in Example \ref{eg:threesequences}. For each $n\in\N$ we have
\begin{align}
	f(x_n)&=1+\frac{1}{n},\quad
	f(y_n)=1-\frac{1}{n},\quad\textnormal{and}\quad
	f(z_n)=1+\frac{1(-1)^n}{n}.
\end{align}
See Figure \ref{fig:fthreesequences}. We have 
\begin{align}
	\lim_{n\to\infty}x_n&=2
	\quad\textnormal{while}\quad
	\lim_{n\to\infty}f(x_n)=1=f(2),\\		
	\lim_{n\to\infty}y_n&=2
	\quad\textnormal{while}\quad
	\lim_{n\to\infty}f(y_n)=1=f(2),\qquad \textnormal{and}\\
	\lim_{n\to\infty}z_n&=2
	\quad\textnormal{while}\quad
	\lim_{n\to\infty}f(z_n)=1=f(2).
\end{align}

In fact, if we wanted to prove these limits are all equal to 1, the same choice for the threshold suffices in all three cases. Namely, any positive integer $n_\varepsilon$ where $n_\varepsilon>1/\varepsilon$ works. 
\end{example}

\begin{figure}
\centering
\begin{tikzpicture} 
\draw (-2,2) node {graph};
\draw (-2,1.5) node {of $g$};
	\draw (-0.5,3) node {$3$};
	\draw (0,3) node {$-$};	
	\draw (-0.5,2) node {$2$};
	\draw (0,2) node {$\bullet$};
	\draw (-0.5,1) node {$1$};
	\draw (0,1) node {$-$};
	\draw (-0.5,-0.5) node {$0$};		
  \draw[-To] (-0.5,0) -- (4.2,0) node[right] {$x$};
  \draw[-To] (0,-0.5) -- (0,4.2) node[above] {$y$};
  \draw[scale=1, domain=0:2, smooth, variable=\x, blue] plot ({\x}, {\x/2});
   \draw[scale=1, domain=2:4, smooth, variable=\x, blue] plot ({\x}, {1+\x/2});
  \draw[blue] (4,2.5) node {$g$};
  \draw (2,0) node {$\bullet$};
  \draw (4,0) node {$|$};  
  \draw[blue] (2,2) node {$\bullet$};
  \draw[blue, fill=white] (2,1) circle (0.08cm);  
  \draw (2,-.5) node {$2$};
  \draw (4,-.5) node {$4$};
\draw[-] (-3,-1) -- (7,-1);

\draw (-2,-1.5) node {$(x_n)$};
\draw (2,-1.5) node {$\circ$};
\draw (2.24,-1.5) node {$...$};
\draw (4,-2) node {$4$};
\draw (3,-2) node {$3$};
\draw (2,-2) node {$2$};
\foreach \Point in {(4,-1.5), (3,-1.5), (2.67,-1.5), (2.5,-1.5)}
{
    \node at \Point {\textbullet};
}
\draw (-2,-3) node {$g(x_n)$};
\draw (2,-3) node {$\circ$};
\draw (2.24,-3) node {$...$};
\draw (3,-3.5) node {$3$};
\draw (2,-3.5) node {$2$};
\foreach \Point in {(3,-3), (2.5,-3)}
{
    \node at \Point {\textbullet};
}	
\draw[-] (-3,-4) -- (7,-4);

\draw (-2,-4.5) node {$(y_n)$};
\draw (2,-4.5) node {$\circ$};
\draw (1.76,-4.5) node {...};
\draw (2,-5) node {$2$};
\draw (1,-5) node {$1$};
\draw (0,-5) node {$0$};
\foreach \Point in {(0,-4.5), (1,-4.5), (1.33,-4.5), (1.5,-4.5)}
{
    \node at \Point {\textbullet};
}
\draw (-2,-6) node {$g(y_n)$};
\draw (1,-6) node {$\circ$};
\draw (0.76,-6) node {...};
\draw (1,-6.5) node {$1$};
\draw (0,-6.5) node {$0$};
\foreach \Point in {(0,-6), (0.5,-6)}
{
    \node at \Point {\textbullet};
}
\draw[-] (-3,-7) -- (7,-7);

\draw (-2,-7.5) node {$(z_n)$};
\draw (2,-7.5) node {$\circ$};
\draw (1.76,-7.5) node {...};
\draw (2.24,-7.5) node {...};
\draw (0,-8) node {$0$};
\draw (3,-8) node {$3$};
\draw (1.15,-8.02) node {$4/3$};
\draw (2,-8) node  {$2$};
\foreach \Point in {(0,-7.5), (3,-7.5), (1.33,-7.5), (2.5, -7.5)}
{
    \node at \Point {\textbullet};
}	
\draw (-2,-9) node {$g(z_n)$};
\draw (1,-9) node {$\circ$};
\draw (0.76,-9) node {...};
\draw (0,-9) node {$\bullet$};
\draw (0,-9.5) node {$0$};
\draw (1,-9.5) node  {$1$};
\draw (2,-9) node {$\circ$};
\draw (2.24,-9) node {...};
\draw (2.5,-9) node {$\bullet$};
\draw (2,-9.5) node {$2$};
\draw (2.7,-9.52) node {$5/2$};
\draw[-] (-3,-10) -- (7,-10);	
\end{tikzpicture}
\caption{The images of the sequences $(x_n),(y_n)$, and $(z_n)$ from Example \ref{eg:threesequences} under the function $g$. See Example \ref{eg:gthreesequences}. The sequences $(g(x_n)),(g(y_n))$, and $(g(z_n))$ exhibit different limiting behavior. Note $g(2)=2$ is indicated with a $\bullet$ in the graph of $g$ but with a $\circ$ with the sequences $(g(x_n)),(g(y_n))$, and $(g(z_n))$.}
\label{fig:gthreesequences}
\end{figure}

\begin{example}\label{eg:gthreesequences}
As in Example \ref{eg:splitg}, define $g:[0,4]\to\R$ define by
\begin{align}
	g(x)&=
	\begin{cases}
	x/2,& \textnormal{ if } 0\leq x< 2,\\
	1+(x/2),& \textnormal{ if } 2\leq x\leq 4,
	\end{cases}
\end{align}
and define $(x_n),(y_n)$, and $(z_n)$ as in Example \ref{eg:threesequences}. For each $n\in\N$ we have
\begin{align}
	g(x_n)&=2+\frac{1}{n},\qquad\\
	g(y_n)&=1-\frac{1}{n},\qquad\textnormal{and}\\
	g(z_n)&=	\begin{cases}
	\displaystyle 1-\frac{1}{n},& \textnormal{ if }n\textnormal{ is odd, }\\
	&\\ 
	\displaystyle 2+\frac{1}{n},& \textnormal{ if }n\textnormal{ is even. }
	\end{cases}
\end{align}
See Figure \ref{fig:gthreesequences}. We have 
\begin{align}
	\lim_{n\to\infty}x_n&=2
	\quad\textnormal{while}\quad
	\lim_{n\to\infty}g(x_n)=2=g(2),\\		
	\lim_{n\to\infty}y_n&=2
	\quad\textnormal{while}\quad
	\lim_{n\to\infty}g(y_n)=1\neq g(2),\qquad \textnormal{and}\\
	\lim_{n\to\infty}z_n&=2
	\quad\textnormal{while}\quad
	\lim_{n\to\infty}g(z_n)\textnormal{ does not exist}.
\end{align}

Therefore, $g$ does not preserve the convergence of sequences. Specifically, $(z_n)$ converges in the domain, but $g(z_n)$ diverges in the range. Also, $(y_n)$ converges to $c=2$ in the domain, but $g(y_n)$ converges to $1$ while $g(c)=g(2)=2$. 
\end{example}

\begin{figure}
\centering
\begin{tikzpicture} 
\draw (-2,2) node {graph};
\draw (-2,1.5) node {of $h$};
	\draw (-0.5,3) node {$3$};
	\draw (0,3) node {$-$};	
	\draw (-0.5,2) node {$2$};
	\draw (0,2) node {$\bullet$};
	\draw (-0.5,0.5) node {$1/2$};
	\draw (-0.5,-0.5) node {$0$};		
  \draw[-To] (-0.5,0) -- (4.2,0) node[right] {$x$};
  \draw[-To] (0,-0.5) -- (0,4.2) node[above] {$y$};
  \draw[scale=1, domain=0:1.77, smooth, variable=\x, blue] plot ({\x}, {1/(2-\x)});
   \draw[scale=1, domain=2:4, smooth, variable=\x, blue] plot ({\x}, {1+\x/2});
  \draw[blue] (4,2.5) node {$h$};
  \draw (2,0) node {$\bullet$};
  \draw (4,0) node {$|$};  
  \draw[blue] (2,2) node {$\bullet$};
  \draw (2,-.5) node {$2$};
  \draw (4,-.5) node {$4$};
\draw[-] (-3,-1) -- (7,-1);  
\draw[-] (-3,-1) -- (7,-1);

\draw (-2,-1.5) node {$(x_n)$};
\draw (2,-1.5) node {$\circ$};
\draw (2.24,-1.5) node {$...$};
\draw (4,-2) node {$4$};
\draw (3,-2) node {$3$};
\draw (2,-2) node {$2$};
\foreach \Point in {(4,-1.5), (3,-1.5), (2.67,-1.5), (2.5,-1.5)}
{
    \node at \Point {\textbullet};
}
\draw (-2,-3) node {$h(x_n)$};
\draw (2,-3) node {$\circ$};
\draw (2.24,-3) node {$...$};
\draw (3,-3.5) node {$3$};
\draw (2,-3.5) node {$2$};
\foreach \Point in {(3,-3), (2.5,-3)}
{
    \node at \Point {\textbullet};
}	
\draw[-] (-3,-4) -- (7,-4);

\draw (-2,-4.5) node {$(y_n)$};
\draw (2,-4.5) node {$\circ$};
\draw (1.76,-4.5) node {...};
\draw (2,-5) node {$2$};
\draw (1,-5) node {$1$};
\draw (0,-5) node {$0$};
\foreach \Point in {(0,-4.5), (1,-4.5), (1.33,-4.5), (1.5,-4.5)}
{
    \node at \Point {\textbullet};
}
\draw (-2,-6) node {$h(y_n)$};
\foreach \Point in {(0.5,-6), (1,-6), (1.5,-6), (2,-6), (2.5,-6), (3,-6), (3.5,-6), (4,-6), (4.5,-6)}
{
    \node at \Point {\textbullet};
}
\draw (5,-6) node {...};
\draw (0.3,-6.5) node {$1/2$};
\draw (1,-6.5) node {$1$};
\draw (2,-6.5) node {$2$};
\draw (3,-6.5) node {$3$};
\draw (4,-6.5) node {$4$};
\draw[-] (-3,-7) -- (7,-7);

\draw (-2,-7.5) node {$(z_n)$};
\draw (2,-7.5) node {$\circ$};
\draw (1.76,-7.5) node {...};
\draw (2.24,-7.5) node {...};
\draw (0,-8) node {$0$};
\draw (3,-8) node {$3$};
\draw (1.15,-8.02) node {$4/3$};
\draw (2,-8) node  {$2$};
\foreach \Point in {(0,-7.5), (3,-7.5), (1.33,-7.5), (2.5, -7.5)}
{
    \node at \Point {\textbullet};
}	
\draw (-2,-9) node {$h(z_n)$};
\foreach \Point in {(0.5,-9), (1.5,-9), (2.5,-9), (3.5,-9), (4.5,-9)}
{
    \node at \Point {\textbullet};
}
\draw (0.3,-9.5) node {$1/2$};
\draw (1.3,-9.5) node {$3/2$};
\draw (2,-9) node {$\circ$};
\draw (2.24,-9) node {...};
\draw (2.5,-9) node {$\bullet$};
\draw (2,-9.5) node {$2$};
\draw (2.7,-9.52) node {$5/2$};
\draw (4.5,-9.5) node {$9/2$};
\draw (5,-9) node {...};
\draw[-] (-3,-10) -- (7,-10);	
\end{tikzpicture}
\caption{The images of the sequences $(x_n),(y_n)$, and $(z_n)$ from Example \ref{eg:threesequences} under the function $h$. See Example \ref{eg:hthreesequences}. The sequences $(h(x_n)),(h(y_n))$, and $(h(z_n))$ exhibit different limiting behavior. Note $h(2)=2$ is indicated with a $\bullet$ in the graph of $h$ but with a $\circ$ with the sequences $(h(x_n)),(h(y_n))$, and $(h(z_n))$.}
\label{fig:hthreesequences}
\end{figure}

\begin{example}\label{eg:hthreesequences}
As in Example \ref{eg:splith}, define $h:[0,4]\to\R$ define by
\begin{align}
	h(x)&=
	\begin{cases}
	\frac{1}{2-x},& \textnormal{ if } 0\leq x< 2,\\
	1+(x/2),& \textnormal{ if } 2\leq x\leq 4.
	\end{cases}
\end{align}
and define $(x_n),(y_n)$, and $(z_n)$ as in Example \ref{eg:threesequences}. For each $n\in\N$ we have
\begin{align}
	h(x_n)&=2+\frac{1}{n},\qquad\\
	h(y_n)&=\frac{n}{2},\qquad\textnormal{and}\\
	h(z_n)&=	\begin{cases}
	\displaystyle \frac{n}{2},& \textnormal{ if }n\textnormal{ is odd, }\\
	&\\ 
	\displaystyle 2+\frac{1}{n},& \textnormal{ if }n\textnormal{ is even. }
	\end{cases}
\end{align}
See Figure \ref{fig:hthreesequences}. We have
\begin{align}
	\lim_{n\to\infty}x_n&=2
	\quad\textnormal{while}\quad
	\lim_{n\to\infty}h(x_n)=2=h(2),\\		
	\lim_{n\to\infty}y_n&=2
	\quad\textnormal{while}\quad
	\lim_{n\to\infty}h(y_n)=\infty\neq h(2),\qquad \textnormal{and}\\
	\lim_{n\to\infty}z_n&=2
	\quad\textnormal{while}\quad
	\lim_{n\to\infty}h(z_n)\textnormal{ does not exist}.
\end{align}

Therefore, $h$ does not preserve the convergence of sequences. Specifically, $(y_n)$ and $(z_n)$ converge in the domain, but $h(y_n)$ and $h(z_n)$ diverge in the range. You might recall that $\lim_{n\to\infty}h(y_n)=\infty$ means $h(y_n)$ diverges in a particular way: The terms $h(y_n)$ get larger and larger without bound in the positive direction as the index $n$ increases.
\end{example}

The following definition provides a formal notion of what it means when a function preserves the convergence---or limits---of sequences. I call it {\em sequential continuity} since it involves converging sequences and it reminds me of how I was taught continuity as a calculus student.

\begin{definition}\label{def:sequentialcontinuity}
Suppose $D\subseteq \R^k, \bfc\in D$, and $f:D\to\R^m$. We say $f$ is \textit{sequentially continuous at} $\bfc$ if 
\begin{align}
(\bfx_n)\subseteq D\,\textnormal{ with }\, \lim_{n\to\infty} \bfx_n=\bfc
\qquad & \implies\qquad
\lim_{n\to\infty} f(\bfx_n)=f(\bfc).
\end{align}
If $f$ is sequentially continuous at every point in the domain $D$, then $f$ is \textit{sequentially continuous}. 
\end{definition}

\begin{remark}\label{rmk:sequentialcontinuitymovelimitinside}
We can think of sequential continuity symbolically as  moving the limit notation from outside $f$ to inside:
\begin{align}
\lim_{n\to\infty} f(\bfx_n)&=f\left(\lim_{n\to\infty} \bfx_n\right)=f(\bfc).
\end{align}
Sequential continuity basically says that either way---whether we plug $(\bfx_n)$ into $f$ first then take the limit of $f(\bfx_n)$, or we take the limit of $(\bfx_n)$ first then plug the limit $\bfc$ into $f$---we end up with the same value $f(\bfc)$. See Figure \ref{fig:continuitycommutativediagram}.
\end{remark}

\begin{figure}
\centering
\begin{tikzpicture} 
\draw (0,3) node {$\bfx_n$};
\draw (0,0) node {$f(\bfx_n)$};
\draw (4.7,3) node {$\displaystyle \lim_{n\to\infty}\bfx_n=\bfc$};
\draw (6.6,0) node {$\displaystyle \lim_{n\to\infty}f(\bfx_n)=f\left(\lim_{n\to\infty}\bfx_n\right)=f(\bfc)$};
	\draw[-To, semithick] (0, 2.5) -- (0, 0.5);
	\draw[-To, semithick] (0.6, 3) -- (3.3, 3);
	\draw[-To, semithick] (0.7, 0) -- (3.3, 0);
	\draw[-To, semithick] (4, 2.5) -- (4, 0.5);
\draw (2,2.5) node {$n\to\infty$};	
\draw (2,-0.5) node {$n\to\infty$};	
\draw (-1, 1.75) node {plug};
\draw (-1, 1.24) node {into $f$};
\draw (5, 1.75) node {plug};
\draw (5, 1.24) node {into $f$};
\end{tikzpicture}
\caption{As in Remark \ref{rmk:sequentialcontinuitymovelimitinside}, this commutative diagram visualizes sequential continuity in that whether we take the limit as $n$ tends to infinity or evaluate $f$ first, we end up with the same result $f(\bfc)$. Also, sequential continuity symbolically amounts to moving the ``$\lim$'' symbol in and out of the function $f$ has no effect on the resulting value.}
\label{fig:continuitycommutativediagram}
\end{figure}

Preserving closeness (Definition \ref{def:preservecloseness}) and sequential continuity (Definition \ref{def:sequentialcontinuity}) are equivalent. They each provide different perspectives on what continuity can mean, and they are also equivalent to the classic definition of continuity (Definition \ref{def:continuity}) developed and explored in Section \ref{sec:continuity} below.

\begin{theorem}\label{thm:preservingclosnesssequentialequivalence}
Suppose $D\subseteq \R^k, \bfc\in D$, and $f:D\to\R^m$. Then the following statements are equivalent:
\begin{enumerate}
	\item $f$ preserves closeness at $\bfc$.
	\item $f$ is sequentially continuous at $\bfc$.
\end{enumerate}
\end{theorem}

\begin{scratch}\label{scr:preservingclosnesssequentialequivalence}
The definitions of preserving closeness (Definition \ref{def:preservecloseness}) and sequential continuity (Definition \ref{def:sequentialcontinuity}) are related to each other through what I consider to be one of the most important theorems in this book, Theorem \ref{thm:exercise0}, which establishes the fundamental connection between the concepts of arbitrarily close (Definition \ref{def:acl}) and limits of sequences (Definition \ref{def:sequentiallimit}). I encourage you to familiarize yourself with the negations of the definitions of preserving closeness (Definition \ref{def:preservecloseness}) and sequential continuity (Definition \ref{def:sequentialcontinuity}).
\end{scratch}

\begin{proof} Let's show both (i)$\implies$(ii) and (ii)$\implies$(i) by contraposition.

\underline{(i)$\implies$(ii)}: Suppose there is a sequence $(\bfx_n)$ of points in $D$ and a point $\bfc$ in $D$ where $\lim_{n\to\infty}\bfx_n=\bfc$ but $\lim_{n\to\infty}f(\bfx_n)\neq f(\bfc)$. Then there must be some $\varepsilon_0>0$ such that no matter which positive integer $n_0$ we consider as a potential threshold, there is a positive integer $q\geq n_0$ where $\bfx_q$ is in $D$ and
\begin{align}
	d_m(f(\bfx_q),f(\bfc))&\geq \varepsilon_0.
\end{align}

We can use the previous statement to construct a suitable subsequence $(\bfx_{n_j})$ of $(\bfx_n)$ whose image stays away from $f(\bfc)$. That is, let $n_1\geq 1$ be a positive integer that satisfies the previous statement. Proceeding inductively, for each positive integer $j$, there is a positive integer $n_j>n_{j-1}$ where $d_m(f(\bfx_{n_j}),f(\bfc))\geq \varepsilon_0$. 

Now, since $\lim_{n\to\infty}\bfx_n=\bfc$ and $(\bfx_{n_j})$ is a subsequence of $(\bfx_n)$, we have that $\lim_{j\to\infty}\bfx_{n_j}=\bfc$ as well. By Theorem \ref{thm:exercise0}, we have $\bfc \acl (\bfx_{n_j})$. But we also have $f(\bfc)\awf (f(\bfx_{n_j}))$. Therefore, $f$ does not preserve closeness at $\bfc$.

\underline{(ii)$\implies$(i)}: Suppose $f$ does not preserve closeness at a point $\bfc$ in $D$. Then there must be some nonempty $E\subseteq D$ where $\bfc \acl E$ but $f(\bfc)\awf f(E)$. So, there is some $\varepsilon_0>0$ such that for every point $\bfx$ in $E$ we have
\begin{align}
	d_m(f(\bfx),f(\bfc))&\geq\varepsilon_0.
\end{align} 
Using Theorem \ref{thm:exercise0}, we can use the previous statement to find a suitable sequence. Since $\bfc \acl E$, there must be some sequence $(\bfx_n)$ of points in $E\subseteq D$ where 
$\lim_{n\to\infty}\bfx_n=\bfc$. However, we have $d_m(f(\bfx_n),f(\bfc))\geq \varepsilon_0$ for every $\bfx_n$, so $\lim_{n\to\infty}f(\bfx_n)\neq f(\bfc)$. Therefore, condition (ii) fails.
\end{proof} 

Consider the following example of a function which maps a disk in the plane $\R^2$ to a subset of the real line $\R$ but does not preserve connectedness. More examples in higher dimensions are explored in the context of linear algebra in Section \ref{sec:segueintolinearalgebra}.

\begin{example}\label{eg:notpreserveconnectedR2}
Consider the function $s:\R^2\to\R$ given by
\begin{align}
s(\bfx)&=
	\begin{cases}
	\|\bfx\|,& \textnormal{ if } 0\leq \|\bfx\| < 1,\\
	2,& \textnormal{ if } \|\bfx\| \geq 1.
	\end{cases}
\end{align}
The function $s$ maps every point in the open unit disk $V_1(\mathbf{0})$ to its length from the origin and maps every point outside of $V_1(\mathbf{0})$ to 2. Note that $s$ is not a linear transformation. 

In Figure \ref{fig:notpreserveconnectedR2} featuring the function $s$, the standard basis vectors $\bfe_1$ and $\bfe_2$ are plotted as points ($\bullet$) instead of arrows. Note that $s(\bfe_1)=s(\bfe_2)=2$.

\begin{figure}
\centering
\begin{tikzpicture}
\draw[fill=blue!15] (-1,0) ellipse (1cm and 1cm);
	\draw (-1,1) node {$\bullet$};
	\draw (-.75,1.25) node {$\bfe_2$};
	\draw (0,0) node {$\bullet$};
	\draw (0.25,-0.25) node {$\bfe_1$};
	\draw (-1,-1.5) node {$V_1(\mathbf{0})\cup C$};
	\draw (1.5,0) node {\Large{$\mapsto$}};
\draw (4,-1.5) node {$s(V_1(\mathbf{0})\cup C)=[0,1)\cup\{2\}$};	
	\draw[-,blue!75,semithick] (3,0) -- (4,0);
	\draw[blue!75] (3,0) node {$[$};
	\draw[blue!75] (4,0) node {$)$};
	\draw (5,-0.02) node {$\bullet$};
	\draw (3,-0.5) node {$0$};	
	\draw (4,-0.5) node {$1$};	
	\draw (5,-0.5) node {$2$};	
\end{tikzpicture}
\caption{The function $s$ from Example \ref{eg:notpreserveconnectedR2} does not preserve closeness since $\bfe_1\acl{V_1(\mathbf{0})}$ in the domain but $s(\bfe_1)=2$ is away from $s(V_1(\mathbf{0}))=[0,1)$ in the range.}
\label{fig:notpreserveconnectedR2}
\end{figure}

The set $C$ denotes the unit circle centered at the origin (the black circle on the left). That is, $C=\{\bfx\in\R^2:\|\bfx\|=1\}$. So, the set $V_1(\mathbf{0})\cup C$ is connected (it's the closed disk $\overline{V_1(\mathbf{0})}$), but its image under $s$ is disconnected since 
\begin{align}
s(V_1(\mathbf{0})\cup C)=[0,1)\cup\{2\}.
\end{align}
In particular, $\bfe_1$ is in $C$ and arbitrarily close to $V_1(\mathbf{0})$ in the domain, but $s(\bfe_1)=2$ is away from $s(V_1(\mathbf{0}))=[0,1)$ in the range. 
\end{example}

Before moving on to the next section, consider what it means for a function to be {\em bounded}. This definition is a generalization of the definition of a bounded set of real numbers (see Definition \ref{def:max}).

\begin{definition}\label{def:boundedfunction}
Let $D\subseteq\R^k$ and $f:D\to\R^m$. The function $f$ is {\em bounded} if its range $f(D)$ is a bounded set. That is, $f$ is bounded if there exists some $b\geq 0$ where
\begin{align}
	\|f(\bfx)\|\leq b \quad\textnormal{for all}\quad\bfx\in D.
\end{align}
Such a nonnegative real number $b$ is called a {\em bound} for $f$ and we say $f$ is {\em bounded} by $b$.
\end{definition}

\begin{example}\label{eg:triobounded}
Which of the functions $f,g$, $h$, and $s$ from Examples \ref{eg:linef}, \ref{eg:splitg}, \ref{eg:splith}, and \ref{eg:notpreserveconnectedR2} are bounded? For every $x\in[0,4]$ we have $|f(x)|\leq 2$ and $|g(x)|\leq 3$, so both $f$ and $g$ are bounded. On the other hand, $h$ is unbounded. The function $s$ is bounded since $|s(\bfx)|\leq 2$ for every $\bfx$ in the domain $V_1(\mathbf{0})\cup C$. 
\end{example}

There are actually lots of interesting results we can prove with the mathematics we've developed so far, but we are also prepared for one of the most important topics in all of mathematics: {\em Continuity}. So, the next section develops the formal ``$\varepsilon$-$\delta$'' definition for continuity (Definition \ref{def:continuity}) and, from there, develops a number of interesting results involving all sorts of material from throughout the book.

\vs
\section*{Exercises}
\setcounter{theorem}{0}

Exercises are for play: Do scratch work, draw stuff, and make mistakes---make {\em lots} of mistakes---before worrying about writing proofs. {\em Have fun!}


\vs
\section{Continuity}
\label{sec:continuity}

This book started off unconventionally with a formal notion for arbitrarily close, and to keep the theme going we'll buck convention again by discussing continuity before limits. The classic ``$\varepsilon$-$\delta$'' definitions of continuity and limits for functions are perhaps the most difficult concepts for new mathematicians to fully understand. (See Definitions \ref{def:continuity} and \ref{def:functionlimit}, respectively.) This perspective is corroborated by \cite{Barnes,Borzellino,Cornu,LarsenSwinyard,Seager} and just about any undergraduate mathematics major (probably). So, this section is designed to face the challenge of understanding continuity and limits for functions by building on the pair of notions arbitrarily close and thresholds. 

Time for the classic definition of continuity in analysis.

\begin{definition}\label{def:continuity}
Let $D\subseteq\R^k$, let $f:D\to \R^m$, and let $\bfc\in D$. We say $f$ is {\em continuous at} $\bfc$ if for every distance $\varepsilon>0$ there is a {\em threshold} $\delta>0$ such that 
\begin{align}
	\bfx\in D \textnormal{ with }d_k(\bfx,\bfc)&<\delta \quad\implies\quad d_m(f(\bfx),f(\bfc))<\varepsilon.
\end{align}
In other words, $f$ is {\em continuous at} $\bfc$ if for every error $\varepsilon>0$ there is some threshold $\delta>0$ where the $\delta$-neighborhood of $\bfc$ in the domain maps into the $\varepsilon$-neighborhood of $f(\bfc)$ in the range:
\begin{align}
	f(V_\delta(\bfc)\cap D)\subseteq V_\varepsilon(f(\bfc)).
\end{align} 
If $f$ is continuous at every point in the domain $D$, we say $f$ is {\em continuous}. If $f$ is not continuous at some point $\bfz$, then we say $f$ is {\em discontinuous at} $\bfz$.
\end{definition}

\begin{remark}\label{rmk:deltathreshold}
The word ``threshold'' appears in the definitions of limit and convergence for sequences (Definition \ref{def:sequentiallimit}) as well as the definition of continuity (Definition \ref{def:continuity}). In both settings, the threshold is in response to a distance (or error)  $\varepsilon>0$ for the range. For limits of sequences, the threshold $n_\varepsilon$ is a positive integer index that depends on $\varepsilon$ (hence the subscript) and how the sequence is defined. For continuity, the threshold $\delta$ is a positive distance for the domain that depends on $\varepsilon$, how the function is defined, {\em and} the point $\bfc$. Since $\delta$ depends on so many things, I decided to leave out subscripts. This is also the conventional notation for $\delta$. Still, it is important to keep in mind the threshold $\delta$ generally depends on the function, $\bfc$, and $\varepsilon$. See Figures \ref{fig:fcontinuous} and \ref{fig:continuous2D}.
\end{remark}

\begin{figure}
\centering
\begin{tikzpicture}
\draw (-2,2) node {graph};
\draw (-2,1.5) node {of $f$};
	\draw (-0.5,2) node {$2$};
	\draw (0,2) node {$-$};
	\draw (-0.6,1.41) node {$\sqrt{2}$};
	\draw (0,1.41) node {$\bullet$};
	\draw (-0.5,-0.5) node {$0$};		
  \draw[-To] (-1, 0) -- (4.2, 0) node[right] {$x$};
  \draw[-To] (0, -1) -- (0, 4.2) node[above] {$y$};
  \draw[scale=1, domain=0:4, smooth, variable=\x] plot ({\x}, {(\x)^(1/2)});
  \draw(4,2.5) node {$f$};
  \draw (2,0) node {$\bullet$};
  \draw (4,0) node {$|$};  
  \draw[blue] (2,1.41) node {$\bullet$};
  \draw (2,-.5) node {$2$};
  \draw (4,-.5) node {$4$};

\draw[dashed,semithick,red] (0,1.81) -- (4,1.81);
\draw[dashed,semithick,red] (0,1.01) -- (4,1.01);
\draw[dashed,semithick,blue] (1.6,0) -- (1.6,1.26);
\draw[dashed,semithick,blue] (2.4,0) -- (2.4,1.55);
\draw[dashed,semithick,blue] (0,1.26) -- (1.6,1.26);
\draw[dashed,semithick,blue] (0,1.55) -- (2.4,1.55);

\draw[red] (-2.5,-2) node {$V_\varepsilon(\sqrt{2})$};
	\draw[-,semithick,red] (1.01,-2) -- (1.81,-2);
	\draw[red] (1.01,-2) node {$($};
	\draw[red] (1.81,-2) node {$)$};
	\draw[red] (0.31,-2.5) node {$\sqrt{2}-\varepsilon$};
	\draw[red] (2.51,-2.5) node {$\sqrt{2}+\varepsilon$};
	\draw (1.41,-2) node {\textbullet};
	\draw (1.41,-2.5) node {$\sqrt{2}$};	

\draw[blue] (-2.5,-3.5) node {$V_\delta(2)\cap D$};
	\draw[-,semithick,blue] (1.6,-3.5) -- (2.4,-3.5);
	\draw[blue] (1.6,-3.5) node {$($};
	\draw[blue] (2.4,-3.5) node {$)$};
	\draw[blue] (1.1,-4) node {$2-\delta$};
	\draw[blue] (2.9,-4) node {$2+\delta$};
	\draw (2,-3.5) node {\textbullet};
	\draw (2,-4) node {$2$};	

\draw[blue] (-3.4,-5) node {$f(V_\delta(2)\cap D)$};
\draw (-1.95,-5) node {$\subseteq$};	
\draw[red] (-1,-4.98) node {$V_\varepsilon(\sqrt{2})$};
	\draw[-,semithick,red] (1.01,-5) -- (1.81,-5);
	\draw[red] (1.01,-5) node {$($};
	\draw[red] (1.81,-5) node {$)$};
	\draw (1.41,-5) node {\textbullet};
	\draw[-,semithick,blue] (1.26,-5) -- (1.55,-5);
	\draw[blue] (1.26,-5) node {$($};
	\draw[blue] (1.55,-5) node {$)$};	
\end{tikzpicture}
\caption{The graph of a function $f:[0,4]\to\R$ given by $f(x)=\sqrt{x}$ to accompany the definition of continuity (Definition \ref{def:continuity}). See and compare with Example \ref{eg:sqrtcontinuous}. Given a distance $\varepsilon>0$ from $f(2)=\sqrt{2}$ in the range, there is a threshold $\delta>0$ providing a distance from $c=2$ in the domain ensuring $f(V_\delta(2)\cap D)\subseteq V_\varepsilon(f(2))$.}
\label{fig:fcontinuous}
\end{figure}

\begin{figure}
\centering
\begin{tikzpicture}
\draw (-4,0) node {Let $\varepsilon>0$.};
	\draw (-1,0) node {$\bullet$};		
	\draw (-1,-0.4) node {$\bfc$};
	\draw (1,0) node {\Large{$\mapsto$}};
\draw[dashed,red,fill=red!15] (5,0) ellipse (2cm and 2cm);
	\draw (5,0) node {$\bullet$};		
	\draw (5,-0.5) node {$f(\bfc)$};
	\draw (5,-2.5) node {$V_\varepsilon(f(\bfc))$};
	\draw[red] (5,0) -- (6.42,1.42);
	\draw[red] (5.9,0.6) node {$\varepsilon$};
	
\draw (-4,-4.5) node {Choose $\delta>0$.};
\draw[dashed,blue,fill=blue!15] (-1,-4.5) ellipse (0.67cm and 0.67cm);	
\draw[dashed,blue,fill=blue!15] (5,-4.5) ellipse (2cm and 0.33cm);
	\draw[blue] (-1,-4.5) -- (-0.53,-4.03);
	\draw[blue] (-0.88,-4.11) node {$\delta$};
	\draw (-1,-4.5) node {$\bullet$};		
	\draw (-1,-4.9) node {$\bfc$};
	\draw (-1,-5.7) node {$V_\delta(\bfc)\cap D$};
	\draw (1,-4.5) node {\Large{$\mapsto$}};
	\draw (5,-4.5) node {$\bullet$};		
	\draw (5.6,-4.5) node {$f(\bfc)$};
	\draw (5,-5.3) node {$f(V_\delta(\mathbf{\bfc})\cap D)$};	
	
	\draw (-2.5,-9) node {Verify $f(V_\delta(\mathbf{\bfc})\cap D)\subseteq V_\varepsilon(f(\bfc)).$};		
\draw[dashed,red,fill=red!15] (5,-9) ellipse (2cm and 2cm);
\draw[dashed,blue,fill=blue!15] (5,-9) ellipse (2cm and 0.33cm);
	\draw (5,-9) node {$\bullet$};		
	\draw (5.6,-9) node {$f(\bfc)$};
	\draw (5,-9.8) node {$f(V_\delta(\mathbf{\bfc})\cap D)$}	;
	\draw (5,-11.5) node {$V_\varepsilon(f(\bfc))$};	
\end{tikzpicture}
\caption{A visualization of the definition of continuity in the context of a function $f$ mapping $\R^2$ to $\R^2$. See Definition \ref{def:continuity}. This figure also happens to serve as a visual summary for proofs showing $f$ is continuous at $\bfc$ as well as an example of the continuity of a linear transformation defined by a matrix. See Section \ref{sec:segueintolinearalgebra}.}
\label{fig:continuous2D}
\end{figure}

There are lots of functions from basic algebra, trigonometry, and calculus that are continuous. The following lemma on a small but important class of continuous functions has a nice and instructive proof.

\begin{lemma}\label{lem:linesarecontinuous}
Suppose $m$ and $b$ are real numbers and let $f:\R\to\R$ be given by $f(x)=mx+b$. Then $f$ is continuous at every real number $c$.
\end{lemma}

\begin{remark}\label{rmk:provingcontinuity}
The scratch work and proofs for showing a function is continuous by directly verifying Definition \ref{def:continuity} holds might look  familiar. The work that comes with directly verifying a sequence converges to its limit as in Definition \ref{def:sequentiallimit} is very similar. Both can involve starting with some arbitrary distance $\varepsilon>0$ and finding a way to respond with a suitable threshold $\delta>0$ or $n_\varepsilon\in\N$, respectively.
\end{remark}

\begin{scratch}\label{scr:linesarecontinuous}
Try starting scratch work with the inequality 
\begin{align}\label{eqn:continuityepsilon}
	d_m(f(\bfx),f(\bfc))<\varepsilon
\end{align}
and somehow find a way to determine a formula for a suitable threshold $\delta$ to give us 
\begin{align}
	\bfx\in D \textnormal{ with }d_k(\bfx,\bfc)&<\delta \quad\implies\quad d_m(f(\bfx),f(\bfc))<\varepsilon.
\end{align}
Once we have a suitable $\delta$, we can reorganize the information into a formal proof. See Figure \ref{fig:continuous2D}.

Taking the function $f$ defined in Lemma \ref{lem:linesarecontinuous} and sticking everything into inequality \eqref{eqn:continuityepsilon} yields
\begin{align}
	d_{\R}(f(x),f(c))=|f(x)-f(c)|=|(mx+b)-(mc+b)|<\varepsilon.
\end{align}
So, let's simplify or rearrange $|f(x)-f(c)|$ and see what we can get.
\begin{align}
	|f(x)-f(c)|&=|(mx+b)-(mc+b)|\\
	&=|mx-mc|\\
	&=|m||x-c| \ldots \label{eqn:xcdifference}
\end{align}
$\ldots$ but what are we looking for? Unlike the scratch work for convergent sequences where a variable $n$ is often readily available, the variable $\delta$ does not explicitly appear in the work we have so far. 

In this case, we are looking for a suitable $\delta$ where 
\begin{align}
	d_{\R}(x,c)=|x-c|<\delta.
\end{align}
So, instead of finding a formula for $\delta$ directly, try finding a formula for the distance $|x-c|$. Note that $|x-c|$ appears in line \eqref{eqn:xcdifference}, so solving for $|x-c|$ in terms of $\varepsilon$ comes out of
\begin{align}
	|f(x)-f(c)|&=|(mx+b)-(mc+b)|=|m||x-c|<\varepsilon.
\end{align}
Dividing the rightmost inequality by $|m|$ yields
\begin{align}
	|x-c|<\frac{\varepsilon}{|m|}.
\end{align}
From here, it's reasonable to try setting 
\begin{align}
	\delta&=\frac{\varepsilon}{|m|}
\end{align}
and see what we get. Careful! This expression for $\delta$ is not defined when $m=0$, so a proof by cases where $m=0$ and $m\neq 0$ is in order.
\end{scratch}

\begin{proof}[Proof of Lemma \ref{lem:linesarecontinuous}]  
Let $m,b,$ and $c$ be real numbers and let $f:\R\to\R$ be given by $f(x)=mx+b$.

\underline{Case (i)}: Suppose $m\neq 0$ and let $\varepsilon>0$. Then $|m|\neq 0$. Define
\begin{align}
	\delta&=\frac{\varepsilon}{|m|}.
\end{align}
(Note that $\delta>0$.) Then for every real number $x$ where 
\begin{align}\label{eqn:linesdelta}
	d_{\R}(x,c)=|x-c|<\delta&=\frac{\varepsilon}{|m|},
\end{align}
we have 
\begin{align}
	|f(x)-f(c)|&=|(mx+b)-(mc+b)|\\
	&=|m||x-c|\\
	&<|m|\frac{\varepsilon}{|m|} \quad\textnormal{(by \eqref{eqn:linesdelta})}\\
	&=\varepsilon.
\end{align}
Therefore, $\delta=\varepsilon/|m|$ is a suitable threshold and $f$ is continuous at $c$ by Definition \ref{def:continuity}.

\underline{Case (ii)}: Suppose $m=0$ and let $\varepsilon>0$. Define $\delta=17>0$. Then for every real number $x$ where 
\begin{align}\label{eqn:linesdeltazero}
	d_{\R}(x,c)=|x-c|<\delta&=17,
\end{align}
we have 
\begin{align}
	|f(x)-f(c)|&=|(mx+b)-(mc+b)|=|0-0|=0<\varepsilon.
\end{align}
Therefore, $\delta=17$ is a suitable threshold and $f$ is continuous at $c$ by Definition \ref{def:continuity}. (Actually, any positive number would suffice for $\delta$.)
\end{proof}

\begin{example}\label{eg:sqrtcontinuous}
The square root function $f:[0,\infty)\to\R$ where $f(x)=\sqrt{x}$ is continuous at every $c\geq 0$. See Figure \ref{fig:fcontinuous}.
\end{example}

\begin{scratch}\label{scr:sqrtcontinuous}
We want the proof to end up with
\begin{align}\label{eqn:sqrtxepsilon}
	d_{\R}(f(x),f(c))&=|\sqrt{x}-\sqrt{c}|<\varepsilon.
\end{align}
To get there, we need a suitable threshold $\delta$ where
\begin{align}
	|x-c|&<\delta\quad\implies\quad |\sqrt{x}-\sqrt{c}|<\varepsilon.
\end{align}	
An idea from calculus or perhaps precalculus can help us here: Conjugates. As long as we consider $c>0$ to ensure our denominators are never zero, since $\sqrt{x}\geq 0$ we have
\begin{align}
	|\sqrt{x}-\sqrt{c}|
	&=|\sqrt{x}-\sqrt{c}|\left(\frac{|\sqrt{x}+\sqrt{c}|}{|\sqrt{x}+\sqrt{c}|}\right)\\
	&=\frac{|x-c|}{\sqrt{x}+\sqrt{c}}\\
	&\leq\frac{|x-c|}{\sqrt{c}}.
\end{align}
From here---and as long as $c>0$---we can ensure 
\begin{align}
	|\sqrt{x}-\sqrt{c}| &\leq\frac{|x-c|}{\sqrt{c}}<\varepsilon
\end{align}
by choosing 
\begin{align}
	|x-c|<\delta=\varepsilon\sqrt{c}.
\end{align}

But what about when $c=0$? In this case, we want to end up with
\begin{align}
	|x-0|=x<\delta \quad\implies\quad 
	|\sqrt{x}-\sqrt{0}|=\sqrt{x}<\varepsilon \label{eqn:sqrtxepsilon}
\end{align}	
Keeping in mind $x\geq 0$ and $0\leq y<z$ implies $y^2<z^2$, squaring both sides of the rightmost inequality in \ref{eqn:sqrtxepsilon} yields
\begin{align}
	|x-c|=x=(\sqrt{x})^2<\varepsilon^2.
\end{align}	
So, let's try using $\delta=\varepsilon^2$ when $c=0$.
\end{scratch}

\begin{proof}[Proof for Example \ref{eg:sqrtcontinuous}] The proof is handled in two cases.

\underline{Case (i)}: Let $\varepsilon>0$, let $c=0$, and define $\delta=\varepsilon^2>0$. Then for every nonnegative real number $x$ where 
\begin{align}\label{eqn:sqrtdelta}
	d_{\R}(x,c)=|x-0|=x<\delta&=\varepsilon^2,
\end{align}
since $0\leq x< \varepsilon^2$ implies $0\leq\sqrt{x}<\sqrt{\varepsilon^2}=\varepsilon$, we have 
\begin{align}
	d_{\R}(f(x),f(c))&=|\sqrt{x}-0|=\sqrt{x}<\sqrt{\varepsilon^2}=\varepsilon.
\end{align}
Therefore, $\delta=\varepsilon^2$ is a suitable threshold and $f$ is continuous at $c=0$ by Definition \ref{def:continuity}.

\underline{Case (ii)}: Let $\varepsilon>0$, let $c>0$, and define $\delta=\varepsilon\sqrt{c}>0$.
Then for every nonnegative real number $x$ where 
\begin{align}\label{eqn:sqrtdeltacpositive}
	d_{\R}(x,c)=|x-c|=x<\delta&=\varepsilon\sqrt{c},
\end{align}
we have
\begin{align}
	|\sqrt{x}-\sqrt{c}|
	&=|\sqrt{x}-\sqrt{c}|\left(\frac{|\sqrt{x}+\sqrt{c}|}{|\sqrt{x}+\sqrt{c}|}\right)\\
	&=\frac{|x-c|}{\sqrt{x}+\sqrt{c}}\\
	&\leq\frac{|x-c|}{\sqrt{c}}\\
	&<\varepsilon.
\end{align}
Therefore, $\delta=\varepsilon\sqrt{c}$ is a suitable threshold and $f$ is continuous at $c>0$ by Definition \ref{def:continuity}.
\end{proof}

\begin{example}\label{eg:squarefunctioncontinuous}
Let $f:\R\to\R$ be given by $f(x)=x^2$. Then $f$ is continuous.
\end{example}

\begin{scratch}
We want to end up with 
\begin{align}\label{eqn:squarefunctionepsilon}
	d_{\R}(f(x),f(c))&=|x^2-c^2|<\varepsilon.
\end{align}
Some algebra yields
\begin{align}\label{eqn:squarefunctionepsilon}
	|x^2-c^2|=|x-c||x+c|.
\end{align}
Substituting $\delta_0$ for $|x-c|$ in inequality \eqref{eqn:squarefunctionepsilon} yields
\begin{align}\label{eqn:squarefunctionepsilon2}
	d_{\R}(f(x),f(c))
	&=|x^2-c^2|
	=|x-c||x+c|
	<\delta_0|x+c|
	<\varepsilon.
\end{align}
So a first, naive choice for a threshold $\delta_0$ might be 
\begin{align}\label{eqn:squarefunctionbaddelta}
	\delta_0=\frac{\varepsilon}{|x+c|}.
\end{align}
However, this $\delta_0$ is not an appropriate choice for a threshold. For instance, if $x=-c$, then this $\delta_0$ would be undefined since the denominator would be zero. 

More importantly, our choice for a threshold must not depend on the variable $x$. A suitable threshold $\delta$ imposes a condition of the value of an input $x$ in terms of how close it should be to $c$ in order to ensure the outputs $f(x)$ and $f(c)$ are within $\varepsilon$ of each other. So, we need to control the size of $|x+c|$ in some way, keeping in mind that we generally want $x$ to be near $c$.

This may seem strange, but to get a handle on both $|x-c|$ and $|x+c|$ we are free to choose any positive real number as a preliminary version of $\delta$ and see what happens. After all, the underlying idea is to find a constraint on how close $x$ needs to be to $c$, not to find a nice formula for $\delta$. 

So, let's see what happens when we take $|x-c|<1$. By the reverse triangle inequality \eqref{eqn:reversetriangleinequality} we have
\begin{align}
	|x|-|c|\leq |x-c|<1.
\end{align}
Therefore, adding $|c|$ yields
\begin{align}
	|x|<|c|+1.
\end{align}
By the (regular) triangle inequality \eqref{eqn:triangleinequality2}, we further have
\begin{align}\label{eqn:preliminarycontrol}
	|x+c|\leq |x|+|c|< 2|c|+1.
\end{align}
So, under the preliminary assumption that $|x-c|<1$, our scratch work looks like
\begin{align}\label{eqn:squarefunctionepsilon3}
	|x^2-c^2| &=|x-c||x+c|<|x-c|(2|c|+1)<\varepsilon.
\end{align}
Hence, choosing a threshold so that 
\begin{align}\label{eqn:squarefunctionepsilon4}
	|x-c|&<\frac{\varepsilon}{2|c|+1}
\end{align}
seems in order. However, inequality \eqref{eqn:squarefunctionepsilon4} is only valid when $|x-c|<1$, so we need to keep this constraint in mind. Hence, a suitable choice for a threshold $\delta$ may well be
\begin{align}
	\delta=\min\left\{1,\frac{\varepsilon}{2|c|+1}\right\}.
\end{align}
On to the proof.
\end{scratch}

\begin{proof}[Proof for Example \ref{eg:squarefunctioncontinuous}]
Suppose $c\in\R$ and let $\varepsilon>0$. Define a candidate for a threshold $\delta$ by 
\begin{align}\label{eqn:squarefunctiondelta}
	\delta=\min\left\{1,\frac{\varepsilon}{2|c|+1}\right\}.
\end{align}
Note that $\delta>0$. Now suppose $|x-c|<\delta$. Since 
\begin{align}
	|x-c|<\delta\leq 1,
\end{align}
by the reverse triangle inequality \eqref{eqn:reversetriangleinequality} we have 
\begin{align}\label{eqn:revtrisquarefunction}
	|x|-|c|\leq |x-c| <1 \implies |x|<|c|+1.
\end{align}
Next, by the triangle inequality and inequality \eqref{eqn:revtrisquarefunction} we have
\begin{align}
	|x+c|\leq |x|+|c|<2|c|+1.
\end{align}
By our choice for $\delta$ \eqref{eqn:squarefunctiondelta} we also have
\begin{align}
	|x-c|<\delta\leq\frac{\varepsilon}{2|c|+1}.
\end{align}
Therefore, if $|x-c|<\delta$, then
\begin{align}
	|x^2-c^2|
	&=|x-c||x+c|
	<\left(\frac{\varepsilon}{2|c|+1}\right)(2|c|+1)
	=\varepsilon.
\end{align}
Hence, $f(x)=x^2$ is continuous on $\R$.
\end{proof}

Note that we still have not addressed the following question: Does the function $f:[0,4]\to\R$ from Example \ref{eg:linef} defined by
\begin{align}
	f(x)&=x/2
\end{align}	
preserve closeness at $c=2$? {\em How do we check?} The following section explores the deep connections between the definitions of preserving closeness (Definition \ref{def:preservecloseness}), sequential continuity (Definition \ref{def:sequentialcontinuity}), and ``$\varepsilon$-$\delta$'' continuity (Definition \ref{def:continuity}).

\vs
\section*{Exercises}
\setcounter{theorem}{0}

Exercises are for play: Do scratch work, draw stuff, and make mistakes---make {\em lots} of mistakes---before worrying about writing proofs. {\em Have fun!}

\vs
\section{Properties of continuous functions}
\label{sec:propertiesofcontinuousfunctions}

A major result in this text is the following fact: The definitions of continuity (Definition \ref{def:continuity}), sequential continuity (Definition \ref{def:sequentialcontinuity}), and preserving closeness (Definition \ref{def:preservecloseness}) are equivalent. 

Each provides its own perspective, its own set of benefits and drawbacks, on the concept of continuity. Each can be used to check whether the others hold.  Ultimately, when I think of the meaning of ``continuity'' in analysis, I have each of these equivalent notions in mind knowing I can run with the one that seems to be the easiest to work with.

\begin{theorem}\label{thm:continuityequivalence}
Suppose $D\subseteq \R^k, \bfc\in D$, and $f:D\to\R^m$. Then the following statements are equivalent forms of the continuity of $f$ at $\bfc$:
\begin{enumerate}
	\item $f$ preserves closeness at $\bfc$ 
(Definition \ref{def:preservecloseness}).
	\item $f$ is sequentially continuous at $\bfc$
(Definition \ref{def:sequentialcontinuity}). 
	\item $f$ is continuous at $\bfc$ (Definition \ref{def:continuity}).
\end{enumerate}
\end{theorem}

\begin{proof} Let's show (i)$\implies$(ii), (ii)$\implies$(iii), and (iii)$\implies$(i), all by contraposition.

\underline{(i)$\implies$(ii)}: Suppose there is a sequence $(\bfx_n)$ of points in $D$ and a point $\bfc$ in $D$ where $\lim_{n\to\infty}\bfx_n=\bfc$ but $\lim_{n\to\infty}f(\bfx_n)\neq f(\bfc)$. Then there must be some $\varepsilon_0>0$ such that no matter which positive integer $N$ we consider, there is a positive integer $q\geq N$ where $\bfx_q$ is in $D$ and $d_m(f(\bfx_q),f(\bfc))\geq \varepsilon_0$.

We can use the previous statement to construct a suitable subsequence $(\bfx_{n_j})$ of $(\bfx_n)$ whose image stays away from $f(\bfc)$. That is, let $n_1\geq 1$ be a positive integer that satisfies the previous statement. Proceeding inductively, for each positive integer $j$, there is a positive integer $n_j>n_{j-1}$ where $d_m(f(\bfx_{n_j}),f(\bfc))\geq \varepsilon_0$. 

Now, since $\lim_{n\to\infty}\bfx_n=\bfc$ and $(\bfx_{n_j})$ is a subsequence of $(\bfx_n)$, we have that $\lim_{j\to\infty}\bfx_{n_j}=\bfc$ as well. By Theorem \ref{thm:exercise0}, we have $\bfc \acl (\bfx_{n_j})$. But we also have $f(\bfc)\awf (f(\bfx_{n_j}))$. Therefore, $f$ does not preserve closeness at $\bfc$. (See Definition \ref{def:preservecloseness}.)

\underline{(ii)$\implies$(iii)}: Suppose $f$ is not continuous at a point $\bfc$ in $D$. Then there must be some $\varepsilon_0>0$ such that no matter which value we take for $\delta>0$, there is a point $\bfx_\delta$ in $D$ with $d_k(\bfx_\delta,\bfc)<\delta$ and $d_m(f(\bfx_\delta),f(\bfc))\geq \varepsilon_0$. 

Much as in the proof of Theorem \ref{thm:exercise0}, we can use the previous statement to construct a suitable sequence.  For each positive integer $n$, there must be a point $\bfx_n$ in $D$ where 
$d_k(\bfx_n,\bfc)<1/n$ and $d_m(f(\bfx_n),f(\bfc))\geq \varepsilon_0$. (Note that $1/n$ plays the role of $\delta$ here). Hence, $(\bfx_n)$ is a sequence of points in $D$ where $\lim_{n\to\infty}\bfx_n=\bfc$ but $\lim_{n\to\infty}f(\bfx_n)\neq f(\bfc)$. Therefore, $f$ is not sequentially continuous at $\bfc$. (See Definition \ref{def:sequentialcontinuity}.)

\underline{(iii)$\implies$(i)}: Finally, suppose $f$ does not preserve closeness at a point $\bfc$ in $D$. Then there must be some $E\subseteq D$ where $\bfc \acl E$ but $f(\bfc)\awf f(E)$. So, there is some $\varepsilon_0>0$ such that for every point $\bfx$ in $E$ we have $d_m(f(\bfx),f(\bfc))\geq\varepsilon_0$. 

Now, let $\delta>0$. Since $\bfc\acl E$, there is a point $\bfy_\delta$ in $E$ where we have both $d_k(\bfy_\delta,\bfc)<\delta$ and $d_m(f(\bfy_\delta),f(\bfc))\geq\varepsilon_0$. Therefore, $f$ is not continuous at $\bfc$. (See Definition \ref{def:continuity}.)
\end{proof} 

Now that we've established Theorem \ref{thm:continuityequivalence}, we can prove some results reminiscent of classic theorems on continuity through the lenses of the preservation of closeness, sequential continuity, and $\varepsilon$-$\delta$ continuity.

\begin{example}\label{eg:halvingpreservescloseness}
Let's finally get to something that seemed to be true back in Section \ref{sec:functionsandimages}: The function $f:[0,4]\to\R$ given by $f(x)=x/2$ in Example \ref{eg:linef} preserves closeness at $c=2$. 

Since $f$ here is a special case of the function $f$ in Lemma \ref{lem:linesarecontinuous} where $m=1/2$, $b=0$, and $c=2$, we can choose $\delta$  as it described in the scratch work and proof for Lemma \ref{lem:linesarecontinuous}. Therefore, $\delta=2\varepsilon$ is a suitable threshold for the continuity of the function given by $f(x)=x/2$ at $c=2$. So by the equivalence of continuity and the preservation of closeness (Theorem \ref{thm:continuityequivalence}), we have $f(x)=x/2$ preserves closeness at $c=2$.
\end{example}

The following theorem summarizes some key algebraic properties of continuous functions. These should all be familiar from calculus as well as the algebraic limit theorems for sequences (Theorems \ref{thm:algebraiclimitseuclidean} and \ref{thm:algebraiclimitsreal}). In fact, the proof of Theorem \ref{thm:algebracontinuity} follows from combining Theorems \ref{thm:algebraiclimitseuclidean} and \ref{thm:algebraiclimitsreal} with the equivalent forms of continuity in Theorem \ref{thm:continuityequivalence}

Additionally, and alternatively, the proof of each statement in Theorem \ref{thm:algebracontinuity} is a fundamental exercise in analysis for exploring $\varepsilon$-$\delta$ continuity.

\begin{theorem}\label{thm:algebracontinuity} Suppose $D\subseteq \R^k, \bfc\in D$, and $f,g:D\to\R^m$. If $f$ and $g$ are continuous at $\bfc$, then
\begin{itemize}
	\item[(i)] $\alpha f$ is continuous at $\bfc$ for every $\alpha\in\R$; $\quad$ and 
	\item[(ii)] $f+g$ is continuous at $\bfc$.
\end{itemize}
If we further assume $k=m=1$, then with $D\subseteq\R$ and $c\in D$ we have
\begin{itemize}
	\item[(iii)] $fg$ is continuous at $c$; $\quad$ and 
	\item[(iv)] $f/g$ is continuous at $c$, provided $g(c)\neq 0$.
\end{itemize}
\end{theorem}

As a first idea from calculus which be proven as a result of Theorem \ref{thm:algebracontinuity}, do you recall hearing that polynomials are continuous? 

\begin{definition}\label{def:polynomial}
A {\em polynomial} is a function $p:\R\to\R$ defined by 
\begin{align}
	p(x)&=a_kx^k+a_{k-1}x^{k-1}+\cdots+a_1x+a_0
\end{align}
where $a_k,a_{k-1},\ldots,a_1,a_0\in\R$ are constants called the {\em coefficients} of $p$.
\end{definition}

\begin{corollary}\label{cor:polynomialscontinuous}
Polynomials are continuous on the real line $\R$.
\end{corollary}

Many more properties follow from continuity, especially in terms of preserving nice properties in the domain as a continuous function transforms points, sequences, and sets into their images in the range. For instance, continuous functions preserve connectedess.

\begin{theorem}\label{thm:preserveconnectedness}
If $f:D\to\R^m$ is continuous and $D$ is connected, then the range $f(D)$ is connected. 
\end{theorem}

\begin{proof}
Suppose $f:D\to\R^m$ is continuous and $D$ is connected. By Theorem \ref{thm:continuityequivalence}, $f$ preserves closeness at every $\bfc$ in the domain $D$. 

Now, suppose $f(D)=A\cup B$ where $A$ and $B$ are nonempty. Let $f^{-1}(A)$ and $f^{-1}(B)$ be the sets of points in the domain $D$ that map into $A$ and $B$, respectively. That is, let
\begin{align}
	f^{-1}(A)&=\{\bfx\in D:f(\bfx)\in A\}
	\quad\textnormal{and}\\
	f^{-1}(B)&=\{\bfx\in D:f(\bfx)\in B\}.
\end{align}
Since $f(D)=A\cup B$ and both $A$ and $B$ are nonempty, we have 
\begin{align}
	D=f^{-1}(A)\cup f^{-1}(B)
\end{align}
where both $f^{-1}(A)$ and $f^{-1}(B)$ are nonempty. Since $D$ is connected, by Definition \ref{def:connected} and without loss of generality,  there is a point $\bfx\in f^{-1}(A)$ where $\bfx\acl{(f^{-1}(B))}$. By the definition of $f^{-1}(A)$, we have $f(\bfx)\in A$. By the preservation of closeness (Definition \ref{def:preservecloseness}) and the definition of $f^{-1}(B)$, we have 
\begin{align}
	\bfx\acl{(f^{-1}(B))} \qquad\implies\qquad 
	f(\bfx)\acl{B}.
\end{align}

Hence, by the definition of connected (Definition \ref{def:connected}), the range $f(D)$ is connected.
\end{proof}

A corollary of Theorem \ref{thm:preserveconnectedness} is the {\em Intermediate Value Theorem}. (Cf. \cite[Theorem 4.5.1, p.136]{Abbott}.) Its proof is left as an exercise.

\begin{corollary}[Intermediate Value Theorem]\label{cor:intermediatevaluetheorem}
Suppose $f:[a,b]\to\R$ preserves closeness. If $\ell$ is a real number satisfying $f(a)<\ell<f(b)$ or  $f(a)>\ell>f(b)$, then there is some $c\in (a,b)$ such that $f(c)=\ell$.
\end{corollary}

Continuous functions preserve compactness, too.

\begin{theorem}\label{thm:preservecompactness}
If $f:K\to\R^m$ is continuous and $K$ is compact, then the range $f(K)$ is compact. 
\end{theorem}

\begin{proof}
Suppose $f:K\to\R^m$ is continuous and $K$ is compact. By Theorem \ref{thm:continuityequivalence}, $f$ is sequentially continuous (Definition \ref{def:sequentialcontinuity}). By the Heine-Borel Theorem \ref{thm:heineborel}, $K$ is sequentially compact (Definition \ref{def:sequentialcontinuity}).

So, let $(\bfy_n)$ be a sequence of points in the range $f(K)$. Then for each index $n\in\N$, there is a point $\bfx_n$ in the domain $K$ where $f(\bfx_n)=\bfy_n$ and so $(\bfy_n)=(f(\bfx_n))$. Since $K$ is sequentially compact, there is a convergent subsequence $(\bfx_{n_k})$ of the sequence $(\bfx_n)$ where 
\begin{align}
	\lim_{k\to\infty}\bfx_{n_k}=\bfc \quad\textnormal{and}\quad \bfc\in K.
\end{align}
Since $f$ is sequentially continuous at every point in the domain $K$, we have 
\begin{align}
	\lim_{k\to\infty}\bfx_{n_k}=\bfc\in K \quad\implies\quad \lim_{k\to\infty}f(\bfx_{n_k})=f(\bfc)\in f(K).
\end{align}
Therefore, $(f(\bfx_{n_k}))$ is a subsequence of $(\bfy_n)=(f(\bfx_n))$ that converges to $f(\bfc)$, and $f(\bfc)$ is in $f(K)$. Therefore, $f(K)$ is sequentially compact and also compact by the Heine-Borel Theorem \ref{thm:heineborel}.
\end{proof}

Theorem \ref{thm:preservecompactness} yields the Extreme Value Theorem as a corollary.

\begin{corollary}[Extreme Value Theorem]\label{cor:extremevaluetheorem} If $K$ is a  compact subset of $\R^k$ and $f:K\to\R$ is continuous, then the range $f(K)$ attains its minimum and maximum values.
\end{corollary}

\begin{proof}
Suppose $f:K\to\R$ is continuous and $K$ is compact. By Theorem \ref{thm:preservecompactness}, $f(K)$ is compact as well. By the  Heine-Borel Theorem \ref{thm:heineborel} (or just Lemmas \ref{lem:compactimpliesbounded} and \ref{lem:compactimpliesclosed}), $f(K)$ is closed and bounded. As a bounded set of real numbers, both $\sup{f(K)}$ and $\inf{f(K)}$ exist. By the definitions of supremum and infimum (Definition \ref{def:supremumacl}), we have both
\begin{align}
	\sup{f(K)}\acl f(K)\quad\textnormal{and}\quad
	\inf{f(K)}\acl f(K).
\end{align}
Since $f(K)$ is closed, by Definition \ref{def:closureclosed} we have both
\begin{align}
	\sup{f(K)}\in f(K)\quad\textnormal{and}\quad
	\inf{f(K)}\in f(K).
\end{align}
Therefore, both $\max{f(K)}$ and $\min{f(K)}$ exist.
\end{proof}

The contrapositions of the implications in Theorem \ref{thm:continuityequivalence} are also useful for showing a function is discontinuous at $\bfc$ or fails to preserve closeness at $\bfc$.

\begin{corollary}[Discontinuity Criteria]\label{cor:discontinuitycriteria}
Suppose $D\subseteq \R^k, f:D\to\R^m$, and $\bfc\in D$. 
\begin{enumerate}
	\item If there is a set $E\subseteq D$ where $\bfc\acl{E}$ but $f(\bfc)\awf{f(E)}$, then $f$ is discontinuous at $\bfc$.
	\item If there is a sequence $(\bfx_n)$ of points in $D$ and a point $\bfc$ in $D$ where $\lim_{n\to\infty}\bfx_n=\bfc$ but $\lim_{n\to\infty}f(\bfx_n)\neq f(\bfc)$, then $f$ is discontinuous at $\bfc$ and $f$ does not preserve closeness at $\bfc$.
	\item If there is a set $E\subseteq D$ where $E$ is connected but $f(E)$ is disconnected, then $f$ is discontinuous.
	\item If there is a set $K\subseteq D$ where $K$ is compact but $f(E)$ is not compact, then $f$ is discontinuous.
\end{enumerate}
\end{corollary}

The following section explores the continuity of linear transformations between Euclidean spaces. In particular, linear transformations provide a nice setting for considering how continous functions behave in higher dimensions.

\vs
\section*{Exercises}
\setcounter{theorem}{0}

Exercises are for play: Do scratch work, draw stuff, and make mistakes---make {\em lots} of mistakes---before worrying about writing proofs. {\em Have fun!}

\xca Given a function $f:\R\to\R$, a {\em level set} is set of the form
 \begin{align}
 	L_k&=\{x\in\R:f(x)=k\}
\end{align} 
for some fixed real number $k$. Prove that if $f$ is continuous, then its level sets are closed.

\xca This exercise is about proving all polynomials are continuous on $\R$ using the $\varepsilon$-$\delta$ definition of continuity (Definition \ref{def:continuity}). This is not strictly necessary since we can leverage Lemma \ref{lem:linesarecontinuous} for parts \textbf{(a)} and \textbf{(b)} while rest stems from the equivalent forms of continuity provided by Theorem \ref{thm:continuityequivalence} and other results. They make for an interesting string of exercises nonetheless. 
\begin{itemize}
	\item[{\bf (a)}] Prove that if $f:\R\to\R$ is constant (so $f(x)=k$ for some $k\in\R$), then $f$ is continuous at every $c\in\R$.
	\item[{\bf (b)}] Prove that if $g:\R\to\R$ is given by $g(x)=mx+b$ where $m,b\in\R$ with $m\neq 0$, then $g$ is continuous at every $c\in\R$.	 
	\item[{\bf (c)}] Prove sums of continuous functions from $\R$ to $\R$ are continuous.
	\item[{\bf (d)}] Prove that if $n\in\N$ with $n\geq 2$ and $h:\R\to\R$ with $h(x)=x^n$, then $h$ is continuous at every $c\in\R$. Be sure to show both scratch work and the proof. HINTS: $c=0$ can be a special case, and $a^n-b^n=(a-b)(a^{n-1}+a^{n-2}b+\cdots+ab^{n-2}+b^{n-1})$.
	\item[{\bf (e)}] Use an induction argument along with the results of parts \textbf{(b)}, \textbf{(c)}, and \textbf{(d)} to prove polynomials are continuous on $\R$. Keep in mind that an arbitrary polynomial is given by
	\[
	p(x)=a_0+a_1x+a_2x^2+\cdots +a_nx^n.
	\]
\end{itemize}

\vs
\section{Segue into linear algebra}
\label{sec:segueintolinearalgebra}

What do continuous functions look like in higher dimensions? The graph of a function between Euclidean spaces---which combines the domain and range into a single figure---is generally not easy to construct. So instead of graphs, it will help to consider domains and ranges separately.

The context of {\em linear transformations} from one Euclidean space to another given by {\em matrix-vector multiplication} provides an especially nice setting to explore the behavior of some continuous functions whose graphs are difficult or seemingly impossible to envision.

In this section I assume you have taken a course in linear algebra and are familiar with topics and terminology such as linear transformations, matrix-vector multiplication, and eigenvalues. So, if this doesn't describe you, you may want to skip this section. 

\begin{example}\label{eg:basiclineartransformation}
For starters, consider the $2\times 2$ matrix 
\begin{align}
M=
\left[
	\begin{array}{rc}
   	3 & 0 \\
 	0 & 1/2
	\end{array}
\right]
\end{align}
and the function $t:\R^2\to\R^2$ given by 
\begin{align}
t(\bfx)=M\bfx, \quad\textnormal{where}\quad
\bfx=\left[
	\begin{array}{c}
   	x_1 \\
 	x_2 
	\end{array}
	\right]
\end{align}
and $M\bfx$ is matrix-vector multiplication defined by 
\begin{align}
M\bfx &=
\left[
	\begin{array}{rc}
   	3 & 0 \\
 	0 & 1/2
	\end{array}
\right]
\left[
	\begin{array}{r}
   	x_1 \\
 	x_2 
	\end{array}
\right]
=x_1\left[
	\begin{array}{r}
   	3 \\
 	0 
	\end{array}
\right]+
x_2\left[
	\begin{array}{c}
   	0 \\
 	1/2 
	\end{array}
\right]
=\left[
	\begin{array}{c}
   	3x_1 \\
 	x_2/2 
	\end{array}
\right].
\end{align}

Thus, $t$ acts on every vector (so every point) in $\R^2$ by tripling the first coordinate and cutting the second coordinate in half. In Figure \ref{fig:neighborhoodandimage}, $t$ transforms $V_1(\mathbf{0})$ (the open disk centered at the origin 
$\mathbf{0}$  with radius $1$) into $t(V_1(\mathbf{0}))$ by stretching the disk horizontally by a scale of 3 and shrinking the disk vertically by a scale of $1/2$. As a result, the image $t(V_1(\mathbf{0}))$ is an ellipse with a semimajor axis of length 3 and semiminor axis of length $1/2$. In fact, $t$ happens to transform every disk to an ellipse, but proving this is beyond the scope of this section.
\end{example}

\begin{figure}
\centering
\begin{tikzpicture}
\draw[dashed,blue,fill=blue!15] (-1,0) ellipse (1cm and 1cm);
	\draw (-1,0) node {$\bullet$};		
	\draw (-1,-0.5) node {$\mathbf{0}$};
	\draw (-1,-1.5) node {$V_1(\mathbf{0})$};
	\draw (1,0) node {\Large{$\mapsto$}};
\draw[dashed,blue,fill=blue!15] (5,0) ellipse (3cm and 0.5cm);
	\draw (5,0) node {$\bullet$};
	\draw (6,0) node {$t(\mathbf{0})=\mathbf{0}$};
	\draw (5,-1.5) node {$t(V_1(\mathbf{0}))$};
\end{tikzpicture}
\caption{The set $V_1(\mathbf{0})$ in the domain is the $1$-neighborhood of the origin. In other words, $V_1(\mathbf{0})$ is the open unit disk. Its image  $t(V_1(\mathbf{0}))$ under the linear transformation $t$ in Example \ref{eg:basiclineartransformation} is an ellipse. Note that this figure is not the graph of $t$!}
\label{fig:neighborhoodandimage}
\end{figure}

In general, functions like $t$ defined by matrix-vector multiplication are {\em linear transformations} (see Definition \ref{def:lineartransformation}). Loosely speaking, a function is a linear transformation if it distributes across addition and scalars factor out. All linear transformations scale, stretch, or shrink input vectors in ways that can be quantified and controlled to our advantage.

\begin{definition}\label{def:lineartransformation}
A function $f:\R^k\to\R^m$ is a {\em linear transformation} if for every $\bfx,\bfy\in\R^k$ and every $\alpha\in\R$ (that is, any {\em scalar}), we have
\begin{align}
	f(\bfx+\bfy)&=f(\bfx)+f(\bfy)\qquad \textnormal{and}\\
	f(\alpha\bfx)&=\alpha f(\bfx).
\end{align}
\end{definition}

A theorem from linear algebra tells us linear transformations from one Euclidean space to another are completely determined by matrix-vector multiplication (Theorem \ref{thm:lineartransformationmatrix}). Its proof makes use of the {\em vector space} structures inherent to Euclidean spaces and is omitted.

\begin{theorem}\label{thm:lineartransformationmatrix}
For every linear transformation $f:\R^k\to\R^m$, there is a $m\times k$ matrix $A$ with real-valued entries where $f(\bfx)=A\bfx$. Also, for every $m\times k$ matrix $B$ with real-valued entries, the function $g:\R^k\to\R^m$ defined by $g(\bfx)=B\bfx$ is a linear transformation.
\end{theorem}

Note how $t$ preserves the connectedness of the unit disk $V_1(\mathbf{0})$ by transforming it into the ellipse $t(V_1(\mathbf{0}))$, which is connected as well. In fact, much more is true.

\begin{theorem}\label{thm:lineartransformationscontinuous}
Suppose $f:\R^k\to\R^m$ is a linear transformation defined by $f(\bfx)=A\bfx$ where $A$ is a $m\times k$ matrix with real-valued entries. Then $f$ is continuous.
\end{theorem}

The remainder of this section is dedicated to building up a proof of the continuity of linear transformations from one Euclidean space to another (Theorem \ref{thm:lineartransformationscontinuous}) by  taking advantage of linear algebra. Feel free to skip ahead if you would to see the proof right away.

\begin{scratch}
If we try to prove Theorem \ref{thm:lineartransformationscontinuous} by directly showing $f$ is continuous via Definition \ref{def:continuity}, the scratch work would entail finding a suitable threshold $\delta>0$ in terms of an error $\varepsilon>0$ where 
\begin{align}
	\|\bfx-\bfc\|_m <\delta \implies \|f(\bfx)-f(\bfc)\|_m=\|A\bfx-A\bfc\|_m<\varepsilon.
\end{align}

As a linear transformation, the linearity of the function $f$ gets us part of the way to defining a suitable threshold $\delta$.  By Theorem \ref{thm:lineartransformationmatrix} we have
\begin{align}\label{eqn:lineartransformationscratch}
	\|f(\bfx)-f(\bfc)\|_m=\|A\bfx-A\bfc\|_m=\|A(\bfx-\bfc)\|_m.
\end{align}
So, understanding how $A$ acts on $\bfx-\bfc$ might get us somewhere. Let's pause the scratch work for a moment to get a better handle on the linear algebra.
\end{scratch}

To set the stage, for points (or {\em vectors}) $\bfx,\bfc\in\R^k$ we have
\begin{align}
\bfx
&=\left[
	\begin{array}{c}
	x_1\\
	x_2\\
	\vdots\\
	x_k
\end{array}
\right]
\quad \textnormal{and}\quad
\bfc
=\left[
	\begin{array}{c}
	c_1\\
	c_2\\
	\vdots\\
	c_k
\end{array}
\right].
\end{align}
We also have the {\em standard basis vectors} of $\R^k$ denoted by $\bfe_1,\bfe_2,\ldots,\bfe_k$ where
\begin{align}
\bfe_1
=\left[
	\begin{array}{c}
	1\\
	0\\
	\vdots\\
	0
\end{array}
\right],\quad
\bfe_2
=\left[
	\begin{array}{c}
	0\\
	1\\
	\vdots\\
	0
\end{array}
\right],
\quad \textnormal{and}\quad
\bfe_k
=\left[
	\begin{array}{c}
	0\\
	0\\
	\vdots\\
	1
\end{array}
\right].
\end{align}
That is, for each $j=1,\ldots,k$, the $j$th {\em standard basis vector} $\bfe_j$ has 1 for its $j$th component and 0 otherwise. 

Since Euclidean spaces are vector spaces where addition and scalar multiplication are both defined work well together, we immediately have the following fact: For every $\bfx\in\R^k$ we have
\begin{align}
\bfx
&=\left[
	\begin{array}{c}
	x_1\\
	x_2\\
	\vdots\\
	x_k
\end{array}
\right]
=x_1\left[
	\begin{array}{c}
	1\\
	0\\
	\vdots\\
	0
\end{array}
\right]+x_2
=\left[
	\begin{array}{c}
	0\\
	1\\
	\vdots\\
	0
\end{array}
\right]
+\cdots+
x_k
\left[
	\begin{array}{c}
	0\\
	0\\
	\vdots\\
	1
\end{array}
\right]
=\sum_{j=1}^k x_j\bfe_j.
\end{align}
Therefore, we also have 
\begin{align}\label{eqn:vectordifference}
\bfx-\bfc
&=\left[
	\begin{array}{c}
	x_1-c_1\\
	x_2-c_2\\
	\vdots\\
	x_k-c_k
\end{array}
\right]
=\sum_{j=1}^k(x_j-c_j)\bfe_j.
\end{align}

Now, let's get back to some scratch work for proving Theorem \ref{thm:lineartransformationscontinuous}.

\begin{scratch}
Applying the linearity of matrix-vector multiplication described in Theorem \ref{thm:lineartransformationmatrix} multiple times allows us to merge portions of equations \eqref{eqn:lineartransformationscratch} and \eqref{eqn:vectordifference} to get
\begin{align}\label{eqn:applyinglinearity}
	f(\bfx)-f(\bfc)&=A\bfx-A\bfc=A(\bfx-\bfc)=\sum_{j=1}^k(x_j-c_j)A\bfe_j.
\end{align}
Taking norms and applying various portions of Corollary \ref{cor:distancesaremetrics}, including the triangle inequality \eqref{eqn:triangleinequality3}, gives us 
\begin{align}
	\|f(\bfx)-f(\bfc)\|_m
	&=\|A\bfx-A\bfc\|_m\\
	&=\|A(\bfx-\bfc)\|_m\\
	&=\left\|\sum_{j=1}^k(x_j-c_j)A\bfe_j\right\|_m\\
	&\leq\sum_{j=1}^k|x_j-c_j|\|A\bfe_j\|_m.
\end{align}
Now what? Let's pause the scratch work one more time. 
\end{scratch}

What are the $A\bfe_j$ and their norms $\|A\bfe_j\|_m$ for $j=1,\ldots,k$? These objects are actually visual in nature and stem from some nifty perspectives on linear algebra. To see them in action, let's revisit the $2\times 2$ matrix $M$ from Example \ref{eg:basiclineartransformation}.

\begin{example}\label{eg:basiclineartransformationrevisit}
Once again, consider
\begin{align}
M=
\left[
	\begin{array}{rc}
   	3 & 0 \\
 	0 & 1/2
	\end{array}
\right]
\end{align}
The transformation $t(\bfx)=M\bfx$ acts on the standard basis vectors for $\R^2$ by identifying the corresponding column of $M$. See Figure \ref{fig:neighborhoodandimage2} and note:
\begin{align}
M\bfe_1 
&=
1\left[
	\begin{array}{r}
   	3 \\
 	0 
	\end{array}
\right]+
0\left[
	\begin{array}{c}
   	0 \\
 	1/2 
	\end{array}
\right]
=\left[
	\begin{array}{c}
   	3 \\
 	0 
	\end{array}
\right]\quad\textnormal{and}\\
M\bfe_2 
&=
0\left[
	\begin{array}{r}
   	3 \\
 	0 
	\end{array}
\right]+
1\left[
	\begin{array}{c}
   	0 \\
 	1/2 
	\end{array}
\right]
=\left[
	\begin{array}{c}
   	0 \\
 	1/2
	\end{array}
\right].
\end{align}
Actually, in more general settings, multiplication of an $m\times k$ matrix $A$ with a standard basis vector $\bfe_j$ amounts to identifying the product $A\bfe_j$ as the $j$th column of A! 

Furthermore,
\begin{align}
	\|M\bfe_1\|_2&=\sqrt{3^2+0^2}=3\quad \textnormal{and}\\
	\|M\bfe_2\|_2&=\sqrt{0^2+(1/2)^2}=1/2.
\end{align} 
So, in general, we can identify the largest $\|A\bfe_j\|_m$ {\em just by finding the largest column of} $A$.
\end{example}
 
\begin{figure}
\centering
\begin{tikzpicture}
\draw[dashed,blue,fill=blue!15] (-1,0) ellipse (1cm and 1cm);
	\draw[line width=1.5pt,red,-stealth](-1,0)--(-1,1);	
	\draw (-1.25,0.5) node {$\bfe_2$};	
	\draw[line width=1.5pt,-stealth](-1,0)--(0,0);
	\draw (-0.5,-0.25) node {$\bfe_1$};
	\draw (-1,-1.5) node {$V_1(\mathbf{0})$};
	\draw (1,0) node {\Large{$\mapsto$}};
\draw[dashed,blue,fill=blue!15] (5,0) ellipse (3cm and 0.5cm);
	\draw[line width=1.5pt,red,-stealth](5,0)--(5,0.5);	
	\draw (4.5,0.25) node {$t(\bfe_2)$};			
	\draw[line width=1.5pt,-stealth](5,0)--(8,0);
	\draw (6,-0.25) node {$t(\bfe_1)$};	
	\draw (5,-1.5) node {$t(V_1(\mathbf{0}))$};
\end{tikzpicture}
\caption{The linear transformation $t$ from Example \ref{eg:basiclineartransformation} transforms the basis vector $\bfe_1$ (in black) into its image $t(\bfe_1)=M\bfe_1$ by stretching by a factor of $3$. The basis vector $\bfe_2$ (in red) is shrunk by a factor of $1/2$.}
\label{fig:neighborhoodandimage2}
\end{figure}

Back to the scratch work for proving Theorem \ref{thm:lineartransformationscontinuous}.

\begin{scratch}
Just a reminder, so far we have 
\begin{align}
	\|f(\bfx)-f(\bfc)\|_m&=\|A(\bfx-\bfc)\|_m \leq\sum_{j=1}^k|x_j-c_j|\|A\bfe_j\|_m.
\end{align}
In light of Example \ref{eg:basiclineartransformationrevisit}, we can identify the largest norm among the images of the standard basis vectors by looking at the columns of $A$ and defining 
\begin{align}
	q&=\max\{\|A\bfe_j\|_m:j=1,\ldots,k\}.
\end{align}
So, this brings us to 
\begin{align}\label{eqn:lineartransformationwithlargestcolumn}
	\|f(\bfx)-f(\bfc)\|_m&\leq\sum_{j=1}^k|x_j-c_j|\|A\bfe_j\|_m \leq q\sum_{j=1}^k|x_j-c_j|.
\end{align}
Keep in mind, our goal is to end up with something like
\begin{align}
	\|f(\bfx)-f(\bfc)\|_m\leq q\sum_{j=1}^k|x_j -c_j| \leq\varepsilon.
\end{align}
So how big are the values of $|x_j-c_j|$?  Since the distance
\begin{align}
d_k(\bfx,\bfc)=\|\bfx-\bfc\|_k
\end{align} 
appears along with the threshold $\delta$ in the definition of continuity (Definition \ref{def:continuity}), it's sensible to wonder about the relationship between $\|\bfx-\bfc\|_k$ and the values of $|x_j-c_j|$.

Time for a little lemma.
\end{scratch}

\begin{lemma}\label{lem:standardmetricversustaxicab}
Suppose $\delta>0$ and let $\bfu\in\R^k$ where
\begin{align}
\bfu
&=\left[
	\begin{array}{c}
	u_1\\
	u_2\\
	\vdots\\
	u_k
\end{array}
\right].
\end{align}
Then $\|\bfu\|_k<\delta$ implies $|u_j|<\delta$ for all $j=1,\ldots,k$. 
\end{lemma}

Equivalently, Lemma \ref{lem:standardmetricversustaxicab} can be visually interpreted to say that every $\delta$-neighborhood of the origin in $\R^k$ is contained in the open $k$-cube of side-length $2\delta$ centered at the origin. See Figure \ref{fig:diskinabox}.

\begin{figure}
\centering
\begin{tikzpicture}
\draw[dashed,black,fill=black!10] (-2,-2) rectangle (2,2);	
\draw[dashed,blue,fill=blue!15] (0,0) ellipse (2cm and 2cm);
	\draw (0,0) node {\textbullet};
	\draw (0,-0.4) node {$\mathbf{0}$};
	\draw (0,-1.5) node {$V_\delta(\mathbf{0})$};	
	\draw[line width=1.5pt,-stealth](0,0)--(1,1.73);
	\draw (0.65,0.6) node {$\bfu$};
	\draw[dashed,black] (0,0) rectangle (1,1.73);
	\draw (0.7,-0.3) node {$u_1$};
	\draw (-0.3,0.86) node {$u_2$};		
	\draw[blue] (0,0) -- (-2,0);
	\draw[blue] (-1,-0.4) node {$\delta$};
\end{tikzpicture}
\caption{In the plane $\R^2$, any point $\bfu$ within a positive distance $\delta$ of the origin $\mathbf{0}$ has entries $u_1$ and $u_2$ whose lengths (absolute values) are strictly less than $\delta$. See Lemma \ref{lem:standardmetricversustaxicab}. Similarly, every $\delta$-neighborhood in $\R^2$ is contained in an open square of side-length $2\delta$ with the same center.}
\label{fig:diskinabox}
\end{figure}

A proof of Lemma \ref{lem:standardmetricversustaxicab} follows quickly from contraposition. 

\begin{proof}
Suppose $\delta>0,\bfu\in\R^k$, and $|u_{j_0}|\geq \delta$ for some $j_0\in\{1,\ldots,k\}$. Since the square root function is increasing, we have 
\begin{align}
	\delta\leq |u_{j_0}|=\sqrt{u_{j_0}^2}\leq \sqrt{u_1^2+u_2^2+\cdots+u_k^2}=\|\bfu\|_k.
\end{align}
\end{proof}

One last time, let's get back to scratch work for proving Theorem \ref{thm:lineartransformationscontinuous}.

\begin{scratch}
Repeating line \eqref{eqn:lineartransformationwithlargestcolumn}, so far we have
\begin{align}\label{eqn:lineartransformationwithlargestcolumnrepeat}
	\|f(\bfx)-f(\bfc)\|_m&\leq\sum_{j=1}^k|x_j-c_j|\|A\bfe_j\|_m \leq q\sum_{j=1}^k|x_j-c_j|
\end{align}
where $q=\max\{\|A\bfe_j\|_m:j=1,\ldots,k\}$.
Substituting $\bfx-\bfc$ for $\bfu$ in Lemma \ref{lem:standardmetricversustaxicab} yields
\begin{align}\label{eqn:diskinaboxdelta}
	\|\bfx-\bfc\|_k<\delta \implies |x_j-c_j|<\delta \textnormal{ for all } j=1,\ldots,k.
\end{align}
Therefore, by \eqref{eqn:lineartransformationwithlargestcolumnrepeat} we have
\begin{align}\label{eqn:lineartransformationwithdelta}
	\|f(\bfx)-f(\bfc)\|_m&\leq\sum_{j=1}^k|x_j-c_j|\|A\bfe_j\|_m \leq q\sum_{j=1}^k|x_j-c_j| <qk\delta.
\end{align}
So, it seems like taking
\begin{align}
	\delta=\frac{\varepsilon}{qk} 
\end{align}
may prove to be a good choice for a threshold $\delta$, as long as $q$ is not zero.
\end{scratch}

Finally, let's prove linear transformations between Euclidean spaces are continuous (Theorem \ref{thm:lineartransformationscontinuous}).

\begin{proof}
Suppose $f:\R^k\to\R^m$ is defined by $f(\bfx)=A\bfx$ where $A$ is a $m\times k$ matrix with real-valued entries.  Let $\varepsilon>0$ and define $q$ to be the length of the largest column vector in $A$. That is,
\begin{align}
	q&=\max\{\|A\bfe_j\|_m:j=1,\ldots,k\}.
\end{align}
So by the linearity of the function $f$ (Theorem \ref{thm:lineartransformationmatrix} tells us $f$ is a linear transformation) and the repeated application of various portions of Corollary \ref{cor:distancesaremetrics}, including the triangle inequality \eqref{eqn:triangleinequality3}, for every $\bfx$ and every $\bfc$ in $\R^k$ we have
\begin{align}
	\|f(\bfx)-f(\bfc)\|_m
	&=\|A\bfx-A\bfc\|_m\\
	&=\|A(\bfx-\bfc)\|_m\\
	&=\left\|\sum_{j=1}^k(x_j-c_j)A\bfe_j\right\|_m\\
	&\leq\sum_{j=1}^k|x_j-c_j|\|A\bfe_j\|_m\\
	&\leq q\sum_{j=1}^k|x_j-c_j|.\label{eqn:finallineartransformationcontinuous}
\end{align}
From here, there are two cases to consider.

\underline{Case (i)}: Suppose $q=0$, which means $A$ is the zero matrix and $f$ is the constant function that sends every vector in $\R^k$ to the origin $\mathbf{0}$ in $\R^m$. Choose $\delta$ to be any positive real number, say $\delta=3140>0$. Then for every $\bfx$ and every $\bfc$ in $\R^k$ we have
\begin{align}
	\|f(\bfx)-f(\bfc)\|_m \leq q\sum_{j=1}^k|x_j-c_j|=0<\varepsilon.
\end{align}
Therefore, $f=\mathbf{0}$ is continuous at every $\bfc$ in $\R^k$.

\underline{Case (ii)}: Suppose $q>0$ and define 
\begin{align}
	\delta=\frac{\varepsilon}{qk}.
\end{align}
(Note $\delta>0$.) Fix $\bfc$ in $\R^k$. Then for every $\bfx$ in $\R^k$ where
\begin{align}
	\|\bfx-\bfc\|_k<\delta=\frac{\varepsilon}{qk},
\end{align}
we have by Lemma \ref{lem:standardmetricversustaxicab} that
\begin{align}
	|x_j-c_j|<\delta \quad\textnormal{for all}\quad j=1,\ldots,k. 
\end{align}
Hence, referring back to \eqref{eqn:finallineartransformationcontinuous}, we have $\|\bfx-\bfc\|_k<\delta$ implies
\begin{align}
	\|f(\bfx)-f(\bfc)\|_m \leq q\sum_{j=1}^k|x_j-c_j|<qk\delta=qk\left(\frac{\varepsilon}{qk}\right)=\varepsilon.
\end{align}
Therefore, $f$ is continuous at every $\bfc$ in $\R^k$.
\end{proof}

\vs
\section*{Exercises}
\setcounter{theorem}{0}

Exercises are for play: Do scratch work, draw stuff, and make mistakes---make {\em lots} of mistakes---before worrying about writing proofs. {\em Have fun!}



%
%
%


\bibliographystyle{amsplain}


\end{document}